\documentclass[12pt,british,openany]{book}

\usepackage{babel}
\usepackage{amsmath,amsfonts,amssymb,amsthm}
\usepackage{ucs} 
\usepackage[utf8x]{inputenc} 
\usepackage[T1]{fontenc} 
\usepackage{enumerate}
\usepackage{textcomp}
\usepackage[obeyspaces,lowtilde]{url}
\usepackage{tikz}
\usetikzlibrary{calc}

\usepackage{eucal}

\usepackage[noBBpl]{mathpazo}
\usepackage{mathabx}
\changenotsign

\usepackage[matrix,arrow,curve]{xy}

\usepackage{chngcntr} 
\counterwithin{section}{part}

\counterwithout{equation}{chapter}
\counterwithout{figure}{chapter}

\makeatletter

\renewcommand{\p@enumii}{}
\renewcommand{\p@enumiii}{}
\def\@enum@{\list{\csname label\@enumctr\endcsname}%
          {\usecounter{\@enumctr}\def\makelabel##1{
\normalfont\ignorespaces\emph{{##1}~}}
\setlength{\labelsep}{3pt}
\setlength{\parsep}{0pt}
\setlength{\itemsep}{0pt}
\setlength{\leftmargin}{0pt}
\setlength{\labelwidth}{0pt}
\setlength{\listparindent}{\parindent}
\setlength{\itemsep}{0pt}
\setlength{\itemindent}{0pt}
\topsep=3pt plus 1pt minus 1 pt}}

\makeatother

\renewcommand{\epsilon}{\ensuremath{\varepsilon}}
\renewcommand{\phi}{\ensuremath{\varphi}}

\renewcommand{\to}{\ensuremath{\longrightarrow}}
\renewcommand{\mapsto}{\ensuremath{\longmapsto}}

\newcommand{\N}{\ensuremath{\mathbb N}}
\newcommand{\Z}{\ensuremath{\mathbb Z}}
\newcommand{\dt}{\ensuremath{\mathbb D}^{2}}
\newcommand{\St}[1][2]{\ensuremath{\mathbb S}^{#1}}
\newcommand{\FF}{\ensuremath{\mathbb F}}
\newcommand{\F}[1][n]{\ensuremath{\FF_{{#1}}}}
\newcommand{\rp}{\ensuremath{\mathbb{R}P^2}}

\newcommand{\sn}[1][n]{\ensuremath{S_{{#1}}}}
\newcommand{\an}[1][n]{\ensuremath{A_{{#1}}}}

\DeclareRobustCommand*{\up}[1]{\textsuperscript{#1}}
\renewcommand{\th}{\up{th}}
\newcommand{\ft}[1][n]{\ensuremath{{\Delta_{#1}^2}}}
\newcommand{\garside}[1][n]{\ensuremath{\Delta_{#1}}}

\newcommand{\fti}[1][n]{\ensuremath{\Delta_{#1}^{-2}}}
\newcommand{\ftalt}[2][n]{\ensuremath{\Delta_{#1}^{2#2}}}

\renewcommand{\ker}[1]{\ensuremath{\operatorname{\text{Ker}}\left({#1}\right)}}

\newcommand{\aut}[2][]{\ensuremath{\operatorname{\text{Aut}}_{#1}\left({#2}\right)}}
\newcommand{\out}[1]{\ensuremath{\operatorname{\text{Out}}\left({#1}\right)}}

\newcommand{\id}{\ensuremath{\operatorname{\text{Id}}}}

\newcommand{\dih}[1]{\ensuremath{\operatorname{\text{Dih}}_{#1}}}
\newcommand{\dic}[1]{\ensuremath{\operatorname{\text{Dic}}_{#1}}}
\newcommand{\quat}[1][8]{\ensuremath{\mathcal{Q}_{#1}}}

\newcommand{\homeo}[1][X]{\ensuremath{\operatorname{\text{Homeo}^+}(\St,#1)}}
\newcommand{\mcg}[1][n]{\ensuremath{\operatorname{\mathcal{MCG}}(\St,#1)}}
\newcommand{\pmcg}[1][n]{\ensuremath{\operatorname{\mathcal{P\!MCG}}(\St,#1)}}

\newcommand{\tonestar}{\ensuremath{T^{\ast}}}
\newcommand{\oonestar}{\ensuremath{O^{\ast}}}
\newcommand{\istar}{\ensuremath{I^{\ast}}}

\makeatletter
\def\@map#1#2[#3]{\mbox{$#1 \colon\thinspace #2 \to #3$}}
\def\map#1#2{\@ifnextchar [{\@map{#1}{#2}}{\@map{#1}{#2}[#2]}}
\makeatother

\newcommand{\brak}[1]{\ensuremath{\left\{ #1 \right\}}}
\newcommand{\ang}[1]{\ensuremath{\left\langle #1\right\rangle}}
\newcommand{\set}[2]{\ensuremath{\left\{#1 \,\mid\, #2\right\}}}

\newcommand{\setangr}[2]{\ensuremath{\ang{#1 \,\left\lvert \, #2 \right.}}}

\newcommand{\ord}[1]{\ensuremath{\left\lvert #1\right\rvert}}

\newcommand{\setr}[2]{\ensuremath{\brak{#1 \,\left\lvert \, #2 \right.}}}
\newcommand{\setl}[2]{\ensuremath{\brak{\left. #1 \,\right\rvert \, #2}}}

\newtheoremstyle{theoremm}{}{}{\itshape}{}{\scshape}{.}{ }{}

\theoremstyle{theoremm}
\newtheorem{thm}{Theorem}
\newtheorem{lem}[thm]{Lemma}
\newtheorem{prop}[thm]{Proposition}
\newtheorem{cor}[thm]{Corollary}

\newtheoremstyle{remarkk}{}{}{}{}{\scshape}{.}{ }{}

\theoremstyle{remarkk}
\newtheorem{defn}[thm]{Definition}

\newtheorem{rem}[thm]{Remark}
\newtheorem{rems}[thm]{Remarks}

\newcommand{\inn}[1]{\ensuremath{\operatorname{\text{Inn}}\left({#1}\right)}}
\newcommand{\reth}[1]{Theorem~\protect\ref{th:#1}}
\newcommand{\relem}[1]{Lemma~\protect\ref{lem:#1}}
\newcommand{\repr}[1]{Proposition~\protect\ref{prop:#1}}
\newcommand{\repart}[1]{Part~\protect\ref{part:#1}}
\newcommand{\reco}[1]{Corollary~\protect\ref{cor:#1}}
\newcommand{\resec}[1]{Section~\protect\ref{sec:#1}}
\newcommand{\rerem}[1]{Remark~\protect\ref{rem:#1}}
\newcommand{\rerems}[1]{Remarks~\protect\ref{rem:#1}}
\newcommand{\redef}[1]{Definition~\protect\ref{def:#1}}

\newcommand{\req}[1]{equation~(\protect\ref{eq:#1})}
\newcommand{\reqref}[1]{(\protect\ref{eq:#1})}

\let\newbigast=\bigast
\renewcommand{\bigast}{\mathbin{\newbigast}}

\newcommand{\resecglobal}[2]{Section~\protect\ref{part:#1}.\ref{sec:#2}}

\AtBeginDocument{%
}

\begin{document}

\frontmatter

\title{The classification of the virtually cyclic subgroups\\ of the sphere
braid groups}

\author{DACIBERG~LIMA~GON\c{C}ALVES\\
Departamento de Matem\'atica - IME-USP,\\
Caixa Postal~66281~-~Ag.~Cidade de S\~ao Paulo,\\ 
CEP:~05314-970 - S\~ao Paulo - SP - Brazil.\\
e-mail:~\url{dlgoncal@ime.usp.br}\vspace*{4mm}\\
JOHN~GUASCHI\\
CNRS, Laboratoire de Math\'ematiques Nicolas Oresme
UMR~6139,\\ Université de Caen Basse-Normandie BP 5186,\\
14032 Caen Cedex, France,\vspace*{1.5mm}\\
and\vspace*{1.5mm}\\
Université de Caen Basse-Normandie,\\ Laboratoire de Mathématiques Nicolas Oresme UMR~CNRS~6139,\\ BP 5186, 14032 Caen Cedex, France.\\
%
%
e-mail:~\url{john.guaschi@unicaen.fr}}

\date{27th October 2011}

\maketitle

\newenvironment{abstract}{\chapter*{Abstract}}{} 

\begin{abstract}
\emph{Let $n\geq 4$, and let $B_{n}(\St)$ denote the $n$-string braid group of the sphere. In~\cite{GG7}, we showed that the isomorphism classes of the maximal finite subgroups of $B_{n}(\St)$ are comprised of cyclic, dicyclic (or generalised quaternion) and binary polyhedral groups. In this paper, we study the infinite virtually cyclic groups of $B_{n}(\St)$, which are in some sense, its `simplest' infinite subgroups. As well as helping to understand the structure of the group $B_{n}(\St)$, the knowledge of its virtually cyclic subgroups is a vital step in the calculation of the lower algebraic $K$-theory of the group ring of $B_{n}(\St)$ over $\Z$, via the Farrell-Jones fibred isomorphism conjecture~\cite{GJM}.}

\emph{The main result of this manuscript is to classify, with a finite number of exceptions and up to isomorphism, the virtually cyclic subgroups of $B_{n}(\St)$. As corollaries, we obtain the complete classification of the virtually cyclic subgroups of $B_{n}(\St)$ when $n$ is either odd, or is even and sufficiently large. 
Using the close relationship between $B_{n}(\St)$ and the mapping class group $\mcg$ of the $n$-punctured sphere, another consequence is the classification (with a finite number of exceptions) of the isomorphism classes of the virtually cyclic subgroups of $\mcg$.}

\emph{The proof of the main theorem is divided into two parts: the reduction of a list of possible candidates for the virtually cyclic subgroups of $B_{n}(\St)$ obtained using a general result due to Epstein and Wall to an almost optimal family $\mathbb{V}(n)$ of virtually cyclic groups; and the realisation of all but a finite number of  elements of $\mathbb{V}(n)$. The first part makes use of a number of techniques, notably the study of the periodicity and the outer automorphism groups of the finite subgroups of $B_{n}(\St)$, and the analysis of the conjugacy classes of the finite order elements of $B_{n}(\St)$. In the second part, we construct subgroups of $B_{n}(\St)$ isomorphic to the elements of $\mathbb{V}(n)$ using mainly an algebraic point of view that is strongly inspired by geometric observations, as well as explicit geometric constructions in $\mcg$ which we translate to $B_{n}(\St)$.}

\emph{In order to classify the isomorphism classes of the virtually cyclic subgroups of $B_{n}(\St)$, we obtain a number of results that we believe are interesting in their own right, notably the characterisation of the centralisers and normalisers of the maximal cyclic and dicyclic subgroups of $B_{n}(\St)$, a generalisation to $B_{n}(\St)$ of a result due to Hodgkin for the mapping class group of the punctured sphere concerning conjugate powers of torsion elements, the study of the isomorphism classes of those virtually cyclic groups of $B_{n}(\St)$ that appear as amalgamated products, as well as an alternative proof of a result due to~\cite{BCP,FZ} that the universal covering of the $n\up{th}$ configuration space of $\St$, $n\geq 3$, has the homotopy type of $\St[3]$.}
\end{abstract}

\begingroup
\renewcommand{\thefootnote}{}
\footnotetext{\noindent 2010 AMS Subject Classification: 20F36 (primary), 20E07, 20F50, 55R80, 55Q52 (secondary)\\
Keywords: Sphere braid groups, configuration space, virtually cyclic group, mapping class group}
\endgroup


\tableofcontents

\mainmatter

\chapter*{Introduction and statement of the main results}\label{part:intro}

The braid groups $B_n$ of the plane were introduced by E.~Artin
in~1925 and further studied in 1947~\cite{A1,A2}. They were later generalised by Fox to braid groups of arbitrary topological spaces via the following definition~\cite{FoN}. Let $M$ be a compact, connected surface, and let $n\in\N$. We denote the set of all ordered $n$-tuples of distinct
points of $M$, known as the \emph{$n\th$ configuration space of $M$},
by:
\begin{equation*}
F_n(M)=\setr{(p_1,\ldots,p_n)}{\text{$p_i\in M$ and $p_i\neq p_j$ if $i\neq j$}}.
\end{equation*}
Configuration spaces play an important r\^ole in several branches of
mathematics and have been extensively studied, see~\cite{Bi,CG,FH,H} for example.

The symmetric group $\sn$ on $n$ letters acts freely on $F_n(M)$ by
permuting coordinates. The corresponding quotient space $F_n(M)/\sn$ will be denoted by
$D_n(M)$. The \emph{$n\th$ pure braid group $P_n(M)$} (respectively
the \emph{$n\th$ braid group $B_n(M)$}) is defined to be the
fundamental group of $F_n(M)$ (respectively of $D_n(M)$). 

Together with the real projective plane $\rp$, the braid groups of the
$2$-sphere $\St$ are of particular interest, notably because they have
non-trivial centre~\cite{GVB,GG2}, and torsion elements~\cite{vB,M}.
Indeed, Fadell and Van Buskirk showed that among the braid groups of compact, connected surfaces, $B_n(\St)$ and $B_n(\rp)$ are the only ones to have torsion~\cite{FvB,vB}. Let us recall briefly some of the properties
of $B_n(\St)$~\cite{FvB,GVB,vB}.

If $\dt\subseteq \St$ is a topological disc, there is a 
homomorphism $\map {\iota}{B_n}[B_n(\St)]$ induced by the inclusion.
If $\beta\in B_n$ then we shall denote its image $\iota(\beta)$ simply
by $\beta$. Then $B_n(\St)$ is generated by
$\sigma_1,\ldots,\sigma_{n-1}$ which are subject to the following
relations:
\begin{gather}
\text{$\sigma_{i}\sigma_{j}=\sigma_{j}\sigma_{i}$ if $\lvert i-j\rvert\geq 2$
and $1\leq i,j\leq n-1$}\label{eq:Artin1}\\
\text{$\sigma_{i}\sigma_{i+1}\sigma_{i}=\sigma_{i+1}\sigma_{i}\sigma_{i+1}$ for
all $1\leq i\leq n-2$, and}\label{eq:Artin2}\\
\text{$\sigma_1\cdots \sigma_{n-2}\sigma_{n-1}^2 \sigma_{n-2}\cdots \sigma_1=1
$.}\label{eq:surface}
\end{gather}
Consequently, $B_n(\St)$ is a quotient of $B_n$. The first three sphere braid
groups are finite: $B_1(\St)$ is trivial, $B_2(\St)$ is cyclic of order~$2$,
and $B_3(\St)$ is a $\text{ZS}$-metacyclic group (a group whose Sylow
subgroups, commutator subgroup and commutator quotient group are all cyclic) of order~$12$, isomorphic to the semi-direct product $\Z_3 \rtimes \Z_4$ of cyclic groups, the action being the non-trivial one. For $n\geq 4$, $B_n(\St)$ is infinite. The Abelianisation of $B_n(\St)$ is isomorphic to the cyclic group $\Z_{2(n-1)}$. The kernel of the associated projection 
\begin{equation}\label{eq:abdefine}
\left\{ \begin{aligned}
\xi\colon\thinspace B_n(\St) &\to \Z_{2(n-1)}\\
\sigma_i &\mapsto \overline{1} \quad\text{for all $1\leq i\leq n-1$}
\end{aligned}\right.
\end{equation}
is the commutator subgroup
$\Gamma_2\left(B_n(\St) \right)$. If $w\in B_n(\St)$ then $\xi(w)$ is the
exponent sum (relative to the $\sigma_i$) of $w$ modulo $2(n-1)$. Further, we have a natural short exact sequence 
\begin{equation}\label{eq:defperm}
1 \to P_{n}(\St) \to B_{n}(\St) \stackrel{\pi}{\to} \sn \to 1,
\end{equation}
$\pi$ being the homomorphism that sends $\sigma_{i}$ to the transposition $(i,i+1)$.

Gillette and Van Buskirk showed that if $n\geq 3$ and $k\in \N$ then
$B_n(\St)$ has an element of order $k$ if and only if $k$ divides one
of $2n$, $2(n-1)$ or $2(n-2)$~\cite{GVB}. The torsion elements of
$B_n(\St)$ and $B_n(\rp)$ were later characterised by
Murasugi:
\begin{thm}[Murasugi~\cite{M}]\label{th:murasugi}
Let $n\geq 3$. Then up to conjugacy, the torsion elements of $B_n(\St)$ are precisely the powers of
the following three elements:
\begin{enumerate}[(a)]
\item $\alpha_0= \sigma_1\cdots \sigma_{n-2}
\sigma_{n-1}$ (of order $2n$).
\item $\alpha_1=\sigma_1\cdots
\sigma_{n-2} \sigma_{n-1}^2$ (of order $2(n-1)$).
\item $\alpha_2=\sigma_1\cdots \sigma_{n-3} \sigma_{n-2}^2$ (of order
$2(n-2)$).
\end{enumerate}
\end{thm}
So the maximal finite cyclic subgroups of $B_n(\St)$ are isomorphic to $\Z_{2n}$, $\Z_{2(n-1)}$ or $\Z_{2(n-2)}$. In~\cite{GG4}, we showed that $B_n(\St)$ is generated by $\alpha_0$ and $\alpha_1$. Let $\ft= (\sigma_1\cdots\sigma_{n-1})^n$ denote the so-called `full twist' braid of $B_n(\St)$. If $n\geq 3$, $\ft[n]$ is the unique element of $B_n(\St)$ of order $2$, and it generates the centre of $B_n(\St)$. It is also the square of the `half twist' element defined by:
\begin{equation}\label{eq:defgarside}
\garside = (\sigma_1 \cdots \sigma_{n-1}) (\sigma_1 \cdots \sigma_{n-2}) \cdots
(\sigma_1 \sigma_2)\sigma_1.
\end{equation}
It is well known that:
\begin{equation}\label{eq:garsideconj}
\garside \sigma_{i} \garside^{-1}= \sigma_{n-i} \quad \text{for all $i=1,\ldots,n-1$.}
\end{equation}
The uniqueness of the element of order $2$ in $B_{n}(\St)$ implies that the three elements $\alpha_0$, $\alpha_1$ and $\alpha_2$ are respectively $n\th$, $(n-1)\th$ and $(n-2)\th$ roots of $\ft$, and this yields the useful relation:
\begin{equation}\label{eq:uniqueorder2}
\ft=\alpha_{i}^{n-i} \quad \text{for all $i\in\brak{0,1,2}$.}
\end{equation}
In what follows, if $m\geq 2$, $\dic{4m}$ will denote the \emph{dicyclic group} of order $4m$. It admits a presentation of the form 
\begin{equation}\label{eq:presdic}
\setangr{x,y}{x^m=y^2,\; yxy^{-1}=x^{-1}}.
\end{equation}
If in addition $m$ is a power of $2$ then we will also refer to the dicyclic group of order $4m$ as the \emph{generalised quaternion group} of order $4m$, and denote it by $\quat[4m]$. For example, if $m=2$ then we obtain the usual quaternion group $\quat$ of order $8$. Further, $\tonestar$ (resp.\ $\oonestar$, $\istar$) will denote the \emph{binary tetrahedral group} of order $24$ (resp.\ the \emph{binary octahedral group} of order $48$, the \emph{binary icosahedral group} of order $120$). We will refer collectively to $\tonestar,\oonestar$ and $\istar$ as the \emph{binary polyhedral groups}. More details on these groups may be found in~\cite{AM,Co,CM,Wo}, as well as in \resecglobal{generalities}{autout} and the Appendix. 

In order to understand better the structure of $B_{n}(\St)$, one may study (up to isomorphism) the finite subgroups of $B_{n}(\St)$. From \reth{murasugi}, it is clear that the finite cyclic subgroups of $B_{n}(\St)$ are isomorphic to the subgroups of $\Z_{2(n-i)}$, where $i\in \brak{0,1,2}$. Motivated by a question of the realisation of $\quat$ as a subgroup of $B_n(\St)$ of R.~Brown~\cite{DD} in connection with the Dirac string trick~\cite{Fa,Ne}, as well as the study of the case $n=4$ by J.~G.~Thompson~\cite{Thp}, we obtained partial results on the classification of the isomorphism classes of the finite subgroups of $B_{n}(\St)$ in~\cite{GG6,GG5}. The complete classification was given in~\cite{GG7}:
\begin{thm}[\cite{GG7}]\label{th:finitebn}
Let $n\geq 3$. Up to isomorphism, the maximal finite subgroups of $B_n(\St)$ are:
\begin{enumerate}[(a)]
\item\label{it:fina} $\Z_{2(n-1)}$ if $n\geq 5$.
\item\label{it:finb} $\dic{4n}$.
\item\label{it:finc} $\dic{4(n-2)}$ if $n=5$ or $n\geq 7$.
\item\label{it:find} $\tonestar$ if $n\equiv 4 \bmod 6$.
\item\label{it:fine} $\oonestar$ if $n\equiv 0,2\bmod 6$.
\item\label{it:finf} $\istar$ if $n\equiv 0,2,12,20\bmod 30$.
\end{enumerate}
\end{thm}

\begin{rems}\label{rem:finitesub}\mbox{}
\begin{enumerate}[(a)]
\item\label{it:finitesuba} By studying the subgroups of dicyclic and binary polyhedral groups, it is not difficult to show that any finite subgroup of $B_{n}(\St)$ is cyclic, dicyclic or binary polyhedral (see \repr{maxsubgp}).
\item\label{it:finitesubb} As we showed in~\cite{GG6,GG7}, for $i\in\brak{0,2}$, 
\begin{equation}\label{eq:basicconj}
\garside \alpha_{i}'\garside^{-1}=\alpha_{i}'^{-1}, \quad\text{where $\alpha_{i}'=\alpha_{0}\alpha_{i}\alpha_{0}^{-1}= \alpha_{0}^{i/2}\alpha_{i}\alpha_{0}^{-i/2}$,}
\end{equation}
and the dicyclic group of order $4(n-i)$ is realised in terms of the generators of $B_{n}(\St)$ by:
\begin{equation*}
\ang{\alpha_{i}', \garside},
\end{equation*}
which we shall refer to hereafter as the \emph{standard copy} of $\dic{4(n-i)}$ in $B_{n}(\St)$.
\end{enumerate}
\end{rems}

A key tool in the proof of \reth{finitebn} is the strong relationship due to Magnus of $B_{n}(\St)$ with the mapping class group $\mcg$ of the $n$-punctured sphere, $n\geq 3$, given by the short exact sequence~\cite{FM,MKS}:
\begin{equation}\label{eq:mcg}
1\to \ang{\ft} \to B_{n}(\St)\stackrel{\phi}{\to} \mcg \to 1. 
\end{equation}
As we shall see, it will also play an important rôle in various parts of this paper, notably in the study of the centralisers and conjugacy classes of the finite order elements in \repart{generalities}, as well as in some of the constructions in \repart{realisation}. There is a short exact sequence for the mapping class group analogous to \req{defperm}; the kernel of the homomorphism $\mcg\to \sn$ is the pure mapping class group $\pmcg$, which may also be seen as the image of $P_{n}(\St)$ under $\phi$. In particular, since for $n\geq 4$, $P_{n}(\St)\cong P_{n-3}(\St\setminus \brak{x_{1},x_{2},x_{3}})\times \Z_{2}$~\cite{GG2}, where the second factor is identified with $\ang{\ft}$, it follows from the restriction of \req{mcg} to $P_{n}(\St)$ that $\pmcg\cong P_{n-3}(\St\setminus \brak{x_{1},x_{2},x_{3}})$, in particular $\pmcg$ is torsion free for all $n\geq 4$.

In this paper, we go a stage further by classifying (up to isomorphism) the virtually cyclic subgroups of $B_{n}(\St)$. Recall that a group is said to be \emph{virtually cyclic} if it contains a cyclic subgroup of finite index (see also \resecglobal{generalities}{generalities}). It is clear from the definition that any finite subgroup is virtually cyclic, so in view of \reth{finitebn}, it suffices to concentrate on the \emph{infinite} virtually cyclic subgroups of $B_{n}(\St)$, which are in some sense its `simplest' infinite subgroups.
The classification of the virtually cyclic subgroups of $B_{n}(\St)$ is an interesting problem in its own right. As well as helping us to understand better the structure of these braid groups, the results of this paper give rise to some $K$-theoretical applications. We remark that our work was partially motivated by a question of S.~Mill\'an-L\'opez and S.~Prassidis concerning the calculation of the algebraic $K$-theory of the braid groups of $\St$ and $\rp$. It was shown recently that the full and pure braid groups of these two surfaces satisfy the Fibred Isomorphism Conjecture of F.~T.~Farrell and L.~E.~Jones~\cite{BJL,JPM1,JPM2}. This implies that the algebraic $K$-theory groups of their group rings (over $\Z$) may be computed by means of the algebraic $K$-theory groups of their virtually cyclic subgroups via the so-called `assembly maps'. More information on these topics may be found in~\cite{BLR, FJ,JP}. The main theorem of this paper, \reth{main}, is currently being applied to the calculation of the lower algebraic $K$-theory of $\Z[B_{n}(\St)]$~\cite{GJM}, which generalises results of the thesis of Mill\'an-L\'opez~\cite{JPM3,ML} where she calculated the lower algebraic $K$-theory of the group rings of $P_n(\St)$ and $P_n(\rp)$ using our classification of the virtually cyclic subgroups of $P_{n}(\rp)$~\cite{GG8}. This application to $K$-theory thus provides us with additional reasons to find the virtually cyclic subgroups of $B_{n}(\St)$.



As we observed previously, if $n\leq 3$ then $B_{n}(\St)$ is a known finite group, and so we shall suppose in this paper that $n\geq 4$. Our main result is \reth{main}, which yields the complete classification of the infinite virtually cyclic subgroups of $B_{n}(\St)$, with a small number of exceptions, that we indicate below in \rerem{exceptions}. Recall that by results of Epstein and Wall~\cite{Ep,W} (see also \reth{wall} in \resecglobal{generalities}{generalities}), any infinite virtually cyclic group $G$ is isomorphic to $F\rtimes \Z$ or $G_{1}\bigast_{F} G_{2}$, where $F$ is finite and $[G_{i}:F]=2$ for $i\in\brak{1,2}$ (we shall say that $G$ is of \textit{Type~I} or \textit{Type~II} respectively). Before stating \reth{main}, we define two families of virtually cyclic groups. If $G$ is a group, let $\aut{G}$ (resp.\ $\out{G}$) denote the group of its automorphisms (resp.\ outer automorphisms).

\begin{defn}\label{def:v1v2}
Let $n\geq 4$.
\begin{enumerate}[(1)]
\item\label{it:mainIdef} Let $\mathbb{V}_{1}(n)$ be the union of the following Type~I virtually cyclic groups:
\begin{enumerate}[(a)]
\item\label{it:mainzq} $\Z_{q}\times \Z$, where $q$ is a strict divisor of $2(n-i)$, $i\in \brak{0,1,2}$, and $q\neq n-i$ if $n-i$ is odd.
\item\label{it:mainzqt} $\Z_{q}\rtimes_{\rho} \Z$, where $q\geq 3$ is a strict divisor of $2(n-i)$, $i\in \brak{0,2}$, $q\neq n-i$ if $n$ is odd, and $\rho(1)\in \aut{\Z_{q}}$ is multiplication by $-1$.
\item\label{it:maindic} $\dic{4m}\times \Z$, where $m\geq 3$ is a strict divisor of $n-i$ and $i\in \brak{0,2}$.
\item\label{it:maindict} $\dic{4m}\rtimes_{\nu} \Z$, where $m\geq 3$ divides $n-i$, $i\in \brak{0,2}$, $(n-i)/m$ is even, and where $\nu(1)\in \aut{\dic{4m}}$ is defined by:
\begin{equation}\label{eq:actdic4m}
\left\{
\begin{aligned}
\nu(1)(x)&=x\\
\nu(1)(y)&=xy
\end{aligned}\right.
\end{equation}
for the presentation~\reqref{presdic} of $\dic{4m}$. 

\item\label{it:mainq8} $\quat\rtimes_{\theta} \Z$, for $n$ even and $\theta\in \operatorname{Hom}(\Z,\aut{\quat})$, for the following actions:
\begin{enumerate}[(i)]
\item $\theta(1)=\id$.
\item\label{it:mainIcii} $\theta=\alpha$, where $\alpha(1)\in \aut{\quat}$ is given by $\alpha(1)(i)=j$ and $\alpha(1)(j)=k$, where $\quat=\brak{\pm 1, \pm i, \pm j, \pm k}$.
\item\label{it:mainIciii} $\theta=\beta$, where $\beta(1)\in \aut{\quat}$ is given by $\beta(1)(i)=k$ and $\beta(1)(j)=j^{-1}$.
\end{enumerate}

\item\label{it:maint} $\tonestar \times \Z$ for $n$ even.
\item\label{it:maing} $\tonestar \rtimes_{\omega} \Z$ for $n\equiv 0,2 \bmod 6$, where $\omega(1)\in \aut{\tonestar}$ is the automorphism defined as follows. Let $\tonestar$ be given by the presentation~\cite[p.~198]{Wo}:
\begin{equation}\label{eq:preststar}
\setangr{P,Q,X}{X^3=1,\, P^2=Q^2,\, PQP^{-1}=Q^{- 1},\, XPX^{-1}=Q, \, XQX^{-1}=PQ},
\end{equation}
and let $\omega(1)\in \aut{\tonestar}$ be defined by
\begin{equation}\label{eq:nontrivacttstar}
\left\{
\begin{aligned}
P &\mapsto QP\\
Q &\mapsto Q^{-1}\\
X &\mapsto X^{-1}.
\end{aligned}\right.
\end{equation}
More details concerning this automorphism will be given in \resecglobal{generalities}{autout}. 

\item\label{it:maino} $\oonestar \times \Z$ for $n\equiv 0,2 \bmod 6$.
\item\label{it:maini} $\istar \times \Z$ for $n\equiv 0,2,12,20 \bmod{30}$.
\end{enumerate}
\item\label{it:mainIIdef} Let $\mathbb{V}_{2}(n)$ be the union of the following Type~II virtually cyclic groups:
\begin{enumerate}[(a)]
\item\label{it:mainIIa} $\Z_{4q}\bigast_{\Z_{2q}} \Z_{4q}$, where $q$ divides $(n-i)/2$ for some $i\in\brak{0,1,2}$.

\item\label{it:mainIIb} $\Z_{4q}\bigast_{\Z_{2q}} \dic{4q}$, where $q\geq 2$ divides $(n-i)/2$ for some $i\in\brak{0,2}$.

\item\label{it:mainIIc} $\dic{4q}\bigast_{\Z_{2q}} \dic{4q}$, where $q\geq 2$ divides $n-i$ strictly for some $i\in\brak{0,2}$.

\item\label{it:mainIId} $\dic{4q}\bigast_{\dic{2q}} \dic{4q}$, where $q\geq 4$ is even and divides $n-i$ for some $i\in\brak{0,2}$. 

\item\label{it:mainIIe} $\oonestar \bigast_{\tonestar} \oonestar$, where $n\equiv 0,2 \bmod{6}$. 
\end{enumerate}
\end{enumerate}
Finally, let $\mathbb{V}(n)=\mathbb{V}_{1}(n) \bigcup \mathbb{V}_{2}(n)$. Unless indicated to the contrary, in what follows, $\rho, \nu,\alpha,\beta$ and $\omega$ will denote the actions defined in parts~(\ref{it:mainIdef})(\ref{it:mainzqt}), (\ref{it:maindict}), (\ref{it:mainq8})(\ref{it:mainIcii}), (\ref{it:mainq8})(\ref{it:mainIciii}) and (\ref{it:maing}) respectively.
\end{defn}

\pagebreak

The main result of this paper is the following, which classifies (up to a finite number of exceptions), the infinite virtually cyclic subgroups of $B_{n}(\St)$.

\begin{thm}\label{th:main}
Suppose that $n\geq 4$. 
\begin{enumerate}[(1)]
\item\label{it:mainI} 
If $G$ is an infinite virtually cyclic subgroup of $B_{n}(\St)$ then $G$ is isomorphic to an element of $\mathbb{V}(n)$.

\item\label{it:mainII} Conversely, let $G$ be an element of $\mathbb{V}(n)$. Assume that the following conditions hold:
\begin{enumerate}[(a)]
\item\label{it:excepa} if $G\cong \quat \rtimes_{\alpha} \Z$ then $n\notin \brak{6,10,14}$.

\item if $G\cong \tonestar \times \Z$ then $n\notin \brak{4,6,8,10,14}$.

\item if $G\cong \oonestar \times \Z$ or $G\cong \tonestar \rtimes_{\omega} \Z$ then $n\notin \brak{6,8,12,14,18,20,26}$.

\item\label{it:excepd} if $G\cong \istar \times \Z$ then $n\notin \brak{12,20,30,32,42,50,62}$.

\item\label{it:excepe} if $G\cong \oonestar \bigast_{\tonestar} \oonestar$ then $n\notin \brak{6,8,12,14,18,20,24,26,30,32,38}$.
\end{enumerate}
Then there exists a subgroup of $B_{n}(\St)$ isomorphic to $G$.

\item\label{it:mainIII} Let $G$ be isomorphic to $\tonestar\times \Z$ (resp.\ to $\oonestar\times \Z$) if $n=4$ (resp.\ $n=6$). Then $B_{n}(\St)$ has no subgroup isomorphic to $G$.
\end{enumerate}
\end{thm}

\begin{rem}\label{rem:exceptions}
Together with \reth{finitebn}, \reth{main} yields a complete classification of the virtually cyclic subgroups of $B_{n}(\St)$ with the exception of a small (finite) number of cases for which the problem of their existence is open. These cases are as follows:
\begin{enumerate}
\item\label{it:exceptions} Type~I subgroups of $B_{n}(\St)$ (see Propositions~\ref{prop:ttimesz} and~\ref{prop:trtimesz}, as well as Remarks~\ref{rem:opentypeI} and~\ref{rem:opentypeIa}):
\begin{enumerate}
\item the realisation of $\quat \rtimes_{\alpha} \Z$ as a subgroup of $B_{n}(\St)$, where $n$ belongs to $\brak{6,10,14}$ and $\alpha(1)\in \aut{\quat}$ is as in \redef{v1v2}(\ref{it:mainIdef})(\ref{it:mainq8})(\ref{it:mainIcii}).

\item the realisation of $\tonestar\times \Z$ as a subgroup of $B_{n}(\St)$, where $n$ belongs to $\brak{6,8,10,14}$.
\item the realisation of $\tonestar\rtimes_{\omega} \Z$ as a subgroup of $B_{n}(\St)$, where the action $\omega$ is given by \redef{v1v2}(\ref{it:mainIdef})(\ref{it:maing}), and $n\in\brak{6,8,12,14,18,20,26}$.
\item the realisation of $\oonestar\times \Z$ as a subgroup of $B_{n}(\St)$, where $n\in\brak{8,12,14,18,20,26}$. 
\item the realisation of $\istar\times \Z$ as a subgroup of $B_{n}(\St)$, where $n\in\brak{12,20,30,32,42,50,62}$.
\end{enumerate}
\item Type~II subgroups of $B_{n}(\St)$ (see \rerem{notostar} and \repr{realV2bis}):
\begin{enumerate}
\item\label{it:exceptionsostar} for $n\in \brak{6,8,12,14,18,20,24,26,30,32,38}$, the realisation of the group $\oonestar \bigast_{\tonestar} \oonestar $ as a subgroup of $B_{n}(\St)$.
\end{enumerate}
\end{enumerate}
\end{rem}

Since the above open cases occur for even values of $n$, the complete classification of the infinite virtually cyclic subgroups of $B_{n}(\St)$ for all $n\geq 5$ odd is an immediate consequence of \reth{main}.
\begin{thm}\label{th:mainodd}
Let $n\geq 5$ be odd. Then up to isomorphism, the following groups are the infinite virtually cyclic subgroups of $B_{n}(\St)$.

\begin{enumerate}[(I)]
\item \begin{enumerate}[(a)]
\item $\Z_{m} \rtimes_{\theta} \Z$, where $\theta(1)\in\brak{\id, -\id}$, $m$ is a strict divisor of $2(n-i)$, for $i\in \brak{0,2}$, and $m\neq n-i$.
\item $\Z_{m} \times \Z$, where $m$ is a strict divisor of $2(n-1)$.

\item $\dic{4m} \times \Z$, where $m\geq 3$ is a strict divisor of $n-i$ for $i\in \brak{0,2}$.

\end{enumerate}

\item 
\begin{enumerate}[(a)]
\item $\Z_{4q}\bigast_{\Z_{2q}}\Z_{4q}$, where $q$ divides $(n-1)/2$.

\item $\dic{4q}\bigast_{\Z_{2q}}\dic{4q}$, where $q\geq 2$ is a strict divisor of $n-i$, and $i\in \brak{0,2}$

\end{enumerate}
\end{enumerate}

\end{thm}


Most of this manuscript is devoted to proving \reth{main}, and is broadly divided into two parts, \ref{part:generalities} and \ref{part:realisation}, together with a short Appendix. The aim of \repart{generalities} is to prove \reth{main}(\ref{it:mainI}). In conjunction with \reth{finitebn}, \reth{wall} gives rise to a family $\mathcal{VC}$ of virtually cyclic groups, defined in \resecglobal{generalities}{generalities}, with the property that any infinite virtually cyclic subgroup of $B_{n}(\St)$ belongs to $\mathcal{VC}$. In that section, we shall discuss a number of properties pertaining to virtually cyclic groups. \repr{vcmcg} describes the correspondence in general between the virtually cyclic subgroups of a group $G$ possessing a unique element $x$ of order $2$ and its quotient $G/\ang{x}$. By the short exact sequence~\reqref{mcg}, this proposition applies immediately to $B_{n}(\St)$ and $\mcg$, and will be used at various points, notably to obtain the classification of the virtually cyclic subgroups of $\mcg$ from that of $B_{n}(\St)$. Two other results of \resecglobal{generalities}{generalities} that will prove to be useful in \resecglobal{realisation}{isoclasses} are \repr{semiamalg} which shows that almost all elements of $\mathbb{V}_{2}(n)$ of the form $G\bigast_{H} G$ may be written as a semi-direct product $\Z\rtimes G$, and \repr{amalgiso} which will be used to determine the number of isomorphism classes of the elements of $\mathbb{V}_{2}(n)$.

The principal difficulty in proving \reth{main} is to decide which of the elements of $\mathcal{VC}$ are indeed realised as subgroups of $B_{n}(\St)$. This is achieved in two stages, \emph{reduction} and \emph{realisation}. In the first stage, we reduce the subfamily of $\mathcal{VC}$ of Type~I groups in several ways. To this end, in \resecglobal{generalities}{gencent}, we obtain a number of results of independent interest concerning structural aspects of $B_{n}(\St)$. The first of these is the calculation of the centraliser and normaliser of its maximal finite cyclic and dicyclic subgroups. Note that if $i\in\brak{0,1}$, the centraliser of $\alpha_{i}$, considered as an element of $B_{n}$, is equal to $\ang{\alpha_{i}}$~\cite{BDM,GW}. A similar equality holds in $B_{n}(\St)$ and is obtained using \req{mcg} and the corresponding result for $\mcg$ due to L.~Hodgkin~\cite{Ho}:


\begin{prop}\label{prop:genhodgkin1}
Let $i\in\brak{0,1,2}$, and let $n\geq 3$. 
\begin{enumerate}[(a)]
\item\label{it:centalphai} The centraliser of $\ang{\alpha_{i}}$ in $B_{n}(\St)$ is equal to $\ang{\alpha_{i}}$, unless $i=2$ and $n=3$, in which case it is equal to $B_{3}(\St)$.
\item\label{it:normcyclic} The normaliser of $\ang{\alpha_{i}}$ in $B_{n}(\St)$ is equal to:
\begin{equation*}
\begin{cases}
\ang{\alpha_{0},\garside} \cong \dic{4n} & \text{if $i=0$}\\
\ang{\alpha_{2},\alpha_{0}^{-1}\garside \alpha_{0}} \cong \dic{4(n-2)} & \text{if $i=2$}\\
\ang{\alpha_{1}}\cong \Z_{2(n-1)} & \text{if $i=1$,}
\end{cases}
\end{equation*}
unless $i=2$ and $n=3$, in which case it is equal to $B_{3}(\St)$.
\item If $i\in\brak{0,2}$, the normaliser of the standard copy of $\dic{4(n-i)}$ in $B_{n}(\St)$ is itself, except when $i=2$ and $n=4$, in which case the normaliser is equal to $\alpha_{0}^{-1} \sigma_{1}^{-1} \ang{\alpha_{0},\garside[4]} \sigma_{1}\alpha_{0}$, and is isomorphic to $\quat[16]$. 
\end{enumerate}
\end{prop}
If $F$ is a maximal dicyclic or finite cyclic subgroup of $B_{n}(\St)$, parts~(\ref{it:centalphai}) and~(\ref{it:normcyclic}) imply immediately that $B_{n}(\St)$ has no Type~I subgroup of the form $F\rtimes \Z$. 

The second reduction, given in \repr{isoout} in \resecglobal{generalities}{autout}, will make use of the fact that if $\map{\theta}{\Z}[\aut{F}]$ is an action of $\Z$ on the finite group $F$, the isomorphism class of the semi-direct product $F\rtimes_{\theta} \Z$ depends only on the class of $\theta(1)$ in $\out{F}$. Since we are interested in the realisation of isomorphism classes of virtually finite subgroups in $B_{n}(\St)$, it will thus be sufficient to study the Type~I groups of the form $F\rtimes_{\theta} \Z$, where $\theta(1)$ runs over a transversal of $\out{F}$ in $\aut{F}$. To this end, in \resecglobal{generalities}{autout}, we recall the structure of $\out{F}$ for the binary polyhedral groups. One could also carry out this analysis for the other finite subgroups of $B_{n}(\St)$ given by \reth{finitebn}, but the resulting conditions on $\theta$ are weaker than those obtained from a generalisation of a second result of L.~Hodgkin concerning the powers of $\alpha_{i}$ that are conjugate in $B_{n}(\St)$. More precisely, in \resecglobal{generalities}{conjfinite}, we prove the following proposition.
\begin{prop}\label{prop:genhodgkin2}
Let $n\geq 3$ and $i\in \brak{0,1,2}$, and suppose that there exist $r,m\in \Z$ such that $\alpha_{i}^m$ and $\alpha_{i}^r$ are conjugate in $B_{n}(\St)$. 
\begin{enumerate}
\item\label{it:conjpowera} If $i=1$ then $\alpha_1^m=\alpha_1^r$.
\item\label{it:conjpowerb} If $i\in \brak{0,2}$ then $\alpha_{i}^m=\alpha_{i}^{\pm r}$.
\end{enumerate}
\end{prop}
In particular, conjugate powers of the $\alpha_{i}$ are either equal or inverse. So if $F$ is a finite cyclic subgroup of $B_{n}(\St)$ then by \reth{murasugi} the only possible actions of $\Z$ on $F$ are the trivial action and multiplication by $-1$. This also has consequences for the possible actions of $\Z$ on dicyclic subgroups of $B_{n}(\St)$. As in \repr{genhodgkin1}, the proof of \repr{genhodgkin2} will make use of a similar result for the mapping class group and the relation \reqref{mcg}.

The final reduction, described in \resecglobal{generalities}{percohI}, again affects the possible Type~I subgroups that may occur, and is a manifestation of the periodicity (with least period $2$ or $4$) of the subgroups of $B_{n}(\St)$ that was observed in \cite{GG7} for the finite subgroups. The following proposition will be applied to rule out Type~I subgroups of $B_{n}(\St)$ isomorphic to $F\rtimes_{\theta} \Z$ with non-trivial action $\theta$, where $F$ is either $\oonestar$ or $\istar$ (one could also apply the result to the other possible finite groups $F$, but this is not necessary in our context in light of the consequences of \repr{genhodgkin2} mentioned above). The following proposition may be found in~\cite{BCP,FZ}, and may be compared with the analogous result for $\rp$~\cite[Proposition 6]{GG3}. We shall give an alternative proof in \resecglobal{generalities}{homotopytype}.
\begin{prop}[\cite{BCP,FZ}]\label{prop:homot}\mbox{}
\begin{enumerate}[(a)]
\item\label{it:homottype2} The space $F_2(\St)$ (resp.\ $D_2(\St)$) has the homotopy type of $\St$ (resp.\ of $\rp$). Hence the universal covering space of $D_2(\St)$ is $F_2(\St)$.
\item If $n\geq 3$, the universal covering space of $F_n(\St)$ or $D_n(\St)$ has the homotopy type of the $3$--sphere $\St[3]$.
\end{enumerate}
\end{prop}

Putting together these reductions will allow us to prove \reth{main}(\ref{it:mainI}), first for the groups of Type~I in \resecglobal{generalities}{defV1}, and then for those of Type~II in \resecglobal{generalities}{typeII}. The structure of the finite subgroups of $B_{n}(\St)$ imposes strong constraints on the possible Type~II subgroups, and the proof in this case is more straightforward than that for Type~I subgroups. 

The second part of the paper, \repart{realisation}, is devoted to the analysis of the realisation of the elements of $\mathbb{V}_{1}(n)\coprod\mathbb{V}_{2}(n)$ as subgroups of $B_{n}(\St)$ and to proving parts~(\ref{it:mainII}) and (\ref{it:mainIII}) of \reth{main}. With the exception of the values of $n$ excluded by the statement of part~(\ref{it:mainII}), we prove the existence of the elements of $\mathbb{V}(n)$ as subgroups of $B_{n}(\St)$, first those of Type~I in 
Sections~\ref{part:realisation}.\ref{sec:centcycdic}--~\ref{part:realisation}.\ref{sec:typeIbinpoly}, and then those of Type~II in \resecglobal{realisation}{realtypeII}. The results of these sections are gathered together in \repr{realV1} (resp.\ \repr{realV2bis}) which proves \reth{main}(\ref{it:mainII}) for the subgroups of Type~I (resp.\ Type~II). The construction of the elements of $\mathbb{V}(n)$ involving finite cyclic and dicyclic groups are largely algebraic, and will rely heavily on \relem{commalphaigen}, as well as on \relem{funda} which describes the action by conjugation of the $\alpha_{i}$ on the generators of $B_{n}(\St)$. In contrast, the realisation of the elements of $\mathbb{V}(n)$ involving the binary polyhedral groups is geometric in nature, and occurs on the level of mapping class groups via the relation~\reqref{mcg}. The constraints involved in the constructions indicate why the realisation of such elements is an open problem for the values of $n$ given in \rerem{exceptions}. For $n\in \brak{4,6}$, in \repr{ttimesz}(\ref{it:ttimeszd}) we are also able to rule out the existence of the virtually cyclic groups given in \reth{main}(\ref{it:mainIII}).


In \resecglobal{realisation}{isoclasses}, we discuss the isomorphism problem for the amalgamated products that occur as elements of $\mathbb{V}_{2}(n)$. It turns out that with one exception, abstractly there is only one way (up to isomorphism) to embed the amalgamating subgroup in each of the two factors. With the help of \repr{amalgiso}, we are able to prove the following result.

\begin{prop}\label{prop:isoamalg}
For each of the amalgamated products given in \redef{v1v2}(\ref{it:mainIIdef}), abstractly there is exactly one isomorphism class, with the exception of $\quat[16]\bigast_{\quat} \quat[16]$, for which there are exactly two isomorphism classes.
\end{prop}

Note that \repr{isoamalg} refers to abstract isomorphism classes, and does not depend on the fact that the amalgamated products occurring as elements of $\mathbb{V}_{2}(n)$ are realised as subgroups of $B_{n}(\St)$. In the exceptional case, that of $\quat[16]\bigast_{\quat} \quat[16]$, abstractly there are two isomorphism classes defined by equations~\reqref{presK1} and~\reqref{presK2}. In \reco{semiamalg}, we show that abstractly, all but one of the isomorphism classes of the elements of $\mathbb{V}_{2}(n)$ of the form $G\bigast_{H} G$ may be written as a semi-direct product of $\Z$ by $G$. In Propositions~\ref{prop:o2k2} and~\ref{prop:existk1k2}, if $n\geq 4$ is even we show that one of these isomorphism classes is always realised as a subgroup of $B_{n}(\St)$, while the other isomorphism class is realised as a subgroup of $B_{n}(\St)$ for all $n\notin\brak{6,14,18,26,30,38}$. It is an open question as to whether this second isomorphism class is realised  as a subgroup of $B_{n}(\St)$ for $n\in\brak{6,14,18,26,30,38}$

In \resecglobal{realisation}{genmcg}, we deduce the classification of the virtually cyclic subgroups of $\mcg$ (with a finite number of exceptions). As we shall see, it will follow from \repr{vcmcg} that the homomorphism $\phi$ of the short exact sequence~\reqref{mcg} induces a correspondence that is one-to-one, with the exception of subgroups of $B_{n}(\St)$ that are isomorphic to $\Z_{m}\rtimes_{\theta} \Z$ or $\Z_{2m}\rtimes_{\theta} \Z$ for $m$ odd, which are sent to the same subgroup $\Z_{m}\rtimes_{\theta'} \Z$ of $\mcg$, the action $\theta'$ being given as in \repr{corrbnmcg}(\ref{it:corrbnmcgii}) below.
\begin{prop}\label{prop:corrbnmcg}
Let $n\geq 4$, and let $\map{\phi}{B_{n}(\St)}[\mcg]$ be the epimorphism given by \req{mcg}.
\begin{enumerate}
\item Let $H'$ be an infinite virtually cyclic subgroup of $\mcg$ of Type~I (resp.\ Type~II). Then $\phi^{-1}(H')$ is a virtually cyclic subgroup of $B_{n}(\St)$ of Type~I (resp.\ Type~II). 
\item\label{it:corrbnmcgii} Let $H$ be a Type~I virtually cyclic subgroup of $B_{n}(\St)$, isomorphic to $F\rtimes_{\theta} \Z$, where $F$ is a finite subgroup of $B_{n}(\St)$ and $\theta\in \operatorname{Hom}(\Z,\aut{F})$. Then $\phi(H)\cong \phi(F)\rtimes_{\theta'} \Z$, where $\theta'\in \operatorname{Hom}(\Z,\aut{F'})$ satisfies $\theta'(1)(f')=\phi(\theta(1)(f))$ for all $f'\in F'$ and $f\in F$ for which $\phi(f)=f'$.
\item Let $H$ be a Type~II virtually cyclic subgroup of $B_{n}(\St)$ isomorphic to $G_{1}\bigast_{F} G_{2}$, where $G_{1},G_{2}$ and $F$ are finite subgroups of $B_{n}(\St)$, and $F$ is an index $2$ subgroup of $G_{1}$ and $G_{2}$. Then $\phi(H)\cong \phi(G_{1}) \bigast_{\phi(F)} \phi(G_{2})$.
\end{enumerate}
\end{prop}
Equation~\reqref{mcg} and \redef{v1v2} together imply that the following virtually cyclic groups are those that will appear in the classification of the virtually cyclic subgroups of $\mcg$. If $m\geq 2$, let $\dih{2m}$ denote the \emph{dihedral group} of order $2m$.

\begin{defn}\label{def:v1v2mcg}
Let $n\geq 4$.
\begin{enumerate}[(1)]
\item Let $\widetilde{\mathbb{V}}_{1}(n)$ be the union of the following Type~I virtually cyclic groups:
\begin{enumerate}[(a)]
\item\label{it:mainzqmcg} $\Z_{q}\times \Z$, where $q$ is a strict divisor of $n-i$, $i\in \brak{0,1,2}$.
\item\label{it:mainzqtmcg} $\Z_{q}\rtimes_{\widetilde{\rho}} \Z$, where $q\geq 3$ is a strict divisor of $n-i$, $i\in \brak{0,2}$, and $\widetilde{\rho}(1)\in \aut{\Z_{q}}$ is multiplication by $-1$.
\item $\dih{2m}\times \Z$, where $m\geq 3$ is a strict divisor of $n-i$ and $i\in \brak{0,2}$.
\item $\dih{2m}\rtimes_{\widetilde{\nu}} \Z$, where $m\geq 3$ divides $n-i$, $i\in \brak{0,2}$, $(n-i)/m$ is even, and where $\widetilde{\nu}(1)\in \aut{\dih{2m}}$ is defined by:
\begin{equation*}
\left\{
\begin{aligned}
\widetilde{\nu}(1)(x)&=x\\
\widetilde{\nu}(1)(y)&=xy
\end{aligned}\right.
\end{equation*}
for the presentation of $\dih{2m}$ given by $\setangr{x,y}{x^m=y^2=1,\; yxy^{-1}=x^{-1}}$.
\item $(\Z_{2}\oplus \Z_{2})\rtimes_{\widetilde{\theta}} \Z$, for $n$ even and $\widetilde{\theta}\in \operatorname{Hom}(\Z,\Z_{2}\oplus \Z_{2})$, for the following actions:
\begin{enumerate}[(i)]
\item $\widetilde{\theta}(1)=\id$.
\item $\widetilde{\theta}=\widetilde{\alpha}$, where $\widetilde{\alpha}(1)\in \aut{\Z_{2}\oplus \Z_{2}}$ is given by $\widetilde{\alpha}(1)((\overline{1},\overline{0}))=(\overline{0},\overline{1})$ and $\widetilde{\alpha}(1)((\overline{0},\overline{1}))=(\overline{1},\overline{1})$.
\item $\widetilde{\theta}=\widetilde{\beta}$, where $\widetilde{\beta}(1)\in \aut{\Z_{2}\oplus \Z_{2}}$ is given by $\widetilde{\beta}(1)((\overline{1},\overline{0}))=(\overline{1},\overline{1})$ and $\widetilde{\beta}(1)((\overline{0},\overline{1}))=(\overline{0},\overline{1})$.
\end{enumerate}

\item $\an[4] \times \Z$ for $n$ even.
\item $\an[4] \rtimes_{\widetilde{\omega}} \Z$ for $n\equiv 0,2 \bmod 6$, where $\widetilde{\omega}(1)\in \aut{\an[4]}$ is the automorphism defined as follows. Let $\an[4]=(\Z_{2}\oplus\Z_{2})\rtimes \Z_{3}$ where the action of $\Z_{3}$ on $\Z_{2}\oplus\Z_{2}$ permutes cyclically the three elements $(\overline{1},\overline{0})$, $(\overline{0},\overline{1})$ and $(\overline{1},\overline{1})$, and let $\widetilde{X}$ be a generator of the $\Z_{3}$-factor. Then we define $\widetilde{\omega}(1)\in \aut{\an[4]}$ by:
\begin{equation*}
\left\{
\begin{aligned}
(\overline{1},\overline{0}) &\mapsto (\overline{1},\overline{1})\\
(\overline{0},\overline{1}) &\mapsto (\overline{0},\overline{1})\\
\widetilde{X} &\mapsto \widetilde{X}^{-1}.
\end{aligned}\right.
\end{equation*}

\item $\sn[4] \times \Z$ for $n\equiv 0,2 \bmod 6$.
\item $\an[5] \times \Z$ for $n\equiv 0,2,12,20 \bmod{30}$.
\end{enumerate}
\item\label{it:mainIIdefmcg} Let $\widetilde{\mathbb{V}}_{2}(n)$ be the union of the following Type~II virtually cyclic groups:
\begin{enumerate}[(a)]
\item $\Z_{2q}\bigast_{\Z_{q}} \Z_{2q}$, where $q$ divides $(n-i)/2$ for some $i\in\brak{0,1,2}$.

\item $\Z_{2q}\bigast_{\Z_{q}} \dih{2q}$, where $q\geq 2$ divides $(n-i)/2$ for some $i\in\brak{0,2}$.

\item $\dih{2q}\bigast_{\Z_{q}} \dih{2q}$, where $q\geq 2$ divides $n-i$ strictly for some $i\in\brak{0,2}$.

\item $\dih{2q}\bigast_{\dih{q}} \dih{2q}$, where $q\geq 4$ is even and divides $n-i$ for some $i\in\brak{0,2}$. 

\item $\sn[4] \bigast_{\an[4]} \sn[4]$, where $n\equiv 0,2 \bmod{6}$. 
\end{enumerate}
\end{enumerate}
Finally, let $\widetilde{\mathbb{V}}(n)=\widetilde{\mathbb{V}}_{1}(n) \bigcup \widetilde{\mathbb{V}}_{2}(n)$. 
\end{defn}

\pagebreak

We thus obtain the classification of the virtually cyclic subgroups of $\mcg$ (with a finite number of exceptions)
\begin{thm}\label{th:classvcmcg}
Let $n\geq 4$. Every infinite virtually cyclic subgroup of $\mcg$ is the image under $\phi$ of an element of\, $\mathbb{V}(n)$, and so is an element of\, $\widetilde{\mathbb{V}}(n)$. Conversely, if $G$ is an element of $\mathbb{V}(n)$ that satisfies the conditions of \reth{main}(\ref{it:mainII}) then $\phi(G)$ is an infinite virtually cyclic subgroup of $\mcg$. 
\end{thm}

In \repr{isoamalgmcg}, we prove a result similar to that of \repr{isoamalg} for the Type~II subgroups of $\mcg$ that appear in \redef{v1v2mcg}(\ref{it:mainIIdefmcg}), namely that there is a single isomorphism class for such groups, with the exception of the amalgamated product $\dih{8} \bigast_{\dih{4}} \dih{8}$, for which there are exactly two isomorphism classes. In an analogous manner to that of $B_{n}(\St)$, if $n$ is even then \repr{s4a4} 
shows that each of these two classes is realised as a subgroup of $\mcg$, with the possible exception of the second isomorphism class when $n$ belongs to $\brak{6,14,18,26,30,38}$.


As we mentioned previously, the real projective plane $\rp$ is the only other surface whose braid groups have torsion. In light of the results of this paper, it is thus natural to consider the problem of the classification of the virtually cyclic subgroups of $B_{n}(\rp)$ up to isomorphism. This is the subject of work in progress~\cite{GG11}. The first step, the classification of the finite subgroups of $B_{n}(\rp)$, was carried out in~\cite[Theorem~5]{GG10}. As in this paper, the classification of the infinite virtually cyclic subgroups of $B_{n}(\rp)$ is rather more difficult than in the finite case, but the combination of~\cite[Corollary~2]{GG10}, which shows that $B_{n}(\rp)$ embeds in $B_{2n}(\St)$, with~\reth{main} should be helpful in this respect.

\subsection*{Acknowledgements}

\enlargethispage{5mm}

This work took place during the visit of the second author to the Departmento de Mate\-m\'atica do IME~--~Universidade de S\~ao Paulo during the periods~14\up{th}~--~29\up{th}~April~2008, 18\up{th}~July~--~8\up{th} August~2008, 31\up{st}~October~--~10\up{th}~November~2008, 20\up{th}~May~--~3\up{rd}~June 2009, 11\up{th}~--~26\up{th}~April~2010, 4\up{th}~--~26\up{th}~October~2010, 24\up{th}~February~--~6\up{th}~March~2011 and~17\up{th}~--~25\up{th}~October~2011, and of the visit of the first author to the Laboratoire de Math\'ematiques Nicolas Oresme, Universit\'e de Caen Basse-Normandie during the periods 21\up{st}~November~--~21\up{st}~December~2008 and 5\up{th}~November--~5\up{th}~December~2010. This work was supported by the international Cooperation USP/Cofecub project n\up{o} 105/06, by the CNRS/FAPESP project n\up{o}~24460 (CNRS) et n\up{o} 2009/54745-1 (FAPESP), by the Fapesp `projeto tem\'atico Topologia alg\'ebrica, geom\'etria e diferencial' n\up{o} 08/57607-6, and by the~ANR project TheoGar n\up{o} ANR-08-BLAN-0269-02. The authors wish to thank Silvia~Mill\'an-L\'opez and Stratos~Prassidis for having posed the question of the virtually cyclic subgroups of surface braid groups, \'Etienne Ghys, Gilbert Levitt and Luis Paris for helpful conversations concerning the centraliser of finite order elements in the mapping class groups of the sphere, and Jes\'us Gonz\'alez for pointing out useful references related to the configuration space of the sphere. The second author would like to thank the CNRS for having granted him a `délégation' during the writing of part of this paper, and CONACYT (Mexico) for partial financial support through its programme `Estancias postdoctorales y sab\'aticas vinculadas al fortalecimiento de la calidad del posgrado nacional'.

\chapter{Virtually cyclic groups: generalities, reduction and the mapping class group}\label{part:generalities}

In \repart{generalities}, we start by recalling the definition of virtually cyclic groups and their characterisation due to Epstein and Wall. In \resec{generalities}, applying \reth{finitebn}, in \repr{possvcbnS2} we obtain a family $\mathcal{VC}$ of virtually cyclic groups that are potential candidates to be subgroups of $B_n(\St)$. The initial aim is to whittle down $\mathcal{VC}$ to the subfamily $\mathbb{V}(n)$ of infinite virtually cyclic groups described in \redef{v1v2} with the property that any infinite virtually cyclic subgroup of $B_{n}(\St)$ is isomorphic to an element of $\mathbb{V}(n)$.  
In \resec{generalities}, we also prove a number of results concerning infinite virtually cyclic groups, in particular \repr{vcmcg}, which will be used in \repart{realisation} to construct certain Type~II subgroups, and to prove \reth{classvcmcg}. Also of interest is \repr{amalgiso}, which will play an important rôle in \resecglobal{realisation}{isoclasses} in the study of the isomorphism classes of the Type~II subgroups of $B_{n}(\St)$, notably in the proof of \repr{isoamalg}, which shows that there is just one isomorphism class of each such subgroup, with the exception of $\quat[16]\bigast_{\quat} \quat[16]$, for which there are two isomorphism classes. Another result that shall be applied in \resecglobal{realisation}{isoclasses} is \repr{semiamalg} which implies that almost all elements of $\mathbb{V}_{2}(n)$ of the form $G\bigast_{H} G$ may be written as semi-direct products $\Z\rtimes G$. In \resecglobal{realisation}{genmcg} we will see that a similar result holds for the isomorphism classes of the Type~II subgroups of $\mcg$, the exceptional case being $\dih{8}\bigast_{\dih{4}} \dih{8}$. 

In \repart{generalities}, we then study the elements of $\mathcal{VC}$ of the form $F\rtimes_{\theta} \Z$, where $F$ is one of the finite groups occurring as a finite subgroup of $B_{n}(\St)$. One of the main difficulties that we face initially is that in general there are many possible actions of $\Z$ on $F$. However, as we shall see in Sections~\ref{sec:gencent}--\ref{sec:reducperiod}, a large number of these actions are incompatible with the structure of $B_{n}(\St)$. In \resec{gencent}, we prove \repr{genhodgkin1}, which will enable us to rule out the case where $F$ is a maximal finite cyclic or dihedral group. In \resec{autout}, we obtain a second reduction using the fact that the isomorphism class of $F\rtimes_{\theta} \Z$ depends only on the outer automorphism induced by $\theta(1)$ in $\out{F}$. Since we are primarily interested in the isomorphism classes of the virtually cyclic subgroups of $B_{n}(\St)$, it follows that it suffices to consider automorphisms of $F$ belonging to a transversal of $\out{F}$ in $\aut{F}$. The subsequent study of the structure of $\out{F}$, where $F$ is either $\quat$ or one of the binary polyhedral groups, then narrows down the possible Type~I subgroups of $B_{n}(\St)$. If $F=\tonestar,\oonestar$ or $\istar$ then $\out{F}\cong \Z_{2}$, so we have just two possible actions to consider, the trivial one, and a non-trivial one, which we shall describe. In \resec{conjfinite}, we obtain in \repr{genhodgkin2} an extension to $B_{n}(\St)$ of a result of Hodgkin concerning the centralisers of finite order elements of $\mcg$. This allows us to reduce greatly the number of possible actions in the case where $F$ is cyclic or dicyclic. In \resec{homotopytype}, in \repr{homot} we give an alternative proof of a result of~\cite{BCP,FZ} that says that if $n\geq 3$, the universal covering space of the $n\up{th}$ permuted configuration space $D_{n}(\St)$ of $\St$ has the homotopy type of $\St[3]$. This fact will then be used in \resec{percohI} to show in \relem{per24} that the non-trivial subgroups of $B_{n}(\St)$ have cohomological period $2$ or $4$. The ensuing study of the cohomology of the groups of the form $F\rtimes_{\theta}\Z$, where $F=\oonestar$ or $\istar$, will allow us to exclude the possibility of the non-trivial action in these cases. Putting together the analysis of Sections~\ref{sec:gencent}--\ref{sec:reducperiod} will lead us to the proof of \reth{main}(\ref{it:mainI}) for the Type~I subgroups. In \resec{typeII}, we study the infinite virtually cyclic groups of the form $G_{1}\bigast_{F} G_{2}$, where $F,G_{1},G_{2}$ are finite and $[G_{i}:F]=2$ for $i=1,2$. Using the cohomological properties obtained in \resec{percohI} and the relation with the groups of the form $F\rtimes_{\theta} \Z$, we show that any group of this form that is realised as a subgroup of $B_{n}(\St)$ is isomorphic to an element of $\mathbb{V}_{2}(n)$. This will enable us to prove \reth{main}(\ref{it:mainI}) in \resec{typeII}.

\section{Virtually cyclic groups: generalities}\label{sec:generalities}

We start by recalling the definition and Epstein and Wall's characterisation of virtually cyclic groups. We then proceed to prove some general results concerning these groups, notably Propositions~\ref{prop:isoamalg} and~\ref{prop:vcmcg}, that will be used in \repart{realisation} of the manuscript.

\begin{defn}
A group is said to be \emph{virtually cyclic} if it contains a cyclic subgroup of finite index.
\end{defn}

\begin{rems}\mbox{}\label{rem:vcgens}
\begin{enumerate}[(a)]
\item Every finite group is virtually cyclic.
\item Every infinite virtually cyclic group contains a normal subgroup of finite index.
\end{enumerate}
\end{rems}

The following criterion is well known; most of the first part is due to Epstein and Wall~\cite{Ep,W}.
\begin{thm}\label{th:wall}
Let $G$ be a group. Then the following statements are equivalent.
\begin{enumerate}[(a)]
\item\label{it:walla} $G$ is a group with two ends.
\item\label{it:wallb} $G$ is an infinite virtually cyclic group.
\item\label{it:wallc} $G$ has a finite normal subgroup $F$ such that $G/F$ is isomorphic to $\Z$ or to the infinite dihedral group $\Z_{2}\bigast \Z_{2}$.
\end{enumerate}
Equivalently, $G$ is of the form:
\begin{enumerate}[(i)]
\item\label{it:semi} $F \rtimes_{\theta} \Z$ for some action $\theta \in \operatorname{\text{Hom}}(\Z,\aut{F})$, or
\item\label{it:amalg} $G_{1} \bigast_{F} G_{2}$, where $[G_{i} : F]=2$ for $i=1,2$,
\end{enumerate}
where $G_{1},G_{2}$ and $F$ are finite groups. 
\end{thm}

\begin{defn}
An infinite virtually cyclic group will be said to be of \emph{Type~I} (resp.\ \emph{Type~II}) if it is of the form given by~(\ref{it:semi}) (resp.\ by~(\ref{it:amalg})).
\end{defn}

\begin{proof}[Proof of \reth{wall}.]
The equivalence of parts~(\ref{it:walla}) and~(\ref{it:wallb}) may be found in~\cite{Ep}, and the implication~(\ref{it:walla}) implies~(\ref{it:wallc}) is proved in~\cite{W}. So to prove the first part, it suffices to show that~(\ref{it:wallc}) implies~(\ref{it:wallb}). Suppose then that $G$ has a finite normal subgroup $F$ such that $G/F$ is isomorphic to $\Z$ or to $\Z_{2}\bigast \Z_{2}$. Clearly $G$ is infinite. Assume first that $G \cong F\rtimes_{\theta} \Z$, where $\theta\in \operatorname{Hom}(\Z,\aut{F})$, let $k$ be the order of the automorphism $\theta(1)\in \aut{F}$, let $\map{s}{G/F}[G]$ be a section for the canonical projection $\map{p}{G}[G/F]$, and let $x$ be a generator of the infinite cyclic group $G/F$. Since $\theta(x)(f)=s(x)f s(x^{-1})$ for all $f\in F$, it follows that the infinite cyclic subgroup $\ang{s(x^{k})}$ is central in $G$, and that there exists a commutative diagram of short exact sequences of the form:
\begin{equation}\label{eq:walltypeI}
\begin{xy}*!C\xybox{%
\xymatrix{%
& & 1 \ar[d] & 1 \ar[d]\\
& & \ang{s(x^{k})} \ar[r]^{p\left\lvert_{\ang{s(x^{k})}}\right.}_{\cong} \ar[d]  & \ar[d] \ang{x^{k}}\cong k\Z\\
1 \ar[r]  & F \ar[r] \ar[d]_{\phi\left\lvert_{F}\right.}^{\cong} & G \ar[d]^{\phi} \ar[r]^(.42){p} & G/F\cong \Z
\ar[d]^{\widehat{\phi}} \ar[r] & 1\\
1 \ar[r] & \ker{\widehat{p}} \ar[r] & G\left/\ang{s(x^{k})}\right. \ar[d]\ar@{.>}[r]_{\widehat{p}} & \Z/k\Z \ar[d]\ar[r] & 1,\\
& & 1 & 1}}
\end{xy}
\end{equation}
the left-hand vertical extension being central, where
\begin{equation*}
\text{$\map{\phi}{G}[G/\ang{s(x^{k})}]$ and $\map{\widehat{\phi}}{\Z}[\Z/k\Z]$}
\end{equation*}
are the canonical projections, and $\map{\widehat{p}}{G/\ang{s(x^{k})}}[\Z/k\Z]$ is the epimorphism induced on the quotients. Since the restriction of $p$ to $\ang{s(x^{k})}$ is an isomorphism, it follows that $\map{\phi\left\lvert_{F}\right.}{F}[\ker{\widehat{p}}]$ is too. Thus $G/\ang{s(x^{k})}$ is of order $k\ord{F}$. Since $\ang{s(x^{k})}$ is infinite cyclic, the left-hand vertical extension then implies that $G$ is virtually cyclic.

Now suppose that $G/F\cong\Z_{2}\bigast \Z_{2}$. Then $G/F\cong \Z \rtimes \Z_2$, where the action of $\Z_{2}$ on $\Z$ is non trivial. So there exist a short exact sequence
\begin{equation*}
1 \to F \to G \stackrel{p}{\to} \Z \rtimes \Z_2 \to 1
\end{equation*}
and a split extension
\begin{equation*}
1 \to F \to \widehat{G} \stackrel{p\left\lvert_{\widehat{G}}\right.}{\to} \Z \to 1,
\end{equation*}
where $\widehat{G}$ is the inverse image of the $\Z$-factor of $\Z\rtimes \Z_{2}$ under $p$. Let $x$ be a generator of $\Z$, and let $\map{s}{\Z}[\widehat{G}]$ be a section for $p\left\lvert_{\widehat{G}}\right.$. Applying the argument of the previous paragraph to $\widehat{G}\cong F\rtimes \Z$, there exists a central extension
\begin{equation*}
1 \to \ang{s(x^{k})} \to \widehat{G} \to \widehat{G}\left/\ang{s(x^{k})}\right. \to 1,
\end{equation*}
where $k$ is the order of $\aut{F}$. 
Let $m=\ord{F}$. We claim that $\ang{s(x^{mk})}$ is normal in $G_1\bigast_{F} G_{2}$. To see this, first note that $\ang{s(x^{mk})}\subset Z(\widehat{G})$. Now let $g\in G\setminus \widehat{G}$. Then $p(gs(x^{k})g^{-1})=x^{-k}$ since $p(g)$ is sent to an element of the form $(x^{q},\overline{1})$ in $\Z\rtimes \Z_{2}$, where $q\in \Z$. Hence $gs(x^{k})g^{-1}=s(x^{-k})f$, where $f\in F$. Since $s(x^{k})\in Z(\widehat{G})$, $s(x^{k})$ commutes with $f$, and so
\begin{equation}\label{eq:noncentral}
gs(x^{mk})g^{-1}=(s(x^{-k})f)^{m}=s(x^{-mk}).
\end{equation}
We thus have the following commutative diagram of short exact sequences:
\begin{equation}\label{eq:walltypeII}
\begin{xy}*!C\xybox{%
\xymatrix{%
& & 1 \ar[d] & 1 \ar[d]\\
1 \ar[r]  & \ang{s(x^{mk})} \ar[r] \ar@{=}[d] & \widehat{G} \ar[d] \ar[r] & \widehat{G}\left/\ang{s(x^{mk})}\right.\ar[d] \ar[r] & 1\\
1\ar[r] & \ang{s(x^{mk})} \ar[r] & G \ar[d] \ar[r] & G \left/\ang{s(x^{mk})}\right. \ar[d] \ar[r] & 1. \\
&  & \Z_{2} \ar[d]\ar@{=}[r]& \Z_{2} \ar[d] &\\
& & 1 & 1}}
\end{xy}
\end{equation}
An argument similar to that of the previous paragraph shows that $\ord{\widehat{G}\left/\ang{s(x^{mk})}\right.}=m^{2}k$, and so $\ord{G\left/\ang{s(x^{mk})}\right.}=2m^{2}k$. Since $\ang{s(x^{mk})}\cong \Z$, it follows from the second row of~\reqref{walltypeII} that $G$ is virtually cyclic. This shows that~(\ref{it:wallc}) implies~(\ref{it:wallb}), and thus completes the proof of the first part of the statement.

We now prove the second statement of the theorem. First note that in part~(\ref{it:wallc}), the fact that $G/F$ is isomorphic to $\Z$ is clearly equivalent to condition~(\ref{it:semi}). Suppose then that condition~(\ref{it:amalg}) holds. Since $[G_{i}:F]=2$ for $i=1,2$, $F$ is normal in $G_{i}$, so is normal in $G=G_{1} \bigast_{F} G_{2}$, and $G/F\cong \Z_{2}\bigast \Z_{2}$. Finally, suppose that $G$ has a finite normal subgroup $F$ such that $G/F\cong\Z_{2}\bigast \Z_{2}$. Let $\map{\Pi}{G}[G/F]$ denote the canonical projection. For $i=1,2$, let $y_{i}\in G/F$ be such that $G/F=\setangr{y_{1},y_{2}}{y_{1}^{2}=y_{2}^{2}=1}$, and let $G_{i}=\Pi^{-1}(\ang{y_{i}})$. Then the groups $G_{i}$ are finite and each contain $F$ as a subgroup of index $2$. We can thus form the amalgamated product $G_{1}\bigast_{F} G_{2}$. So $F$ is normal in $G_{1}\bigast_{F} G_{2}$, and the quotient $(G_{1}\bigast_{F} G_{2})/F$ is isomorphic to $\Z_{2}\bigast \Z_{2}$. By standard properties of amalgamated products, there exists a unique (surjective) homomorphism $\map{\phi}{G_{1}\bigast_{F} G_{2}}[G]$ that makes the following diagram of short exact sequences commutative:
\begin{equation*}
\xymatrix{%
1 \ar[r]  & F \ar[r] \ar@{=}[d] & G_1\bigast_{F}G_2 \ar[d]^{\phi} \ar[r]^(.42)q & (G_{1}\bigast_{F} G_{2})/F
\ar@{.>}[d]^{\widehat{\phi}} \ar[r] & 1\\
1 \ar[r] & F \ar[r] & G \ar[r]^(.45){\Pi} & G/F \ar[r] & 1,}
\end{equation*}
$q$ being the canonical projection, and where $\widehat{\phi}$ is the induced homomorphism on the quotients. Now for $i=1,2$, $\phi(g)=g$ for all $g\in G_{i}$, and so $\widehat{\phi}(q(x_{i}))=y_{i}$. In particular, $\widehat{\phi}$ sends the $\Z_{2}$-factors of $(G_{1}\bigast_{F} G_{2})/F$ isomorphically onto those of $G/F$, and thus $\widehat{\phi}$ is an isomorphism. It follows from the $5$-Lemma that $\phi$ is also an isomorphism.
%
\end{proof}

The following result shows that the type of an infinite virtually cyclic group is determined by the (non) centrality of the extension given by \reth{wall}(\ref{it:wallb}).

\begin{prop}\label{prop:centralext}
Let $G$ be an infinite virtually cyclic group. Then $G$ is of Type~I (resp.\ of Type~II) if and only if the extension
\begin{equation}\label{eq:seszgf}
1\to \Z\to G\to F\to 1
\end{equation}
arising in the definition of virtually cyclic group is central (resp.\ is not central).
\end{prop}

\begin{proof}
In order to prove the proposition, we start by showing that if
\begin{equation*}
1\to \Z \stackrel{\iota_{j}}{\to} G \to F_{j} \to 1 \quad\text{for $j=1,2$,}
\end{equation*}
are extensions of $G$, with $F_{j}$ finite, then they are either both central or both non central. Note that the intersection $\iota_1(\Z) \cap \iota_2(\Z)$ is a normal subgroup of $G$ of finite index, and so is infinite cyclic. Since an automorphism of $\Z$ is completely determined by its restriction to the subgroup $k\Z\subset \Z$ for any $k\neq 0$, (as the automorphism and its restriction are either both equal to $\id$ or to $-\id$), the two extensions are thus either both central or both non central. 

To prove the necessity of the condition, consider the extension~\reqref{seszgf} given by the definition of virtually cyclic group. Assume first that $G$ is of Type~I. By the first part of \reth{wall}, there exists a finite subgroup $F'$ of $G$ and $\theta\in \operatorname{Hom}(\Z,\aut{F'})$ such that $F'\rtimes_{\theta} \Z$. Using the notation of the first part of the proof of \reth{wall}, as in the commutative diagram~\reqref{walltypeI}, we obtain a central extension
\begin{equation*}
1\to \ang{s(x^{k})}\to G\to G\left/\ang{s(x^{k})}\right.\to 1.
\end{equation*}
Since $\ang{s(x^{k})}\cong \Z$, it follows from the first paragraph that the extension~\reqref{seszgf} is central.

Now suppose that $G$ is of Type~II. From the proof of the first part of \reth{wall}, from the commutative diagram~\reqref{walltypeII} we obtain an extension
\begin{equation*}
1\to \ang{s(x^{mk})}\to G\to G\left/\ang{s(x^{mk})}\right.\to 1,
\end{equation*}
where $\ang{s(x^{mk})}\cong \Z$, $G\left/\ang{s(x^{mk})}\right.$ is finite. Equation~\reqref{noncentral} implies that this extension is non central. Using the first paragraph once more, it follows that the extension~\reqref{seszgf} is non central.
This proves the necessity of the conditions.

%

Conversely, if the extension~\reqref{seszgf} is central (resp.\ non central)  then from \reth{wall}, it must be of Type~I (resp.\ Type~II) because as we saw in the two previous paragraphs, any group of Type~I (resp.\ Type~II) is the middle group of a central (resp.\ non central) extension. But by the first paragraph of this proof, this property is independent of the short exact sequence.
\end{proof}

The following proposition will be used in \resecglobal{realisation}{isoclasses} to give an alternative description of the elements of $\mathbb{V}_{2}(n)$ as semi-direct products.

\begin{prop}\label{prop:semiamalg}
Let $G_{1}$ and $G_{2}$ be isomorphic groups, and consider the amalgamated product $G=G_{1}\bigast_{H} G_{2}$ defined by 
\begin{equation*}
\xymatrix{%
& H_{1} \ar[r] & G_{1} \ar[rd] & \\
H \ar[ru]^{i_{1}} \ar[rd]_{i_{2}} & & & G_1\bigast_{H} G_2,\\
& H_{2} \ar[r] & G_{2} \ar[ru]  & }
\end{equation*}
where for $j=1,2$, $H_{j}$ is a subgroup of $G_{j}$ of index $2$ and $\map{i_{j}}{H}[H_{j}]$ is an embedding of the abstract group $H$ in $G_{j}$, the remaining arrows being inclusions. Suppose that the isomorphism $\map{i_{2}\circ i_{1}^{-1}}{H_{1}}[H_{2}]$ extends to an isomorphism $\map{\iota}{G_{1}}[G_{2}]$. Then $G \cong \Z \rtimes G_{i}$, where the action is given by
\begin{equation}\label{eq:actamalg}
g_{i} t g_{i}^{-1}=
\begin{cases}
t & \text{if $g_{i}\in H_{i}$}\\
t^{-1} & \text{if $g_{i}\in G_{i}\setminus H_{i}$,}
\end{cases}
\end{equation}
$t$ being a generator of the $\Z$-factor.
\end{prop}

\begin{proof}
We start by constructing a homomorphism $\map{\alpha}{G_{1}\bigast_{H} G_{2}}[G_{2}]$. It suffices to define $\alpha$ on the elements of $G_{1}$ and $G_{2}$. Let
\begin{equation*}
\alpha(x)=
\begin{cases}
\iota(x) & \text{if $x\in G_{1}$}\\
x & \text{if $x\in G_{2}$.}
\end{cases}
\end{equation*}
Then $\alpha$ is well defined, since if $h\in H$ then $\alpha(i_{1}(h))=\iota(i_{1}(h))=i_{2}(h)=\alpha(i_{2}(h))$ since $i_{j}(h)\in H_{j}$ for $j\in \brak{1,2}$.
Hence we obtain a split short exact sequence:
\begin{equation}\label{eq:sesalpha}
1 \to \ker{\alpha} \to G_{1}\bigast_{H} G_{2} \to G_{2} \to 1,
\end{equation}
where a section $\map{s}{G_{2}}[G_{1}\bigast_{H} G_{2}]$ is just given by inclusion. It remains to show that $\ker{\alpha}\cong \Z$, and to determine the action. 

Let $\map{p}{G_{1}\bigast_{H} G_{2}}[\Z_{2}\bigast \Z_{2}]$ be the canonical projection of $G_{1}\bigast_{H} G_{2}$ onto the quotient $(G_{1}\bigast_{H} G_{2})/H$. If $h\in H$ then $\alpha(i_{2}(h))=i_{2}(h)$, so the lower left-hand square of the following diagram of short exact sequences is commutative:
\begin{equation*}
\xymatrix{%
& & 1\ar[d] & 1\ar[d] &\\
& & \ker{\alpha}\ar[d] & \Z\ar[d] & \\
1 \ar[r] & H \ar[r]^{i_{2}} \ar@{=}[d] & G_{1}\bigast_{H} G_{2} \ar[r]^{p} \ar[d]^{\alpha} & \Z_{2} \bigast \Z_{2} \ar[r] \ar@{.>}[d]^{\widehat{\alpha}} & 1\\
1 \ar[r] & H \ar[r]^{i_{2}} \ar[d] &  G_{2} \ar[r]^{p\left\lvert_{G_{2}}\right.}  \ar[d] & \Z_{2} \ar[r]\ar[d] & 1.\\
& 1 & 1 & 1 &}
\end{equation*}
Thus $\alpha$ induces a homomorphism $\map{\widehat{\alpha}}{\Z_{2} \bigast \Z_{2}}[\Z_{2}]$ that makes the lower right-hand square commute. Let $i\in \brak{1,2}$, and suppose that $g_{i}\in G_{i}\setminus H_{i}$. If $i=1$ then $\alpha(g_{1})=\iota(g_{1})\in G_{2}\setminus H_{2}$ because $\iota$ is an isomorphism that sends $H_{1}$ to $H_{2}$, while if $i=2$ then $\alpha(g_{2})=g_{2}\in G_{2}\setminus H_{2}$. We conclude that $p(\alpha(g_{i}))=\overline{1}$. Setting $x_{i}=p(g_{i})$, the commutativity of the above diagram implies firstly that $\widehat{\alpha}(x_{i})=\overline{1}$, and hence $\ker{\widehat{\alpha}}= \ang{x_{1}x_{2}}\cong \Z$, and secondly that the restriction
\begin{equation*}
\map{p\left\lvert_{\ker{\alpha}}\right.}{\ker{\alpha}}[\ker{\widehat{\alpha}}]
\end{equation*}
is an isomorphism and that $\ker{\alpha}=\ang{g_{1}(\iota(g(1)))^{-1}}\cong \Z$ for any $g_{1}\in G_{1}\setminus H_{1}$. Thus $G_{1}\bigast_{H} G_{2}\cong \Z\rtimes G_{2}$ by \req{sesalpha}. Further, if $g_{2}\in G_{2}$ then
\begin{align*}
p\biggl(g_{2} \bigl(g_{1}(\iota(g(1)))^{-1}\bigr) g_{2}^{-1}\biggr)&= p(g_{2}) x_{1}x_{2} p(g_{1}^{-1})\\
&=
\begin{cases}
x_{1}x_{2} & \text{if $g_{2}\in i_{2}(H)=H_{2}$}\\
x_{2}x_{1}x_{2}x_{2}^{-1}=x_{2}x_{1}= (x_{1}x_{2})^{-1} & \text{if $g_{2}\in G_{2}\setminus H_{2}$.}
\end{cases}
\end{align*}
The action given by \req{actamalg} then follows from the commutativity of the above diagram, where $t$ is taken to be the element $g_{1}(\iota(g(1)))^{-1}$.
\end{proof}

We now turn our attention to the virtually cyclic subgroups of $B_{n}(\St)$. 

\begin{defn}
Given $n\geq 4$, let $\mathcal{VC}$ denote the family of virtually cyclic groups consisting of all groups of Type~I and Type~II whose factors $F,G_{1}$ and $G_{2}$, as described by \reth{wall}, are subgroups of $\Z_{2(n-1)}$, $\dic{4n}$, $\dic{4(n-2)}$, $\tonestar$, $\oonestar$ or $\istar$.
\end{defn}

The family $\mathcal{VC}$ thus consists of the infinite virtually cyclic groups that are formed using the finite subgroups of $B_{n}(\St)$. The following proposition is an immediate consequence of Theorems~\ref{th:finitebn} and~\ref{th:wall}.
\begin{prop}\label{prop:possvcbnS2}
Let $G$ be a virtually cyclic subgroup of $B_n(\St)$. 
\begin{enumerate}[(a)]
\item\label{it:parta} If $G$ is finite then it is isomorphic to a subgroup of one of $\Z_{2(n-1)}$, $\dic{4n}$, $\dic{4(n-2)}$, $\tonestar$, $\oonestar$ or $\istar$.
\item\label{it:partb} If $G$ is infinite then it is isomorphic to an element of $\mathcal{VC}$.
\end{enumerate}
\end{prop}

We recall the following general result from~\cite{GG8}, which will prove to be very useful when it comes to constructing subgroups of $B_{n}(\St)$ of Type~II.

\begin{prop}[{\cite[Lemma~15]{GG8}}]\label{prop:infincard}
Let $G=G_{1} \bigast_{F} G_{2}$ be a virtually cyclic group of Type~II, and let
$\map{\phi}{G_{1} \bigast_{F} G_{2}}[H]$ be a homomorphism such that for $i=1,2$, the restriction of $\phi$
to $G_i$ is injective. Then $\phi$ is injective if and only if $\phi(G)$ is infinite.
\end{prop}

\begin{rem}
\repr{infincard} will be applied in the following manner: we will be given finite subgroups $\widetilde{G}_{1}, \widetilde{G}_{2}$ of $B_{n}(\St)$ such that $\widetilde{F}=\widetilde{G}_{1}\bigcap \widetilde{G}_{2}$ is of index two in both $\widetilde{G}_{1}$ and $\widetilde{G}_{2}$. The aim will be to prove that the subgroup $\ang{\widetilde{G}_{1}\bigcup \widetilde{G}_{2}}$ is the amalgamated product of $\widetilde{G}_{1}$ and $\widetilde{G}_{2}$ along $\widetilde{F}$. It will suffice to show that $\ang{\widetilde{G}_{1}\bigcup \widetilde{G}_{2}}$ is infinite. Suppose that this is indeed the case. Let $G_{1}$ and $G_{2}$ be abstract groups isomorphic respectively to $\widetilde{G}_{1}$ and $\widetilde{G}_{2}$ whose intersection is an index two subgroup $F$. We define a map $\map{\phi}{G_{1} \bigast_{F} G_{2}}[\ang{\widetilde{G}_{1}\bigcup \widetilde{G}_{2}}]$ that sends $F$ onto $\widetilde{F}$ and $G_{i}$ onto $\widetilde{G}_{i}$ isomorphically for $i=1,2$. Then $\phi$ is a surjective homomorphism, and by \repr{infincard}, is an isomorphism.
\end{rem}

As an easy exercise, we may deduce the classification of the virtually cyclic subgroups of $P_{n}(\St)$. If $n\leq 3$ then $P_{n}(\St)$ is trivial if $n\leq 2$, and $P_{3}(\St)\cong \Z_{2}$. So suppose that $n\geq 4$. The only finite subgroups of $P_{n}(\St)$ are $\brak{e}$ and $\ang{\ft}$, both of which are central. 
\begin{prop}
Let $n\geq 4$. The virtually cyclic subgroups of $P_{n}(\St)$ are $\brak{e}$, $\ang{\ft}$, $\ang{x}\cong \Z$ and
$\ang{\ft,x} \cong \Z_{2} \times \Z$, where $x$ is any element of $P_{n}(\St) \setminus \brak{\ft}$.
\end{prop}

\begin{proof}
Let $G$ be an infinite virtually cyclic subgroup of $P_{n}(\St)$. The Type~I subgroups are $\Z$ and $\Z_{2} \times \Z$ (both are realised, by taking $\ang{x}$ and $\ang{\ft, x}$ respectively, where $x$ is any element of $P_{n}(\St)\setminus \ang{\ft}$). As for the Type~II subgroups, the only possibility is $F=\brak{e}$ and $G_{1}= G_{2}= \ang{\ft}$, but then $G \ncong \Z_{2} \bigast \Z_{2}$ since $\ang{\ft}$ is the unique subgroup of $P_{n}(\St)$ of order two. 
\end{proof}



The following result will be used later on to show that there is an almost one-to-one correspondence between the virtually cyclic subgroups of $B_{n}(\St)$ and those of $\mcg$. This will also enable us to construct copies of $\tonestar\times \Z$ (\repr{ttimesz}) and $\oonestar \bigast_{\tonestar} \oonestar$ (\repr{oto}) in $B_{n}(\St)$ for certain values of $n$, as well as to prove \reth{classvcmcg}.

\begin{prop}\label{prop:vcmcg}
Let $G$ be a group that possesses a unique element $x$ of order $2$, let $G'=G/\ang{x}$, and let $\map{p}{G}[G']$ denote the canonical projection. 
\begin{enumerate}[(a)]
\item\label{it:vcmcga} Let $H$ be a virtually cyclic subgroup of $G$. 
\begin{enumerate}[(i)]
\item\label{it:vcmcgai} $H'=p(H)$ is a virtually cyclic subgroup of $G'$ of the same type (finite, of Type~I or of Type~II) as $H$.
\item\label{it:vcmcgaii} Let $H\cong F\rtimes_{\theta} \Z$, where $F$ is a finite subgroup of $G$ and $\theta\in \operatorname{Hom}(\Z,\aut{F})$. Then $p(H)\cong p(F)\rtimes_{\theta'} \Z$, where $\theta'\in \operatorname{Hom}(\Z,\aut{F'})$ is the action induced by $\theta$, and defined by $\theta'(1)(f')=p\bigl(\theta(1)(f)\bigr)$ for all $f'\in F'$, where $f\in F$ satisfies $p(f)=f'$.

\item\label{it:vcmcgaiii} Let $H\cong G_{1} \bigast_{F} G_{2}$, where $G_{1},G_{2}$ are subgroups of $H$, and $F=G_{1} \cap G_{2}$ is of index $2$ in $G_{1}$ and $G_{2}$. Then $p(H)\cong p(G_{1})\bigast_{p(F)} p(G_{2})$.
\end{enumerate}

\item\label{it:vcmcgb} Let $H'$ be a virtually cyclic subgroup of $G'$.
\begin{enumerate}
\item\label{it:vcmcgbi} $H=p^{-1}(H')$ is a virtually cyclic subgroup of $G$ of the same type (finite, of Type~I or of Type~II) as $H'$.
\item\label{it:vcmcgbii} If $H' \cong G_{1}' \bigast_{F'} G_{2}'$, where $G_{1}',G_{2}'$ are subgroups of $H'$, and $F'=G_{1}' \cap G_{2}'$ is of index $2$ in $G_{1}'$ and $G_{2}'$, then $H\cong p^{-1}(G_{1}') \bigast_{p^{-1}(F')} p^{-1}(G_{2}')$.
\end{enumerate}
\item\label{it:vcmcgc} Let $H_{1}$ and $H_{2}$ be isomorphic subgroups of $G$. Then $p(H_{1})$ and $p(H_{2})$ are isomorphic subgroups of $G'$.
\end{enumerate}
\end{prop}


\begin{proof}
First note that since $x$ is the unique element of $G$ of order $2$, the subgroup $\ang{x}$ is characteristic in $G$, in particular, $x\in Z(G)$.
We start by proving parts~(\ref{it:vcmcga})(\ref{it:vcmcgai}) and~(\ref{it:vcmcgb})(\ref{it:vcmcgbi}). The result is clear if either $H$ or $H'$ is finite, so it suffices to consider the cases where they are infinite. Before proving the statement in these cases, let us introduce some notation. Suppose that $H$ (resp.\ $H'$) is an infinite virtually cyclic subgroup of $G$ (resp.\ $G'$). Then by \reth{wall}, $H$ (resp.\ $H'$) has a finite normal subgroup $F$ (resp.\ $F'$) such that $H/F$ (resp.\ $H'/F'$) is isomorphic to $\Z$ if $H$ (resp.\ $H'$) is of Type~I, and to $\Z_{2}\bigast \Z_{2}$ if $H$ (resp.\ $H'$) is of Type~II.
Let $H'=p(H)$ and $F'=p(F)$ (resp.\ $H=p^{-1}(H')$ and $F=p^{-1}(F')$). So $H'$ (resp.\ $H$) is infinite, and $F'$ (resp.\ $F$) is finite. Further, $\map{{p\left\lvert_{F}\right.}}{F}[F']$ and $\map{{p\left\lvert_{H}\right.}}{H}[H']$ are surjective, $F'$ (resp.\ $F$) is normal in $H'$ (resp.\ $H$), and
\begin{equation}\label{eq:inclex}
\brak{e}\subset \ker{p\left\lvert_{F}\right.}\subset \ker{p\left\lvert_{H}\right.} \subset \ker{p}=\ang{x}. 
\end{equation}
Then we have the following commutative diagram of short exact sequences:
\begin{equation}\label{eq:mcgvc1}
\begin{xy}*!C\xybox{%
\xymatrix{%
& 1 \ar[d] & 1 \ar[d] & &\\
& \ker{p\left\lvert_{F}\right.} \ar[d] \ar[r] & \ker{p\left\lvert_{H}\right.} \ar[d] & &\\
1 \ar[r]  & F \ar[r] \ar[d]_{p\left\lvert_{F}\right.} & H \ar[d]^{p\left\lvert_{H}\right.} \ar[r]^{q} & H/F \ar@{.>}[d]^{\widehat{p}} \ar[r] & 1\\
1 \ar[r] & F' \ar[r] \ar[d] & H' \ar[r]^{q'} \ar[d] & H'/F' \ar[r] & 1,\\
& 1 & 1 &&
}}
\end{xy}
\end{equation}
where $\map{q}{H}[H/F]$ and $\map{q'}{H'}[H'/F']$ are the canonical projections, the map
\begin{equation*}
\map{\widehat{p}}{H/F}[H'/F']
\end{equation*}
is the induced surjective homomorphism on the quotients and $\ker{p\left\lvert_{F}\right.} \to \ker{p\left\lvert_{H}\right.}$ is inclusion. We claim that $\ker{p\left\lvert_{F}\right.}= \ker{p\left\lvert_{H}\right.}$. This being the case, $\widehat{p}$ is an isomorphism, and thus $H$ and $H'$ are virtually cyclic groups of the same type, which proves the proposition. If $x\notin H$ then $\ker{p\left\lvert_{F}\right.}= \ker{p\left\lvert_{H}\right.}=\brak{e}$ trivially by \req{inclex}. So assume that $x\in H$. To prove the claim, by \req{inclex}, it suffices to show that $x\in F$. We separate the two cases corresponding to parts~(\ref{it:vcmcga})(\ref{it:vcmcgai}) and~(\ref{it:vcmcgb})(\ref{it:vcmcgbi}) of the statement.
\begin{enumerate}
\item[(\ref{it:vcmcga})(\ref{it:vcmcgai})] If $H$ is of Type~I then $H\cong F\rtimes \Z$, and so $x\in F$ since $x$ is of finite order. So suppose that $H$ is of Type~II. Then $H\cong G_{1}\bigast_{F} G_{2}$, where $G_{1},G_{2}$ are subgroups of $H$ that contain $F$ as a subgroup of index $2$. 
By standard properties of amalgamated products, $x$ belongs to a conjugate in $H$ of one of the $G_{i}$ because it is of finite order, and since $x\in Z(G)$, it belongs to one of the $G_{i}$, which shows that $G_{1}$ and $G_{2}$ are of (the same) even order. The fact that $x$ is the unique element of $G$ of order $2$ implies that $x\in G_{1}\cap G_{2}=F$ as required.
\item[(\ref{it:vcmcgb})(\ref{it:vcmcgbi})] In this case, $\ker{p\left\lvert_{F}\right.}= \ker{p\left\lvert_{H}\right.}=\ker{p}=\ang{x}$ by construction.
\end{enumerate}
This proves the claim, and thus we obtain parts~(\ref{it:vcmcga})(\ref{it:vcmcgai}) and~(\ref{it:vcmcgb})(\ref{it:vcmcgbi}).

We now prove part~(\ref{it:vcmcga})(\ref{it:vcmcgaii}). Let $H$ be an infinite Type~I subgroup of $G$ and let $F$ be a finite normal subgroup of $H$ such that there exists a short exact sequence of the form
\begin{equation*}
1\to F \to H\stackrel{q}{\to} H/F \to 1,
\end{equation*}
where $H/F\cong \Z$, and where $\map{q}{H}[H/F]$ is the canonical projection. By the previous paragraph, we thus have the commutative diagram~\reqref{mcgvc1}, $\widehat{p}$ being an isomorphism. Let $z$ be a generator of $H/F$, let $\map{s}{H/F}[H]$ be a section for $q$ such that $\theta(z)(f)=s(z)\ldotp f \ldotp s(z^{-1})$ for all $f\in F$, where $\theta\in \operatorname{Hom}(H/F,\aut{F})$ is given. The commutativity of the diagram~\reqref{mcgvc1} implies that $\map{s'=p\circ s\circ \widehat{p}^{-1}}{H'/F'}[H']$ is a section for $q'$. Since $x\in Z(G)$, if $x\in F$ then $\theta(z)(x)=x$, and so $p$ induces a homomorphism $\map{\Phi}{\aut{F}}[\aut{F'}]$ satisfying $\Phi(\alpha)(p(f))=p(\alpha(f))$ for all $f\in F$ and $\alpha\in \aut{F}$. We thus obtain a homomorphism $\map{\theta'}{H'/F'}[\aut{F'}]$ defined by $\theta'=\Phi\circ \theta \circ \widehat{p}^{-1}$ that makes the following diagram commute:
\begin{equation*}
\begin{xy}*!C\xybox{%
\xymatrix{%
H/F \ar[r]^>>>>>>{\theta}\ar[d]^{\widehat{p}} & \aut{F} \ar[d]^{\Phi}\\
H'/F' \ar@{.>}[r]^>>>>>{\theta'} & \aut{F'}.}}
\end{xy}
\end{equation*}
In particular, if $f'\in F'$ and if $f\in F$ is such that $p(f)=f'$ then:
\begin{align*}
s'(z') \ldotp f'\ldotp s'(z'^{-1})&=p\circ s(z)\ldotp p(f) \ldotp p\circ s(z^{-1})= p (s(z)\ldotp f \ldotp s(z^{-1}))=p (\theta(z))(f)\\
&= \Phi\circ \theta(z)(f')=\theta'(z')(f'),
\end{align*}
and thus $H'\cong F'\rtimes_{\theta'} \Z$, where $\theta'\in \operatorname{Hom}(\Z,\aut{F'})$ is the homomorphism induced by 
$\theta\in \operatorname{Hom}(\Z,\aut{F})$ given by $\theta'(1)(f')=p(\theta(1)(f))$ for all $f'\in F'$, where $f\in F$ satisfies $p(f)=f'$, and where we write the generators of $H/F$ and $H'/F'$ as $1$. This proves part~(\ref{it:vcmcga})(\ref{it:vcmcgaii}).

We now prove part~(\ref{it:vcmcga})(\ref{it:vcmcgaiii}). Let $H,G_{1},G_{2}$ and $F$ be as in the statement, and let $H',G_{1}',G_{2}'$ and $F'$ be their respective images under $p$. Then $H/F\cong \Z_{2}\bigast \Z_{2}$, and once more we have the commutative diagram~\reqref{mcgvc1}, $\widehat{p}$ being an isomorphism. By part~(\ref{it:vcmcga})(\ref{it:vcmcgai}), $H'$ is a Type~II subgroup of $G$. Now $F'$ is of index $2$ in both $G_{1}'$ and $G_{2}'$, and the inclusions $F'\subset G_{i}'$ give rise to an amalgamated product $G_{1}' \bigast_{F'} G_{2}'$ whose quotient by $F'$ is isomorphic to $\Z_{2}\bigast \Z_{2}$. Since $H=\ang{G_{1}\cup G_{2}}$ and $\map{p\left\lvert_{H}\right.}{H}[H']$ is surjective, we have that $H'=\ang{G_{1}'\cup G_{2}'}$. By the universality property of amalgamated products, there exists a surjective homomorphism $\map{\alpha}{G_1'\bigast_{F'} G_2'}[H']$ satisfying $\alpha(g_{i}')=g_{i}'$ for all $g_{i}'\in G_{i}'$. We thus obtain the following commutative diagram of short exact sequences:
\begin{equation*}
\begin{xy}*!C\xybox{%
\xymatrix{%
1 \ar[r]  & F' \ar[r] \ar@{=}[d] & G_1'\bigast_{F'} G_2' \ar[d]^{\alpha} \ar[r] & \Z_{2}\bigast \Z_{2} \ar@{.>}[d]^{\widehat{\alpha}} \ar[r] & 1\\
1 \ar[r] & F' \ar[r] & H' \ar[r]^{q'} & H'/F' \ar[r] & 1,}}
\end{xy}
\end{equation*}
where $\widehat{\alpha}$ is the homomorphism induced on the quotients. The surjectivity of $\alpha$ implies that of $\widehat{\alpha}$. The finiteness of $\Z_{2}$ implies that the free product $\Z_{2}\bigast \Z_{2}$, which is finitely generated, is residually finite~\cite[Proposition~22]{Coh}. It thus follows that $H'/F'\cong \Z_{2}\bigast \Z_{2}$ is Hopfian~\cite[see the proof of the Corollary, page~12]{Coh}, so $\widehat{\alpha}$ is an isomorphism. Using the $5$-Lemma, we see that $\alpha$ is an isomorphism as required. 

We now prove part~(\ref{it:vcmcgb})(\ref{it:vcmcgbii}). For $i=1,2$, let $G_{i}=p^{-1}(G_{i}')$. Since $H' \cong G_{1}' \bigast_{F'} G_{2}'$ and $\widehat{p}$ is an isomorphism, we have that $H'/F'\cong H/F\cong \Z_{2}\bigast \Z_{2}$. Now $F$ is a subgroup of both $G_{1}$ and $G_{2}$, and the corresponding inclusions give rise to an amalgamated product $G_{1}\bigast_{F} G_{2}$ whose quotient by $F$ is isomorphic to $\Z_{2}\bigast \Z_{2}$. The equality $H'=\ang{G_{1}' \cup G_{2}'}$ implies that $H=\ang{G_{1} \cup G_{2}}$, and it follows from the universality property of amalgamated products that there exists a (unique) surjective homomorphism $\map{\alpha}{G_{1}\bigast_{F} G_{2}}[H]$ satisfying $\alpha(g_{i})=g_{i}$ for all $g_{i}\in G_{i}$. We thus have a commutative diagram of the form
\begin{equation*}
\begin{xy}*!C\xybox{%
\xymatrix{%
1 \ar[r]  & F \ar[r] \ar@{=}[d] & G_{1}\bigast_{F} G_{2} \ar[d]^{\alpha} \ar[r] & \Z_{2}\bigast \Z_{2} \ar@{.>}[d]^{\widehat{\alpha}} \ar[r] & 1\\
1 \ar[r] & F \ar[r] & H \ar[r]^{q} & H/F \ar[r] & 1,}}
\end{xy}
\end{equation*}
where $\widehat{\alpha}$ is the homomorphism induced by $\alpha$ that makes the diagram commute, which is surjective because $\alpha$ is, and is thus an isomorphism since $\Z_{2}\bigast \Z_{2}$ is Hopfian. The $5$-Lemma implies the result.

Finally, we prove part~(\ref{it:vcmcgc}). Let $\map{\psi}{H_{1}}[H_{2}]$ be an isomorphism between $H_{1}$ and $H_{2}$. Since $x$ is the  unique element of $G$ of order $2$, then $x\in H_{1}$ if and only if $x\in H_{2}$, and since $\ker{p}=\ang{x}$, we have $\ker{p\left\lvert_{H_{1}}\right.}=\ker{p\left\lvert_{H_{2}}\right.}$. We thus have the following commutative diagram of short exact sequences:
\begin{equation*}
\begin{xy}*!C\xybox{%
\xymatrix{%
1 \ar[r]  & \ker{p\left\lvert_{H_{1}}\right.} \ar[r] \ar@{=}[d] & H_{1} \ar[d]^{\cong}_{\psi} \ar[r]^{p} & p(H_{1}) \ar@{.>}[d]^{\widehat{\psi}} \ar[r] & 1\\
1 \ar[r] & \ker{p\left\lvert_{H_{2}}\right.} \ar[r] & H_{2} \ar[r]^{p} & p(H_{2}) \ar[r] & 1,}}
\end{xy}
\end{equation*}
where $\map{\widehat{\psi}}{p(H_{1})}[p(H_{2})]$ is the surjective homomorphism induced by $\psi$. The $5$-Lemma then implies that $\widehat{\psi}$ is an isomorphism.
\end{proof}

We thus obtain directly \repr{corrbnmcg}:
\begin{proof}[Proof of \repr{corrbnmcg}.]
Taking $G=B_{n}(\St)$, $G'=\mcg$ and $\phi$ as given in \req{mcg}, and applying \repr{vcmcg} yields the result.
\end{proof}

We finish this section with the following result that will be applied in \resecglobal{realisation}{isoclasses} to study the isomorphism classes of the elements of $\mathbb{V}_{2}(n)$. For $k=1,2$, let $G_k, F$ be finite groups such that $F$ is abstractly isomorphic to a subgroup of $G_{k}$ of index $2$, and let $\map{i_{k},j_{k}}{F}[G_k]$ be pairs of embeddings. We can then form two amalgamated products, $G_1\bigast_{F}G_2$ (with respect to the embeddings $i_1,i_2$) and $G_1\bigast_{F}'G_2$ (with respect to the embeddings $j_1,j_2$).  Suppose that for $k=1,2$, there exist automorphisms $\map{\theta_k}{G_k}$ satisfying $\theta_k \circ i_k=j_k$.

\begin{prop}\label{prop:amalgiso}
Under the above hypotheses, the two amalgamated products $G_1\bigast_{F}G_2$ and $G_1\bigast_{F}'G_2$ are isomorphic.
\end{prop}

\begin{proof}
The hypotheses imply the existence of the following commutative diagram:
\begin{equation*}
\xymatrix{%
& G_{1} \ar[r]^{\theta_{1}} \ar[rd] & G_{1} \ar[rd] & \\
F \ar[ru]^{i_{1}} \ar[rd]_{i_{2}} &  & G_1\bigast_{F} G_2 \ar@{.>}[r] &  G_1\bigast_{F}'G_2,\\
& G_{2} \ar[r]_{\theta_{2}} \ar[ru] & G_{2} \ar[ru] & }
\end{equation*}
where for $l=1,2$, the homomorphisms from $G_{l}$ to $G_1\bigast_{F} G_2$ and $G_1\bigast_{F}'G_2$ are inclusions. By the universal property of amalgamated products, there exists a unique (surjective) homomorphism $G_1\bigast_{F} G_2\to G_1\bigast_{F}'G_2$. We obtain the inverse of this homomorphism in a similar manner, by replacing $i_{1}$, $i_{2}$, $\theta_{1}$ and $\theta_{2}$ by $j_{1}$, $j_{2}$, $\theta_{1}^{-1}$ and $\theta_{2}^{-1}$ respectively and by exchanging the rôles of $G_1\bigast_{F} G_2$ and $G_1\bigast_{F}'G_2$.
\end{proof}


\section{Centralisers and normalisers of some maximal finite subgroups of $B_{n}(\St)$}\label{sec:gencent}

\reth{murasugi} asserts that up to conjugacy, the maximal finite order cyclic subgroups of $B_{n}(\St)$ are of the form $\ang{\alpha_{i}}$ for $i\in\brak{0,1,2}$. On the other hand,~\cite[Theorem~1.3 and Proposition~1.5(1)]{GG7} implies that up to conjugacy, the maximal dicyclic subgroups of $B_{n}(\St)$ are the standard dicyclic subgroups of \rerem{finitesub}(\ref{it:finitesubb}). In this section, we determine the centralisers and normalisers of these subgroups.  In the cyclic case, our results mirror those for finite order elements of $\mcg$, and shall be used to construct the possible actions of $\Z$ on cyclic and dicyclic subgroups of $B_{n}(\St)$.

We first prove the following proposition which states that an infinite subgroup of $B_{n}(\St)$ cannot be formed solely of elements of finite order.
\begin{prop}\label{prop:finord}
Any infinite subgroup of $B_{n}(\St)$ contains an element of infinite order. In particular, any subgroup of $B_{n}(\St)$ consisting entirely of elements  of finite order is itself finite.
\end{prop}

\begin{proof}
Let $H$ be an infinite subgroup of $B_{n}(\St)$. Consider the following restriction of the short exact sequence~\reqref{defperm}:
\begin{equation*}
1 \to P_{n}(\St)\cap H \to H \stackrel{\pi\left\lvert_{H}\right.}\to \pi(H) \to 1,
\end{equation*}
where $\pi(H)$ is a subgroup of $\sn$. If $P_{n}(\St)\cap H$ is finite then it follows that $H$ is finite, which contradicts the hypothesis. So $P_{n}(\St)\cap H$ is infinite, but since the torsion of $P_{n}(\St)$ is precisely $\brak{e, \ft}$, $H$ must contain an element of infinite order.
\end{proof}

The following lemma will play a fundamental rôle in the rest of the paper.
\begin{lem}\label{lem:funda}
Let $i\in \brak{0,1,2}$. Then:
\begin{gather}
\alpha_{i}^l \sigma_{j} \alpha_{i}^{-l}=\sigma_{j+l} \quad \text{for all $j,l\in \N$ satisfying $j+l\leq n-i-1$,}\label{eq:fundaa}\\
\sigma_{1}=\alpha_{i}^2 \sigma_{n-i-1} \alpha_{i}^{-2}.\label{eq:fundab}
\end{gather}
Further, if $0\leq q\leq n$, we have:
\begin{equation}\label{eq:alpha0q}
\alpha_{0}^q=(\sigma_{1}\cdots \sigma_{q-1})^q \cdot \prod_{k=1}^{q} (\sigma_{q-k+1}\cdots \sigma_{n-k}).
\end{equation}
\end{lem}

\begin{rems}\mbox{}\label{rem:funda}
\begin{enumerate}
\item\label{it:fundaa} An alternative formulation of equations~\reqref{fundaa} and~\reqref{fundab} is that conjugation by $\alpha_{i}$ permutes the $n-i$ elements
\begin{equation*}
\sigma_{1},\ldots, \sigma_{n-i-1}, \alpha_{i} \sigma_{n-i-1} \alpha_{i}^{-1}
\end{equation*}
cyclically. 
\item \label{it:fundac} If $0\leq q\leq n$ then using \req{alpha0q}, $\alpha_{0}^q$ may be interpreted geometrically as a full twist on the first $q$ strings, followed by the passage of these $q$ strings over the remaining $n-q$ strings (see Figure~\ref{fig:alpha0q} for an example). If further $q$ divides $n$ then $\alpha_{0}^q$ admits a block structure (see also \rerem{nt}(\ref{it:nta})).
\end{enumerate}
\end{rems}

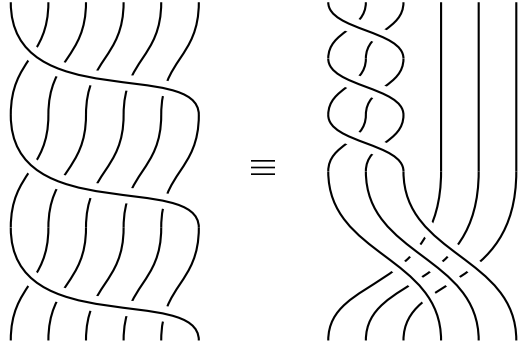
\begin{figure}[h]
\hfill
\begin{tikzpicture}[scale=0.5]
\foreach \k in {6,3,0}
{\foreach \j in {2,3,4,5,6}
{\draw[thick] (\j,\k).. controls (\j,\k-1.5) and (\j-1,\k-1.5) .. (\j-1,\k-3);};
\draw[white,line width=6pt] (1,\k) .. controls (1,\k-3) and (6,\k-1.5) .. (6,\k-3);
\draw[thick] (1,\k) .. controls (1,\k-3) and (6,\k-1.5) .. (6,\k-3);};
\end{tikzpicture}\quad
\begin{tikzpicture}[scale=0.5]
\draw (0,1.5) node {\large$\equiv$};
\draw (0,-3) node {};
\end{tikzpicture}\quad
\begin{tikzpicture}[scale=0.5]
\foreach \j in {1,2,3}
{\draw[thick] (\j+3,1.5).. controls (\j+3,-1.75) and (\j,-1.25) .. (\j,-3);};
\foreach \k in {6,4.5,3}
\foreach \j in {2,3}
{\draw[thick] (\j,\k).. controls (\j,\k-0.75) and (\j-1,\k-0.75) .. (\j-1,\k-1.5);};;
\foreach \k in {6,4.5,3}
{\draw[white,line width=6pt] (1,\k).. controls (1,\k-0.75) and (3,\k-0.75) .. (3,\k-1.5);
\draw[thick] (1,\k).. controls (1,\k-0.75) and (3,\k-0.75) .. (3,\k-1.5);};
\foreach \j in {1,2,3}
{\draw[thick] (\j+3,6) -- (\j+3,1.5);
\draw[white,line width=6pt] (\j,1.5).. controls (\j,-0.75) and (\j+3,-0.75) .. (\j+3,-3);
\draw[thick] (\j,1.5).. controls (\j,-0.75) and (\j+3,-0.75) .. (\j+3,-3);
};
\end{tikzpicture}
\hspace*{\fill}
\caption{The braid $\alpha_{0}^{3}$ in $B_{6}(\St)$, first in its usual form, and then in the form $(\sigma_{1}\sigma_{2})^{3} (\sigma_{3}\sigma_{4}\sigma_{5}) (\sigma_{2}\sigma_{3}\sigma_{4}) (\sigma_{1}\sigma_{2}\sigma_{3})$ of \req{alpha0q}.}\label{fig:alpha0q}
\end{figure}


\begin{proof}[Proof of \relem{funda}.]
Let $i\in \brak{0,1,2}$. We start by establishing equations~\reqref{fundaa} and~\reqref{fundab}. First note that $\alpha_{1}=\alpha_{0}\sigma_{n-1}$ and $\alpha_{2}=\alpha_{0}\sigma_{n-1}^{-1}\sigma_{n-2}$, so if $1\leq j\leq n-i-2$, $\alpha_{i} \sigma_{j} \alpha_{i}^{-1}=\alpha_{0} \sigma_{j} \alpha_{0}^{-1}= \sigma_{j+1}$ using standard properties of $\alpha_{0}$. If further $l\in \N$ and $j+l\leq n-i-1$, $\alpha_{i}^l \sigma_{j} \alpha_{i}^{-l}=\sigma_{j+l}$, which proves \req{fundaa}. Since $n-i-2\geq 0$, we obtain
\begin{equation*}
\sigma_{1}=\alpha_{i}^{n-i} \sigma_{1} \alpha_{i}^{-(n-i)}= \alpha_{i}^2 \alpha_{i}^{n-i-2} \sigma_{1} \alpha_{i}^{-(n-i-2)} \alpha_{i}^{-2}=\alpha_{i}^2 \sigma_{n-i-1} \alpha_{i}^{-2},
\end{equation*}
using equations~\reqref{uniqueorder2} and~\reqref{fundaa}, which proves \req{fundab}. We now prove \req{alpha0q}. Let us prove by induction that for all $m\in \brak{0,\ldots,q}$,
\begin{equation}\label{eq:induction}
\alpha_{0}^q=(\sigma_{1}\cdots \sigma_{q-1})^m \alpha_{0}^{q-m} \cdot \prod_{k=1}^{m} (\sigma_{m-k+1}\cdots \sigma_{n-q+m-k}).
\end{equation}
Clearly the equality holds if $m=0$. So suppose that it is true for some $m\in \brak{0,\ldots,q-1}$. Then $q-(m+1)\geq 0$, and:
\begin{align*}
\alpha_{0}^q &=(\sigma_{1}\cdots \sigma_{q-1})^m \alpha_{0}^{q-m} \cdot \prod_{k=1}^{m} (\sigma_{m-k+1}\cdots \sigma_{n-q+m-k})\\
&=(\sigma_{1}\cdots \sigma_{q-1})^{m+1} \sigma_{q}\cdots \sigma_{n-1} \alpha_{0}^{q-(m+1)} \cdot \prod_{k=1}^{m} (\sigma_{m-k+1}\cdots \sigma_{n-q+m-k})\\
&=(\sigma_{1}\cdots \sigma_{q-1})^{m+1} \alpha_{0}^{q-(m+1)} \sigma_{m+1}\cdots \sigma_{n-q+m} \cdot \prod_{k=2}^{m+1} (\sigma_{m-k+2}\cdots \sigma_{n-q+m-k+1})\\
&=(\sigma_{1}\cdots \sigma_{q-1})^{m+1} \alpha_{0}^{q-(m+1)} \cdot \prod_{k=1}^{m+1} (\sigma_{(m+1)-k+1}\cdots \sigma_{n-q+(m+1)-k})
\end{align*}
using \req{fundaa}, which gives \req{induction}. Taking $m=q$ in that equation yields \req{alpha0q}.
\end{proof}

As well as being of interest in its own right, the following result will prove to be useful when we come to discussing the possible Type~I subgroups whose finite factor is cyclic. If $H$ is a subgroup of a group $G$ then we denote the centraliser (resp.\ normaliser) of $H$ in $G$ by $Z_{G}(H)$ (resp.\ $N_{G}(H)$).
\begin{prop}\label{prop:luis}
Let $n\geq 4$, and let $i\in \brak{0,1,2}$. Then $Z_{B_{n}(\St)}(\ang{\alpha_{i}})=\ang{\alpha_{i}}$.
\end{prop}


In order to prove \repr{luis}, we first state a result due to L.~Hodgkin concerning the centralisers of finite order elements in $\mcg$. 

\begin{prop}[\cite{Ho}]\label{prop:hodgkin1}
Let $n\geq 3$, let $\gamma\in \mcg$ be an element of finite order $r\geq 2$, and let $f$ be a rotation of $\St$ by angle $2\pi m/r$ about the axis passing through the poles which represents $\gamma$, where $\gcd{(m,r)}=1$. Let $\Lambda$ be the subgroup of the mapping class group of the quotient space $\St/\ang{f}$ whose set of marked points is the union of the image of the $n$ marked points under the quotient map $\St\to \St/\ang{f}$ with the two poles of $\St/\ang{f}$, and whose elements fix these two poles if $r\neq 2$ or $r$ divides $n-1$, and leaves the set of poles invariant if $r=2$ and $r$ does not divide $n-1$. Then there is an exact sequence 
\begin{equation}\label{eq:mcghodgkin}
1 \to \Z_{r}\to Z_{\mcg}(\ang{\gamma}) \to \Lambda \to 1.
\end{equation}
\end{prop}

\begin{rem}
Hodgkin's proof of the result is for elements of prime power order~\cite[Proposition~2.5]{Ho}, but one may check that it holds for any finite order element.
\end{rem}


We now come to the proof of \repr{luis}.

\begin{proof}[Proof of \repr{luis}.]
Let $z$ belong to $Z_{B_{n}(\St)}(\ang{\alpha_{i}})$. We start by showing that either $Z_{\mcg}(\ang{a_{i}})=\ang{a_{i}}$, or in the case $n=4$, $i=2$, the possibility that $Z_{\mcg[4]}(\ang{a_{2}})\cong \Z_{2}\oplus \Z_{2}$ is also allowed. Consider the short exact sequence~\reqref{mcg}.
Take $m=1$ and $r=n-i\geq 2$ in the statement of \repr{hodgkin1}. Up to conjugacy, we may suppose that $\phi(\alpha_{i})=a_{i}$, where we denote the mapping class of the rotation $f$ of that proposition by $a_{i}$. Then $\St/\ang{f}$ may be regarded as a sphere with three marked points, two of which are the poles, and the other marked point corresponds to the single orbit of $r$ marked points in $\St$. Suppose first that $r\neq 2$ or $i=1$ (in the latter case, $r$ clearly divides $n-1$). Then $\Lambda$ is the subgroup of the mapping class group of $\St/\ang{f}$ whose elements fix each of the poles, as well as the remaining marked point. Hence $\Lambda$ is the pure mapping class group of $\St/\ang{f}$, which is trivial. It follows from \req{mcghodgkin} that $Z_{\mcg}(a_{i})$ is cyclic of order $r$, and so is equal to $\ang{a_{i}}$. Now suppose that $r=2$ and $i\neq 1$. Since $n\geq 3$, we have that $n=4$ and $i=2$. In this case, $\Lambda$ is the subgroup of the mapping class group of $\St/\ang{f}$ whose elements leave the set of poles invariant (and fix the remaining marked point), and so is isomorphic to $\Z_{2}$. By \req{mcghodgkin}, $Z_{\mcg}(\ang{a_{2}})$ is an extension of $\Z_{2}$ by $\Z_{2}$, thus $Z_{\mcg[4]}(\ang{a_{2}})$ is isomorphic to either $\Z_{4}$ or $\Z_{2}\oplus \Z_{2}$. In the former case, we obtain $Z_{\mcg[4]}(\ang{a_{2}})=\ang{a_{2}}$. 

We first consider the case where $Z_{\mcg}(\ang{a_{i}})=\ang{a_{i}}$ (so either $r\neq 2$ or $i=1$, or $n=4$, $i=2$ and $Z_{\mcg[4]}(\ang{a_{2}})=\ang{a_{2}}$). Now $z'=\phi(z)$ belongs to the centraliser of $a_{i}$, and so may be written in the form $z'=a_{i}^t$, where $t\in \brak{0,\ldots,n-i-1}$. By equations~\reqref{uniqueorder2} and~\reqref{mcg}, $z=\alpha_{i}^t \ftalt{\epsilon}=\alpha_{i}^{t+ \epsilon(n-i)}$, where $\epsilon\in \brak{0,1}$, and hence $z\in \ang{\alpha_{i}}$. Since $Z_{B_{n}(\St)}(\ang{\alpha_{i}})$ clearly contains $\ang{\alpha_{i}}$, it follows that  $Z_{B_{n}(\St)}(\ang{\alpha_{i}})=\ang{\alpha_{i}}$ as required.

Finally, suppose that $n=4$, $i=2$ and $Z_{\mcg[4]}(\ang{a_{2}})\cong \Z_{2}\oplus \Z_{2}$. Since
\begin{equation*}
\phi(Z_{B_{4}(\St)}(\ang{\alpha_{2}}))\subset Z_{\mcg[4]}(\ang{a_{2}}),
\end{equation*}
we have
\begin{equation*}
\ang{\alpha_{2}} \subset Z_{B_{4}(\St)}(\ang{\alpha_{2}}) \subset \phi^{-1}(Z_{\mcg[4]}(\ang{a_{2}}))\cong \quat
\end{equation*}
using \req{mcg}. If $Z_{B_{4}(\St)}(\ang{\alpha_{2}})$ is isomorphic to $\quat$ then there exists at least one element of $Z_{B_{4}(\St)}(\alpha_{2})$ that does not commute with $\alpha_{2}$, which is a contradiction. So $\ang{\alpha_{2}} = Z_{B_{4}(\St)}(\ang{\alpha_{2}})$, and this completes the proof of the proposition.
\end{proof}



\begin{rem}
As we shall see in \resecglobal{realisation}{dirtoi}, in general the binary polyhedral groups $\tonestar,\oonestar$ and $\istar$ have infinite centraliser in $B_{n}(\St)$.
\end{rem}

We now proceed with the proof of \repr{genhodgkin1}.

\begin{proof}[Proof of \repr{genhodgkin1}.] 
We first deal with the case $n=3$ for both parts~(\ref{it:centalphai}) and~(\ref{it:normcyclic}). We have $\alpha_{0}=\sigma_{1}\sigma_{2}$, $\alpha_{1}=\sigma_{1}\sigma_{2}^{2}$ and $\alpha_{2}=\sigma_{1}^{2}$, which are of order $6,4$ and $2$ respectively. In particular, $\alpha_{2}=\ft[3]$, and so the centraliser of $\alpha_{2}$ and the normaliser of $\ang{\alpha_{2}}$ are both equal to $B_{3}(\St)$. As for $\alpha_{0}$ and $\alpha_{1}$, they may be taken to be the generators $x$ and $y$ of $B_{3}(\St)\cong \dic{12}$ appearing in \req{presdic}. Indeed, $\alpha_{0}^{3}=\alpha_{1}^{2}=\ft[3]$ by \req{uniqueorder2}, $\ang{\alpha_{0},\alpha_{1}}=B_{3}(\St)$ since $\ang{\alpha_{0},\alpha_{1}}$ cannot be of order less than $12$, and 
\begin{equation*}
\alpha_{1}\alpha_{0}\alpha_{1}^{-1}= \sigma_{1}\sigma_{2}^{2} \sigma_{1}\sigma_{2} \sigma_{2}^{-2}\sigma_{1}^{-1}= \sigma_{1}\sigma_{2}^{2} \sigma_{1} \sigma_{2}^{-1}\sigma_{1}^{-1} =\sigma_{2}^{-1}\sigma_{1}^{-1} \quad\text{by \req{surface}.}
\end{equation*}
It follows from the presentation of \req{presdic} that $\ang{\alpha_{0},\alpha_{1}}\cong \dic{12}$, and the rest of the statement follows using this presentation in the case $m=3$.

We suppose henceforth that $n\geq 4$.
\begin{enumerate}[(a)]
\item This is the statement of \repr{luis}.

\item Let $i\in {0,1,2}$, let $N=N_{B_{n}(\St)}(\ang{\alpha_{i}})$, and let $x\in N$. Then some power of $x$ belongs to the centraliser of $\ang{\alpha_{i}}$ in $B_{n}(\St)$, which is equal to $\ang{\alpha_{i}}$ by \repr{luis}. So $x$ is of finite order, and thus $N$ is finite by \repr{finord}. Let 
\begin{equation*}
G=
\begin{cases}
\ang{\alpha_{0},\garside} \cong \dic{4n} & \text{if $i=0$}\\
\ang{\alpha_{2},\alpha_{0}^{-1}\garside \alpha_{0}} \cong \dic{4(n-2)} & \text{if $i=2$}\\
\ang{\alpha_{1}}\cong \Z_{2(n-1)} & \text{if $i=1$.}
\end{cases}
\end{equation*}
If $i\in \brak{0,2}$ then $G$ is conjugate to the standard copy of $\dic{4(n-i)}$. Since $[G:\ang{\alpha_{i}}]\in \brak{1,2}$, $\ang{\alpha_{i}}$ is normal in $G$, and so $N \supset G$. If $G=N$ then we are done. So suppose that $G\neq N$, and let $M$ be a maximal finite subgroup containing $N$. Hence $G$ is not maximal, and by \reth{finitebn}, we are in one of the following cases:
\begin{enumerate}[(i)]
\item\label{it:n4i12} $n=4$ and $i\in\brak{1,2}$. If $i=1$ (resp.\ $i=2$) then $G\cong \Z_6$ (resp.\ $G\cong \quat$). Since $G \subsetneqq N \subset M$, it follows from \reth{finitebn} and the subgroup structure of the finite maximal subgroups of $B_{4}(\St)$ (see \repr{maxsubgp}) that $N=M$ and $N\cong \tonestar$ (resp.\ $N\cong \tonestar$ or $N\cong \quat[16]$). Now $\ang{\alpha_{i}}$ is isomorphic to $\Z_{6}$ (resp.\ to $\Z_{4}$), but these subgroups are not normal in $\tonestar$, so $N\ncong \tonestar$. Hence $i=2$ and $N\cong \quat[16]$. Take $N$ to have the presentation~\reqref{presdic} with $m=4$. Since $\ang{\alpha_{2}}$ is isomorphic to $\Z_{4}$ and is normal in $N$, we have that $\ang{\alpha_{2}}=\ang{x^2}$. Hence $\ang{\alpha_2} \subsetneqq \ang{x}$, but this contradicts the fact that $\ang{\alpha_2}$ is maximal cyclic in $B_{4}(\St)$ by \reth{murasugi}.

\item\label{it:n6i2} $n=6$ and $i=2$. Then $\ang{\alpha_2}\cong \Z_8$ and $G\cong \quat[16]$. Since the maximal finite subgroups of $B_{6}(\St)$ are isomorphic to $\dic{24}$, $\Z_{10}$ or $\oonestar$ by \reth{finitebn}, and $G \subsetneqq N \subset M$, it follows that $N=M\cong \oonestar$. However, this contradicts the fact that the copies of $\Z_8$ in $\oonestar$ are not normal by \repr{maxsubgp}. This completes the proof of part~(\ref{it:normcyclic}).
\end{enumerate}

\item Let $i\in \brak{0,2}$, let $G$ denote the standard copy of $\dic{4(n-i)}$, and let $N$ be the normaliser $N_{B_{n}(\St)}(G)$ of $G$ in $B_{n}(\St)$. If $x\in N$ then some power of $x$ centralises $G$, and so centralises its cyclic subgroup $\ang{\alpha_{0}\alpha_{i} \alpha_{0}^{-1}}$ of order $2(n-i)$. It follows from part~(\ref{it:normcyclic}) that $N$ is finite. Since $N\supset G$, if $G$ is maximal finite then $G=N$, and we are done. So suppose that $G$ is not maximal, and let $M$ be a finite maximal subgroup of $B_{n}(\St)$ satisfying $G \subsetneqq N \subset M$. \reth{finitebn} implies that $n\in\brak{4,6}$ and $i=2$.

Suppose first that $n=4$, so $G\cong \quat$. Then $M$ is isomorphic to $\tonestar$ or $\quat[16]$ by \reth{finitebn}, and $N=M$ by \repr{maxsubgp}. Suppose first that $N\cong \tonestar \cong \quat \rtimes \Z_{3}$. Then $G$ is the unique subgroup of $M$ isomorphic to $\quat$. The form of the action of $\Z_{3}$ on $\quat$ implies that the elements of $G$ of order $4$ are pairwise conjugate. However, this is impossible since the  permutations of the order $4$ elements $\alpha_{2}$ and $\alpha_{0}^{-1}\garside[4]\alpha_{0}$ of $G$ have distinct cycle types. Thus $N\cong \quat[16]$. By~\cite[Proposition~1.5 and Theorem~1.6]{GG7}, the standard copy $\ang{\alpha_{0},\garside[4]}$ of $\quat[16]$ in $B_{4}(\St)$, which is a representative of the unique conjugacy class of subgroups isomorphic to $\quat[16]$, contains representatives of the two conjugacy classes of $\quat$, from which it follows that there exists a subgroup $K$ of $B_{4}(\St)$ conjugate to $\ang{\alpha_{0},\garside[4]}$ and containing $G$. Since $[K:G]=2$, $G$ is normal in $K$, so $K\subset N$. The maximality of $\quat[16]$ as a finite subgroup of $B_{4}(\St)$ implies that $K=N\cong \quat[16]$. Furthermore, we claim that $K=\alpha_{0}^{-1} \sigma_{1}^{-1} \ang{\alpha_{0},\garside[4]} \sigma_{1}\alpha_{0}$. Indeed, $K$ has a subgroup isomorphic to $\quat$ that is generated by the following two elements:
\begin{align*}
\alpha_{0}^{-1}\sigma_{1}^{-1} \alpha_{0}^2 \sigma_{1}\alpha_{0}&=\alpha_{0}^{-1}\alpha_{0}^2 \sigma_{3}^{-1} \sigma_{1}= \alpha_{0}^{-1}\garside[4]\alpha_{0}\\
\alpha_{0}^{-1}\sigma_{1}^{-1} \alpha_{0}\garside[4] \sigma_{1}\alpha_{0}&=\alpha_{0}^{-1}\sigma_{1}^{-1} \alpha_{0} \sigma_{3} \garside[4]\alpha_{0}=
\alpha_{0}^{-1} \sigma_{2}\sigma_{3}^2 \alpha_{0}\ldotp \alpha_{0}^{-1} \garside[4]\alpha_{0}=\alpha_{2} \ldotp \alpha_{0}^{-1} \garside[4]\alpha_{0},
\end{align*}
which are also generators of $G$. We have used relations~\reqref{defgarside} and~\reqref{garsideconj}, as well as \relem{funda} to obtain these equalities. This proves the claim, and completes the proof in the case $n=4$.


Finally, suppose that $n=6$, so $G\cong \quat[16]$. If $G \subsetneqq N$ then as in part~(\ref{it:normcyclic})(\ref{it:n6i2}) it follows that $N\cong \oonestar$. But by \repr{maxsubgp}, the copies of $\quat[16]$ in $\oonestar$ are not normal, which yields a contradiction. We thus conclude that $G=N$ as required.\qedhere
\end{enumerate}
\end{proof}

\section{Reduction of isomorphism classes of $F\rtimes_{\theta} \Z$ via $\out{F}$}\label{sec:autout}






If $F$ is a group, let $\inn{F}$ denote the normal subgroup of inner automorphisms of the group $\aut{F}$ of automorphisms of $F$, and recall that $\inn{F}\cong F/Z(F)$, where $Z(F)$ denotes the centre of $F$. By \reth{wall}, any Type~I group is of the form $F\rtimes_{\theta} \Z$ for some action $\theta\in \operatorname{\text{Hom}}(\Z,\aut{F})$, where $F$ is finite. The following proposition asserts that the isomorphism class of such a group depends only on the homomorphism $\map{\overline{\theta}}{\Z}[\out{F}]$ which is the composition of $\theta$ with the canonical projection $\aut{F} \to \out{F}$.

\begin{prop}[{\cite[Chapter~1.2, Proposition~12]{AB}}]\label{prop:isoout}
Let $F$ be a finite group, and let
\begin{equation*}
\map{\theta,\theta'}{\Z}[\aut{F}]
\end{equation*}
be homomorphisms such that $\overline{\theta}=\overline{\theta'}$. Then the groups $F\rtimes_{\theta}\Z$ and $F\rtimes_{\theta'}\Z$ are isomorphic. 
\end{prop}

In order to help us determine (up to isomorphism) the possible Type~I groups arising as subgroups of $B_{n}(\St)$, it will be appropriate at this juncture to describe $\out{F}$, where $F$ is one of the finite subgroups $\quat,\tonestar,\oonestar$ or $\istar$ of $B_{n}(\St)$. By choosing a transversal in $\aut{F}$ of $\out{F}$, from \repr{isoout} we may obtain all possible isomorphism classes of the groups $F\rtimes_{\theta} \Z$ (we shall always choose the identity as the representative of the trivial element of $\out{F}$). It then follows directly from \repr{possvcbnS2}(\ref{it:partb}) that any Type~I subgroup of $B_{n}(\St)$ involving $F$ is isomorphic to one of the groups belonging to this family. Cohomological considerations will then be applied in \resec{reducperiod} to rule out those subgroups involving $\oonestar$ and $\istar$ for all but the trivial action. Note that we could carry out the study of $\out{F}$ for the other finite subgroups of $B_{n}(\St)$, but in \resec{conjfinite} we will obtain stronger conditions on the possible actions of $\Z$ on $F$ using \repr{genhodgkin2}.


\begin{enumerate}[(1)]

\item\label{it:caseq8} $F=\quat$: we have $\aut{\quat}\cong \sn[4]$~\cite[p.~149]{AM}, $Z(\quat)\cong\Z_2$ and $\inn{\quat}\cong \Z_2 \oplus\Z_2$. Therefore $\out{\quat}\cong \sn[4]/(\Z_2 \oplus\Z_2)\cong \sn[3]$.

\item\label{it:casetstar} $F=\tonestar$. Writing $\quat=\brak{\pm 1,\pm i,\pm j,\pm k}$, it is well known that $\tonestar$ is isomorphic to $\quat\rtimes \Z_{3}$~\cite{AM,CM}, where the action of $\Z_{3}$ permutes cyclically the elements $i,j,k$ of $\quat$. From \cite[Theorem~3.3]{gg5}, we have $\aut{\tonestar} \cong \sn[4]$.  Now $Z(\tonestar)\cong \Z_2$, so $\inn{\tonestar}\cong (\Z_2 \oplus \Z_2)\rtimes \Z_3\cong \an[4]$, where the action permutes cyclically the three non-trivial elements of $\Z_2 \oplus \Z_2$. Therefore $\out{\tonestar}\cong \Z_2$. Let $\tonestar$ be given by the presentation \reqref{preststar}. 
%
The non-trivial element of $\out{\tonestar}$ is represented by the automorphism $\omega(1)$ of $\tonestar$ defined by \req{nontrivacttstar}.
Indeed, if $S\in\ang{P,Q}$ then any automorphism of $\tonestar$ which sends $X$ to $SX^{-1}$ is not an inner automorphism. This follows since $PXP^{-1}=PQ^{-1}X$, and $QXQ^{-1}=P^{-1}X$, so any conjugate of $X$ in $\tonestar$ belongs to the coset $\ang{P,Q}\ldotp X$, but on the other hand, $SX^{-1}$ belongs to the coset $\ang{P,Q}\ldotp X^{-1}$ which is distinct from $\ang{P,Q}\ldotp X$. As we shall see presently in case~(\ref{it:caseostar}), the automorphism given by~\reqref{nontrivacttstar} is the restriction to $\tonestar$ of conjugation by an element $R\in \oonestar \setminus \tonestar$.

\item\label{it:caseostar} $F=\oonestar$: from~\cite[p.~198]{Wo}, $\oonestar$ is generated by $X,P,Q,R$ which are subject to the following relations:
\begin{equation}\label{eq:presostar}
\left\{ 
\begin{aligned}
& X^{3}=1,\,P^2=Q^2=R^2,\,PQP^{-1}=Q^{-1},\\
& XPX^{-1}=Q,\,XQX^{-1}=PQ,\\
& RXR^{-1}=X^{-1},\, RPR^{-1}=QP,\,RQR^{-1}=Q^{-1}.
\end{aligned}
\right.
\end{equation}
Comparing the presentations given by equations~(\ref{eq:preststar}) and~(\ref{eq:presostar}), we see that $\oonestar$ admits $\ang{P,Q,X}\cong \tonestar$ as its index~$2$ subgroup. So we have the extensions~\cite[page~150]{AM}:
\begin{equation*}
1\to \tonestar \to \oonestar \to \Z_2 \to 1,
\end{equation*}
and~\cite[Proposition~4.1]{gg5}:
\begin{equation}\label{eq:autot}
1\to \Z_2 \to \aut{\oonestar} \to \aut{\tonestar} \to 1.
\end{equation}
Now $Z(\oonestar)\cong \Z_2$, and so $\inn{\oonestar}\cong \oonestar/Z(\oonestar)\cong \sn[4]$. From \req{autot} and part~(\ref{it:casetstar}) above, $\ord{\aut{\oonestar}}=48$, and thus $\out{\oonestar}\cong \Z_2$.
%
The non-trivial element of $\out{\oonestar}$ is represented by the following element of $\aut{\oonestar}$: 
\begin{equation*}
\left\{
\begin{aligned}
&P \mapsto P\\
&Q \mapsto Q\\
&X \mapsto X\\
&R \mapsto R^{-1}.
\end{aligned}
\right.
\end{equation*}
To see this, suppose on the contrary that this automorphism arises as conjugation by some element $S\in \oonestar$. Since $[\oonestar:\tonestar]=2$ and $R\notin \tonestar$, there exists $t\in \tonestar$ such that $S=t$ or $S=tR$. If $S=t$ then $S$ commutes with all of the generators of $\tonestar$, hence belongs to the centre $\ang{P^2}$ of $\tonestar$. But $Z(\oonestar)= \ang{P^2}$, so conjugation by $S$ cannot send $R$ to $R^{-1}$ since $R$ is of order $4$. Thus $S=tR$, and so $X=SXS^{-1}= tX^{-1}t^{-1}$, but this implies that $X$ and $X^{-1}$ belong to the same conjugacy class in $\tonestar$, and as we saw in case~(\ref{it:casetstar}) above, this is impossible. We conclude that the given automorphism is not an inner automorphism, so must represent the non-trivial element of $\out{\oonestar}$.

\item $F=\istar$: we know that $\istar \cong \operatorname{SL}_2(\FF_{5})$~\cite[page~151]{AM}, $Z(\istar)\cong \Z_{2}$, $\inn{\istar}\cong \istar/Z(\istar)\cong \an[5]$, $\aut{\istar}\cong \sn[5]$~\cite[see Theorem~2.1]{gg5}, and $\out{\istar}\cong \Z_2$~(see \cite[page~151]{AM} or \cite[page~207]{gg4}). The non-trivial element of $\out{\istar}$ is represented by the automorphism of $\istar$ which in terms of $\operatorname{SL}_2(\FF_{5})$ is conjugation by the matrix $\left( \begin{smallmatrix}
w & 0\\
0 & 1
\end{smallmatrix}
\right)$, where $w$ is a non square of $\FF_{5}$~\cite[page~152]{AM}. 
\end{enumerate}

Let us come back to case~(\ref{it:caseq8}) 
where $F=\quat$. Since $\out{\quat}\cong\sn[3]$, \emph{a priori}, we need to decide which of the six groups of the form $\quat\rtimes \Z$ are realised. We may however make a minor simplification as follows. Recall from \redef{v1v2}(\ref{it:mainq8}) that $\alpha,\beta \in \operatorname{Hom}(\Z, \aut{\quat})$ are such that $\alpha(1)$ is the automorphism of $\quat$ of order $3$ that permutes $i,j$ and $k$ cyclically, and $\beta(1)$ is the automorphism that sends $i$ to $k$ and $j$ to $j^{-1}$. The following lemma shows that we may reduce further the number of isomorphism classes of $\quat\rtimes \Z$ from the six representatives of the elements of $\out{\quat}$ to just three.


\begin{lem}\label{lem:quatact}
Let $H$ be of the form $\quat\rtimes_{\theta} \Z$, where $\theta\in \operatorname{Hom}(\Z,\aut{\quat})$. Then $H$ is isomorphic to one of $\quat\times \Z$, $\quat\rtimes_{\alpha} \Z$ and $\quat\rtimes_{\beta} \Z$.
\end{lem}

\begin{proof}
Since $\out{\quat}\cong \sn[3]$, there exists $\gamma\in \brak{\id,\alpha,\alpha^2,\beta, \alpha\circ \beta,\alpha^2\circ \beta}$ such that $H$ is isomorphic to $\quat \rtimes_{\gamma} \Z$ by \repr{isoout}.
We claim that:
\begin{enumerate}[(a)]
\item\label{it:caseq8i} $\quat \rtimes_{\alpha} \Z$ and $\quat \rtimes_{\alpha^{2}} \Z$ are isomorphic.
\item\label{it:caseq8ii} $\quat \rtimes_{\beta} \Z$, $\quat \rtimes_{\alpha \circ\beta} \Z$ and $\quat \rtimes_{\alpha^{2} \circ\beta} \Z$ are isomorphic.
\end{enumerate}
To prove the claim, we define isomorphisms $\map{\phi}{\quat \rtimes_{\theta} \Z}[\quat \rtimes_{\theta'} \Z]$, where the actions $\theta,\theta'\in \operatorname{Hom}(\Z,\aut{\quat})$ run through the possible pairs given by~(\ref{it:caseq8i}) and~(\ref{it:caseq8ii}). Let $t$ (resp.\ $t'$) denote the generator of the $\Z$-factor of $\quat \rtimes_{\theta} \Z$ (resp.\ of $\quat \rtimes_{\theta'} \Z$). Defining $\phi$ by:
\begin{enumerate}[(i)]
\item $i \mapsto i$, $j \mapsto k$ and $t \mapsto kt'$ if $\theta=\alpha$ and $\theta'=\alpha^2$,
\item $i\mapsto k$, $k\mapsto j$, $j\mapsto i$ and $t\mapsto jt'$ if $\theta= \beta$ and $\theta'=\alpha \circ \beta$,
\item $i\mapsto j$, $k\mapsto i$, $j\mapsto k$ and $t\mapsto jt'$ if $\theta= \beta $ and $\theta'=\alpha^2 \circ \beta$,
\end{enumerate}
we may check that $\phi$ gives rise to an isomorphism between each pair of groups. In particular, there are only three isomorphism classes of semi-direct products $\quat \rtimes_{\theta} \Z$, namely those for which $\theta(1)\in \brak{\id, \alpha(1),\beta(1)}$.
\end{proof}

\begin{rem}
Since $(\alpha(1))^{3}=(\beta(1))^{2}=\id_{F}$, it will suffice to study the existence of semi-direct products of the form $\quat\rtimes_{\alpha}\Z$ and $\quat\rtimes_{\beta}\Z$.
\end{rem}

\section{Reduction of isomorphism classes of $F\rtimes_{\theta} \Z$ via conjugacy classes}\label{sec:conjfinite}


In this section, we use the relation between $\mcg$ and $B_{n}(\St)$ given by \req{mcg} to prove \repr{genhodgkin2}. As a consequence, the only possible actions on cyclic groups that are realised as subgroups of $B_{n}(\St)$ are the trivial action, and multiplication by $-1$. This will subsequently be used to rule out many Type~I groups involving 
dicyclic factors.


In order to prove \repr{genhodgkin2}, we first state \repr{hodgkin2} whose statement, seemingly well known to the experts in the field, is related to a classical problem of Nielsen concerning the conjugacy problem in the mapping class group. The first proof we found in the literature is due to L.~Hodgkin~\cite{Ho} (see also~\cite{McH} for related results). 

\pagebreak

\begin{prop}\label{prop:hodgkin2}
Let $n,r\geq 2$ be such that $\mcg$ has elements of order $r$.
\begin{enumerate}[(a)]
\item Suppose that either $r\geq 3$, or $r=2$ and $n$ is odd. Then there is a unique value of $i\in \brak{0,1,2}$ for which $r$ divides $n-i$. Let $f$ be a rotation of $\St$ of angle $2\pi m/r$, where $m\in \N$ and $\gcd{(m,r)}=1$, and let $\gamma\in \mcg$ denote the mapping class of $f$. Then any element $\gamma'\in \mcg$ of order $r$ is conjugate to $\gamma$. Further, any two distinct powers of $\gamma$ are conjugate in $\mcg$ if and only if the following conditions hold:
\begin{enumerate}[(i)]
\item they are inverse, and
\item $i\in \brak{0,2}$.
\end{enumerate}
\item If $r=2$ and $n$ is even then $r$ divides both $n$ and $n-2$, and so both the choices $i=0$ and $i=2$ are possible. In the first (resp.\ second) case, we obtain an element $\gamma_{0}$ (resp.\ $\gamma_{2}$) of order $2$ that fixes none (resp.\ exactly two) of the $n$ marked points of $\St$. Further, every element of $\mcg$ of order $2$ is conjugate to exactly one of $\gamma_{0}$ or $\gamma_{2}$.
\end{enumerate}
\end{prop}

The proof of \repr{hodgkin2} may be deduced in a straightforward manner from that of~\cite[Proposition 2.1]{Ho}. Before coming to the proof of \repr{genhodgkin2}, we first define some notation that shall also be used later in Sections~\ref{part:realisation}.\ref{sec:dirtoi} and~\ref{part:realisation}.\ref{sec:oonestaramalg}. If $X$ is an $n$-point subset of $\St$, let $\homeo$ denote the set of orientation-preserving homeomorphisms that leave $X$ invariant. There is a natural surjective homomorphism $\map{\Psi}{\homeo}[\mcg]$, where $\Psi(f)=[f]$ denotes the mapping class of the homeomorphism $f\in \homeo$.


%
%

\begin{proof}[Proof of \repr{genhodgkin2}.]
Let $i\in\brak{0,1,2}$. Let $1\leq m,r\leq 2(n-i)$, and suppose that $\alpha_{i}^m$ and $\alpha_{i}^r$ are conjugate powers of $\alpha_{i}$ in $B_{n}(\St)$. Then there exists $z\in B_{n}(\St)$ such that 
\begin{equation}\label{eq:conjpoweralphai}
z \alpha_{i}^m z^{-1} = \alpha_{i}^r.
\end{equation}
Let $\mu=\gcd{(m,2(n-i))}$, and set $q=2(n-i)/\mu$. Then $\alpha_{i}^m$ and $\alpha_{i}^r$ are both of order $q$, and generate the same subgroup $\ang{\alpha_{i}^{\mu}}$ of $\ang{\alpha_i}$. In particular, there exists $1\leq \tau < q$ with $\gcd{(\tau,q)}=1$ such that $\alpha_{i}^{m\tau}=\alpha_{i}^{\mu}$. Setting $\xi=r\tau$ and raising \req{conjpoweralphai} to the $\tau\up{th}$ power yields $z \alpha_i^{\mu} z^{-1} = \alpha_i^{\xi}$. Now $\alpha_i^{\mu}$ and $\alpha_i^{\xi}$ generate the same subgroup of $\ang{\alpha_i}$, so there exists $1\leq t< q$ with $\gcd{(t,q)}=1$ such that $\alpha_i^{\xi}=\alpha_i^{t\mu}$, and hence
\begin{equation}\label{eq:mutmu}
z \alpha_i^{\mu} z^{-1} = \alpha_i^{\xi}=\alpha_i^{t\mu}.
\end{equation}
We claim that it suffices to show that $\alpha_i^{\xi}\in \brak{\alpha_{i}^{\mu},\alpha_{i}^{-\mu}}$. Suppose for a moment that the claim holds. Since $\gcd{(\tau,q)}=1$, there exist $u,v\in \Z$ such that $u\tau -vq=1$, and so
\begin{equation}\label{eq:alpha1calca}
\alpha_i^{m\tau u} = \alpha_i^{m(1+vq)}=\alpha_i^m \ldotp \left(\alpha_i^{mq}\right)^v= \alpha_i^m
\end{equation}
since $\alpha_i^{m}$ is of order $q$. Similarly,
\begin{equation}\label{eq:alpha1calcbb}
\alpha_i^{r\tau u} = \alpha_i^{r(1+vq)}=\alpha_i^r \ldotp \left(\alpha_i^{rq}\right)^v=\alpha_i^r
\end{equation}
since $\alpha_i^{r}$ is also of order $q$. But 
\begin{equation}\label{eq:alpha1calcc}
\alpha_i^{m\tau u}=\alpha_i^{\mu u}=\alpha_i^{\pm\xi u}=\alpha_i^{\pm r\tau u},
\end{equation}
and putting together equations~\reqref{alpha1calca}, \reqref{alpha1calcbb} and~\reqref{alpha1calcc}, we obtain $\alpha_i^m=\alpha_i^{\pm r}$. As we shall see below, if $i=1$ then in fact $\alpha_1^m=\alpha_1^r$, which will prove the proposition in this case. We now proceed to prove the claim, separating the cases $i=1$ and $i\in\brak{0,2}$.

\begin{enumerate}[(i)]
\item Let $i=1$.
Projecting relation~\reqref{mutmu} onto the Abelianisation $\Z_{2(n-1)}$ of $B_{n}(\St)$, we obtain $n\mu \equiv n\mu t \bmod 2(n-1)$, in other words, there exists $k\geq 0$ such that $n\mu (t-1)=k\ldotp 2(n-1)$. Now $n$ and $n-1$ are coprime, so there exists $l\geq 0$ such that $\mu (t-1) =l(n-1)$ and $2k =nl$. But $1\leq t<q= 2(n-1)/\mu$, thus $\mu\leq \mu t< 2(n-1)$, which implies that
\begin{equation*}
0\leq \mu (t-1) \leq 2(n-1)-\mu< 2(n-1),
\end{equation*}
and thus
\begin{equation*}
0 \leq l(n-1) < 2(n-1).
\end{equation*}
It follows that $l=0$ or $l=1$. If $l=1$ then $n=2k$, so $n$ is even. Further, $t-1=(n-1)/\mu=q/2$, hence $q$ is even. But $\gcd(t,q)=1$, so $t$ is odd, thus $\mu (t-1)=n-1$ is even, and $n$
is odd, a contradiction. We conclude that $l=0$, so $t=1$, and so $\alpha_{1}^{\mu}=\alpha_{1}^{\xi}$. As we saw above, this implies that $\alpha_{1}^{m}= \alpha_{1}^{r}$, which proves part~(\ref{it:conjpowera}) of the proposition.


\item Let $i\in \brak{0,2}$.
Consider \req{mutmu} and the short exact sequence~\reqref{mcg}. Let $w=\phi(z)$, let $a_{i}=\phi(\alpha_{i})$, and let $X$ be an $n$-point subset of $\St$ consisting of $n-i$ equally-spaced points on the equator, with the remaining $i$ points distributed at the poles. Then $w a_{i}^{\mu}w^{-1}=a_{i}^{\xi}$, and we may suppose $a_{i}$ to be represented by the homeomorphism $f_{i}\in\homeo$ that is rigid rotation of $\St$ of angle $2\pi /(n-i)$. It follows from \repr{hodgkin2} that $a_{i}^{\mu}$ and $a_{i}^{\xi}$ are either equal or are inverses, and since $a_{i}$ is of order $n-i$, $\xi\equiv \pm \mu \bmod{n-i}$, so $\xi=\pm \mu+\delta (n-i)$, where $\delta\in \Z$. 
If $\delta$ is even then $\alpha_{i}^{\xi}=\alpha_{i}^{\pm \mu}$ by \req{uniqueorder2}, and as we saw above, this implies that $\alpha_{i}^{m}= \alpha_{i}^{\pm r}$, which proves part~(\ref{it:conjpowerb}) of the proposition in this case. So assume that $\delta$ is odd,
in which case 
\begin{equation}\label{eq:conjalphai}
z\alpha_{i}^{\mu}z^{-1}=\alpha_{i}^{\xi}=\alpha_{i}^{\pm \mu+\delta (n-i)}=\alpha_{i}^{\pm \mu}\ft, 
\end{equation}
also using \req{uniqueorder2}. Conjugating \req{conjalphai} by $\alpha_{0}^{-i/2}\garside \alpha_{0}^{i/2}$, replacing $z$ by the element $\alpha_{0}^{-i/2}\garside \alpha_{0}^{i/2} z$ and using \req{basicconj} if necessary, we may suppose that
\begin{equation}\label{eq:transforma}
z\alpha_{i}^{\mu}z^{-1}=\alpha_{i}^{\mu}\ft.
\end{equation}
Notice however that since $\ft$ is central and of order $2$, the relation 
\begin{equation}\label{eq:transformb}
\alpha_{i}^{\xi}=\alpha_{i}^{\pm \mu}\ft
\end{equation}
of \req{conjalphai} persists under this conjugation. Conjugating \req{transforma} by $z^{-1}$ and multiplying by $\ft$ yields:
\begin{equation}\label{eq:conjalphamft}
\alpha_{i}^{\mu}\ft=z^{-1}\alpha_{i}^{\mu}z.
\end{equation}
The Abelianisation of \req{conjalphamft} yields $n(n-1)\equiv 0 \bmod{2(n-1)}$, so $n$ must be even for a solution to exist. In particular, if $n$ is odd, there is no $z\in B_{n}(\St)$ satisfying \req{conjalphamft}.
So let $n\geq 4$, and suppose that \req{conjalphamft} admits a solution $z\in B_{n}(\St)$. If $\mu\in \brak{n-i,2(n-i)}$ then $\alpha_{i}^{\mu}\in\ang{\ft}$, and this equation implies that $\ft=\id$, hence $n\leq 2$, which gives a contradiction. So $\mu\notin \brak{n-i,2(n-i)}$, and since $\mu\divides 2(n-i)$, we must have $1\leq \mu<n-i$. Moreover, $q=2(n-i)/\mu$ cannot be odd, for if it were then $\alpha_{i}^{\mu}\ft$ would be of order $2q$ because $\alpha_{i}^{\mu}$ is of order $q$ and $\ft$ is central. But this contradicts \req{conjalphamft}, so $q$ is even, and hence $\mu$ divides $n-i$. If $(n-i)/\mu=2$ then $\mu=(n-i)/2$, and
\begin{equation*}
\alpha_{i}^{\xi}=\alpha_{i}^{\pm\frac{(n-i)}{2}+(n-i)}=\alpha_{i}^{\mp (n-i)/2}=\alpha_{i}^{\mp\mu}
\end{equation*}
by \req{transformb}, which proves the result in this case. Since $\mu<n-i$, we suppose henceforth that $(n-i)/\mu\geq 3$.

We first assume that $i=0$, so $\mu$ divides $n$ and $n/\mu\geq 3$. Consider the image of \req{conjalphamft} under the homomorphism $\pi$ of \req{defperm}. Then $\pi(\alpha_{0}^{\mu})=(n-\mu+1,n-2\mu+1,\ldots, \mu+1,1)(n-\mu+2,n-2\mu+2,\ldots, \mu+2,2)\cdots (n,n-\mu,\ldots,2\mu,\mu)$ consists of $\mu$ disjoint $n/\mu$-cycles. For $j=1,\ldots,\mu$, the elements that appear in the $j\up{th}$ such cycle are of the form $\mu\left(\frac{n}{\mu}-k\right)+j$, where $k=1,\ldots,n/\mu$. 
Since $\pi(\ft)$ is trivial, $\pi(z)$ commutes with $\pi(\alpha_{0}^{\mu})$, and so $\pi(z)$ permutes the $\mu$ $n/\mu$-cycles of $\pi(\alpha_{0}^{\mu})$, and preserves the cyclic order of the elements within each cycle. In particular, if $\pi(z)$ sends $j$ to $\mu\left(\frac{n}{\mu}-k'\right)+j'$, where $j'\in \brak{1,\ldots,\mu}$ and $k'\in \brak{1,\ldots,n/\mu}$ then 
\begin{align}
\pi(z) \left(\mu\left(\frac{n}{\mu}-k\right)+j\right)&=\pi\bigl(\alpha_{0}^{k\mu}\bigr) \circ\pi(z)(j)= \pi(z) \circ\pi\bigl(\alpha_{0}^{k\mu}\bigr)(j)\notag\\
&=\pi\left(\alpha_{0}^{k\mu}\right)\left( \mu\left(\frac{n}{\mu}-k'\right)+j'\right)\notag\\
&= \mu\left(\left(\frac{n}{\mu}-k\right)-k'\right)+j' \bmod n.\label{eq:permpiz}
\end{align}
To coincide with the convention that we use for braids, note that we compose permutations from left to right. Now let $j=1$, and let $j'\in \brak{1,\ldots, \mu}$ and $k'\in \brak{1,\ldots,n/\mu}$ be such that $\pi(z)(1)= \mu\left(\frac{n}{\mu}-k'\right)+j'$. Set
\begin{equation*}
\zeta =(\sigma_{1}\cdots \sigma_{j'-1}) (\sigma_{\mu +1}\cdots \sigma_{\mu +j'-1}) \cdots (\sigma_{n-\mu +1}\cdots \sigma_{n-\mu +j'-1}).
\end{equation*}
Since $1\leq j'\leq \mu$, for $k=1,\ldots,n/\mu$, the blocks $\sigma_{\mu(\frac{n}{\mu}-k)+1}\cdots \sigma_{\mu(\frac{n}{\mu}-k)+j'-1}$ commute pairwise. By equations~\reqref{fundaa} and~\reqref{fundab}, $\zeta$ and $\alpha_{0}^{\mu}$ commute, 
hence 
\begin{equation}\label{eq:conjxi}
\zeta z^{-1}\alpha_{0}^{\mu}z \zeta ^{-1}=\alpha_{0}^{\mu}\ft.
\end{equation}
Now $\pi(\zeta)(1)=j'$, so for all $k=1,\ldots,n/\mu$,
\begin{align*}
\pi(\zeta z^{-1})\left(\mu\left(\frac{n}{\mu}-k\right)+1\right)&= \pi(z^{-1})\left(\mu\left(\frac{n}{\mu}-k\right)+j'\right)\\
&= \mu\left(\frac{n}{\mu}-k+k'\right)+1,
\end{align*}
by \req{permpiz}. Thus $\zeta z^{-1}$ and $\alpha_{0}^{\mu}$ belong to the subgroup $B_{n/\mu,n-n/\mu}(\St)$ of $B_{n}(\St)$ which here denotes the subgroup of those braids whose permutation leaves the set $\brak{1, \mu +1,\ldots, n-\mu +1}$ invariant. Let $z'$ denote the image of $z \zeta^{-1}$ under the projection onto $B_{n/\mu}(\St)$. 
Since the kernel
\begin{equation*}
B_{n-n/\mu}(\St\setminus \brak{x_{1}, x_{\mu +1},\ldots, x_{n-\mu +1}})
\end{equation*}
of the surjective homomorphism $B_{n/\mu,n-n/\mu}(\St)\to B_{n/\mu}(\St)$ is torsion free (this follows for example from~\cite[Proposition~2.5]{GG9}), the element $\alpha_{0}^{\mu}$, which is of order $q$, is sent to an element $\beta$ of $B_{n/\mu}(\St)$ of order $q$, and $\ft$ is sent to $\ft[n/\mu]$, the unique element of $B_{n/\mu}(\St)$ of order $2$ (using \req{alpha0q}, it is in fact possible to show that $\beta$ is equal to the element $\alpha_{0}$ of $B_{n/\mu}(\St)$, see Figure~\ref{fig:projnovermu} for an example in the case $n=6$ and $\mu=2$).
\begin{figure}[h]
\hfill
\begin{tikzpicture}[scale=0.6]
\foreach \k in {2,4}
{\draw[very thick] (\k,5) .. controls (\k,-1) and  (\k-1,1) .. (\k-2,-1);
\draw (\k+1,5) .. controls (\k+1,-1) and  (\k,1) .. (\k-1,-1);}:
\draw[very thick] (0,5) edge[out=-90,in=90] (1,4);
\draw (1,5) edge[out=-90,in=45] (0.57,4.57);
\draw (0.43,4.43) edge[out=-135,in=90] (0,4);
\draw (0,4) edge[out=-90,in=90] (1,3);
\draw[very thick] (1,4) edge[out=-90,in=45] (0.57,3.57);
\draw[very thick] (0.43,3.43) edge[out=-135,in=90] (0,3);
\filldraw[white] (1.8,2) rectangle (2.2,2.4);
\filldraw[white] (1.6,1.2) rectangle (2.2,1.5);
\filldraw[white] (2.8,1.8) rectangle (3.2,2.2);
\filldraw[white] (2.6,1) rectangle (3.2,1.3);
\filldraw[white] (3.7,1.6) rectangle (4.2,2);
\filldraw[white] (3.4,0.4) rectangle (3.6,0.8);
\filldraw[white] (4.5,0.9) rectangle (5.1,1.2);
\filldraw[white] (3.4,-0.04) rectangle (4.7,0.1);
\draw[very thick] (0,3)  .. controls (1,0) and (4,2.5) .. (4,-1);
\draw (1,3)  .. controls (2,1) and (5,3.5) .. (5,-1);
\draw[->,very thick] (7,2.2) -- (9,2.2);
\end{tikzpicture}\hspace{1cm}
\begin{tikzpicture}[scale=0.6]
\foreach \k in {2,4}
{\draw[very thick] (\k,5) .. controls (\k,-1) and  (\k-1,1) .. (\k-2,-1);};
\draw[very thick] (0,5) edge[out=-90,in=90] (1,4);
\draw[very thick] (1,4) edge[out=-90,in=45]  (0,3);
\filldraw[white] (1.6,1.2) rectangle (2.2,1.5);
\filldraw[white] (3.4,0.4) rectangle (3.6,0.8);
\draw[very thick] (0,3)  .. controls (1,0) and (4,2.5) .. (4,-1);
\end{tikzpicture}\hspace*{\fill}
\caption{The element $\alpha_{0}^{2}$ of $B_{3,3}(\St)$ is sent to the element $\alpha_{0}$ of $ B_{3}(\St)$ under the projection $B_{3,3}(\St)\to B_{3}(\St)$.}\label{fig:projnovermu}
\end{figure}
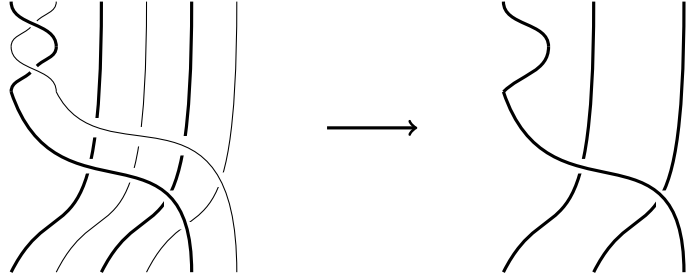
Now $q=2n/\mu\geq 6$, so by \reth{murasugi}, 
there exists $z''\in B_{n/\mu}(\St)$ and $1\leq k< q$, $\gcd{(k,q)}=1$, such that $\beta=z'' \alpha_{0}^k z''^{-1}$ ($\alpha_{0}$ here being considered as the standard finite order element of $B_{n/\mu}(\St)$). The image of \req{conjxi} under this projection yields: 
\begin{equation*}
z'^{-1}z'' \alpha_{0}^k z''^{-1} z'  =z'' \alpha_{0}^k z''^{-1} \ft[n/\mu] \quad\text{in $B_{n/\mu}(\St)$,}
\end{equation*}
so 
\begin{equation}\label{eq:zonenovermu}
z_{1} \alpha_{0}^k z_{1}^{-1} = \alpha_{0}^k \ft[n/\mu] \quad\text{in $B_{n/\mu}(\St)$.}
\end{equation}
where $z_{1}=z''^{-1}z'^{-1}z''$. There exist $\lambda_{1},\lambda_{2}\in \Z$ such that $\lambda_{1}k+\lambda_{2}q=1$, so $\alpha_{0}^{\lambda_{1}k}=\alpha_{0}$ in $B_{n/\mu}(\St)$. Since $q$ is even, $\lambda_{1}$ is odd, and raising \req{zonenovermu} to the $\lambda_{1}\up{th}$ power yields
\begin{equation}\label{eq:znovermu}
z_{1}\alpha_{0}z_{1}^{-1} =\alpha_{0}\ft[n/\mu] =\alpha_{0}^{1+\frac{n}{\mu}}\in B_{n/\mu}(\St).
\end{equation}
Hence $z_{1}\in N_{B_{n/\mu}(\St)}(\ang{\alpha_{0}})$, and so by \repr{genhodgkin1}(\ref{it:normcyclic}), $z_{1}$ is an element of $\ang{\alpha_0,\garside[n/\mu]}\cong \dic{4n/\mu}$, and $z_{1}=\alpha_0^{\lambda}\garside[n/\mu]^{\epsilon}$, where $0\leq \lambda < 2n/\mu$ and $\epsilon\in \brak{0,1}$. Thus 
\begin{equation}\label{eq:conjz}
z_{1}\alpha_0 z_{1}^{-1}=\garside[n/\mu]^{-\epsilon} \alpha_0 \garside[n/\mu]^{\epsilon}=
\begin{cases}
\alpha_0 & \text{if $\epsilon=0$}\\
\alpha_0^{-1} & \text{if $\epsilon=1$.}
\end{cases}
\end{equation}
Combining equations~\reqref{znovermu} and~\reqref{conjz}, we obtain $\ft[n/\mu]\in\brak{\id,\alpha_{0}^2}$. Now $n/\mu\geq 3$, so $\alpha_{0}^2$ (resp.\ $\ft[n/\mu]$) is of order $n/\mu$ (resp.\ $2$), which yields a contradiction.

Suppose finally that $i=2$, so $\mu\divides n-2$ and $(n-2)/\mu \geq 3$. 
Since $n$ must be even in order that \req{conjalphamft} possess a solution, these conditions imply that $n\geq 6$. Let $t\in \brak{n-1,n}$. Projecting \req{conjalphamft} into $\sn$ leads to the equality $(\pi(\alpha_{2}^{\mu})\circ \pi(z))(t)=(\pi(z) \circ \pi(\alpha_{2}^{\mu}))(t)$, and this implies that 
\begin{equation*}
\pi(\alpha_{2}^{\mu})(\pi(z)(t))=\pi(z)(t),
\end{equation*}
so $\pi(z)(t)\in \operatorname{Fix}(\pi(\alpha_{2}^{\mu}))$. Since $1\leq \mu <n-2$, we have
$\operatorname{Fix}(\pi(\alpha_{2}^{\mu}))=\brak{n-1,n}$,
and thus $\pi(z)(t)\in \brak{n-1,n}$. We conclude that $z\in B_{n-2,2}(\St)$, $B_{n-2,2}(\St)$ being the subgroup of $B_{n}(\St)$ whose elements induce permutations that leave $\brak{n-1,n}$ invariant. This permits us to project \req{conjalphamft} onto $B_{n-2}(\St)$ by forgetting the last two strings. It is clear that $\alpha_{2}$ (as an element of $B_{n-2,2}(\St)$) projects to $\alpha_{0}$ (as an element of $B_{n-2}(\St)$), and so $\ft=\alpha_{2}^{n-2}$ (which is an element of $B_{n-2,2}(\St)$) projects to $\alpha_{0}^{n-2}=\ft[n-2]$ (as an element of $B_{n-2}(\St)$) by \req{uniqueorder2}. We thus obtain:
\begin{equation}\label{eq:conjnminus2}
z'^{-1} \alpha_{0}^{\mu} z'=\alpha_{0}^{\mu} \ft[n-2],
\end{equation}
where $z'$ is the image of $z$ under this projection. But $n-2\geq 4$, and applying the analysis of the case $i=0$ to \req{conjnminus2} yields a contradiction. This proves the result in the case $i=2$, and thus completes the proof of the proposition.\qedhere
\end{enumerate}
\end{proof}

\begin{rems}\mbox{}\label{rem:nt}
\begin{enumerate}
\item If $i\in \brak{0,2}$ then the converse of \repr{genhodgkin2}(\ref{it:conjpowerb}) holds using the construction of the corresponding dicyclic groups of \rerem{finitesub}(\ref{it:finitesubb}).
\item\label{it:nta}  If $\mu$ divides $n-i$ where $i\in\brak{0,1,2}$, the braid $\alpha_{i}^{\mu}$ admits a block structure using arguments similar to those of the second part of \relem{funda}. If $q=(n-i)/\mu$ then $\alpha_{i}^{\mu}$ may be thought of as a collection of $q$ blocks, each comprised of $\mu$ strings (see Figures~\ref{fig:alphai} and~\ref{fig:alphaii} for examples where $n-i=12$ and $\mu=4$, as well as Figure~\ref{fig:alpha0q} for the case $i=0$, $n=6$ and $\mu=3$). The first block contains a full twist on its $\mu$ strings, and passes over each of the remaining $q-1$ blocks. If $i=1$ (resp.\ $i=2$) then the last (resp.\ penultimate) string then wraps around this first block. If $i=2$ then there is an additional final vertical string. In terms of the Nielsen-Thurston classification of surface homeomorphisms applied to braid groups, these braids are reducible, and a set of reducing curves may be read off from these braid diagrams (see~\cite{BNG,GW} for more information). 
\end{enumerate}
\end{rems}



\begin{figure}[h]
\hfill
\begin{tikzpicture}[scale=0.5]
\foreach \k in {0,1,2,3}
{\draw[thick] (\k,4) .. controls (\k+1,1.5) and (\k+3,2.5) .. (\k+4,1);
\draw[thick] (\k,-2) -- (\k,1);
\draw[thick] (\k,6.3) -- (\k,7);};
\draw[thick] (4,4) edge[out=-90,in=45] (3.7,3.2);
\draw[thick] (0,1) edge[out=90,in=-135] (1.3,2.1);
\draw[thick] (5,4) edge[out=-90,in=45] (4.2,2.8);
\draw[thick] (1,1) edge[out=90,in=-135] (2.1,1.9);
\draw[thick] (6,4) edge[out=-90,in=45] (4.7,2.5);
\draw[thick] (2,1) edge[out=90,in=-135] (2.6,1.6);
\draw[thick] (7,4) edge[out=-90,in=45] (5.4,2.3);
\draw[thick] (3,1) edge[out=90,in=-135] (3.4,1.4);
\foreach \k in {0,1,2,3}
{\draw[thick, rotate around={180:(5.5,1)}] (\k,4) .. controls (\k+1,1.5) and (\k+3,2.5) .. (\k+4,1);
\draw[thick, rotate around={180:(5.5,1)}] (\k,-2) -- (\k,1);};
\draw[thick, rotate around={180:(5.5,1)}] (4,4) edge[out=-90,in=45] (3.7,3.2);
\draw[thick, rotate around={180:(5.5,1)}] (0,1) edge[out=90,in=-135] (1.3,2.1);
\draw[thick, rotate around={180:(5.5,1)}] (5,4) edge[out=-90,in=45] (4.2,2.8);
\draw[thick, rotate around={180:(5.5,1)}] (1,1) edge[out=90,in=-135] (2.1,1.9);
\draw[thick, rotate around={180:(5.5,1)}] (6,4) edge[out=-90,in=45] (4.7,2.5);
\draw[thick, rotate around={180:(5.5,1)}] (2,1) edge[out=90,in=-135] (2.6,1.6);
\draw[thick, rotate around={180:(5.5,1)}] (7,4) edge[out=-90,in=45] (5.4,2.3);
\draw[thick, rotate around={180:(5.5,1)}] (3,1) edge[out=90,in=-135] (3.4,1.4);
\foreach \k in {4,5,...,11}
{\draw[thick] (\k,4) -- (\k,7);
\draw[thick, rotate around={180:(5.5,1)}] (\k,4) -- (\k,7);
};
\foreach \k in {8,9,...,11}
{\draw[thick] (\k,-5) -- (\k,-2);};
\draw[very thick] (-0.3,4) rectangle  (3.3,6.3);
\draw (1.5,5) node {\large $\ft[4]$};
\end{tikzpicture}\hspace*{\fill}
\caption{The braid $\alpha_{0}^{4}$ in $B_{12}(\St)$.}\label{fig:alphai}
\end{figure}
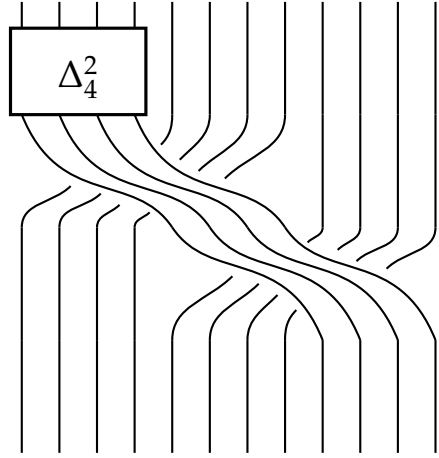
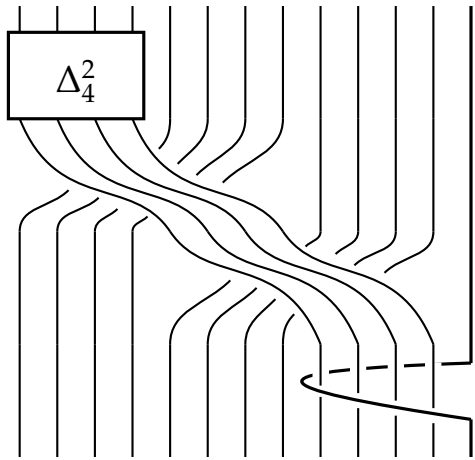
\begin{figure}[h]\hfill
\begin{tikzpicture}[scale=0.5]
\foreach \k in {0,1,2,3}
{\draw[thick] (\k,4) .. controls (\k+1,1.5) and (\k+3,2.5) .. (\k+4,1);
\draw[thick] (\k,-2) -- (\k,1);
\draw[thick] (\k,6.3) -- (\k,7);};
\draw[thick] (4,4) edge[out=-90,in=45] (3.7,3.2);
\draw[thick] (0,1) edge[out=90,in=-135] (1.3,2.1);
\draw[thick] (5,4) edge[out=-90,in=45] (4.2,2.8);
\draw[thick] (1,1) edge[out=90,in=-135] (2.1,1.9);
\draw[thick] (6,4) edge[out=-90,in=45] (4.7,2.5);
\draw[thick] (2,1) edge[out=90,in=-135] (2.6,1.6);
\draw[thick] (7,4) edge[out=-90,in=45] (5.4,2.3);
\draw[thick] (3,1) edge[out=90,in=-135] (3.4,1.4);
\foreach \k in {0,1,2,3}
{\draw[thick, rotate around={180:(5.5,1)}] (\k,4) .. controls (\k+1,1.5) and (\k+3,2.5) .. (\k+4,1);
\draw[thick, rotate around={180:(5.5,1)}] (\k,-2) -- (\k,1);};
\draw[thick, rotate around={180:(5.5,1)}] (4,4) edge[out=-90,in=45] (3.7,3.2);
\draw[thick, rotate around={180:(5.5,1)}] (0,1) edge[out=90,in=-135] (1.3,2.1);
\draw[thick, rotate around={180:(5.5,1)}] (5,4) edge[out=-90,in=45] (4.2,2.8);
\draw[thick, rotate around={180:(5.5,1)}] (1,1) edge[out=90,in=-135] (2.1,1.9);
\draw[thick, rotate around={180:(5.5,1)}] (6,4) edge[out=-90,in=45] (4.7,2.5);
\draw[thick, rotate around={180:(5.5,1)}] (2,1) edge[out=90,in=-135] (2.6,1.6);
\draw[thick, rotate around={180:(5.5,1)}] (7,4) edge[out=-90,in=45] (5.4,2.3);
\draw[thick, rotate around={180:(5.5,1)}] (3,1) edge[out=90,in=-135] (3.4,1.4);
\foreach \k in {4,5,...,11}
{\draw[thick] (\k,4) -- (\k,7);
\draw[thick, rotate around={180:(5.5,1)}] (\k,4) -- (\k,7);
};
\draw[very thick] (-0.3,4) rectangle  (3.3,6.3);
\draw (1.5,5) node {\large $\ft[4]$};
\draw[very thick] (12,7) -- (12,-2.5);
\draw[very thick] (12,-2.5) .. controls (6,-2.7) and (6,-3.1) .. (12,-4);
\draw[very thick] (12,-4) -- (12,-5);
\filldraw[white] (7.8,-3) rectangle  (8.2,-2.5);
\filldraw[white] (8.8,-3) rectangle  (9.2,-2.5);
\filldraw[white] (9.8,-3) rectangle  (10.2,-2.5);
\filldraw[white] (10.8,-3) rectangle  (11.2,-2.5);
\draw[thick] (8,-2) -- (8,-3.1);
\draw[thick] (9,-2) -- (9,-3.3);
\draw[thick] (10,-2) -- (10,-3.5);
\draw[thick] (11,-2) -- (11,-3.7);
\draw[thick] (8,-5) -- (8,-3.4);
\draw[thick] (9,-5) -- (9,-3.7);
\draw[thick] (10,-5) -- (10,-3.9);
\draw[thick] (11,-5) -- (11,-4.1);
\end{tikzpicture}\hspace*{\fill}
\caption{The braid $\alpha_{1}^{4}$ in $B_{13}(\St)$. The braid $\alpha_{2}^{4}$ in $B_{14}(\St)$ is obtained by adding an extra vertical string on the right of this braid.}\label{fig:alphaii}
\end{figure}

One immediate consequence of \repr{genhodgkin2} is that it allows us to narrow down the possible Type~I subgroups of $B_{n}(\St)$ involving cyclic or dicyclic factors, with the exception of $\quat$.

\begin{cor}\label{cor:cycaction}
Let $G$ be a Type~I subgroup of $B_{n}(\St)$ of the form $F\rtimes_{\theta} \Z$.
\begin{enumerate}[(a)]
\item\label{it:cycactiona} Suppose that $F$ is cyclic.
\begin{enumerate}[(i)]
\item If $\ord{F}$ divides $2(n-1)$ then $G\cong F\times \Z$.
\item\label{it:cycactionaii} If $\ord{F}$ divides $2(n-i)$, where $i\in \brak{0,2}$, then either $G\cong F\times \Z$ or $G\cong F\rtimes_{\rho} \Z$, where $\rho$ is the action defined in \redef{v1v2}(\ref{it:mainzqt}) (multiplication by $-1$).
\end{enumerate}
\item\label{it:cycactionb} Let $m \geq 3$ divide $n-i$, where $i\in \brak{0,2}$, and let $F$ be dicyclic of order $4m$ with the presentation given by \req{presdic}. Then either $G\cong F\times\Z$ or $G\cong F\rtimes_{\nu} \Z$, where $\nu$ is the action defined by \req{actdic4m}.
\end{enumerate}
\end{cor}


\pagebreak

\begin{proof}\mbox{}
\begin{enumerate}[(a)]
\item Let $G$ be a Type~I subgroup of $B_{n}(\St)$ of the form $F\rtimes_{\theta} \Z$, where $F$ is cyclic. Up to conjugacy, we may suppose by \reth{murasugi} that there exist $i\in\brak{0,1,2}$ and $1\leq l\leq 2(n-i)$ such that $l$ divides $2(n-i)$, and $F=\ang{\alpha_{i}^l}$, with $\ord{F}=2(n-i)/l$. There exists $z\in B_{n}(\St)$ of infinite order such that the action $\theta$ on $F$ is realised by conjugation by $z$, so $z\alpha_{i}^l z^{-1}=\alpha_{i}^{lm}$, where $\gcd{(m,2(n-i)/l)}=1$. From \repr{genhodgkin2}, $\alpha_{i}^{lm}=\alpha_{i}^l$ if $i=1$ and $\alpha_{i}^{lm} \in\brak{\alpha_{i}^l,\alpha_{i}^{-l}}$ if $i\in \brak{0,2}$, which implies the result.

\item Let $G$ be a Type~I subgroup of $B_{n}(\St)$ of the form $F\rtimes_{\theta} \Z$, where $F\cong \dic{4m}$ has the given presentation, and let the action $\theta$ of $\Z$ on $F$ be realised by conjugation by $z$, where $z\in B_{n}(\St)$ is of infinite order. Since $m\geq 3$, $\ang{x}$ is the unique cyclic subgroup of $F$ of order $2m$, so is invariant under conjugation by $z$. By part~(\ref{it:cycactiona})(\ref{it:cycactionaii}), $zxz^{-1}=\theta(1)(x)=x^{\epsilon}$, where $\epsilon\in \brak{1,-1}$. Further, the elements of $F\setminus \ang{x}=\ang{x}y$ are permuted by the action, so $zyz^{-1}=\theta(1)(y)=x^{2k+\delta}y$ for some $k\in \brak{0,1,\ldots, m-1}$ and $\delta\in \brak{0,1}$. If $\epsilon=1$ (resp.\ $\epsilon=-1$) then consider the action $\theta'$ defined by $\theta'(1)=\iota \circ \theta(1)$, where $\iota\in \inn{F}$ is conjugation by $x^{-k}$ (resp.\ by $x^{k+\delta}y$). So $\theta'(1)(x)=x$, and $\theta'(1)(y)= x^{\delta}y$, which gives rise to the two possible actions given in the statement. Since the automorphisms $\theta(1)$ and $\theta'(1)$ of $F$ differ by an inner automorphism, it follows from \repr{isoout} that $G$ and $F\rtimes_{\theta'} \Z$ are isomorphic.\qedhere
\end{enumerate}
\end{proof}

\section{Reduction of isomorphism classes of $F\rtimes_{\theta} \Z$ via periodicity}\label{sec:reducperiod}

We now turn our attention to the Type~I subgroups $G$ of $B_{n}(\St)$ of the form $F\rtimes_{\theta} \Z$, where $F$ is equal to $\oonestar$ or $\istar$. The arguments of \resec{autout} showed that there are two possible actions. The aim of this section is to rule out the non-trivial action in each case, which will imply that $G$ is isomorphic to $F \times \Z$. This is achieved in two stages. First, in \resec{homotopytype} we give an alternative proof of the fact that the homotopy type of the universal covering space of the configuration spaces $F_{n}(\St)$ and $D_{n}(\St)$ is that of $\St$ if $n\leq 2$, and that of $\St[3]$ otherwise. This result appears to be an interesting fact in its own right, and mirrors that for the projective plane $\rp$~\cite{GG3}. As a consequence, in \relem{per24} we generalise the fact that any nontrivial finite subgroup of $B_{n}(\St)$ is periodic of least period $2$ or $4$~\cite{GG5} to its infinite subgroups. Secondly, if $F\in\brak{\oonestar,\istar}$, in \repr{perhom} we recall some facts concerning the cohomology of $F$. From this, it will follow in these cases that $\theta(1)$ is an inner automorphism, and so by \repr{isoout}, $F\rtimes_{\theta} \Z$ is isomorphic to $F\times \Z$.

\subsection{The homotopy type of the configuration spaces $F_{n}(\St)$ and $D_{n}(\St)$}\label{sec:homotopytype}


The purpose of this section is to describe the homotopy type of the universal covering space of $F_n(\St)$ and $D_n(\St)$. For $n=1$, we have $F_1(\St)=D_1(\St)=\St$, which is simply connected. So assume from now on that $n\geq 2$. We give an alternative proof of \repr{homot} which is due to~\cite{BCP,FZ}. 


\begin{proof}[Proof of \repr{homot}.]
First observe that $F_n(\St)$ and $D_n(\St)$ have the same universal covering space because $F_n(\St)$ is a finite $n!$-fold regular covering space of $D_n(\St)$.
\begin{enumerate}[(a)]
\item This was proved in \cite[Lemma~8]{GG10}.



\item Let $n\geq 1$. Consider the Fadell-Neuwirth fibration $\map{p_{n+1}}{F_{n+1}(\St)}[F_n(\St)]$ obtained by forgetting the last coordinate. The fibre over a point $(x_1,\ldots,x_n) \in F_{n}(\St)$ may be identified with $F_1(\St\setminus\brak{x_1,\ldots,x_n})$. The related long exact sequence in homotopy is:
\begin{multline*}
\ldots \to \pi_{m+1}(F_{n+1}(\St)) \to \pi_{m+1}(F_n(\St)) \to
\pi_m(F_1(\St\setminus\brak{x_1,\ldots,x_n})) \to\\
\pi_m(F_{n+1}(\St)) \to \pi_m(F_n(\St))\to \pi_{m-1}(F_1(\St\setminus\brak{x_1,\ldots,x_n})) \to \ldots
\end{multline*}
The fact that $F_1(\St\setminus\brak{x_1,\ldots,x_n})$ is a $K(\pi,1)$-space implies that the homomorphism $\pi_m(F_{n+1}(\St))\to \pi_m(F_n(\St))$ induced by $p_{n+1}$ is an isomorphism for all $m\geq 3$ and all $n\geq 1$. It remains to study the case $m=2$.

First suppose that $n=3$. From part~(\ref{it:homottype2}), $F_2(\St)$ has the homotopy type of $\St$, and 
$\pi_2(F_3(\St))=\brak{1}$ and $\pi_1(F_3(\St))\cong \Z_2$ by~\cite{FvB}. 
Let $\map{\phi}{\St[3]}[F_3({\St})]$ be such that $p_{2}\circ p_{3}\circ \phi$ is homotopic to the Hopf map $\eta$ (such a $\phi$ exists because $\pi_3(F_3(\St))$ is isomorphic  to $\pi_3(F_2(\St))$). We thus have the following diagram that commutes up to homotopy:
\begin{equation*}
\xymatrix{%
\St[3] \ar[r]^{\phi} \ar[d]^{\eta}  & F_{3}(\St)
\ar[d]^{p_{2}\circ p_{3}}\\
\St  \ar@{=}[r] & \St.}
\end{equation*}
If $m\geq 3$, $\eta$ induces an isomorphism $\pi_{m}(\St[3])\to \pi_{m}(\St)$ and $p_{2}\circ p_{3}$ induces an isomorphism $\pi_m(F_{3}(\St))\to \pi_m(\St)$. Since $\pi_2(\St[3])$ and $\pi_2(F_3(\St))$ are trivial, the commutativity of the above diagram implies that $\phi$ induces an isomorphism $\pi_{m}(\St[3]) \to \pi_{m}(F_{3}(\St))$ for all $m\geq 2$. Lifting to the corresponding universal covering spaces gives rise to an isomorphism $\pi_{m}(\St[3]) \to \pi_{m}(\widetilde{F_{3}(\St)})$ for all $m\in \N$, $\widetilde{F_{3}(\St)}$ being the universal covering space of $F_{3}(\St)$, and so by Whitehead's theorem, $\widetilde{F_{3}(\St)}$ has the homotopy type of $\St[3]$.

Let $n\geq 3$. Then $\pi_{2}(F_{n}(\St))=\brak{1}$~\cite{FvB} and so the homomorphism
\begin{equation*}
\pi_m(F_{n+1}(\St))\to \pi_m(F_n(\St))
\end{equation*}
induced by $p_{n+1}$ is an isomorphism for all $m\geq 2$. Lifting to the universal covering spaces and applying Whitehead's Theorem, part~(\ref{it:homottype2}) and induction gives the result.\qedhere
\end{enumerate}
\end{proof}

\subsection{A cohomological condition for the realisation of Type~I virtually cyclic groups}\label{sec:percohI}


In this section we apply \repr{homot} to derive a necessary cohomological condition for an abstract group to be realised as a subgroup of $B_n(\St)$. If $F=\oonestar,\istar$, this will allow us to rule out the possibility of $F\rtimes_{\theta} \Z$ for the non-trivial action for each of these groups described in \resec{autout}. Following~\cite{AS}, we recall the definition of a periodic group which extends the classical definition for finite groups. 
By Definition~2.1 and the definition given just before Corollary~2.10 in~\cite{AS}, we say that a group $G$ is \emph{periodic} of period $d\geq 1$ if there exist a non-negative integer $r_0\geq 0$ and a cohomology class $u \in H^d(G,\Z)$ such that the homomorphism $H^r(G, A) \stackrel{\cup u}{\to} H^{r+d}(G, A)$ is an isomorphism for all $r \geq r_0$ and for all local coefficient systems $A$. From~\cite[Corollary~2.14]{AS}, if a discrete group acts freely on a finite-dimensional $CW$-complex of dimension $m$ whose homotopy type is that of the sphere $\St[d-1]$ then the group $G$ is periodic. By a standard argument using the spectral sequence associated to the covering of the orbit space, it is not hard to see that $d$ is a period, and that we can take $r_0=m+1$. An obvious consequence of the above is the following lemma.

\begin{lem}\label{lem:per24}
Let $n\geq 3$, and let $G$ be a group abstractly isomorphic to a subgroup of $B_{n}(\St)$. Then there exists $r_{0}\geq 1$ such that $H^r(G, \Z) \cong H^{r+4}(G, \Z)$ for all $r\geq r_{0}$. 
\end{lem}


\begin{proof} 
Since the universal covering space $\widetilde{D_{n}(\St)}$ of $D_{n}(\St)$ is a finite-dimensional $CW$-complex, it is a homotopy $3$--sphere by \repr{homot}(\ref{it:homottype2}). Any subgroup of $B_n(\St)$ acts freely on $\widetilde{D_{n}(\St)}$, and thus is periodic of period $4$. Taking $A=\Z$ yields the result.
%
\end{proof} 

We now apply \relem{per24} to the Type~I groups of the form $F\rtimes_{\theta} \Z$.
%
If a group $G$ acts on a module $A$, let $A^G$ denote the submodule of $A$ fixed by $G$, and let $A_G$ denote the quotient of $A$ by the submodule generated by $\set{a-ga}{a\in A,\, g\in G}$.

\begin{lem}\label{lem:cohomtypeI}
Let $G=F\rtimes_{\theta} \Z$, where $F$ is a finite periodic group 
and $\theta\in \operatorname{Hom}(\Z,\aut{F})$, and let $\map{\theta(1)^{(i)}}{H^i(F,\Z)}[H^i(F,\Z)]$ be the
induced automorphism on cohomology in dimension $i$.  Then $H^{\ast}(G, \Z)$ is as follows: $H^0(G, \Z)=\Z$,  $H^1(G, \Z)=\Z$,  and for all $i\in \N$, $H^{2i}(G,\Z)=H^{2i}(F,\Z)^{\Z}$ and $H^{2i+1}(G, \Z)=H^{2i}(F,\Z)_{\Z}$ with respect to the $\Z$-module structure on
$H^{2i}(F,\Z)$ induced by $\theta$.
\end{lem}

\begin{proof}
Consider the Lyndon-Hochschild-Serre spectral sequence associated with the short exact sequence
\begin{equation*}
1 \to F \to G \to \Z \to 1.
\end{equation*}
The $E_2$-term of this spectral sequence, given by $H^{p}(\Z, H^q(F,\Z))$, vanishes if $p \notin \brak{0,1}$ because the cohomological dimension of $\Z$ is equal to one. So outside of the two vertical lines given by $p=0$ and $p=1$, the terms vanish which implies that all differentials are necessarily trivial, and so the spectral sequence collapses. Further, since the cohomology of $F$ in odd dimension vanishes, there is at most one non-trivial group $E_2^{p,q}$ with $p+q=r$ for each given $r$.  Hence there is no extension problem from $E_{\infty}$ to $H^{\ast}(G)$, and it suffices to compute the $E_2$-term. The result follows from the well-known description of  the cohomology of $\Z$ with coefficients in $A$ (see  \cite[Chapter III, Section 1, Example 1]{Br}).
\end{proof}

We now seek necessary conditions for the group $G$ to have least period either $2$ or $4$. Let $d$ be the least period of $F$. Then $d$ is the least integer for which $H^{d}(F,\Z)\cong \Z_{\ord{F}}$, and if $H^{2i}(G,\Z)=H^{2i}(F,\Z)^{\Z}\cong \Z_{\ord{F}}$ then $\theta(1)^{(2i)}=\id$. So there exists $k\in \N$ such that $2i=kd$. Let $k_{0}$ be the least integer for which $\theta(1)^{(k_{0}d)}=\id$. If $G$ is periodic, its period is necessarily a multiple of $k_{0}d$. In particular, if the least period of $G$ is equal to either $2$ or $4$ then $k_{0}\in \brak{1,2}$ if $d=2$, and $k_{0}=1$ if $d=4$.


\begin{rem}
The additive structure of the cohomology of the virtually cyclic groups of Type~I with integer coefficients was computed in detail in~\cite{Je} for the cases where $F$ is one of the groups of the form $\Z_a\rtimes \Z_b$ or $\Z_a\rtimes( \Z_b \times \quat[2^i])$. This corresponds to the first two families of the classification of the finite periodic groups given by the Suzuki-Zassenhaus Theorem~\cite[Theorem 6.15]{AM}.
\end{rem}

Based on \relem{cohomtypeI} and the knowledge of the cohomology of finite periodic groups, we obtain the following result.
\begin{prop}\label{prop:perhom}
Let $F\in \brak{\oonestar,\istar}$, and let $G\cong F\rtimes_{\theta}\Z$ be a Type~I subgroup of $B_{n}(\St)$. Then $\theta(1)$ is an inner automorphism of $F$.
\end{prop}

\begin{proof}
Suppose first that $F\cong\oonestar$. By~\cite[p.~39]{gg3}, the group $\oonestar$ has period $4$, and the induced automorphism on $H^4(\oonestar,\Z)\cong \Z_{48}$ is trivial if $\theta(1)$ is an inner automorphism, and multiplication by $9$ (so is non trivial) otherwise, thus the result follows.

Now suppose that $F\cong \istar$, which we interpret as $\operatorname{SL}_2(\FF_{5})$. The non-trivial element of $\out{\istar}$ is represented by the automorphism of $\istar$ which is conjugation by the matrix $\left( \begin{smallmatrix}
w & 0\\
0 & 1
\end{smallmatrix}
\right)$, where $w$ is a non square of $\FF_{5}$~\cite[page~152]{AM}.  From~\cite[Proposition~1.5]{gg4}, the induced automorphism on the $5$-primary component of the group $H^4(\istar,\Z)\cong \Z_{120}$, which is isomorphic to $\Z_5$, is multiplication by $-1$. For the trivial element of $\out{\istar}$, the induced homomorphism is trivial and the result follows.
\end{proof}



\section{Necessity of the conditions on $\mathbb{V}_{1}(\lowercase{n})$ and $\mathbb{V}_{2}(\lowercase{n})$}
\label{sec:defV1V2}

Let $n\geq 4$. In this section, we prove \reth{main}(\ref{it:mainI}), which shows the necessity of the conditions on $\mathbb{V}_{1}(n)$ and $\mathbb{V}_{2}(n)$. We start by considering the subgroups of $B_{n}(\St)$ of Type~I, and then go on to study those of Type~II.


\subsection{Necessity of the conditions on $\mathbb{V}_{1}(n)$}\label{sec:defV1}

%

We gather together the results of the previous sections to prove the following proposition, which is the statement of \reth{main}(\ref{it:mainI}) for the Type~I subgroups of $B_{n}(\St)$.

\begin{prop}\label{prop:necV1}
Let $n\geq 4$. Then every virtually cyclic subgroup of $B_{n}(\St)$ of Type~I is isomorphic to an element of $\mathbb{V}_{1}(n)$.
\end{prop}

Before proving \repr{necV1}, we state and prove the following result which shows that if $F$ is a dicyclic subgroup of $B_{n}(\St)$ then up to conjugacy, it may be taken to be a subgroup of one of the maximal dicyclic subgroups $\dic{4(n-i)}$, $i\in \brak{0,2}$.

\begin{lem}\label{lem:maxdicyclic2}
Let $n\geq 4$, and let $H$ be a subgroup of $B_{n}(\St)$ isomorphic to $\dic{4m}$, where $m\geq 2$. Then there exists $i\in \brak{0,2}$ such that $H$ is conjugate to a subgroup of the standard maximal dicyclic subgroup $\dic{4(n-i)}$ of \rerem{finitesub}(\ref{it:finitesubb}).
\end{lem}

\begin{rem}\label{rem:maxdicyclic2}
Under the hypotheses of \relem{maxdicyclic2}, we have that $m\divides n-i$.
\end{rem}

\begin{proof}[Proof of \relem{maxdicyclic2}.] 
Let $H\cong \dic{4m}$, where $m\geq 2$. By \cite[Proposition~1.5(2)]{GG7}, $H$, as an abstract finite group, is realised as a single conjugacy class in $B_{n}(\St)$ with the exception that when $n$ is even and $m$ divides $(n-i)/2$, $i\in\brak{0,2}$, there are exactly two conjugacy classes. Using the subgroup structure of dicyclic groups and the construction of \cite[Theorem~1.6]{GG7}, it follows that $H$ is conjugate to a subgroup of the one of the standard maximal dicyclic subgroups $\dic{4(n-i)}$ of $B_{n}(\St)$, where $i\in \brak{0,2}$.
\end{proof}

\begin{proof}[Proof of \repr{necV1}.]
Let $G$ be a infinite virtually cyclic subgroup of $B_{n}(\St)$ of Type~I. Then $G$ is of the form $F\rtimes_{\theta} \Z$, where $F$ is a finite subgroup of $B_{n}(\St)$, and $\theta(1)\in \operatorname{Hom}(\Z,\aut{F})$. We separate the discussion into two cases.
\begin{enumerate}[(a)]
\item Suppose that $F$ is isomorphic to one of the binary polyhedral groups $\tonestar,\oonestar, \istar$. Then $n$ must satisfy the conditions given in \reth{finitebn} for the existence of $F$ as a subgroup of $B_{n}(\St)$. Applying \repr{isoout}, up to isomorphism, we may restrict ourselves to representative automorphisms $\theta(1)$ of the elements of $\out{F}\cong \Z_{2}$ given in \resec{autout}. If $\theta(1)=\id_{F}$ then $G\cong F \times \Z$, and these are the elements of $\mathbb{V}_{1}(n)$ given by \redef{v1v2}(\ref{it:mainI})(\ref{it:maint}), (\ref{it:maino}) and (\ref{it:maini}) for the given values of $n$. So suppose that $\theta(1)$ represents the nontrivial element of $\out{F}$. By \repr{perhom}, $F\ncong \oonestar, \istar$, so $F\cong \tonestar$, and $G$ is isomorphic to the element of $\mathbb{V}_{1}(n)$ given by \redef{v1v2}(\ref{it:mainI})(\ref{it:maing}), the action $\omega$ being that of \req{nontrivacttstar}. Since $n$ must be even for the existence of $\tonestar$, it remains to show that $n\equiv 0,2 \bmod{6}$. Suppose on the contrary that $n=6l+4$, where $l\in\N$, and suppose that $\tonestar\rtimes_{\omega} \Z$ is realised as a subgroup $L$ of $B_{n}(\St)$, with the $\tonestar$-factor (resp.\ the $\Z$-factor) realised as a subgroup $H$ (resp.\  $\ang{z}$) of $B_{n}(\St)$. Let~\reqref{preststar} denote a presentation of $H$. By the definition of $\omega $, we have that $\omega(1)(X)=X^{-1}$ by \req{nontrivacttstar}. On the other hand, $X$ is of order $3$, so up to conjugacy and inverses, it follows from \reth{murasugi} that $X=\alpha_{1}^{2(n-1)/3}=\alpha_{1}^{4l+2}$. Since the action $\omega $ of $\Z$ on $H$ is realised by conjugation by $z$, we have $\omega(1)(X)=zXz^{-1}$ in $L$, which implies that $zXz^{-1}=X^{-1}$. Abelianising this relation in $\Z_{2(n-1)}$ yields $\xi(X)=\xi(X^{-1})$. However,
\begin{equation*}
\xi(X)=\xi(\alpha_{1}^{4l+2})=\overline{n(4l+2)}=\overline{(6l+4)(4l+2)}=\overline{4l+2}
\end{equation*}
in $\Z_{12l+6}$, so $\xi(X)\neq \xi(X^{-1})$, and we obtain a contradiction.


\item\label{it:necV1notpoly} Suppose that $F$ is not isomorphic to any of the three binary polyhedral groups $\tonestar,\oonestar, \istar$. \repr{maxsubgp} implies that $F$ is cyclic or dicyclic. If $F$ is dicyclic, isomorphic to $\dic{4m}$ for some $m\geq 2$, then \relem{maxdicyclic2} implies that up to conjugation, $F$ is a subgroup of one of the standard dicyclic groups $\dic{4(n-i)}$, $i\in\brak{0,2}$, and $m\divides n-i$. If $m=2$ then $F\cong \quat$ and $n$ is even. Furthermore, by \relem{quatact}, up to an element of $\inn{F}$, $\theta(1)\in \brak{\id_{\quat},\alpha,\beta}$, so $G$ is isomorphic to an element of $\mathbb{V}_{1}(n)$ given by \redef{v1v2}(\ref{it:mainI})(\ref{it:mainq8}). If $m\geq 3$ then \reco{cycaction}(\ref{it:cycactionb}) applies, and up to an element of $\inn{F}$, there are two cases to consider:
\begin{enumerate}[(i)]
\item\label{it:necV1dictrivial} $\theta(1)=\id_{F}$, in which case $G\cong \dic{4m}\times \Z$. Since $F$ admits a cyclic subgroup of order $2m$, the realisation of $G$ implies that of $\Z_{2m}\times \Z$. If $m=n-i$ then up to conjugacy, we may suppose by \reth{murasugi} that the cyclic factor is generated by $\alpha_{i}$, but this contradicts \repr{luis}, and hence $m<n-i$. Thus $G$ is an element of $\mathbb{V}_{1}(n)$ given by \redef{v1v2}(\ref{it:mainI})(\ref{it:maindic}).

\item\label{it:necV1dicnontrivial} $G\cong \dic{4m}\rtimes_{\nu} \Z$, where $\nu(1)$ is given by \req{actdic4m}. Let $F$ have the presentation given by \req{presdic}. Abelianising the relation $\nu(1)(y)=xy$ in $B_{n}(\St)$ implies that the exponent sum of $x$ is congruent to zero modulo $2(n-1)$. On the other hand, $x$ is of order $2m$, so by \reth{murasugi} is conjugate to $\alpha_{i}^{l(n-i)/m}$, where $\gcd{(l,2m)}=1$. In particular, $l$ is odd. Now the exponent sum of $\alpha_{i}$ is congruent to $n-1$ modulo $2(n-1)$, and since that of $x$ is congruent to zero modulo $2(n-1)$, it follows that $l(n-i)/m$ is even, and consequently $(n-i)/m$ is even. Thus $G$ is an element of $\mathbb{V}_{1}(n)$ given by \redef{v1v2}(\ref{it:mainI})(\ref{it:maindict}).
\end{enumerate}

Finally, suppose that $F$ is cyclic of order $q$, say. By \reth{murasugi} there exists $i\in\brak{0,1,2}$ such that $q$ divides $2(n-i)$, and up to conjugacy, $F=\ang{\alpha_{i}^{2(n-i)/q}}$. Applying \reco{cycaction}(\ref{it:cycactiona}), we have that $\theta(1)\in \brak{\id_{F},-\id_{F}}$ up to an element of $\inn{F}$. If $\theta(1)=\id_{F}$ then $G\cong F \times \Z$. But $F$ cannot be maximal cyclic, for then its centraliser would contain an element of infinite order, which contradicts \repr{luis}, so $q\neq 2(n-i)$. Further, if $n-i$ is odd then $q\neq n-i$, for then $\ang{\alpha_{i}^2\ft}=\ang{\alpha_{i}}$ would be of order $2(n-i)$, and its centraliser would contain an element of infinite order, which contradicts \repr{luis} once more. Hence $G$ is isomorphic to an element of $\mathbb{V}_{1}(n)$ given by \redef{v1v2}(\ref{it:mainI})(\ref{it:mainzq}). So suppose that $\theta(1)=-\id_{F}$. Then $G\cong F \rtimes_{\rho} \Z$, where $\rho$ is the action by conjugation for which $\rho(1)$ is multiplication by $-1$. By \reco{cycaction}(\ref{it:cycactiona}), we have $i\in \brak{0,2}$. Further, the subgroup of $G$ isomorphic to $F \rtimes_{\rho} 2\Z$ is abstractly isomorphic to $F\times \Z$, and so we conclude from the previous case that $q\neq 2(n-i)$, and that $q\neq n-i$ if $n$ is odd. Hence $G$ is isomorphic to an element of $\mathbb{V}_{1}(n)$ given by \redef{v1v2}(\ref{it:mainI})(\ref{it:mainzqt}). This shows that any virtually cyclic subgroup of $B_{n}(\St)$ is isomorphic to an element of the family $\mathbb{V}_{1}(n)$ as required. \qedhere
\end{enumerate}
\end{proof}

\subsection{Necessity of the conditions on $\mathbb{V}_{2}(n)$}\label{sec:typeII}

We now prove \reth{main}(\ref{it:mainI}) for the Type~II subgroups of $B_{n}(\St)$. 

\begin{prop}\label{prop:necV2}
Let $n\geq 4$. Then every virtually cyclic subgroup of $B_{n}(\St)$ of Type~II is isomorphic to an element of $\mathbb{V}_{2}(n)$.
%
\end{prop}

\begin{rem}
Combining Propositions~\ref{prop:necV1} and~\ref{prop:necV2} yields the proof of \reth{main}(\ref{it:mainI}).
\end{rem}

\begin{proof}[Proof of \repr{necV2}.]
Let $G$ be an infinite virtually cyclic subgroup of $B_{n}(\St)$ of Type~II. Then $G=G_{1} \bigast_{F} G_{2}$, where $F,G_{1}$ and $G_{2}$ are finite subgroups of $B_{n}(\St)$, and $F$ is of index $2$ in $G_j$, $j=1,2$. Suppose first that one of the $G_{j}$, $G_{1}$ say, is binary polyhedral. Then $G_{1}\cong \oonestar$ since $\tonestar,\istar$ have no index $2$ subgroup, $F\cong \tonestar$ since $\tonestar$ is the unique index $2$ subgroup of $\oonestar$, and $G_{2}\cong \oonestar$ since $\oonestar$ is the only finite subgroup of $B_{n}(\St)$ to have $\tonestar$ as an index $2$ subgroup. Thus $G\cong \oonestar \bigast_{\tonestar} \oonestar$, which is the element of $\mathbb{V}_{2}(n)$ given by \redef{v1v2}(\ref{it:mainIIdef})(\ref{it:mainIIe}).

Assume now that the $G_{j}$ are not binary polyhedral. By \rerem{finitesub}(\ref{it:finitesuba}),  the $G_{j}$ are cyclic or dicyclic, and since they possess an even index subgroup, they are of even order, so both contain the unique element $\ft$ of order $2$. This implies that $F=G_{1}\cap G_{2}$ is of even order, so the $G_{j}$ are in fact of order $4q$ for some $q\in \N$. 

Suppose that one of the $G_{j}$, $G_{1}$ say, is cyclic. Then $G_{1}\cong \Z_{4q}$ and $F\cong \Z_{2q}$. By \reth{murasugi}, there exists $i\in\brak{0,1,2}$ such that $4q \divides 2(n-i)$, so $q\divides (n-i)/2$. If $G_{2}\cong \Z_{4q}$, $G$ is isomorphic to the element of $\mathbb{V}_{2}(n)$ given by \redef{v1v2}(\ref{it:mainIIdef})(\ref{it:mainIIa}). If $G_{2}\cong \dic{4q}$ then $q\geq 2$, and there exists $i'\in \brak{0,2}$ such that $q \divides n-i'$ by \relem{maxdicyclic2}. But $n-i'=2\left(\frac{n-i}{2}\right)+ (i-i')$, so $q\divides i-i'$, and since $q\geq 2$, we must have $i\in \brak{0,2}$. In this case, $G$ is isomorphic to the element of $\mathbb{V}_{2}(n)$ given by \redef{v1v2}(\ref{it:mainIIdef})(\ref{it:mainIIb}).

Finally, suppose that $G_{1}\cong G_{2}\cong \dic{4q}$, where $q\geq 2$. Then $F\cong \Z_{2q}$ or $F\cong \dic{2q}$, and there exists $i\in\brak{0,2}$ such that $q\divides n-i$ by \relem{maxdicyclic2}. If $F\cong \Z_{2q}$ then by standard properties of the amalgamated product $G=G_{1} \bigast_{F} G_{2}$, $G$ has an index $2$ subgroup $G'$ isomorphic to $F\rtimes_{\theta} \Z$ for some $\theta\in \operatorname{Hom}(\Z,\aut{F})$. Since $F$ is cyclic, $\theta(1)\in \brak{\id_{F}, -\id_{F}}$ by \reco{cycaction}(\ref{it:cycactiona})(\ref{it:cycactionaii}), and hence the subgroup $F\rtimes_{\theta} 2\Z$ is abstractly isomorphic to $F\times \Z$. It follows from \reth{murasugi} and \repr{luis} that $q\neq n-i$, and so $G$ is isomorphic to the element of $\mathbb{V}_{2}(n)$ given by \redef{v1v2}(\ref{it:mainIIdef})(\ref{it:mainIIc}). Now suppose that $F\cong \dic{2q}$. Then $q\geq 4$ is even, and hence $G$ is isomorphic to the element of $\mathbb{V}_{2}(n)$ given by \redef{v1v2}(\ref{it:mainIIdef})(\ref{it:mainIId}).
\end{proof}

\enlargethispage{6mm}
\vspace*{-2mm}

\begin{rem}
The cohomological property of \resec{percohI} used to define the family $\mathbb{V}_{1}(n)$ appears to be important in this case. We do not know of an example of two finite periodic groups $G_{1}, G_{2}$ of the same period $d$ for which the amalgamated product $G_{1} \bigast_{F} G_{2}$ does not have period $d$. The Mayer-Vietoris sequence~\cite[Chapter~II, Section~7, Corollary~7.7]{Br} suggests that such an example may not even exist.
%
%
%
\end{rem}

\chapter{Realisation of the elements of $\mathbb{V}_{1}(n)$ and $\mathbb{V}_{2}(n)$ in $B_{n}(\St)$}\label{part:realisation}

In this Part, we prove that with a small number of exceptions (those described in \rerem{exceptions}), the isomorphism classes of $\mathbb{V}(n)$ given in the statement of \reth{main}(\ref{it:mainII}) are indeed realised as subgroups of $B_{n}(\St)$. For the realisation of the Type~I groups, the cases $F=\Z_{q}$, $F=\dic{4m}$ ($m\geq 3$), $F=\quat$ and $F=\tonestar,\oonestar,\istar$ will be treated in Sections~\ref{sec:centcycdic}, \ref{sec:realdic}, \ref{sec:realquat} and~\ref{sec:typeIbinpoly} respectively, and the results will be brought together in \resec{realtypeI}. The realisation of the Type~II groups will be dealt with in \resec{realtypeII}, and this will enable us to prove \reth{main}(\ref{it:mainII}) in \resec{proofthm5}. In the first three cases, the constructions are algebraic, but are heavily inspired by geometric considerations, and it may be helpful for the reader to draw some pictures. If $F$ is binary polyhedral, the corresponding virtually cyclic groups will be obtained geometrically by considering certain multitwists in $\mcg$, and then lifting the corresponding mapping class to an element of $B_{n}(\St)$ via \req{mcg}. \reth{main}(\ref{it:mainIII}) will be proved in \resec{dirtoi}. In \resec{isoclasses}, we discuss the question of the number of isomorphism classes of the Type~II virtually cyclic subgroups of $B_{n}(\St)$, which will enable us to prove \repr{isoamalg}. Finally, in \resec{genmcg}, we apply \reth{main} and \repr{corrbnmcg} to the problem of the classification of the virtually cyclic subgroups of $\mcg$, from which we will obtain directly \reth{classvcmcg}.

\section{Type~I subgroups of $B_{\lowercase{n}}(\St)$ of the form $F\rtimes \Z$ with $F$ cyclic}\label{sec:centcycdic}

Let $n\geq 4$, and let $F$ be a finite cyclic subgroup of $B_{n}(\St)$. In order to construct elements of $\mathbb{V}_{1}(n)$ involving $F$, we require elements of $B_{n}(\St)$ of infinite order whose action on $F$ by conjugation is compatible with \repr{genhodgkin2}. Since these actions are given by multiplication by $\pm 1$, we will be interested in finding elements $z\in B_{n}(\St)$ of infinite order for which $zxz^{-1}=x^{\pm 1}$ for all $x\in F$. This comes down to studying the centraliser and normaliser of $F$ in $B_{n}(\St)$. Note that by \reth{murasugi}, there exist $i\in\brak{0,1,2}$ and $0\leq m< 2(n-i)$, $m\divides 2(n-i)$, such that $F$ is conjugate to $\ang{\alpha_{i}^m}$. Since conjugate subgroups have conjugate centralisers and normalisers, we may suppose for our purposes that $F=\ang{\alpha_{i}^m}$.

\subsection{Type~I subgroups of the form $\Z_{q}\times \Z$}\label{sec:constzqz}

We first study the centralisers of powers of the $\alpha_{i}$, $i\in \brak{0,1,2}$, which will give rise to Type~I subgroups of the form $\Z_{q}\times \Z$.

\begin{lem}\label{lem:commalphaigen}
Let $n\geq 4$, and let $i\in \brak{0,1,2}$. Suppose that $m\in \N$ divides $2(n-i)$, and let
\begin{equation*}
r=
\begin{cases}
m & \text{if $m \divides n-i$}\\
\frac{m}{2} & \text{if $m \ndivides n-i$.}
\end{cases}
\end{equation*}
Then:
\begin{enumerate}[(a)]
\item\label{it:gen012} $r \divides n-i$, and $Z_{B_{n}(\St)}(\ang{\alpha_{i}^r})=Z_{B_{n}(\St)}(\ang{\alpha_{i}^m})$.
\item If $r=1$ then $Z_{B_{n}(\St)}(\ang{\alpha_{i}^m})=\ang{\alpha_{i}}$.
\item\label{it:rgeq2} If $r\geq 2$ then $Z_{B_{n}(\St)}(\ang{\alpha_{i}^m}) \supset \ang{\delta_{r,i}}$, where the element
\begin{equation}\label{eq:delta1}
\delta_{r,i}=\sigma_{1}\sigma_{r+1} \cdots \sigma_{n-i-r+1}= \prod_{k=0}^{(n-i-r)/r} \sigma_{kr+1}
\end{equation}
is of infinite order.
\end{enumerate}
\end{lem}

\begin{proof}\mbox{}
\begin{enumerate}[(a)]
\item The statement clearly holds if $m \divides n-i$. So suppose that $m\ndivides n-i$. Since $qm=2(n-i)$ for some $q\in \N$, we have $q/2=(n-i)/m$. Thus $q$ is odd, $m$ is even, $r=m/2$ is an integer and $q=(n-i)/r$, which proves the first part of the statement. For the second part, note first that $Z_{B_{n}(\St)}(\ang{\alpha_{i}^r}) \subset Z_{B_{n}(\St)}(\ang{\alpha_{i}^m})$. Conversely, suppose that $z\in B_{n}(\St)$ commutes with $\alpha_{i}^m$. Then $z$ commutes with $\alpha_{i}^m \ft=\alpha_{i}^{n+m-i}$ by \req{uniqueorder2}. Further, $\alpha_{i}^m$ is of order $q$, which is odd. Hence $\alpha_{i}^m \ft$ is of order $2q$. Since $\alpha_{i}^m \ft\in \ang{\alpha_{i}}$ and $\ord{\ang{\alpha_{i}^r}}=2q$, we have $\ang{\alpha_{i}^m \ft}= \ang{\alpha_{i}^r}$, so $z$ commutes with $\alpha_{i}^r$, and this completes the proof of part~(\ref{it:gen012}).
\item If $r=1$ then $Z_{B_{n}(\St)}(\ang{\alpha_{i}^m}) =Z_{B_{n}(\St)}(\ang{\alpha_{i}}) = \ang{\alpha_{i}}$ by part~(\ref{it:gen012}) and \repr{genhodgkin1}.
\item Suppose that $r\geq 2$. We first show that $\delta_{r,i}$ is of infinite order. Assume on the contrary that $\delta_{r,i}$ is of finite order. By \reth{murasugi}, there exist $l\in \brak{0,1,2}$ and $0\leq \mu <2(n-i)$ such that $\delta_{r,i}$ is conjugate to $\alpha_{l}^{\mu}$. Since $r\geq 2$, the permutation $\pi(\delta_{r,i})$ consists of the product of $s$ disjoint transpositions, plus $n-2s$ fixed points, where $s=(n-i)/r$. In particular, $\delta_{r,i}\notin P_{n}(\St)$, so $\delta_{r,i}\neq \ft$, $\mu\neq n-l$ by \req{uniqueorder2}, and $\pi(\delta_{r,i})$ has exactly $l$ fixed points. Suppose first that $l\in \brak{0,2}$. Since $\xi(\alpha_{l})=\overline{n-1}$ in $\Z_{2(n-1)}$, $\xi(\delta_{r,i})=\xi(\alpha_{l}^{\mu})=\overline{s}$ belongs to the subgroup $\ang{\overline{n-1}}$, so there exists $\lambda\in \N$ such that $\lambda(n-1)=s$. But 
\begin{equation*}
n-1\leq\lambda(n-1)=s=(n-i)/r\leq n/2,
\end{equation*}
so $n\leq 2$, which yields a contradiction. Hence $l=1$, $\pi(\delta_{r,i})$ has a single fixed point, thus $1=n-2s=(rs+i)-2s=s(r-2)+i$, and $i\in \brak{0,1}$. If $i=0$ then $s=1$ and $r=n=3$, which gives rise to a contradiction. So $i=1$, $r=2$, $n=2s+1$ (which implies that $n\geq 5$) and $\delta_{r,i}=\delta_{2,1}=\sigma_{1}\sigma_{3}\cdots \sigma_{n-2}$. But $\delta_{2,1}$ belongs to the subgroup $B_{n-1,1}(\St)$ of $n$-string braids whose permutation fixes the element $n$. Under the projection $B_{n-1,1}(\St)\to B_{n-1}(\St)$ given geometrically by forgetting the last string, $\delta_{2,1}$ is sent to the element $\delta_{2,0}$ of $B_{n-1}(\St)$, which must then also be of finite order. However, using the fact that $n-1\geq 4$, the above discussion implies that the element $\delta_{2,0}$ of $B_{n-1}(\St)$ is of infinite order, hence the element $\delta_{2,1}$ of $B_{n}(\St)$ is also of infinite order.

It remains to prove that $\delta_{r,i}$ commutes with $\alpha_{i}^m$. By part~(\ref{it:gen012}), it suffices to show that it commutes with $\alpha_{i}^r$. First note that the product in \req{delta1} is taken over $k=0,1,\ldots,s-1$. If $0\leq k\leq s-2$, we have
\begin{equation*}
1\leq r+(kr+1)\leq r(s-1)+1=n-i-(r-1)\leq n-i-1 \quad\text{since $r\geq 2$,}
\end{equation*}
and hence $\alpha_{i}^r \sigma_{kr+1}\alpha_{i}^{-r}=\sigma_{(k+1)r+1}$ by \req{fundaa}. If $k=s-1$ then
\begin{align*}
\alpha_{i}^r \sigma_{(s-1)r+1}\alpha_{i}^{-r}&=\alpha_{i}^r \sigma_{n-i-(r-1)}\alpha_{i}^{-r}= \alpha_{i}^2 \sigma_{n-i-1}\alpha_{i}^{-2} \quad\text{by \req{fundaa}}\\
&= \sigma_{1} \quad\text{by \req{fundab}.}
\end{align*}
Since $r\geq 2$, the terms $\sigma_{kr+1}$, $0\leq k\leq s-1$, commute pairwise, and so
\begin{equation*}
\alpha_{i}^r \delta_{r,i}\alpha_{i}^{-r} = \alpha_{i}^r \left( \prod_{k=0}^{s-1} \sigma_{kr+1}\right) \alpha_{i}^{-r}=  \left( \prod_{k=1}^{s-2} \sigma_{kr+1}\right) \sigma_{1}=\delta_{r,i},
\end{equation*}
using the previous calculations. This completes the proof of the proposition.\qedhere
\end{enumerate}
\end{proof}

\begin{prop}\label{prop:necsuftimes}
Let $n\geq 4$, and let $q\in \N$. Then $B_{n}(\St)$ possesses a subgroup isomorphic to $\Z\times \Z_{q}$ if and only if there exists $i\in \brak{0,1,2}$ such that the following three conditions are satisfied:
\begin{enumerate}[(i)]
\item\label{it:ztimeszq1} $q$ divides $2(n-i)$.
\item\label{it:ztimeszq2} $1\leq q\leq n-i$.
\item\label{it:ztimeszq3} $q<n-i$ if $n-i$ is odd.
\end{enumerate}
\end{prop}

\begin{proof}
The necessity of conditions~(\ref{it:ztimeszq1})--(\ref{it:ztimeszq3}) was proved in \repr{necV1}.
%
Conversely, suppose that there exists $i\in \brak{0,1,2}$ such that the conditions~(\ref{it:ztimeszq1})--(\ref{it:ztimeszq3}) are satisfied. Then $m=2(n-i)/q$ is an integer greater than or equal to two. Consider the subgroup $\ang{\alpha_{i}^m}$ of $B_{n}(\St)$, which is isomorphic to $\Z_{q}$. With the notation of \relem{commalphaigen}:
\begin{enumerate}[\textbullet]
\item if $m\divides n-i$ then $r=m\geq 2$.
\item if $m\ndivides n-i$ then $q$ is odd, $m$ is even and $r=m/2$. If $r=1$ then $m=2$ and so $q=n-i$, but this contradicts condition~(\ref{it:ztimeszq3}). Hence $r\geq 2$.
\end{enumerate}
So by \relem{commalphaigen}(\ref{it:rgeq2}), $\delta_{r,i}\in Z_{B_{n}(\St)}(\ang{\alpha_{i}^m})$, and thus the subgroup $\ang{\alpha_{i}^m,\delta_{r,i}}$ of $B_{n}(\St)$ is isomorphic to $\Z\times \Z_{q}$ as required.
\end{proof}

\subsection{Type~I subgroups of the form $\Z_{q}\rtimes_{\rho} \Z$}

In this section, we consider the realisation in $B_{n}(\St)$ of Type~I groups $\Z_{q}\rtimes_{\rho} \Z$, where $\rho\in \operatorname{Hom}(\Z,\aut{\Z_{q}})$, and $\rho(1)$ is multiplication by $-1$.

\pagebreak

\begin{prop}\label{prop:rtimes}
Let $n\geq 4$, and let $q\in \N$. Then $B_{n}(\St)$ possesses a subgroup isomorphic to $\Z_{q}\rtimes_{\theta} \Z$, for some action $\theta\in \operatorname{Hom}(\Z,\aut{\Z_{q}})$, $\theta(1)\neq \id_{\Z_{q}}$, if and only if the following conditions are satisfied:
\begin{enumerate}[(i)]
\item\label{it:rtimesi} $q$ divides $2(n-i)$, where $i\in \brak{0,2}$.
\item\label{it:rtimesii} $3\leq  q\leq n-i$, and $q<n-i$ if $n$ is odd.
\item\label{it:rtimesiii} $\theta(1)=\rho(1)$.
\end{enumerate}
\end{prop}

\begin{proof}
Suppose first that $B_{n}(\St)$ possesses a subgroup isomorphic to $\Z_{q}\rtimes_{\theta} \Z$, where $\theta(1)\neq \id_{\Z_{q}}$. \repr{necV1} implies that conditions~(\ref{it:rtimesi})--(\ref{it:rtimesiii}) are satisfied (note that if $q\in\brak{1,2}$ then $\theta(1)=\id_{\Z_{q}}$).

Conversely, suppose that conditions~(\ref{it:rtimesi})--(\ref{it:rtimesiii}) are satisfied, and let $m=2(n-i)/q$. From the proof of \repr{necsuftimes}, and making use of the notation of \relem{commalphaigen}, we know that $r\geq 2$ and that $\ang{\alpha_{i}^m,\delta_{r,i}}$ is isomorphic to $\Z_{q}\times \Z$. We will modify slightly the generator $\delta_{r,i}$ of the $\Z$-factor in order to obtain an action on $\alpha_{i}^m$ that is multiplication by $-1$. To achieve this, let $\garside'=\alpha_{0}^{-1}\garside \alpha_{0}$. Equation~\reqref{basicconj} implies that $\alpha_{i}^{-1}=\garside' \alpha_{i}\garside'^{-1}$. Now let $\delta_{r,i}'=\garside'\delta_{r,i}$. Since $\delta_{r,i}$ commutes with $\alpha_{i}^m$, we have that $\delta_{r,i}' \alpha_{i}^m \delta_{r,i}'^{-1}=\alpha_{i}^{-m}$, which will give rise to the required action on $\ang{\alpha_{i}^m}$. We claim that $\delta_{r,i}'$ is of infinite order. This being the case, the subgroup $\ang{\delta_{r,i}',\alpha_{i}^m}$ of $B_{n}(\St)$ is isomorphic to $\Z_{q}\rtimes_{\rho} \Z$, where $\rho(1)=-\id_{\Z_{q}}$, which will prove that the conditions~(\ref{it:rtimesi})--(\ref{it:rtimesiii}) are sufficient. To prove the claim, first note that since $\ft$ is central and of order $2$, it suffices to prove that $\beta=\delta_{r,i}'^2 \fti$ is of infinite order. Further:
\begin{align*}
\beta& \! =\! (\garside'\delta_{r,i})^2 \fti=\alpha_{0}^{-1}\garside \alpha_{0} \delta_{r,i} \alpha_{0}^{-1}\garside^{-1} \alpha_{0} \delta_{r,i}\\
& \! =\!\alpha_{0}^{-2}\garside \sigma_{1}\sigma_{r+1}\cdots \sigma_{n-i-2r+1} \sigma_{n-i-r+1} \garside^{-1} \alpha_{0}^2 \delta_{r,i} \;\text{by equations~\reqref{basicconj} ($i=0$), and~\reqref{delta1}}\\
& \! =\! \alpha_{0}^{-2} \sigma_{i+r-1} \sigma_{i+2r-1} \cdots \sigma_{n-r-1}\sigma_{n-1}  \alpha_{0}^2 \delta_{r,i} \; \text{using \req{garsideconj} and the fact that $r\geq 2$}\\
& \! =\!
\begin{cases}
\alpha_{0}^{-2}\sigma_{i+r-1} \alpha_{0}^2 \ldotp \sigma_{i+2r-3} \cdots \sigma_{n-r-3}\sigma_{n-3} \delta_{r,i} & \text{if $i+r\leq3$}\\
\sigma_{i+r-3} \sigma_{i+2r-3} \cdots \sigma_{n-r-3}\sigma_{n-3} \sigma_{1} \sigma_{r+1} \cdots \sigma_{n-2r-i+1}\sigma_{n-r-i+1} & \text{if $i+r\geq 4$,}
\end{cases}
\end{align*}
using \req{fundaa}. We distinguish these two cases:
\begin{enumerate}[(a)]
\item $i+r\geq 4$. Then $n-(i+r)+1\leq n-3$, and the last two strings of $\beta$ are vertical. 
If $\beta$ were of finite order, it would have to be conjugate to a power of $\alpha_{2}$ using \reth{murasugi} (observe that this is also the case if $\beta$ is pure, since the only nontrivial torsion element of $P_{n}(\St)$ is $\ft$, which is a power of $\alpha_{2}$ by \req{uniqueorder2}), and so its Abelianisation $\xi(\beta)$ would be congruent to $0$ modulo $n-1$. On the other hand, $\xi(\beta)=\xi(\delta_{r,i}^{2})$ is congruent to $2(n-i)/r \bmod{2(n-1)}$. So there exists $\lambda\in \N$ such that $2(n-i)/r=\lambda(n-1)$. Hence $\lambda r(n-1)=2(n-i)=2(n-1)+2(1-i)$, and since $1-i\in \brak{1,-1}$, this implies that $n-1\divides 2$, which is impossible. So $\beta$ is of infinite order.
%
%

\item $i+r\leq 3$. Since $r\geq 2$ and $i\in\brak{0,2}$, we must have $i=0$ and $r\in \brak{2,3}$. Suppose first that $r=3$. Using \req{fundaa}, we obtain:
\begin{align*}
\beta &=\alpha_{0}^{-1} (\alpha_{0}^{-1} \sigma_{2} \sigma_{5} \cdots \sigma_{n-4}\sigma_{n-1}  \alpha_{0}) (\alpha_{0}
\sigma_{1} \sigma_{4} \cdots \sigma_{n-5}\sigma_{n-2} \alpha_{0}^{-1}) \alpha_{0}\\
&= \alpha_{0}^{-1} (\sigma_{1} \sigma_{2} \sigma_{4} \sigma_{5}\cdots \sigma_{n-5} \sigma_{n-4}  \sigma_{n-2} \sigma_{n-1}) \alpha_{0}.
\end{align*}
Now $3\divides n$ by \relem{commalphaigen}(\ref{it:gen012}) and $n\geq 4$, so $n\geq 6$. Thus the projection of $\alpha_{0}\beta \alpha_{0}^{-1}$ onto the first six strings is the braid $\beta'=\sigma_{1}\sigma_{2} \sigma_{4} \sigma_{5}\in B_{6}(\St)$. If $\beta'$ were of finite order, by \reth{murasugi}, it would be conjugate in $B_{6}(\St)$ to some power of $\alpha_{0}=\sigma_{1}\cdots\sigma_{5}$ (because its permutation has no fixed point), so its exponent sum would be congruent to $5$ modulo $10$. But this is clearly not the case, and so $\beta'$ and $\beta$ are of infinite order in their respective groups. 
Now assume that $r=2$. Then
\begin{equation*}
\beta =\alpha_{0}^{-2} \sigma_{1} \sigma_{3} \cdots \sigma_{n-3}\sigma_{n-1}  \alpha_{0}^2 
\sigma_{1} \sigma_{3} \cdots \sigma_{n-3}\sigma_{n-1}=  \sigma_{1}^2 \sigma_{3}^2 \cdots \sigma_{n-3}^2\sigma_{n-1}^2
\end{equation*}
by equations~\reqref{fundaa} and~\reqref{fundab}.
The projection of $\beta$ onto $B_{4}(\St)$ by forgetting all but the first four strings gives rise to the element $\sigma_{1}^2 \sigma_{3}^2$ of $P_{4}(\St)$, which is equal to $\delta_{2,0}^{2}$ by \req{delta1}, and so is of infinite order by \relem{commalphaigen}(\ref{it:rgeq2}). This implies that
%
$\beta$ is also of infinite order.
\end{enumerate}
So in both cases, $\beta$ is of infinite order, and hence so is $\delta'_{r,i}$. This completes the proof of the claim, and thus that of the proposition.
\end{proof}

\section{Type~I subgroups of $B_{n}(\St)$ of the form $F\rtimes \Z$  with $F$ dicyclic, $F\neq \quat$}\label{sec:realdic}


Let $n\geq 4$ and $i\in\brak{0,2}$. In this section, we consider the realisation in $B_{n}(\St)$ of Type~I subgroups of the form $F\rtimes_{\theta} \Z$, where $F\cong \dic{4s}$, $s\geq 3$. By \repr{necV1}, there are two possible actions of $\Z$ on $\dic{4s}$ to be considered. The trivial action, given by Definition~\ref{def:v1v2}(\ref{it:maindic}) will be analysed in \repr{dic4sz}, while the nontrivial action, given by Definition~\ref{def:v1v2}(\ref{it:maindict}) will be studied in \repr{nonk}.

\begin{prop}\label{prop:dic4sz}
Let $n\geq 4$ and let $s\geq 3$. Then $\dic{4s}\times \Z$ is realised as a subgroup of $B_n(\St)$ if and only if there exists $i\in \brak{0,2}$ such that $s$ divides $n-i$ strictly.
\end{prop}

\begin{rem}
In other words, if $i\in \brak{0,2}$ and $s\geq 3$ divides $n-i$ then $\dic{4s}\times \Z$ is realised as a subgroup of $B_n(\St)$ if and only if $\dic{4s}$ is non maximal. Further, the value of $i\in \brak{0,2}$ is unique since $s\geq 3$.
\end{rem}

\begin{proof}[Proof of \repr{dic4sz}.]
The necessity of the condition was shown in the proof of \repr{necV1}(\ref{it:necV1notpoly})(\ref{it:necV1dictrivial}). Conversely, suppose that $i\in \brak{0,2}$, let $s\geq 3$ be such that $s$ divides $n-i$ strictly, so $s\leq (n-i)/2$. Set $m=(n-i)/s\geq 2$. Then $2\leq m\leq (n-i)/3$. Consider the subgroup $\ang{\alpha_i^m, \rho^{m}}$, where
\begin{align}
\rho &=(\sigma_1\cdots \sigma_{m-1})(\sigma_{m+1}\cdots \sigma_{2m-1})
\cdots (\sigma_{(s-1)m+1} \cdots \sigma_{sm-1})\label{eq:rhocomm1}\\
& = \prod_{j=1}^s
\left(\sigma_{(j-1)m+1} \cdots \sigma_{jm-1}\right).\notag
\end{align}
We claim that the bracketed terms of \req{rhocomm1} are permuted cyclically under conjugation by $\alpha_i^m$. To prove the claim, first suppose that $j\in \brak{1,\ldots,s-1}$. Since $jm-1+m=(j+1)m-1\leq sm-1=n-i-1$, it follows from \req{fundaa} that 
\begin{equation}\label{eq:rhocomm2}
\alpha_{i}^m \left(\sigma_{(j-1)m+1} \cdots \sigma_{jm-1}\right)  \alpha_{i}^{-m}= 
\sigma_{jm+1} \cdots \sigma_{(j+1)m-1}.
\end{equation}
Now suppose that $j=s$. Then
\begin{align}
\alpha_{i}^m \left(\sigma_{(s-1)m+1} \cdots \sigma_{sm-1}\right)  \alpha_{i}^{-m}
&= \left( \prod_{k=1}^{m-1} \alpha_{i}^m \sigma_{n-i-m+k} \alpha_{i}^{-m}\right)\notag\\
&=\left( \prod_{k=1}^{m-1} \alpha_{i}^{k+1} \sigma_{n-i-1} \alpha_{i}^{-(k+1)}\right) \notag\\
&= \left( \prod_{k=1}^{m-1} \alpha_{i}^{k-1} \sigma_{1} \alpha_{i}^{-(k-1)}\right)= \sigma_{1} \cdots \sigma_{m-1}, \label{eq:finalconj}
\end{align}
by equations~\reqref{fundaa} and~\req{fundab}.
The claim then follows from equations~\reqref{rhocomm2} and~\reqref{finalconj}. The fact that the bracketed terms of \req{rhocomm1} commute pairwise implies that $\alpha_i^m$ and $\rho$ commute, and that $\rho^m \in P_{n}(\St)$. 
If $\rho^m$ were of finite order then $\rho^m\in \ang{\ft}$, so $\xi(\rho^m)\equiv n-1$ modulo $2(n-1)$, and the exponent sum of $\rho^m$ would be congruent to $0$ modulo $n-1$. On the other hand, the exponent sum of $\rho^m$ modulo $n-1$ is equal to: 
\begin{align}
\overline{sm(m-1)}&=\overline{(n-i)(m-1)}\notag\\
&=\overline{(n-1)(m-1)}+\overline{(1-i)(m-1)} =\overline{(1-i)(m-1)}.\label{eq:xirhom}
\end{align}
Since $1-i\in\brak{1,-1}$, $n-1$ would thus divide $m-1$, which is not possible because $2\leq m\leq (n-i)/3<n$. Thus $\rho$ is of infinite order, and hence $\ang{\alpha_i^m, \rho^m}\cong \Z_{2s}\times \Z$. 

Using the element $\garside$ and the subgroup $\ang{\alpha_i^m,\rho^m}$, we will now construct a subgroup isomorphic to $\dic{4s}\times \Z$, which will complete the proof of the proposition. First note that for all $1\leq j_{1}< j_{2}\leq n-1$, the relation
\begin{equation}\label{eq:revft}
(\sigma_{j_{1}}\cdots \sigma_{j_{2}-1})^{j_{2}-j_{1}+1}=(\sigma_{j_{2}-1}\sigma_{j_{2}-2} \cdots \sigma_{j_{1}})^{j_{2}-j_{1}+1}
\end{equation}
holds in $B_{n}$ (\emph{cf.}~\cite[Chapter~2, Exercise~4.1]{M2}, and using the fact that $B_{j_{2}-j_{1}}$ embeds in $B_{n}$),  and so holds in $B_{n}(\St)$. Now
\begin{align*}
\garside \rho^m \garside^{-1} &= \garside \left( \prod_{j=1}^s \left(\sigma_{(j-1)m+1} \cdots \sigma_{jm-1}\right)^m  \right) \garside^{-1}\\ 
&= \left( \prod_{j=1}^s \left(\sigma_{m(s-j+1)+i-1} \cdots \sigma_{m(s-j)+i+1}\right)^m  \right)\quad \text{by \req{garsideconj}}\\
&= \left( \prod_{j=1}^s \left(\sigma_{m(s-j)+1} \cdots \sigma_{m(s-j+1)-1}\right)^m \right) \quad \text{by \req{revft}}
\\
&= \alpha_{0}^{i}\left( \prod_{j=1}^s \left(\sigma_{m(s-j)+1} \cdots \sigma_{m(s-j+1)-1}\right)^m \right) \alpha_{0}^{-i} \quad\text{by \req{fundaa}}\\
&= \alpha_{0}^{i}\left( \prod_{j'=1}^s \left(\sigma_{m(j'-1)+1} \cdots \sigma_{mj'-1}\right)^m \right) \alpha_{0}^{-i} = \alpha_{0}^{i} \rho^m\alpha_{0}^{-i},
\end{align*}
taking $j'=s-j+1$, and using also the fact that the inner bracketed terms commute pairwise. It follows from \req{basicconj} that $\garside$ commutes with the element $\rho'^m=\alpha_{0}^{i/2} \rho^m\alpha_{0}^{-i/2}$. Since 
$\ang{\alpha_i'^m,\garside} \cong \dic{4s}$ where $\alpha_i'^m=\alpha_{0}^{i/2} \alpha_i^m\alpha_{0}^{-i/2}$, it follows that the group $\ang{\alpha_i'^m,\garside,\rho'^m}$ is isomorphic to $\dic{4s}\times \Z$ as required.
\end{proof}

We now turn our attention to the other possible action in $B_{n}(\St)$ of $\Z$ on the dicyclic subgroups.

\begin{prop}\label{prop:nonk}
Let $n\geq 4$ and $s\geq 3$, and consider the Type~I group $G=\dic{4s}\rtimes_{\nu} \Z$, where $\nu$ is defined by \req{actdic4m}.
Then $B_{n}(\St)$ possesses a subgroup isomorphic to $G$ if and only if the following two conditions are satisfied:
\begin{enumerate}[(i)]
\item\label{it:nis1} $s$ divides $n-i$ for some $i\in \brak{0,2}$, and 
\item\label{it:nis2} $(n-i)/s$ is even.
\end{enumerate}
\end{prop}

\begin{proof}
The necessity of the conditions was obtained in part~(\ref{it:necV1notpoly})(\ref{it:necV1dicnontrivial}) of the proof of \repr{necV1}.
Conversely, suppose that conditions~(\ref{it:nis1}) and~(\ref{it:nis2}) hold. Set $m=(n-i)/s$, and let $\alpha_{i}'=\alpha_{0}\alpha_{i}\alpha_{0}^{-1}= \alpha_{0}^{i/2}\alpha_{i}\alpha_{0}^{-i/2}$.
Since $m/2\in \N$ by condition~(\ref{it:nis2}), we may consider the subgroup $\ang{\alpha_{i}'^{m/2},\garside}$ of $B_{n}(\St)$ which is a dicyclic subgroup (of order $8s$) of the standard copy of $\dic{4(n-i)}$, and which contains the dicyclic subgroup $\ang{\alpha_{i}'^{m},\garside}$ of order $4s$. Taking $x=\alpha_{i}'^{m}$ and $y=\garside$, the action by conjugation of $\alpha_{i}'^{m/2}$ on $\ang{x,y}$ coincides with that given by $\theta(1)$ in the statement of the proposition. From the proof of \repr{dic4sz}, the subgroup $\ang{\alpha_{i}'^{m},\garside,\rho'^m}$ is isomorphic to $\dic{4s}\times \Z$, $\rho'^{m}$ being as defined in that proof.
We claim that $\alpha_{i}'^{m/2}\rho'^m$ is of infinite order. This being the case, the subgroup $\ang{\alpha_{i}'^{m},\garside,\alpha_{i}'^{m/2}\rho'^m}$ is isomorphic to $\dic{4s}\rtimes_{\theta} \Z$, which will complete the proof of the proposition. To prove the claim, we suppose that $\alpha_{i}'^{m/2}\rho'^m$ is of finite order, and argue for a contradiction. Since $\rho'^m\in P_{n}(\St)$, $\pi(\alpha_{i}'^{m/2}\rho'^m)=\pi(\alpha_{i}'^{m/2})$. Now $\alpha_{i}'^{m/2}$ is of order $4s$, and the cycle decompositions of $\pi(\alpha_{i}'^{m/2})$ and $\pi(\alpha_{i}'^{m/2}\rho'^m)$ consist of $m/2$ $2s$-cycles (and $i$ fixed elements). The fact that the finite order elements of $P_{n}(\St)$ are the elements of $\ang{\ft}$ then implies that $\alpha_{i}'^{m/2}\rho'^m$ is of order $ks$, where $k\in \brak{2,4}$. Now $\alpha_{i}'^{2m/k}$ also generates a subgroup of order $ks$, and since $ks\geq 6$, by~\cite[Proposition~1.5(2)]{GG7}, there is a single conjugacy class of such subgroups in $B_{n}(\St)$. So there exist $\gamma\in B_{n}(\St)$ and $\lambda\in \N$, with $\gcd(\lambda,2s)=1$, such that $\alpha_{i}'^{2m\lambda/k}=\gamma \alpha_{i}'^{m/2}\rho'^m \gamma^{-1}$. But $\xi(\alpha_{i}')\equiv 0$ modulo $n-1$, and so it follows that $\xi(\rho'^m)=\xi(\rho^m)\equiv 0$ modulo $n-1$. But using \req{xirhom}, we saw in the proof of \repr{dic4sz} that this is not the case. This yields a contradiction, and proves the claim.
%
%
%
%
\end{proof}

\section{Type~I subgroups of $B_{\lowercase{n}}(\St)$ of the form $\quat\rtimes \Z$}\label{sec:realquat}

The aim of this section is to prove the existence of Type~I subgroups of $B_{n}(\St)$ of the form $\quat\rtimes \Z$. As we saw in \relem{quatact}, up to isomorphism it suffices to consider the two actions $\alpha$ and $\beta$ defined in \redef{v1v2}(\ref{it:mainIdef})(\ref{it:mainq8}), of order $3$ and $2$ respectively. We start by showing that the existence of the Type~I subgroup $\tonestar\times \Z$ (resp.\ of $\tonestar\rtimes_{\omega} \Z$, for the nontrivial action $\omega$ given by \req{nontrivacttstar}) implies that of $\quat\rtimes_{\alpha} \Z$ (resp.\ of $\quat\rtimes_{\beta} \Z$). Using the results of \resecglobal{realisation}{typeIbinpoly}, this will imply the existence of $\quat\rtimes_{\alpha} \Z$ and $\quat\rtimes_{\beta} \Z$ as subgroups of $B_{n}(\St)$ for most even values of $n$. In the second part of this section, we exhibit explicit algebraic constructions of $\quat\rtimes_{\alpha} \Z$ (resp.\ $\quat\rtimes_{\beta} \Z$) for all $n\equiv 0 \bmod 4$, $n\geq 8$ (resp.\ all $n\geq 4$ even).

\begin{prop}\label{prop:tstarxz}\mbox{}
\begin{enumerate}[(a)]
\item\label{it:tstarxza} The group $\tonestar\times \Z$ possesses a subgroup isomorphic to $\quat\rtimes_{\alpha} \Z$.
\item The group $\tonestar\rtimes_{\omega} \Z$ for the action defined by \req{nontrivacttstar} possesses a subgroup isomorphic to $\quat\rtimes_{\beta} \Z$.
\end{enumerate}
\end{prop}

\begin{proof} 
Consider $\tonestar=\quat \rtimes \Z_{3}$ given by the presentation~\reqref{preststar}.
\begin{enumerate}[(a)]
\item Let $G=\tonestar\times \Z$, and let $Z$ be the generator of the $\Z$-factor. Since $X$ and $Z$ commute, the group $\ang{XZ}$ is of infinite order and its action on $\quat$ by conjugation permutes cyclically the elements $P$, $Q$ and $PQ$ of $\ang{P,Q}$. Hence $\ang{P,Q, XZ}\cong \quat\rtimes_{\alpha} \Z$, where $\alpha$ is as defined in \redef{v1v2}(\ref{it:mainIdef})(\ref{it:mainq8}).
\item Let $G=\tonestar\rtimes_{\omega} \Z$, let $Z$ be the generator of the $\Z$-factor. The action of $Z$ on $\tonestar$ by conjugation coincides with that of \req{nontrivacttstar}. The restriction of this action to $\ang{P,Q}$ exchanges $P$ and $QP$, and sends $Q$ to $Q^{-1}$.
%
Thus $\ang{P,Q, Z}\cong \quat\rtimes_{\beta} \Z$, where $\beta$ is as defined in \redef{v1v2}(\ref{it:mainIdef})(\ref{it:mainq8}).\qedhere
\end{enumerate}
\end{proof}

\begin{rem}
The realisation of $\tonestar\times \Z$ (resp.\ $\tonestar\rtimes_{\omega} \Z$) as a subgroup of $B_{n}(\St)$ for $n$ even and satisfying $n=12$ or $n\geq 16$ (resp.\ $n\equiv 0,2 \bmod{6}$ and satisfying $n=24$ or $n\geq 30$) will follow from Propositions~\ref{prop:ttimesz} and~\ref{prop:trtimesz}. \repr{tstarxz} then implies the existence of $\quat\rtimes_{\alpha} \Z$ (resp.\ $\quat\rtimes_{\beta} \Z$) as a subgroup of $B_{n}(\St)$ for these values of $n$.
\end{rem}

We now turn our attention to the problem of the algebraic realisation of Type~I subgroups of the form $\quat\rtimes \Z$. In most cases, the existence of these subgroups follows by combining \repr{tstarxz} with Propositions~\ref{prop:ttimesz} and~\ref{prop:trtimesz}. As we shall see later, we will prove these two propositions using geometric constructions in $\mcg$. Before doing so, we exhibit explicit algebraic representations in terms of the standard generators of $B_{n}(\St)$, and in some cases, we obtain their existence for values of $n$ that are not covered by these propositions. We start by defining certain elements that shall be used in the constructions, and in \relem{propsomega}, we give some of their properties.
Let $n\geq 4$ be even, and let
\begin{equation}\label{eq:omegadef}
\Omega_{1}= \prod_{i=1}^{n/2-1} \; \sigma_{1}\cdots \sigma_{n/2-i} \quad \text{and} \quad \Omega_{2}= \prod_{i=1}^{n/2-1} \; \sigma_{n/2+1}\cdots \sigma_{n-i}.
\end{equation}
Clearly $\Omega_{1}$ and $\Omega_{2}$ commute, and using equations~\reqref{uniqueorder2} and~\reqref{fundaa}, we see that
\begin{equation}\label{eq:omegaconj}
\alpha_{0}^{n/2} \Omega_{1}\alpha_{0}^{-n/2}= \Omega_{2} \quad \text{and} \quad \alpha_{0}^{n/2} \Omega_{2}\alpha_{0}^{-n/2}= \Omega_{1}.
\end{equation}
For $i=1,\ldots,n/2$, set 
\begin{equation}\label{eq:rhoi}
\rho_{i}= \sigma_{i} \cdots \sigma_{i+n/2-1},
\end{equation}
and 
\begin{equation}\label{eq:rhodef}
\rho=\rho_{n/2} \cdots \rho_{1}.
\end{equation}
Geometrically, $\Omega_{1}$ (resp.\ $\Omega_{2}$) is the half twist on the first (resp.\ second) $n/2$ strings, and $\rho$ is the braid that passes the first $n/2$ strings over the second $n/2$ strings (see Figures~\ref{fig:omegai},~\ref{fig:omegai1} and~\ref{fig:omegai2}). 
\begin{figure}[h]
\hfill
\begin{tikzpicture}[scale=0.5]
\foreach \k in {1,2,3,4}
{\draw[thick] (\k+4,8) .. controls (\k+4,4.5) and (\k,4.5) .. (\k,1);};
\foreach \k in {1,2,3,4}
{\draw[draw=white,line width=5pt] (\k,8) .. controls (\k,4.5) and (\k+4,4.5) .. (\k+4,1);
\draw[thick] (\k,8) .. controls (\k,4.5) and (\k+4,4.5) .. (\k+4,1);};
\end{tikzpicture}\hspace*{\fill}
\caption{The braid $\rho$ in $B_8(\St)$.}\label{fig:omegai}
\end{figure}
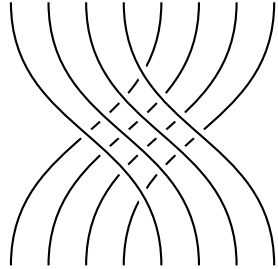
\begin{figure}[h]\hfill
\begin{tikzpicture}[scale=0.5]
\foreach \j in {2,3,4}
{\draw[thick] (\j,8) .. controls (\j,6) and (\j-1,7) .. (\j-1,5);}
{\draw[white,line width=7pt] (1,8) .. controls (1,5.5) and (4,7.5) .. (4,5);}
{\draw[thick] (1,8) .. controls (1,5.5) and (4,7.5) .. (4,5);};
\foreach \j in {2,3}
{\draw[thick] (\j,5) .. controls (\j,3.5) and (\j-1,4.5) .. (\j-1,3);}
{\draw[white,line width=7pt] (1,5) .. controls (1,3.5) and (3,4.5) .. (3,3);}
{\draw[thick] (1,5) .. controls (1,3.5) and (3,4.5) .. (3,3);};
\draw[thick] (2,3) .. controls (2,1.25) and (1,2.75) .. (1,1);
\draw[white,line width=7pt] (1,3) .. controls (1,1.25) and (2,2.75) .. (2,1);
\draw[thick] (1,3) .. controls (1,1.25) and (2,2.75) .. (2,1);
\draw[thick] (3,1) -- (3,3);
\draw[thick] (4,1) -- (4,5);
\foreach \k in {5,6,7,8}
{\draw[thick] (\k,1)--(\k,8);};
\end{tikzpicture}
\hspace*{\fill}
\caption{The braid $\Omega_{1}$ in $B_8(\St)$.}\label{fig:omegai1}
\end{figure}
\begin{figure}[h]\hfill
\begin{tikzpicture}[scale=0.5]
\foreach \j in {6,7,8}
{\draw[thick] (\j,8) .. controls (\j,6) and (\j-1,7) .. (\j-1,5);}
{\draw[white,line width=7pt] (5,8) .. controls (5,5.5) and (8,7.5) .. (8,5);}
{\draw[thick] (5,8) .. controls (5,5.5) and (8,7.5) .. (8,5);};
\foreach \j in {6,7}
{\draw[thick] (\j,5) .. controls (\j,3.5) and (\j-1,4.5) .. (\j-1,3);}
{\draw[white,line width=7pt] (5,5) .. controls (5,3.5) and (7,4.5) .. (7,3);}
{\draw[thick] (5,5) .. controls (5,3.5) and (7,4.5) .. (7,3);};
\draw[thick] (6,3) .. controls (6,1.25) and (5,2.75) .. (5,1);
\draw[white,line width=7pt] (5,3) .. controls (5,1.25) and (6,2.75) .. (6,1);
\draw[thick] (5,3) .. controls (5,1.25) and (6,2.75) .. (6,1);
\draw[thick] (7,1) -- (7,3);
\draw[thick] (8,1) -- (8,5);
\foreach \k in {1,2,3,4}
{\draw[thick] (\k,1)--(\k,8);};
\end{tikzpicture}
\hspace*{\fill}
\caption{The braid $\Omega_{2}$ in $B_8(\St)$.}\label{fig:omegai2}
\end{figure}
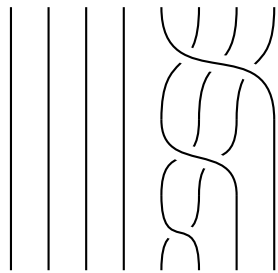

\begin{lem}\label{lem:propsomega}
With the above notation, the following relations hold:
\begin{enumerate}[(a)]
\item\label{it:omegai} $\rho\Omega_{1}=\Omega_{2}\rho$.
\item\label{it:omegaii} $\garside=\Omega_{1}\Omega_{2}\rho$.
\item\label{it:omegaiii} $\rho\Omega_{2}=\Omega_{1}\rho$.
\item\label{it:omegaiiibis} $\Omega_{2}=\sigma_{n-1}(\sigma_{n-2}\sigma_{n-1})\cdots (\sigma_{n/2+2} \cdots \sigma_{n-1})(\sigma_{n/2+1} \cdots \sigma_{n-1})$.
\item\label{it:omegaiv} $\alpha_{0}^{n/2}=\Omega_{1}^2 \rho$.
\item\label{it:omegav} $\garside=\Omega_{2} \alpha_{0}^{n/2}\Omega_{2}^{-1}$ and $\alpha_{0}^{n/2}=\Omega_{1} \garside\Omega_{1}^{-1}$.
\item\label{it:omegavi} $\ft=\Omega_{1}^2 \Omega_{2}^{-2}$.
\end{enumerate}
\end{lem}

\begin{proof}\mbox{}
\begin{enumerate}
\item First observe that
\begin{align}
\rho_{1} \Omega_{1}\rho_{1}^{-1}&= \sigma_{1} \cdots \sigma_{n/2} \left( \prod_{i=1}^{n/2-1} \; \sigma_{1}\cdots \sigma_{n/2-i} \right) \sigma_{n/2}^{-1} \cdots \sigma_{1}^{-1}\notag\\
&= \sigma_{1} \cdots \sigma_{n-1} \left( \prod_{i=1}^{n/2-1} \; \sigma_{1}\cdots \sigma_{n/2-i} \right) \sigma_{n-1}^{-1} \cdots \sigma_{1}^{-1}\notag\\
&= \alpha_{0} \Omega_{1}  \alpha_{0}^{-1}.\label{eq:rho1omega}
\end{align}
For $i=1,\ldots,n/2$, we have
\begin{align*}
\rho_{i} \alpha_{0}^{i-1} \Omega_{1}\alpha_{0}^{-(i-1)}\rho_{i}^{-i} &= \alpha_{0}^{i-1}  (\alpha_{0}^{-(i-1)}\rho_{i} \alpha_{0}^{i-1})  \Omega_{1} (\alpha_{0}^{-(i-1)}\rho_{i}^{-i} \alpha_{0}^{i-1}) \alpha_{0}^{-(i-1)}\\
&= \alpha_{0}^{i-1} \rho_{1} \Omega_{1}\rho_{1}^{-1} \alpha_{0}^{-(i-1)} \quad \text{by equations~\reqref{fundaa} and~\reqref{rhoi}}\\
&= \alpha_{0}^{i} \Omega_{1} \alpha_{0}^{-i} \quad\text{by \req{rho1omega}.}
\end{align*}
By induction on $i$, equations~\reqref{omegaconj} and~\reqref{rho1omega}, it follows that
\begin{equation*}
\rho \Omega_{1} \rho^{-1}= \rho_{n/2}\cdots \rho_{1} \Omega_{1} \rho_{1}^{-1}\cdots \rho_{n}^{-1}=\alpha_{0}^{n/2}\Omega_{1}\alpha_{0}^{-n/2}=\Omega_{2}
\end{equation*}
as required.

\item We have:
\begin{align*}
\garside &= \prod_{i=1}^{n-1} (\sigma_{1}\cdots \sigma_{n-i})=\prod_{i=1}^{n/2-1} (\sigma_{1}\cdots \sigma_{n-i}) \prod_{i=n/2}^{n-1} (\sigma_{1}\cdots \sigma_{n-i})\quad\text{by \req{defgarside}}\\
&=\prod_{i=1}^{n/2-1} (\sigma_{1}\cdots \sigma_{n/2-i})(\sigma_{n/2-i+1}\cdots \sigma_{n-i}) \prod_{i=0}^{n/2-1} (\sigma_{1}\cdots \sigma_{n/2-i})\\
&=\left(\prod_{i=1}^{n/2-1} (\sigma_{1}\cdots \sigma_{n/2-i}) \prod_{i=1}^{n/2-1}(\sigma_{n/2-i+1}\cdots \sigma_{n-i})\right) \rho_{1} \prod_{i=1}^{n/2-1} (\sigma_{1}\cdots \sigma_{n/2-i})\\
&=\Omega_{1} \left(\prod_{i=1}^{n/2-1} \rho_{n/2-i+1}\right)\rho_{1}\Omega_{1}\quad \text{by equations~\reqref{omegadef} and~\reqref{rhoi}}\\
&= \Omega_{1} \rho\Omega_{1} =\Omega_{1}  \Omega_{2} \rho \quad\text{by equations~\reqref{rhoi} and~\reqref{rhodef}, and part~(\ref{it:omegai}).}
%
\end{align*}

\item Since $\alpha_{0}^{n/2} \garside \alpha_{0}^{-n/2}=\garside^{-1}$ by equations~\reqref{uniqueorder2} and~\reqref{basicconj}, we have:
\begin{equation*}
\alpha_{0}^{n/2} \rho \alpha_{0}^{-n/2}=\alpha_{0}^{n/2} \Omega_{2}^{-1}\Omega_{1}^{-1} \garside\alpha_{0}^{-n/2} = \Omega_{1}^{-1}\Omega_{2}^{-1} \garside^{-1}= \rho \ft,\\
\end{equation*}
by part~(\ref{it:omegaii}) and \req{omegaconj}, using the fact that $\Omega_{1}$ and $\Omega_{2}$ commute. Conjugating the relation $\rho\Omega_{1}=\Omega_{2}\rho$ of part~(\ref{it:omegai}) by $\alpha_{0}^{n/2}$ and using \req{omegaconj} gives the result.

\item For $n/2+1\leq i\leq j\leq n-1$, set $\tau_{i,j}=\sigma_{i}\cdots \sigma_{j}$ (so $\tau_{i,i}=\sigma_{i}$). For $k=1,\ldots, n/2-1$, set 
\begin{equation*}
\omega_{k}=\left( \prod_{i=n/2-k}^{n/2-1} \tau_{n-i,n-1}  \right) \left( \prod_{i=n/2-k}^{n/2-1} \tau_{n/2+1,n/2+i}^{-1}  \right).
\end{equation*}
Let $\Omega_{2}'=\sigma_{n-1}(\sigma_{n-2}\sigma_{n-1})\cdots (\sigma_{n/2+2} \cdots \sigma_{n-1})(\sigma_{n/2+1} \cdots \sigma_{n-1})$. Since
\begin{align*}
\omega_{n/2-1}&=\left( \prod_{i=1}^{n/2-1} \tau_{n-i,n-1}  \right) \left( \prod_{i=1}^{n/2-1} \tau_{n/2+1,n/2+i}^{-1}  \right)\\
&=\Omega_{2}' \Omega_{2}^{-1} \quad\text{by \req{omegadef},}
\end{align*}
it suffices to show that $\omega_{n/2-1}=1$. To do so, we shall prove by induction that $\omega_{k}=1$ for all $k=1,\ldots, n/2-1$. If $k=1$ then $\omega_{1}=\tau_{n/2+1,n-1}\tau_{n/2+1,n-1}^{-1}=1$. So suppose that $\omega_{k}=1$ for some $k=1,\ldots, n/2-2$. First note that if $i\leq l<j$,
\begin{align}
\tau_{i,j}\tau_{i,l}^{-1}&=(\sigma_{i}\cdots \sigma_{l}\sigma_{l+1}\sigma_{l+2}\cdots\sigma_{j})(\sigma_{l}^{-1}\cdots \sigma_{i}^{-1})\notag\\
&=(\sigma_{i}\cdots \sigma_{l}\sigma_{l+1}\sigma_{l}^{-1}\cdots \sigma_{i}^{-1}) (\sigma_{l+2}\cdots\sigma_{j})\notag\\
&=\tau_{i,l+1} \sigma_{l}^{-1}\cdots \sigma_{i}^{-1} \tau_{i,l+1}^{-1} \tau_{i,j}= \sigma_{l+1}^{-1}\cdots \sigma_{i+1}^{-1}\tau_{i,j}\notag\\
& =\tau_{i+1,l+1}^{-1} \tau_{i,j},\label{eq:tauij}
\end{align}
using the fact that $\tau_{i,l+1}\tau_{m}\tau_{i,l+1}^{-1}=\tau_{m+1}$ for all $i\leq m\leq l$.
So
\begin{align*}
\omega_{k+1}&=\tau_{n/2+k+1,n-1} \left( \prod_{i=n/2-k}^{n/2-1} \tau_{n-i,n-1}  \right) 
\tau_{n/2+1,n-k-1}^{-1} \left( \prod_{i=n/2-k}^{n/2-1} \tau_{n/2+1,n/2+i}^{-1}  \right)\\ 
&=\tau_{n/2+k+1,n-1} \ldotp 
\tau_{n/2+k,n-1}\cdots \tau_{n/2+1,n-1}\tau_{n/2+1,n-k-1}^{-1} \left( \prod_{i=n/2-k}^{n/2-1} \tau_{n/2+1,n/2+i}^{-1}  \right)\\ 
&= \tau_{n/2+k+1,n-1} \ldotp \tau_{n/2+k+1,n-1}^{-1}
\left( \prod_{i=n/2-k}^{n/2-1} \tau_{n-i,n-1}  \right) \left( \prod_{i=n/2-k}^{n/2-1} \tau_{n/2+1,n/2+i}^{-1}  \right)\\
& =\omega_{k}=1,
\end{align*}
where we have used \req{tauij} $k$ times to obtain the first equality of the last line. The result follows by induction.

\item  Using parts~(\ref{it:omegaiii}) and~(\ref{it:omegaiiibis}), and equations~\reqref{omegadef},~\reqref{rhoi} and~\reqref{rhodef}, we have:
\begin{align*}
\alpha_{0}^{n/2}&= (\sigma_{1}\cdots \sigma_{n-1})^{n/2}= \prod_{i=1}^{n/2} (\sigma_{1}\cdots \sigma_{n/2-i}\sigma_{n/2-i+1} \cdots \sigma_{n-1})\\
&= \prod_{i=1}^{n/2-1} (\sigma_{1}\cdots \sigma_{n/2-i}) \prod_{i=1}^{n/2}(\sigma_{n/2-i+1} \cdots \sigma_{n-1})\\
& = \Omega_{1} \prod_{i=1}^{n/2}(\sigma_{n/2-i+1} \cdots \sigma_{n-i}\sigma_{n-i+1}\cdots \sigma_{n-1})\\
& = \Omega_{1} \prod_{i=1}^{n/2}(\sigma_{n/2-i+1} \cdots \sigma_{n-i}) \prod_{i=2}^{n/2}(\sigma_{n-i+1}\cdots \sigma_{n-1})\\
&= \Omega_{1} \rho \prod_{i=1}^{n/2-1}(\sigma_{n-i}\cdots \sigma_{n-1})=\Omega_{1} \rho \Omega_{2}=\Omega_{1}^{2} \rho.
%
\end{align*}

\item Applying successively parts~(\ref{it:omegaiv}),~(\ref{it:omegaiii}) and~(\ref{it:omegaii}) and using the fact that $\Omega_{1}$ and $\Omega_{2}$ commute yields:
\begin{equation*}
\Omega_{2} \alpha_{0}^{n/2}\Omega_{2}^{-1} = \Omega_{2} \Omega_{1}^2\rho \Omega_{2}^{-1}= \Omega_{2} \Omega_{1}^2 \Omega_{1}^{-1} \rho = \Omega_{1}\Omega_{2}\rho=\garside.
\end{equation*}
Further, using parts~(\ref{it:omegaii}),~(\ref{it:omegai}) and~(\ref{it:omegaiv}), we have:
\begin{equation*}
\Omega_{1} \garside\Omega_{1}^{-1} = \Omega_{1} \Omega_{1}\Omega_{2}\rho\Omega_{1}^{-1}=\Omega_{1}^{2}\rho=\alpha_{0}^{n/2}.
\end{equation*}

\item Using equations~\reqref{omegaconj} and~\reqref{uniqueorder2} and part~(\ref{it:omegav}), we have
\begin{align*}
\Omega_{1}^{2}\Omega_{2}^{-2} &= \Omega_{1}^{2} \alpha_{0}^{n/2}\Omega_{1}^{-2} \alpha_{0}^{-n/2}= \Omega_{1} \left(\Omega_{1} \alpha_{0}^{n/2}\Omega_{1}^{-1}\right) \Omega_{1}^{-1}\alpha_{0}^{-n/2}\\
&= \Omega_{1} \left(\alpha_{0}^{n/2}\Omega_{2}\alpha_{0}^{-n/2}\ldotp \alpha_{0}^{n/2} \ldotp\alpha_{0}^{n/2}\Omega_{2}^{-1}\alpha_{0}^{-n/2}\right) \Omega_{1}^{-1}\alpha_{0}^{-n/2}\\
&= \Omega_{1} \left(\alpha_{0}^{n/2}\garside\alpha_{0}^{-n/2}\right) \Omega_{1}^{-1}\alpha_{0}^{-n/2}= \Omega_{1} \garside^{-1} \Omega_{1}^{-1}\alpha_{0}^{-n/2}\\
&=\alpha_{0}^{-n/2}\ldotp \alpha_{0}^{-n/2}=\ft. \qedhere
\end{align*}
\end{enumerate}
\end{proof}


\begin{prop}\label{prop:constq8}
With the notation defined above,
\begin{enumerate}[(a)]
\item\label{it:q8parta} $\ang{\alpha_{0}^{n/2}, \garside, \alpha_{0}^{n/4}\Omega_{2}}\cong \quat\rtimes_{\alpha} \Z$ for all $n\equiv 0 \bmod 4$, $n\geq 8$.
\item\label{it:q8partb} $\ang{\alpha_{0}^{n/2}, \garside, \Omega_{1}\garside}\cong \quat\rtimes_{\beta} \Z$ for all $n\geq 4$ even.
\end{enumerate}
\end{prop}

\begin{proof}
\rerem{finitesub}(\ref{it:finitesubb}) implies that the subgroup $\ang{\alpha_{0}^{n/2}, \garside}$ of $B_{n}(\St)$ is isomorphic to $\quat$. So to prove the proposition, we must study the action of the third generator in both cases on this subgroup. Let $\Omega_{1}$ and $\Omega_{2}$ be as defined in \req{omegadef}.
\begin{enumerate}[(a)]
\item Let $n \equiv 0\bmod{4}$ where $n\geq 8$, and let $\nu=\alpha_{0}^{n/4}\Omega_{2}$. We have:
\begin{align}
\nu \alpha_{0}^{n/2} \nu^{-1}&= (\alpha_{0}^{n/4}\Omega_{2}) \alpha_{0}^{n/2} (\Omega_{2}^{-1} \alpha_{0}^{-n/4})= \alpha_{0}^{n/4}\garside \alpha_{0}^{-n/4} \quad \text{by \relem{propsomega}(\ref{it:omegav})}\notag\\
&= \alpha_{0}^{n/2}\garside \quad\text{by \req{basicconj}}, \label{eq:nuconj}\\
\nu \alpha_{0}^{n/2} \garside \nu ^{-1}&=\alpha_{0}^{n/2} \garside \ldotp \alpha_{0}^{n/4}\Omega_{2} \garside \Omega_{2}^{-1} \alpha_{0}^{-n/4} \quad\text{by \req{nuconj}}\notag\\
%
& =\alpha_{0}^{n/2} \garside \ldotp \alpha_{0}^{3n/4} \alpha_{0}^{-n/2}\Omega_{2}^2 \alpha_{0}^{n/2} \Omega_{2}^{-2} \alpha_{0}^{-n/4} \quad \text{by \relem{propsomega}(\ref{it:omegav})} \notag\\
&=\alpha_{0}^{n/2} \garside \ldotp \alpha_{0}^{3n/4} \Omega_{1}^2 \Omega_{2}^{-2} \alpha_{0}^{-n/4} \quad \text{by \req{omegaconj}}\notag\\
& =\alpha_{0}^{n/2} \garside \ldotp \alpha_{0}^{n/2} \ft  \quad \text{by \relem{propsomega}(\ref{it:omegavi})} \notag\\
&= \garside^{-1} \quad\text{by equations~\reqref{uniqueorder2} and~\reqref{basicconj}, and}\label{eq:nuconj2}\\
\nu\garside^{-1} \nu^{-1} &= \nu (\alpha_{0}^{n/2} \garside)^{-1} \alpha_{0}^{n/2} \nu^{-1}\\
&= \alpha_{0}^{n/2} \quad \text{by equations~\reqref{uniqueorder2},~\reqref{basicconj},~\reqref{nuconj} and~\reqref{nuconj2}}.\label{eq:nuconjgarside}
\end{align}
Hence conjugation by $\nu$ permutes cyclically the elements $\alpha_{0}^{n/2}$, $\alpha_{0}^{n/2} \garside$ and $\garside^{-1}$, and thus gives rise to the action $\alpha$ on the copy $\ang{\alpha_{0}^{n/2}, \garside}$ of $\quat$ in $B_{n}(\St)$. It remains to show that $\nu$ is of infinite order. Its permutation is:
\begin{align*}
\pi(\nu)=&(1,3n/4+1,n/2+1,n/4+1)(2,3n/4+2,n/2+2,n/4+2) \cdots \\
& (n/4,n,3n/4,n/2)(n/2+1,n)(n/2+2,n-1)\cdots (3n/4,3n/4+1),
\end{align*}
and since $n\geq 8$, the cycle decomposition of $\pi(\nu)$ contains the transposition $(3n/4+1,n)$ and the $6$-cycle $(1,3n/4, n/2, n/4, n/2+1,n/4+1)$. By \reth{murasugi}, $\pi(\nu)$ cannot be the permutation of an element of $B_{n}(\St)$ of finite order. This shows that $\nu$ is of infinite order, and so $\ang{\alpha_{0}^{n/2}, \garside, \alpha_{0}^{n/4}\Omega_{2}}\cong \quat\rtimes_{\alpha} \Z$.

\item Let $n\geq 4$ be even. Set $\zeta=\Omega_{1}\garside$. Then:
\begin{align*}
&\zeta \garside \zeta^{-1}= \Omega_{1}\garside \Omega_{1}^{-1}= 
\alpha_{0}^{n/2} \quad\text{by \relem{propsomega}(\ref{it:omegav}),}\\
&\zeta \alpha_{0}^{n/2} \zeta^{-1}
\begin{aligned}[t]
=&\Omega_{1}\garside \alpha_{0}^{n/2} \garside^{-1} \Omega_{1}^{-1} \alpha_{0}^{-n/2}\alpha_{0}^{n/2}\\
=& \ft \Omega_{1} \Omega_{2}^{-1}\alpha_{0}^{n/2}\,\text{by equations~\reqref{uniqueorder2},~\reqref{basicconj} and~\reqref{omegaconj}}\\
=&\ft\alpha_{0}^{n/2} \garside^{-1}\alpha_{0}^{n/2}\;\text{by \relem{propsomega}(\ref{it:omegaii}) and (\ref{it:omegaiv}), and the commutativity of $\Omega_{1}$ and $\Omega_{2}$}\\
=&\garside \quad\text{by equations~\reqref{uniqueorder2} and~\reqref{basicconj},}
\end{aligned}\\
&\zeta \alpha_{0}^{n/2}\garside \zeta^{-1}= \garside\alpha_{0}^{n/2}=(\alpha_{0}^{n/2}\garside)^{-1} \quad\text{by \req{uniqueorder2} and the above two relations.}
\end{align*}
So conjugation by $\zeta$ exchanges $\garside$ and $\alpha_{0}^{n/2}$, and sends $\alpha_{0}^{n/2} \garside$ to $(\alpha_{0}^{n/2}\garside)^{-1}$, hence gives rise to the action $\beta$ on the copy $\ang{\alpha_{0}^{n/2},\garside}$ of $\quat$ in $B_{n}(\St)$. It remains to show that $\zeta$ is of infinite order. Suppose first that $n\equiv 2 \bmod 4$. Then:
\begin{align*}
\pi(\zeta)=& (1,n/2)(2,n/2-1)\cdots ((n-2)/4,(n+6)/4)\ldotp (1,n)\ldotp\\
&(2,n-1)\cdots (n/2,n/2+1).
\end{align*}
By \reth{murasugi}, $\pi(\zeta)$ cannot be the permutation of a finite-order element of $B_{n}(\St)$ since its cycle decomposition contains the transposition $((n+2)/4,(3n+2)/4)$ and the $4$-cycle $(1,n/2+1,n/2,n)$. So suppose that $n \equiv 0 \bmod 4$. Then:
\begin{equation*}
\pi(\zeta)= (1,n/2)(2,n/2-1)\cdots (n/4,n/4+1)\ldotp (1,n)(2,n-1)\cdots (n/2,n/2+1).
\end{equation*}
Hence the cycle decomposition of $\pi(\zeta)$ consists of the $4$-cycles of the form $(j,n/2+j,n/2+1-j,n+1-j)$, where $1\leq j\leq n/4$, so if $\zeta$ is of finite order then by \reth{murasugi}, it is conjugate to a power of $\alpha_{0}$. Thus the Abelianisation of $\zeta$ is congruent to $0$ modulo $n-1$. On the other hand, the Abelianisation of $\zeta=\Omega_{1}\garside$ is congruent to $\frac{n}{4}\left( \frac{n}{2}-1 \right)+\frac{n}{2}(n-1)$ modulo $2(n-1)$, which is congruent modulo $n-1$ to $\frac{n}{8}(n-2)$. But $\frac{n}{8}(n-2) \nequiv 0$ modulo $n-1$, which gives a contradiction. So $\zeta$ is of infinite order, and $\ang{\alpha_{0}^{n/2}, \garside, \zeta}\cong \quat\rtimes_{\beta} \Z$ as required.\qedhere
\end{enumerate}
\end{proof}

\begin{rems}\mbox{}
\begin{enumerate}[(a)]
\item If we take $n=4$ in the proof of \repr{constq8}(\ref{it:q8parta}) then $\nu=\alpha_{1}$ is of order $6$, and we obtain the subgroup $\ang{\alpha_{0}^{2}, \garside[4], \alpha_{0}\Omega_{2}}$ which is isomorphic to $\tonestar$ (see~\cite[Remark~3.2]{GG7}). However, a copy of $\quat \rtimes_{\alpha} \Z$ in $B_{4}(\St)$ will be exhibited in the proof of \repr{realV1}. Combining this with Propositions~\ref{prop:tstarxz},~\ref{prop:constq8} and~\ref{prop:ttimesz} will prove the existence of Type~I subgroups of $B_{n}(\St)$ of the form $\quat \rtimes_{\alpha} \Z$ for all $n\geq 4$ even, with the exception of $n\in \brak{6,10,14}$. \repr{constq8}(\ref{it:q8partb}) implies the existence of Type~I subgroups of $B_{n}(\St)$ of the form $\quat \rtimes_{\beta} \Z$ for all $n\geq 4$ even.
\item In the case where $n\equiv 2 \bmod 4$, we do not know of an explicit algebraic representation of $\quat \rtimes_{\alpha} \Z$ similar to that of the construction of \repr{constq8}(\ref{it:q8parta}) in the case $n\equiv 0 \bmod{4}$. In order to obtain such a representation, note that by~\cite[Proposition~1.5 and Theorem~1.6]{GG7}, the standard copy $\ang{\alpha_{2}',\garside}$ of $\dic{4(n-2)}$ exhibits both conjugacy classes of subgroups isomorphic to $\quat$ in $B_{n}(\St)$. To construct a copy $H$ of $\quat \rtimes_{\alpha} \Z$, the elements of the copy of $\quat$ of order $4$ must be conjugate, so $H=\ang{\alpha_{2}'^{(n-2)/2},\alpha_{2}' \garside}$ (up to conjugacy). We then need to look for an element $z$ of $B_{n}(\St)$ of infinite order whose action by conjugacy on $H$ permutes cyclically the elements $\alpha_{2}'^{(n-2)/2},\garside$ and $\alpha_{2}'^{(n-2)/2}\garside$ of $H$ (or perhaps their inverses). Propositions~\ref{prop:tstarxz}(\ref{it:tstarxza}) and~\ref{prop:ttimesz}(\ref{it:ttimesza}) imply the existence of $z$, but we have not been able to find explicitly such an element.
\end{enumerate}
\end{rems}

\section{Type~I subgroups of $B_{n}(\St)$ of the form $F\rtimes \Z$  with $F=\tonestar, \oonestar, \istar$}\label{sec:typeIbinpoly}

We now consider the problem of the existence of Type~I subgroups of $B_{n}(\St)$ of the form $F\rtimes \Z$  with $F=\tonestar, \oonestar, \istar$. In the case where the product is direct, the question will be treated in \resec{dirtoi}. \repr{necV1} asserts that the only nontrivial action occurs when $F=\tonestar$, in which case the action is that given by \req{nontrivacttstar}. This possibility will be dealt with in \resec{nontrivtstar}.

\subsection{Type~I subgroups of $B_{n}(\St)$ of the form $F\times \Z$  with $F=\tonestar, \oonestar, \istar$}\label{sec:dirtoi}


In this section, we prove the following result.
\begin{prop}\label{prop:ttimesz}\mbox{}
\begin{enumerate}[(a)]
\item\label{it:ttimesza} Suppose that $n=12$ or that $n\geq 16$ is even. Then the group $\tonestar\times \Z$ is realised as a subgroup of $B_{n}(\St)$.
\item\label{it:ttimeszb} Suppose that $n=24$ or that $n\geq 30$ is congruent to $0$ or $2\bmod 6$. Then the group $\oonestar\times \Z$ is realised as a subgroup of $B_{n}(\St)$. 
\item\label{it:ttimeszc} Suppose that $n=60$ or that $n\geq 72$ is congruent to $0,2,12$ or $20 \bmod 30$. Then the group $\istar\times \Z$ is realised as a subgroup of $B_{n}(\St)$.
\item\label{it:ttimeszd} The group $\tonestar\times \Z$ (resp.\ $\oonestar\times \Z$) is not realised as a subgroup of $B_{4}(\St)$ (resp.\ $B_{6}(\St)$).
\end{enumerate}
\end{prop}

\begin{rems}\mbox{}\label{rem:proddirbin}
\begin{enumerate}
\item Since $\tonestar\times \Z$ is not realised as a subgroup of $B_{4}(\St)$, neither is $\tonestar\rtimes_{\omega} \Z$.
\item\label{it:proddirbinb} \reth{main}(\ref{it:mainIII}) follows immediately from \repr{ttimesz}(\ref{it:ttimeszd}).
\end{enumerate}
\end{rems}


\begin{rem}\label{rem:opentypeI}
For the following values of $n$ not covered by \repr{ttimesz}, the associated binary polyhedral group occurs as a subgroup of $B_{n}(\St)$, but it is an open question as to whether the given direct product is realised or not:
\begin{enumerate}[(i)]
\item $\tonestar\times \Z$, for $n\in \brak{6,8,10,14}$,
\item $\oonestar\times \Z$, for $n\in \brak{8,12,14,18,20,26}$,
\item $\istar\times \Z$, for $n\in \brak{12,20,30,32,42,50,62}$.
\end{enumerate}
\end{rem}

\begin{proof}[Proof of \repr{ttimesz}.]
We start by proving part~(\ref{it:ttimesza}). Suppose that $n=12$ or that $n\geq 16$ is even. Set $n=6l+4m$, where $l\geq 2$ and $m\in\brak{0,1,2}$. Let $\Delta$ be a regular tetrahedron, and let $X\subset \Delta$ be an $n$-point subset invariant under the action of the group $\Gamma\cong \an[4]$ of rotations of $\Delta$. We may suppose that each edge of $\Delta$ contains $l$ equally-spaced points in its interior. If $m\geq 1$ then we place four points of $X$ at the vertices of $\Delta$, and if $m=2$, we add a further four points at the barycentres of the faces. We inscribe $\Delta$ within the sphere $\St$, and from now on, the two shall be identified by radial projection without further comment. 

Recall that $\homeo$ and
\begin{equation*}
\map{\Psi}{\homeo}[\mcg]
\end{equation*}
were defined in \resecglobal{generalities}{conjfinite}.
Now $\Gamma$ is a subgroup of $\homeo$ whose image $\widetilde{\Gamma}=\Psi(\Gamma)$ under $\Psi$ is also isomorphic to $\an[4]$. Indeed, $f \in \homeo$ belongs to $\ker{\Psi}$ if and only if it is isotopic to the identity relative to $X$. Such an $f$ would thus fix $X$ pointwise, but the only element of $\Gamma$ which achieves this is the identity. So the restriction of $\Psi$ to $\Gamma$ is injective.

Since $\widetilde{\Gamma} \cong \an[4]$, the preimage $\Lambda=\phi^{-1}\bigl(\widetilde{\Gamma}\bigr)$ under the homomorphism $\phi$ of \req{mcg} is a copy of $\tonestar$. The aim is to prove the existence of an element $v$ of infinite order belonging to the centraliser of $\Lambda$ in $B_{n}(\St)$. We claim that it suffices to exhibit an element $\widetilde{z}$ of infinite order belonging to the centraliser of $\widetilde{\Gamma}$ in $\mcg$. Indeed, suppose such a $\widetilde{z}$ exists, and let $z\in B_{n}(\St)$ be a preimage of $\widetilde{z}$ under $\phi$. Clearly $z$ is also of infinite order. Let $w\in \Lambda$, and let $\widetilde{w}=\phi(w)\in \widetilde{\Gamma}$. Then $\phi([w,z])=[\widetilde{w},\widetilde{z}]=1$, so $[w,z]=\ftalt{\epsilon}$, where $\epsilon\in \brak{0,1}$. Thus  $wzw^{-1}=z\ftalt{\epsilon}$, hence $wz^2w^{-1}=z^2$ for all $w\in \Lambda$, and so we may take $v=z^2$. It follows that $\ang{\Lambda, v}\cong \tonestar \times \Z$.

To prove the existence of $\widetilde{z}$, denote the edges of $\Delta$ by $e_{1},\ldots,e_{6}$, and for $j=1,\ldots,6$, let $f_{j}\in \Gamma$ be such that $f_{j}(e_{1})=e_{j}$ (we choose $f_{1}=\id$). Let $\mathcal{C}_{1}$ be a positively oriented simple closed curve containing the $l$ points of $X$ belonging to $e_{1}$, and let $\mathcal{A}_{1}$ be a small annular neighbourhood of $\mathcal{C}_{1}$, chosen so that the orbit $\mathcal{C}$ of $\mathcal{C}_{1}$ (resp.\ the orbit $\mathcal{A}$ of $\mathcal{A}_{1}$) under the action of $\Gamma$ consists of the six (disjoint) oriented simple closed curves $\mathcal{C}_{j}=f_{j}(\mathcal{C}_{1})$, $j=1,\ldots,6$ (resp.\ six pairwise-disjoint annuli $\mathcal{A}_{j}=f_{j}(\mathcal{A}_{1})$, $j=1,\ldots,6$) (see Figure~\ref{fig:tetra}). The orbits $\mathcal{C}$ and $\mathcal{A}$ are obviously invariant under this action. Each $\mathcal{C}_{i}$ is associated with the edge $e_{i}$ of $\Delta$, and bounds a disc containing the $l$ points of $X \cap e_{i}$. Let $T_{1}\in \homeo$ be the (positive) Dehn twist along $\mathcal{C}_{1}$ in $\mathcal{A}_{1}$, and set $T_{i}=f_{i}\circ T_{1} \circ f_{i}^{-1}$. Then $T_{i}$ is the (positive) Dehn twist along $\mathcal{C}_{i}$ in $\mathcal{A}_{i}$. Since the $\mathcal{A}_{i}$ are pairwise disjoint, the $T_{i}$ commute pairwise. 

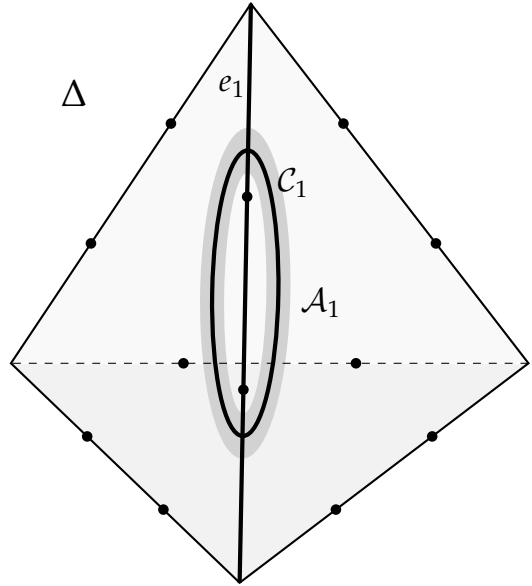
\begin{figure}[h]
\centering
\begin{tikzpicture}[scale=4]
\coordinate (A) at (-0.0625,0.696);
\coordinate (B) at (0.86,-0.5);
\coordinate (C) at (-0.86,-0.5);
\coordinate (D) at (-0.1,-1.23);

\node[black] (a) at (A) {};
\node[black] (b) at (B) {};
\node[black] (c) at (C) {};
\node[black] (d) at (D) {};

\draw[white,fill=gray!5] (A)--(B)--(C)--cycle;
\draw[white,fill=gray!10] (B)--(D)--(C)--cycle;

\fill [black] 
($2/3*(a) + 1/3*(b)$) circle (0.5pt)
($1/3*(a) + 2/3*(b)$) circle (0.5pt)
($2/3*(a) + 1/3*(c)$) circle (0.5pt)
($1/3*(a) + 2/3*(c)$) circle (0.5pt)
($2/3*(a) + 1/3*(d)$) circle (0.5pt)
($1/3*(a) + 2/3*(d)$) circle (0.5pt)
($2/3*(c) + 1/3*(b)$) circle (0.5pt)
($1/3*(c) + 2/3*(b)$) circle (0.5pt)
($2/3*(d) + 1/3*(b)$) circle (0.5pt)
($1/3*(d) + 2/3*(b)$) circle (0.5pt)
($2/3*(d) + 1/3*(c)$) circle (0.5pt)
($1/3*(d) + 2/3*(c)$) circle (0.5pt);

%

\fill[even odd rule,color=gray!35,rotate=89] ($1/2*(a) + 1/2*(d)$) 
ellipse (0.4cm and 0.07cm) ($1/2*(a) + 1/2*(d)$) circle (0.55cm and 0.15cm);

\draw[ultra thick,rotate=89] ($1/2*(a) + 1/2*(d)$) ellipse (0.475cm and 0.11cm);

\draw (-0.65,0.4) node {\large$\Delta$};
\draw (0.17,-0.3) node {$\mathcal{A}_{1}$};
\draw (-0.125,0.42) node {$e_{1}$};
\draw (0.078,0.1) node {$\mathcal{C}_{1}$};

\draw[thick] (A)--(C)--(D)--(B)--cycle;
\draw[ultra thick] (A)--(D);
\draw[dashed] (B)--(C);

\end{tikzpicture}
\caption{The geometric construction of $\widetilde{z}$ in $\mcg[12]$.}\label{fig:tetra}
\end{figure}

Set $T=T_{1}\circ \cdots \circ T_{6}$. Let us prove that $T$ is an element of $\homeo$ of infinite order belonging to the centraliser of $\Gamma$. To see this, let $f\in \Gamma$, and let $x\in \Delta$. First suppose that $x\notin \bigcup_{i=1}^6 \mathcal{A}_{i}$. Then $f(x)\notin \bigcup_{i=1}^6 \mathcal{A}_{i}$ since $\bigcup_{i=1}^6 \mathcal{A}_{i}$ is invariant under the action of $\Gamma$, and so $f\circ T(x)=f(x)=T\circ f(x)$ as required. Now assume that $x\in \mathcal{A}_{j}$ for some $j=1,\ldots,6$. By relabelling the edges of $\Delta$ if necessary, we may suppose that $x\in \mathcal{A}_{1}$. Let $\ang{g}\cong \Z_{2}$ be the stabiliser of $e_{1}$ in $\Gamma$. If we parametrise $\mathcal{A}_{1}$ as $[0,1]\times \St[1]$ then $T_{1}$ is defined by $T_{1}(t,s)=(t,se^{2\pi it})$, and the restriction of $g$ to $\mathcal{A}_{1}$ is given by $g(t,s)=(t,se^{\pi i})$. A straightforward calculation shows that $g\circ T_{1}=T_{1}\circ g$ on $\mathcal{A}_{1}$. By considering the action on the oriented edges of $\Delta$, it follows that there exist $i\in \brak{1,\ldots,6}$ and $\epsilon\in \brak{0,1}$ such that $f=f_{i}\circ g^{\epsilon}$, so $f(x)\in \mathcal{A}_{i}$, and:
\begin{equation*}
T\circ f(x)= T_{i}\circ f(x)=T_{i}\circ f_{i}\circ g^{\epsilon} (x)=f_{i}\circ T_{1}\circ g^{\epsilon}(x)= f_{i}\circ g^{\epsilon}\circ T_{1}(x)=f\circ T(x),
\end{equation*}
using the facts that the $T_{j}$ commute pairwise, and that for $j=1,\ldots,6$, the support of $T_{j}$ is $\mathcal{A}_{j}$. This shows that $T$ belongs to the centraliser of $\Gamma$ in $\homeo$, and so $\widetilde{T}=\Psi(T)$ belongs to the centraliser of $\widetilde{\Gamma}$ in $\mcg$. It remains to show that $\widetilde{T}$ is of infinite order. This is a consequence of a generalisation of the intersection number formula for Dehn twists, see~\cite[Propositions~3.2 and~3.4]{FM} for example. An alternative proof of this fact is as follows. Since $\widetilde{T}$ belongs to the pure mapping class group $\pmcg$ of $\St$ on $n$ points, we may consider its image $\widehat{T}$ under the homomorphism $\pmcg \to \pmcg[4]$, obtained in an analogous manner to the Fadell-Neuwirth homomorphism by removing all but two pairs of points, one pair contained in the small disc bounded by $\mathcal{C}_{1}$, and another pair contained in that bounded by $\mathcal{C}_{2}$. Since $\mathcal{C}_{1}$ and $\mathcal{C}_{2}$ are both positively oriented, $\widehat{T}$ is the image under $\phi$ of a pure braid, which choosing appropriate generators, may be written as $\sigma_{1}^2 \sigma_{3}^2$. We saw at the end of the proof of \repr{rtimes} that this element is of infinite order, and this
%
%
implies that $\widetilde{T}$ is also of infinite order. We have thus shown that there exists an element $\widetilde{z}=\widetilde{T}$ of infinite order belonging to the centraliser of $\widetilde{\Gamma}$ in $\mcg$, and this proves part~(\ref{it:ttimesza}).

Applying a similar construction for $\oonestar$ (taking $\Delta$ to be a cube) and for $\istar$ (taking $\Delta$ to be a dodecahedron) yields parts~(\ref{it:ttimeszb}) and~(\ref{it:ttimeszc}). Note that in the case of $\oonestar$ (resp.\ $\istar$), we have that $n\equiv 0,2 \bmod 6$ (resp.\ $n\equiv 0,2,12,20 \bmod 30$). We set $n=12l+8m+6r$ (resp.\ $n=30l+20m+12r$), where $l\in \N$, and $m,r\in \brak{0,1}$ denote respectively the number of points of $X$ placed at the vertices of $\Delta$ and at the barycentre of the faces of $\Delta$. Since we require $l\geq 2$ in the construction, the excluded values of $n$ are $6,8,12,14,18,20$ and $26$ (resp.\ $12,20,30,32,42,50$ and $62$). 

Finally, we prove part~(\ref{it:ttimeszd}). Suppose that $\tonestar\times \Z$ (resp.\ $\oonestar\times \Z$) is realised as a subgroup $K$ of $B_{4}(\St)$ (resp.\ $B_{6}(\St)$). Since $\tonestar$ (resp.\ $\oonestar$) possesses elements of order $6$ (resp.\ $8$) by \repr{maxsubgp}, $K$ contains a subgroup $H$ isomorphic to $\Z_{6}\times \Z$ (resp.\ $\Z_{8}\times \Z$) whose finite factor is conjugate to $\ang{\alpha_{1}}$  (resp.\ $\ang{\alpha_{2}}$) by \reth{murasugi}. But the existence of $H$ then contradicts \repr{luis}.
\end{proof}

As a consequence of \repr{ttimesz}(\ref{it:ttimeszd}), we obtain the following result which complements that of \repr{genhodgkin1}.
\begin{cor}
If $n=4$ (resp.\ $n=6$), let $H$ be a subgroup of $B_{n}(\St)$ isomorphic to $\tonestar$ (resp.\ $\oonestar$). Then the normaliser of $H$ in $B_{n}(\St)$ is $H$ itself.
\end{cor}

\begin{proof}
In both cases, $H$ is finite maximal by \reth{finitebn}, so it suffices to prove that $N=N_{B_{n}(\St)}(H)$ is finite. If $x\in N$ then some power of $x$ belongs to $Z_{B_{n}(\St)}(H)$, but by \repr{ttimesz}(\ref{it:ttimeszd}), $x$ must be of finite order. Hence $N$ is finite by \repr{finord}.
\end{proof}

\subsection{Realisation of $\tonestar\rtimes_{\omega} \Z$}\label{sec:nontrivtstar}

We now consider the realisation of $\tonestar\rtimes_{\omega} \Z$ as a subgroup of $B_{n}(\St)$, where $\omega(1)$ is as defined in \req{nontrivacttstar}.

\begin{prop}\label{prop:trtimesz}
If $\oonestar \times \Z$ is realised as a subgroup of $B_{n}(\St)$ then so is $\tonestar\rtimes_{\omega} \Z$.
\end{prop}

\begin{rem}\label{rem:opentypeIa}
Let $n\notin\brak{6,8,12,14,18,20,26}$. Putting together the results of Propositions~\ref{prop:ttimesz}(\ref{it:ttimeszb}) and~\ref{prop:trtimesz}, we see that $\tonestar\rtimes_{\omega} \Z$ is realised as a subgroup of $B_{n}(\St)$ if and only if $n\equiv 0,2 \bmod 6$.
\end{rem}


\begin{proof}[Proof of \repr{trtimesz}.] Suppose that $\oonestar \times \Z$ is realised as a subgroup $L$ of $B_{n}(\St)$. Then there exist a subgroup $K$ of $B_{n}(\St)$ isomorphic to $\oonestar$ and an element $z\in B_{n}(\St)$ of infinite order such that $z$ belongs to the centraliser of $K$. Let~\req{presostar} denote a presentation of $K$, and let $H=\ang{P,Q,X}$ denote the subgroup of $K$ isomorphic to $\tonestar$, with the presentation of \req{preststar}. Equations~\reqref{nontrivacttstar} and~\reqref{presostar} imply that the restriction  to $H$ of conjugation by the element $R$ of $K$ represents the nontrivial element of $\out{\tonestar}$. But $z$ commutes with $R$, so $zR$ is of infinite order, and since $z$ also belongs to the centraliser of $H$, it follows that $\ang{H,zR}\cong \tonestar\rtimes_{\omega} \Z$.
\end{proof}

\section{Proof of the realisation of the elements of $\mathbb{V}_{1}(\lowercase{n})$ in $B_{\lowercase{n}}(\St)$}\label{sec:realtypeI}

In this section, we bring together the results of Sections~\ref{part:realisation}.\ref{sec:centcycdic}--\ref{sec:typeIbinpoly} to prove \repr{realV1}. This proposition will imply \reth{main}(\ref{it:mainII}) for the Type~I subgroups of $B_{n}(\St)$, namely the realisation of the virtually cyclic groups given by~(\ref{it:mainIdef})(\ref{it:mainzq})--(\ref{it:maini}) of \redef{v1v2}, with the exception of the values of $n$ given in \rerem{exceptions}(\ref{it:exceptions}) and not covered by \reth{main}(\ref{it:mainII})(\ref{it:excepa})--(\ref{it:excepd}). 

\begin{prop}\label{prop:realV1}
Let $n\geq 4$. The following Type~I virtually cyclic groups are realised as subgroups of $B_{n}(\St)$:
\begin{enumerate}[(a)]
\item\label{it:realV11a} $\Z_{q}\times \Z$, where $q\divides 2(n-i)$ with $i\in\brak{0,1,2}$, $1\leq q\leq n-i$, and $q<n-i$ if $n-i$ is odd.
\item\label{it:realV11b} $\Z_{q}\rtimes_{\rho} \Z$, where $q\divides 2(n-i)$ with $i\in\brak{0,2}$, $3\leq q\leq n-i$, $q<n-i$ if $n-i$ is odd, and $\rho(1)\in \aut{\Z_{q}}$ is multiplication by $-1$.
\item\label{it:realV13a} $\dic{4m}\times \Z$, where $m\divides n-i$ with $i\in \brak{0,2}$, and $3\leq m\leq (n-i)/2$.
\item\label{it:realV13b} $\dic{4m}\rtimes_{\nu} \Z$, where $m\divides n-i$ with $i\in \brak{0,2}$, $m\geq 3$, $(n-i)/m$ is even, and $\nu(1)$ is the automorphism of $\dic{4m}$ given by \req{actdic4m}.
\item\label{it:realV12} \begin{enumerate}[(i)]
\item\label{it:realV12a} $\quat \times \Z$ for all $n$ even.
\item\label{it:realV12b} $\quat \rtimes_{\alpha} \Z$, for all $n$ even, $n\notin\brak{6,10,14}$, where $\alpha(1)\in \aut{\quat}$ is given by $\alpha(1)(i)=j$ and $\alpha(1)(j)=k$, and where $\quat=\brak{\pm 1, \pm i, \pm j, \pm k}$.

\item\label{it:realV12c} $\quat \rtimes_{\beta} \Z$ for all $n$ even, where $\beta(1)\in \aut{\quat}$ is given by $\beta(1)(i)=k$ and $\beta(1)(j)=j^{-1}$.
\end{enumerate}
\item\label{it:realV14a} $\tonestar\times \Z$, where $n=12$ or $n\geq 16$ is even.
\item\label{it:realV14d} $\tonestar\rtimes_{\omega} \Z$, where $n=24$ or $n\geq 30$ and $n\equiv 0,2\bmod{6}$, and $\omega(1)$ is the automorphism of $\tonestar$ given by \req{nontrivacttstar}.

\item\label{it:realV14b} $\oonestar\times \Z$, where $n=24$ or $n\geq 30$ and $n\equiv 0,2\bmod{6}$.
\item\label{it:realV14c} $\istar\times \Z$, where $n=60$ or $n\geq 72$ and $n\equiv 0,2,12,20\bmod{30}$.
\end{enumerate}
\end{prop}


\begin{proof}
Parts~(\ref{it:realV11a}),~(\ref{it:realV11b}),~(\ref{it:realV13a}) and~(\ref{it:realV13b}) are proved in Propositions~\ref{prop:necsuftimes},~\ref{prop:rtimes},~\ref{prop:dic4sz} and~\ref{prop:nonk} respectively.
By \repr{constq8}(\ref{it:q8partb}), $\quat\rtimes_{\beta}\Z$ is realised as a subgroup of $B_{n}(\St)$ for all $n\geq 4$ even, and its subgroup generated by $\quat$ and the square of the $\Z$-factor is abstractly isomorphic to $\quat\times\Z$, which proves parts~(\ref{it:realV12})(\ref{it:realV12a}) and~(\ref{it:realV12c}). We now consider the realisation of $\quat \rtimes_{\alpha}\Z$ as a subgroup of $B_{n}(\St)$. Suppose first that $n\equiv 0 \bmod{4}$. If $n\geq 8$ then the result follows from \repr{constq8}(\ref{it:q8parta}). So suppose that $n=4$. By~\cite[Theorem~1.3(3)]{GG5}, $B_{4}(\St)$ contains a copy of $\quat$ generated by $x=\sigma_{3}\sigma_{1}^{-1}$ and $y =(\sigma_{1}^2 \sigma_{2}\sigma_{1}^{-3})\sigma_{3}\sigma_{1}^{-1}(\sigma_{1}^3\sigma_{2}^{-1} \sigma_{1}^{-2})$,
and the element $a=\sigma_{1}^2 \sigma_{2}\sigma_{1}^{-3}$, which is of infinite order, acts by conjugation on $\ang{x,y}$ by sending $x$ to $y$ and $y$ to $xy$. Hence the subgroup $\ang{x,y,a}$ of $B_{4}(\St)$ is isomorphic to $\quat \rtimes_{\alpha} \Z$ as required. Now suppose that $n\equiv 2 \bmod{4}$. If $n\notin \brak{6,10,14}$ then $n\geq 18$. So by \repr{ttimesz}(\ref{it:ttimesza}), $\tonestar \times \Z$ is realised as a subgroup of $B_{n}(\St)$, and we deduce from \repr{tstarxz}(\ref{it:tstarxza}) that $B_{n}(\St)$ contains a copy of $\quat \rtimes_{\alpha}\Z$, which proves part~(\ref{it:realV12})(\ref{it:realV12b}).
Parts~(\ref{it:realV14a}),~(\ref{it:realV14b}) and~(\ref{it:realV14c}) follow directly from \repr{ttimesz}(\ref{it:ttimesza})--(\ref{it:ttimeszc}). Finally, to prove part~(\ref{it:realV14d}), if $n=24$ or $n\geq 30$ and $n\equiv 0,2\bmod{6}$ then $\oonestar \times \Z$ is realised as a subgroup of $B_{n}(\St)$ by \repr{ttimesz}(\ref{it:ttimeszb}), and so $B_{n}(\St)$ contains a copy of $\tonestar \rtimes_{\omega} \Z$ by \repr{trtimesz}.
\end{proof}

\begin{rem}\label{rem:tstarb6}
In \repr{realV1}(\ref{it:realV12})(\ref{it:realV12b}), we do not know whether the Type~I group $\quat\rtimes_{\alpha} \Z$ is realised as a subgroup of $B_{n}(\St)$ for $n\in \brak{6,10,14}$. In~\cite[Remark~3.3]{GG7}, we exhibited a copy $\ang{\gamma,\delta}$ of $\tonestar$ in $B_{6}(\St)$, where 
\begin{equation*}
\text{$\gamma=\sigma_{5}\sigma_{4} \sigma_{1}^{-1}\sigma_{2}^{-1}$ and $\delta= \sigma_{3}^{-1}\sigma_{4}^{-1} \sigma_{5}^{-1}\sigma_{2}^{-1}\sigma_{1}^{-1} \sigma_{2}^{-1}\sigma_{5}\sigma_{4}\sigma_{5}\sigma_{5}\sigma_{4}\sigma_{3}$}
\end{equation*}
(note that there is a typing error in the original version, the expression for $\delta$ there is missing the terms $\sigma_{5}\sigma_{4}\sigma_{5}$). The action of conjugation by $\gamma$ permutes cyclically the elements $\gamma^{i}\delta\gamma^{-i}$, $i=0,1,2$, and gives rise to the semi-direct product structure $\quat\rtimes \Z_{3}$ of $\tonestar$. In order to obtain a subgroup of $B_{6}(\St)$ isomorphic to $\quat\rtimes_{\alpha} \Z$, the proof of \repr{tstarxz}(\ref{it:tstarxza}) shows that it suffices to exhibit an element $z\in B_{6}(\St)$ of infinite order that commutes with $\gamma$ and $\delta$, but up until now, we have not been able to find such a $z$.
\end{rem}

\section{Realisation of the elements of $\mathbb{V}_{2}(\lowercase{n})$ in $B_{\lowercase{n}}(\St)$}\label{sec:realtypeII}

We now turn our attention to the problem of the realisation in $B_{n}(\St)$ of the virtually cyclic groups of Type~II described in \redef{v1v2}(\ref{it:mainIIdef}). In \resec{realV2}, we consider those groups that contain a cyclic or dicyclic factor. In \resec{oonestaramalg}, we discuss the realisation of $\oonestar \bigast_{\tonestar} \oonestar$ in $B_{n}(\St)$.

\subsection{Realisation of the elements of $\mathbb{V}_{2}(n)$ with cyclic or dicyclic factors}\label{sec:realV2}


\begin{thm}\label{th:realV2}
For all $n\geq 4$, the following Type~II virtually cyclic groups are realised as subgroups of $B_{n}(\St)$:
%
\begin{enumerate}[(a)]
\item\label{it:realV2a} $\Z_{4q} \bigast_{\Z_{2q}} \Z_{4q}$, where $i\in\brak{0,1,2}$ and $q$ divides $(n-i)/2$.
\item\label{it:realV2b} $\Z_{4q} \bigast_{\Z_{2q}} \dic{4q}$, where $i\in\brak{0,2}$, $q\geq 2$ and $q$ divides $(n-i)/2$.
\item\label{it:realV2c} $\dic{4q} \bigast_{\Z_{2q}} \dic{4q}$, where $i\in\brak{0,2}$, $q\geq 2$ and $q$ divides $n-i$ strictly.
\item\label{it:realV2d} $\dic{4q} \bigast_{\dic{2q}} \dic{4q}$, where $i\in\brak{0,2}$, and $q\geq 4$ is an even divisor of $n-i$.
\end{enumerate}
\end{thm}

\begin{proof}  
Let $n\geq 4$. First recall that if $1\leq j\leq n+1$, the kernel of the homomorphism $P_{n+1}(\St)\to P_{n}(\St)$ defined geometrically by deleting the $j\up{th}$ string may be identified with the fundamental group 
\begin{equation*}
\pi_{1}\left(\St\setminus\brak{x_{1},\ldots,x_{j-1},x_{j+1},\ldots, x_{n+1}},x_{j}\right),
\end{equation*}
which is a free group of rank $n-1$ for which a presentation is given by
\begin{equation}\label{eq:prespi1}
\setangr{A_{1,i},\ldots, A_{i-1,i},A_{i,i+1},\ldots, A_{i,n}}{A_{1,i} \cdots A_{i-1,i}A_{i,i+1}\cdots A_{i,n}=1},
\end{equation}
and for which a basis is obtained by selecting any $n-1$ distinct elements of the set $\brak{A_{1,j},\ldots, A_{j-1,j}, A_{j,j+1},\ldots,A_{j,n+1}}$, where for $1\leq i<j\leq n+1$, 
\begin{equation}\label{eq:defaij}
A_{i,j}=\sigma_{j-1}\cdots \sigma_{i+1}\sigma_{i}^{2}\sigma_{i+1}^{-1}\cdots \sigma_{j-1}^{-1}=\sigma_{i}^{-1}\cdots \sigma_{j-2}^{-1}\sigma_{j-1}^{2}\sigma_{j-2}\cdots \sigma_{i}.
\end{equation}
We consider the four cases of the statement of the theorem in turn. 
\begin{enumerate}[(a)]
\item We first treat the case $q=1$, and then go on to deal with the general case $q\geq 2$.


\begin{enumerate}
\item[\underline{$1\up{st}$ case: $q=1$}.] We shall construct a subgroup of $B_n(\St)$ isomorphic to $\Z_{4}\bigast_{\Z_2} \Z_{4}$.  Set $i=2$ if $n$ is even, and $i=1$ if $n$ is odd. Then $(n-i)$ is even, and the condition given in the statement is satisfied. Let $v_{1}=\alpha_{i}^{(n-i)/2}$, $v_{2}=\sigma_{n-i}v_{1}\sigma_{n-i}^{-1}$, and for $j=1,2$, let $G_{j}=\ang{v_{j}}$. Then $\ord{G_{j}}=4$ by \req{uniqueorder2}, and $G_{1}\cap G_{2}\supset \ang{\ft}$ since $\ft$ is the unique element of $B_{n}(\St)$ of order $2$. Let $H=\ang{G_{1}\cup G_{2}}$. By \repr{infincard}, to prove that $H\cong \Z_4\bigast_{\Z_2} \Z_4$, it suffices to show that $H$ is of infinite order, or indeed that $H$ contains an element of infinite order. Consider the element $v_{1}v_{2}$ of $H$. A straightforward calculation shows that:
\begin{align*}
\pi(v_{1})&= \textstyle\left(1,\frac{n-i}{2}+1\right) \left(2, \frac{n-i}{2}+2 \right) \cdots \left(\frac{n-i}{2}-1,n-i-1 \right)\left(\frac{n-i}{2},n-i \right)\\
\pi(v_{2})&= \textstyle\left(1,\frac{n-i}{2}+1\right) \left(2, \frac{n-i}{2}+2 \right) \cdots \left(\frac{n-i}{2}-1,n-i-1 \right)\left(\frac{n-i}{2},n-i+1 \right)\\
\pi(v_{1}v_{2})&= \textstyle\left(\frac{n-i}{2},n-i,n-i+1\right),
\end{align*}
and thus $\pi(v_{1}v_{2})$ consists of one $3$-cycle and $n-3$ fixed points. So if $n\geq 6$, by \reth{murasugi}, $v_{1}v_{2}$ is of infinite order, and this implies that $H$ is infinite as required. It remains to treat the cases $n=4,5$. Suppose first that $n=4$, and assume that $H$ is finite. Then $H$ is contained in a maximal finite subgroup $K$ of $B_{4}(\St)$, where $K$ is isomorphic to $\quat[16]$ or $\tonestar$ by \reth{finitebn}. Since $\pi(v_{1}v_{2})$ is a $3$--cycle and the set of torsion elements of $P_{n}(\St)$ is $\ang{\ft}$, $v_{1}v_{2}$ is of order $3$ or $6$, and so $K\cong \tonestar$. On the other hand, the elements of order $4$ of $\tonestar\cong \quat\rtimes \Z_{3}$ all belong to its subgroup isomorphic to $\quat$, and so the product $v_{1}v_{2}$ of elements of order $4$ is of order $1,2$ or $4$. This yields a contradiction, so $H$ is infinite in this case. Now suppose that $n=5$. Using equations~\reqref{uniqueorder2} and \reqref{fundaa}, as well as the fact that $\alpha_{1}=\alpha_{0}\sigma_{4}$, we obtain:
\begin{align*}
v_{1}v_{2}&= \alpha_1^{2} \sigma_{4} \alpha_1^2 \sigma_{4}^{-1}= \alpha_1^4 \alpha_1^{-2} \sigma_{4} \alpha_1^2 \sigma_{4}^{-1}= \ft[5] \alpha_1^{-1} \sigma_{4}^{-1} \alpha_0^{-1} \sigma_{4} \alpha_0 \sigma_4 \alpha_1 \sigma_{4}^{-1}\\
&=\ft[5] \sigma_{4}^{-1} \alpha_0^{-1} \sigma_{4}^{-1} \sigma_{3} \sigma_4 \alpha_0 = \ft[5] \sigma_{4}^{-1}  \sigma_{3}^{-1} \sigma_{2} \sigma_3= \ft[5] \sigma_{2}(\sigma_{4}^{-1}  \sigma_3) \sigma_{2}^{-1}.
\end{align*}
So to show that $v_{1}v_{2}$ is of infinite order, it suffices to prove that $\sigma_{4}^{-1}  \sigma_3$ is of infinite order. We have 
\begin{equation}\label{eq:sig4sig3}
(\sigma_{4}^{-1}  \sigma_3)^3 = \sigma_{4}^{-2} \sigma_{4} \sigma_3 \sigma_{4}^{-1}  \sigma_3 \sigma_{4}^{-1}  \sigma_3 
=\sigma_{4}^{-2} \sigma_{3}^{-1} \sigma_{4} \sigma_3^2\sigma_{4}^{-1}  \sigma_3=\sigma_{4}^{-2} \sigma_{3}^{-2}\sigma_{4}^2 \sigma_3^2.
\end{equation}
So $(\sigma_{4}^{-1}  \sigma_3)^3$ belongs to the free group $\pi_{1}(\St\setminus \brak{x_{1},x_{2},x_{3},x_{5}},x_{4})$, and in terms of the basis $\brak{A_{2,4},A_{3,4},A_{4,5}}$ of the latter, may be written as the commutator $[A_{4,5}^{-1},A_{3,4}^{-1}]$. It follows that $(\sigma_{4}^{-1}  \sigma_3)^3$ and $H$ are of infinite order, and thus $H\cong \Z_{4}\bigast_{\Z_{2}} \Z_{4}$ by \repr{infincard}.

\item[\underline{$2\up{nd}$ case: $q\geq 2$}.] 
We claim that it suffices to find distinct cyclic subgroups $G_1$, $G_2$ of $B_{n}(\St)$ of order $4q$ for which $G_1\cap G_2$ contains a (cyclic) subgroup of order $2q$. To prove the claim, let $G_{1}$ and $G_{2}$ be subgroups of $B_{n}(\St)$ satisfying these conditions, and suppose that $H=\ang{G_1\cup G_2}$ is finite. Let $K$ be a maximal finite subgroup of $B_{n}(\St)$ containing $H$. Since $G_{1}\neq G_{2}$, $K$ contains two distinct copies of $\Z_{4q}$, and so cannot be cyclic or dicyclic, nor by \repr{maxsubgp} can it be isomorphic to $\tonestar$ or $\istar$, since $q\geq 2$. So suppose that $K \cong \oonestar$. Then $q=2$ by \repr{maxsubgp}, $G_{j}\cong \Z_{8}$, where $j=1,2$, and $G_{1}\cap G_{2}\cong \Z_{4}$ by hypothesis. Under the restriction of the homomorphism $\phi$ of \req{mcg}, $\oonestar$ is sent to $\sn[4]$, the $\phi(G_{j})$ are sent to subgroups of $\sn[4]$ generated by $4$-cycles, and $\ord{\phi(G_{1}\cap G_{2})}=2$. But in $\sn[4]$, the intersection of two subgroups generated by $4$-cycles cannot be of order $2$, so $K\ncong \oonestar$. We conclude that $H$ is infinite, hence $H\cong \Z_{4q} \bigast_{\Z_{2q}} \Z_{4q}$ by \repr{infincard}, which proves the claim. 

We now exhibit subgroups $G_1$ and $G_2$ of $B_n(\St)$ satisfying the  properties of the claim. By hypothesis, $m=(n-i)/2q\in \N$. Let $G_1=\ang{\alpha_i^m}$ and $G_2=\xi G_1 \xi^{-1}$, where
\begin{align*}
\xi&=\delta_{2m,i}=\sigma_1 \sigma_{2m+1} \cdots \sigma_{m(2q-4)+1}\sigma_{m(2q-2)+1}\\
&= \sigma_1 \sigma_{2m+1} \cdots \sigma_{n-i-4m+1}\sigma_{n-i-2m+1},
\end{align*}
using the notation of \req{delta1}. Then $G_{j}\cong \Z_{4q}$ for $j=1,2$ by \req{uniqueorder2}. Now $\pi(\alpha_i^m)$ contains the $2q$-cycle $(1,m(2q-1)+1, m(2q-2)+1, \ldots, m+1)$, and so $\pi(\alpha_i^{mk})(1)\in \brak{1,m+1,\ldots, m(2q-1)+1}$ for all $k\in \N$.
%
%
On the other hand,
\begin{align*}
\pi(\xi \alpha_i^m \xi^{-1})(1) &=\pi(\alpha_i^m \xi^{-1})(2)\\
&=\pi(\xi^{-1})(m(2q-1)+2)=m(2q-1)+2
\end{align*}
(recall that as for braids, we compose permutations from left to right). Thus $\xi \alpha_i^m \xi^{-1}\notin G_1$, so $G_2\neq G_1$. Taking the integer $m$ of \relem{commalphaigen} to be $2m$, we have that $2m$ divides $n-i$ and so $r=2m\geq 2$. By part~(\ref{it:rgeq2}) of that proposition, $\xi$ commutes with $\alpha_{i}^{2m}$.
Thus $G_1\cap G_2=\ang{\alpha_i^{2m}}\cong \Z_{2q}$, and so $G_{1}$ and $G_{2}$ satisfy the hypotheses of the claim.
\end{enumerate}

\item Suppose that $q\geq 2$ divides $(n-i)/2$ for some $i\in \brak{0,2}$, so $n$ is even. Set $m=(n-i)/2q$, and let 
\begin{equation}\label{eq:defxi}
\xi_{i}=\sigma_{1+\frac{i}{2}} \sigma_{1+2m+\frac{i}{2}}\cdots \sigma_{1+n-2m-\frac{i}{2}},
\end{equation}
Equation~\reqref{fundaa} implies that $\xi_{i}=\alpha_{0}^{i/2}\delta_{2m,i}\alpha_{0}^{-i/2}$, where $\delta_{2m,i}$ is as in \req{delta1}.
Taking the integer $m$ of \relem{commalphaigen} to be $2m$, it follows from part~(\ref{it:rgeq2}) of that lemma that $\delta_{2m,i}$ commutes with $\alpha_{i}^{2m}$, and thus $\xi_{i}$ commutes with $\alpha_{i}'^{2m}$, where $\alpha_{i}'$ is given by \req{basicconj}. We analyse separately the two cases $m=1$ and $m\geq 2$.


\begin{enumerate}
\item [\underline{$1\up{st}$ case: $m=1$}.] Then $2q=n-i$. Take $G_1=\ang{\xi_{i}\alpha_i'\xi_{i}^{-1}}$ and $G_2=\ang{\alpha_i'^{2},\alpha_i'\garside}$, where $\xi_{i}$ is as defined above.
Then $G_{1}\cong \Z_{2(n-i)}= \Z_{4q}$, and $G_{2}$ is one of the two dicyclic subgroups of order $2(n-i)$ of the standard copy of $\dic{4(n-i)}$, so $G_{2}\cong \dic{4q}$ and $G_{1}\neq G_{2}$. Since $\xi_{i}$ commutes with $\alpha_{i}'^2$, it follows that $G_{1}\cap G_{2}=\ang{\alpha_{i}'^2}\cong \Z_{2q}$. Set $H=\ang{G_{1}\cup G_{2}}$. By \repr{infincard}, to see that $H\cong \Z_{4q}\bigast_{\Z_{2q}} \dic{4q}$, it suffices to show that $H$ contains an element of infinite order. Consider $\eta=\xi_{i}\alpha_i'\xi_{i}^{-1} \ldotp \alpha_i'\garside\in H$. Since $\xi_{i}$ commutes with $\garside$ by \req{garsideconj} as well as with $\alpha_{i}'^2$, and $n-i$ is even, we have
\begin{align*}
\eta^{2(n-i)}&= (\xi_{i}\alpha_i'\xi_{i}^{-1} \alpha_i'\garside \ldotp \xi_{i}\alpha_i'\xi_{i}^{-1} \alpha_i'\garside^{-1}\ldotp \ft)^{n-i}= (\xi_{i}\alpha_i'\xi_{i}^{-1} \alpha_i' \xi_{i}\alpha_i'^{-1}\xi_{i}^{-1} \alpha_i'^{-1}\ft)^{n-i}\\
&=(\xi_{i}\alpha_i'\xi_{i}^{-1} \alpha_i'\ldotp \alpha_i'^{-2})^{2(n-i)}=(\xi_{i}\alpha_i'\xi_{i}^{-1} \alpha_i')^{2(n-i)} \alpha_i'^{-4(n-i)}=\widetilde{\eta}^{2(n-i)},
\end{align*}
using also \req{uniqueorder2}, where
\begin{equation}\label{eq:defetatilde}
\widetilde{\eta}=\xi_{i}\alpha_i'\xi_{i}^{-1} \alpha_i'.
\end{equation}
So to prove that $\eta$ is of infinite order, it suffices to show that $\widetilde{\eta}$ is of infinite order. Since $\pi(\xi_{i})$ is of order $2$ and
\begin{align*}
\pi(\xi_{i})&= \textstyle\left( 1+\frac{i}{2},2+\frac{i}{2} \right) \left( 3+\frac{i}{2},4+\frac{i}{2} \right) \cdots \left( n-1-\frac{i}{2},n-\frac{i}{2} \right)\\
\pi(\alpha_{i}')&= \textstyle \left( n-\frac{i}{2}, n-1-\frac{i}{2}, \ldots, 2+\frac{i}{2}, 1+\frac{i}{2}\right),
\end{align*}
we have
\begin{equation*}
\textstyle\pi(\widetilde{\eta})=(\pi(\xi_{i}\alpha_i'))^2=\left(n-\frac{i}{2},n-\frac{i}{2}-2, \ldots, 4+\frac{i}{2}, 2+\frac{i}{2}\right)^2,
\end{equation*}
and the cycle decomposition of $\pi(\widetilde{\eta})$ consists of two $(n-i)/4$-cycles (resp.\ one $(n-i)/2$-cycle) if $n-i$ is divisible (resp.\ is not divisible) by $4$, plus $(n+i)/2$ fixed points. If either $i=0$ and $n\geq 6$ or if $i=2$ and $n\geq 8$ then the cycle decomposition of $\pi(\widetilde{\eta})$ contains a cycle of length at least two, plus at least three fixed points, and so $\widetilde{\eta}$ is of infinite order by \reth{murasugi}. Let us deal with the three remaining cases, which are given by $n=4$ and $i\in\brak{0,2}$, and $n=6$ and $i=2$.

\begin{enumerate}[(i)]
\item $i=0$ and $n=4$. Using the presentation of \req{prespi1}, we have
\begin{equation*}
\widetilde{\eta}= \sigma_{1}\sigma_{3}\sigma_{1}\sigma_{2}\sigma_{3}\sigma_{3}^{-1} \sigma_{1}^{-1}\sigma_{1}\sigma_{2}\sigma_{3}=\sigma_{1}^2\sigma_{3} \sigma_{2}^2\sigma_{3}=A_{1,2} A_{2,4} A_{3,4}= A_{1,2} A_{1,4}^{-1},
\end{equation*}
which may be interpreted as an element of $\pi_{1}(\St\setminus \brak{x_{2},x_{3},x_{4}},x_{1})$ for which a basis is $\brak{A_{1,2}, A_{1,4}}$. 
\item $i=2$ and $n=4$. In this case, $\widetilde{\eta}=\sigma_{2}^2 \sigma_{3}^{4}=A_{2,3}A_{3,4}^2$ belongs to the free group $\pi_{1}(\St \setminus \brak{x_{1},x_{2},x_{4}}, x_{3})$ for which a basis is $\brak{A_{2,3},A_{3,4}}$. 
\item $i=2$ and $n=6$. Then by \req{surface},
\begin{align}
\widetilde{\eta}&= \sigma_{2}\sigma_{4}\ldotp \sigma_{2}\sigma_{3}\sigma_{4}\sigma_{5}^2 \ldotp \sigma_{4}^{-1}\sigma_{2}^{-1}\ldotp \sigma_{2}\sigma_{3}\sigma_{4}\sigma_{5}^2
= \sigma_{2}\sigma_{4}\ldotp \sigma_{1}^{-2}\sigma_{2}^{-1}\sigma_{3}^{-1}\sigma_{4}^{-1}
\ldotp \sigma_{4}^{-1}\sigma_{3}\sigma_{4}\sigma_{5}^2\notag\\
&= \sigma_{2}\sigma_{1}^{-2}\sigma_{2}^{-1} \ldotp\sigma_{4} \sigma_{3}^{-1}\sigma_{4}^{-2}
\sigma_{3}\sigma_{4}\sigma_{5}^2
= \sigma_{2}\sigma_{1}^{-2}\sigma_{2}^{-1} \ldotp\sigma_{4}^2
\sigma_{3}^{-2}\sigma_{5}^2= A_{1,3}^{-1}A_{4,5}A_{3,4}^{-1} A_{5,6}.\label{eq:etainfinite}
\end{align}
Projecting $\widetilde{\eta}$ onto $P_{4}(\St)$ by forgetting the $4\up{th}$ and $5\up{th}$ strings yields $A_{1,3}^{-1}$, which is of infinite order. 
\end{enumerate}
In all three cases, we conclude that $\widetilde{\eta}$ is of infinite order, and this completes the proof of the case $m=1$.

\item [\underline{$2\up{nd}$ case: $m\geq 2$}.] Let $G_1=\ang{\xi_{i}\alpha_i'^m\xi_{i}^{-1}}$ and $G_2=\ang{\alpha_i'^{2m},\garside}$, where $\xi_{i}$ is as defined in \req{defxi}.
Then $G_{1}\cong \Z_{4q}$ and $G_{2}\cong \dic{4q}$, so $G_{1}\neq G_{2}$. Since $\xi_{i}$ commutes with $\alpha_{i}'^{2m}$, we have $G_{1}\cap G_{2}=\ang{\alpha_{i}'^{2m}}\cong \Z_{2q}$. By \repr{infincard}, to prove that the group $H=\ang{G_{1}\cup G_{2}}$ is isomorphic to $\Z_{4q}\bigast_{\Z_{2q}} \dic{4q}$, it suffices to show that it contains an element of infinite order. Consider the element $\eta=\xi_{i}\alpha_i'^m\xi_{i}^{-1} \ldotp \garside$ of $H$. Using \req{basicconj}, we have that
\begin{equation*}
\eta^2=\xi_{i}\alpha_i'^m\xi_{i}^{-1} \garside \ldotp \xi_{i}\alpha_i'^m\xi_{i}^{-1} \garside^{-1}\ft= \xi_{i}\alpha_i'^m\xi_{i}^{-1} \xi_{i}'\alpha_i'^{-m}\xi_{i}'^{-1} \ft,
\end{equation*}
where 
\begin{equation}
\xi_{i}' =\garside \xi_{i} \garside^{-1}= \sigma_{2m-1+\frac{i}{2}} \sigma_{4m-1+\frac{i}{2}}\cdots \sigma_{n-2m-1-\frac{i}{2}} \sigma_{n-1-\frac{i}{2}}. \label{eq:defxiprime}
\end{equation}
All of the generators appearing in equations~\reqref{defxi} and~\reqref{defxiprime} commute pairwise, so $\xi_{i}$ commutes with $\xi_{i}'$. Since $\ft$ is central and of order $2$, $\eta$ is of infinite order if and only if $\eta^2 \fti$ is. We now distinguish three subcases.

\begin{enumerate}
\item [\underline{$1\up{st}$ subcase: $m=2$}.] In this case, $4q=n-i$, 
\begin{equation*}
\xi_{i} =\sigma_{1+\frac{i}{2}} \sigma_{5+\frac{i}{2}}\cdots \sigma_{n-3-\frac{i}{2}} \quad \text{and} \quad \xi_{i}' = \sigma_{3+\frac{i}{2}} \sigma_{7+\frac{i}{2}}\cdots \sigma_{n-5-\frac{i}{2}} \sigma_{n-1-\frac{i}{2}},
\end{equation*}
and hence $\alpha_{i}'^2 \xi_{i} \alpha_{i}'^{-2}=\xi_{i}'$ by \req{fundaa}. Since $\xi_{i}$ commutes with $\alpha_{i}'^4$, this implies that $\alpha_{i}'^2 \xi'_{i} \alpha_{i}'^{-2}=\xi_{i}$, and thus 
\begin{equation*}
\eta^2 \fti=\xi_{i}\alpha_i'^2\xi_{i}^{-1} \xi_{i}'\alpha_i'^{-2}\xi_{i}'^{-1}=\xi_{i} \xi_{i}'^{-1} \xi_{i} \xi_{i}'^{-1}=\xi_{i}^2 \xi_{i}'^{-2}.
\end{equation*}
Now $n\geq 8$ since $q\geq 2$, and projecting $\eta^2 \fti$ onto $B_{4}(\St)$ by forgetting all but the $\left( 1+\frac{i}{2}\right)^{\text{th}}$, $\left( 2+\frac{i}{2}\right)^{\text{th}}$, $\left( 5+\frac{i}{2}\right)^{\text{th}}$  and $\left( 6+\frac{i}{2}\right)^{\text{th}}$ strings yields the braid $\sigma_{1}^2\sigma_{3}^2$ of $P_{4}(\St)$, which by \req{delta1} is the element $\delta_{2,0}^{2}$. But this element is of infinite order by \relem{commalphaigen}(\ref{it:rgeq2}), and we conclude that $\eta$ is also of infinite order.

\item [\underline{$2\up{nd}$ subcase: $m=3$}.] Since $q\geq 2$, we have $n\geq 12+i$ and $\xi_{i} =\sigma_{1+\frac{i}{2}} \sigma_{7+\frac{i}{2}}\cdots \sigma_{n-5-\frac{i}{2}}$. So
\begin{align*}
\pi(\eta)=& \textstyle\left( 1+\frac{i}{2}, 2+\frac{i}{2}\right)\left( 7+\frac{i}{2}, 8+\frac{i}{2}\right)\cdots \left( n-5-\frac{i}{2}, n-4-\frac{i}{2}\right) \ldotp\\
& \textstyle\left( 1+\frac{i}{2}, n-2-\frac{i}{2},n-5-\frac{i}{2},\ldots, 4+\frac{i}{2}\right)\ldotp\\
& \textstyle\left( 2+\frac{i}{2}, n-1-\frac{i}{2},n-4-\frac{i}{2},\ldots, 5+\frac{i}{2}\right)\ldotp\\
& \textstyle\left( 3+\frac{i}{2}, n-\frac{i}{2},n-3-\frac{i}{2},\ldots, 6+\frac{i}{2}\right) \ldotp\left( 1+\frac{i}{2}, 2+\frac{i}{2}\right)\left( 7+\frac{i}{2}, 8+\frac{i}{2}\right)\cdots\\
& \textstyle \left( n-5-\frac{i}{2}, n-4-\frac{i}{2}\right) \ldotp \left(\vphantom{\left( \frac{i}{2}\right)}1,n\right)\left(\vphantom{\left( \frac{i}{2}\right)} 2,n-1\right) \cdots \left(\vphantom{\left( \frac{i}{2}\right)}\frac{n}{2}, \frac{n}{2}+1\right).
\end{align*}
If $i=0$ (resp.\ $i=2$),
the cycle decomposition of $\pi(\eta)$ contains the two cycles $(1,2,3)$ and $(4,n-1,6,n-2,5,n)$  
(resp.\ $(1,n)$ and $(2,3,4)$), and we deduce from \reth{murasugi} that $\eta$ is of infinite order.

\item [\underline{$3\up{rd}$ subcase: $m\geq 4$}.] Since $q\geq 2$, we have $n\geq 16$. Using equations~\reqref{fundaa} and~\reqref{fundab}, we have that
\begin{align*}
\alpha_i'^m \sigma_{n-1-\frac{i}{2}}\alpha_i'^{-m}&= \alpha_{0}^{i/2}\alpha_i^m\alpha_{0}^{-i/2} \sigma_{n-1-\frac{i}{2}}\alpha_{0}^{i/2}\alpha_i^{-m} \alpha_{0}^{-i/2}=\alpha_{0}^{i/2}\alpha_i^m \sigma_{n-1-i}\alpha_i^{-m} \alpha_{0}^{-i/2}\\
&=\alpha_{0}^{i/2}\alpha_i^{m-2} \sigma_{1}\alpha_i^{-(m-2)} \alpha_{0}^{-i/2}=\sigma_{m-1+\frac{i}{2}}, 
\end{align*}
from which one may see that
\begin{align}
\alpha_i'^m \xi_{i} \alpha_i'^{-m}&= \sigma_{m+1+\frac{i}{2}} \sigma_{3m+1+\frac{i}{2}}\cdots \sigma_{n-3m+1-\frac{i}{2}} \sigma_{n-m+1-\frac{i}{2}}, \quad\text{and}\label{eq:xiconj}\\
\alpha_i'^m \xi_{i}' \alpha_i'^{-m}&= \sigma_{m-1+\frac{i}{2}} \sigma_{3m-1+\frac{i}{2}}\cdots \sigma_{n-3m-1-\frac{i}{2}} \sigma_{n-m-1-\frac{i}{2}}. \label{eq:xiprimeconj}
\end{align}
The terms in each of the expressions~\reqref{defxi},~\reqref{defxiprime},~\reqref{xiconj} and~\reqref{xiprimeconj} commute pairwise, and since $m\geq 4$, $\xi_{i}$, $\xi_{i}'$, $\alpha_i'^m \xi_{i} \alpha_i'^{-m}$ and $\alpha_i'^m \xi_{i}' \alpha_i'^{-m}$ also commute pairwise. So:
\begin{align*}
\eta^2 \fti&=\xi_{i}\ldotp \alpha_i'^m\xi_{i}^{-1}\alpha_i'^{-m}\ldotp \alpha_i'^{m} \xi_{i}'\alpha_i'^{-m}\ldotp \xi_{i}'^{-1}\\
&= \sigma_{1+\frac{i}{2}} \sigma_{m-1+\frac{i}{2}} \sigma_{m+1+\frac{i}{2}}^{-1} \sigma_{2m-1+\frac{i}{2}}^{-1} \cdots 
\sigma_{n-2m+1-\frac{i}{2}} \sigma_{n-m-1-\frac{i}{2}} \sigma_{n-m+1-\frac{i}{2}}^{-1} \sigma_{n-1-\frac{i}{2}}^{-1},
\end{align*}
and all of the terms in this expression commute pairwise. Projecting $\eta^2 \fti$ onto $B_{4}(\St)$ by forgetting all but the strings numbered $1+\frac{i}{2}$, $2+\frac{i}{2}$, $m-1+\frac{i}{2}$ and $m+\frac{i}{2}$ yields the braid $\sigma_{1}\sigma_{3}$, which we know to be of infinite order from the case $m=2$. So $\eta$ and $H$ are also of infinite order. This completes the proof of the realisation of $\Z_{4q}\bigast_{\Z_{2q}} \dic{4q}$ as a subgroup of $B_{n}(\St)$ for all $q\geqslant 2$ dividing $(n-i)/2$.
\end{enumerate}
\end{enumerate}

\item Let $q\geq 2$ be a strict divisor of $n-i$, where $i\in \brak{0,2}$, and let $m=(n-i)/q$. Then $m\geq 2$. We distinguish the cases $m=2$ and $m\geq 3$. 

\begin{enumerate}

\item[\underline{1\up{st} case: $m=2$.}] Then $2q=n-i$, and $n$ is even. Let $G_{1}=\ang{\alpha_{i}'^2,\alpha_{i}' \garside}$ and $G_{2}=\xi_{i} G_{1} \xi_{i}^{-1}$, where $\xi_{i}=\alpha_{0}^{i/2}\delta_{2,i} \alpha_{0}^{-i/2}$. Then $G_{1},G_{2}\cong\dic{2(n-i)}=\dic{4q}$. As we saw in the case $m=1$ of part~(\ref{it:realV2b}) above, $\xi_{i}$
commutes with $\alpha_{i}'^2$, and so $G_{1}\cap G_{2} \supset F$, where $F=\ang{\alpha_{i}'^2}\cong \Z_{2q}$. Let $H=\ang{G_{1}\cup G_{2}}$. To complete the construction, it suffices once more by \repr{infincard} to show that $H$ contains an element of infinite order. Consider the following element of $H$:
\begin{equation*}
\xi_{i} \alpha_{i}' \garside \xi_{i}^{-1}\ldotp \alpha_{i}' \garside\ldotp \alpha_{i}'^2=\xi_{i} \alpha_{i}'  \xi_{i}^{-1} \alpha_{i}'\ft=\widetilde{\eta}\ft,
\end{equation*}
using \req{basicconj} and the fact that $\garside$ commutes with $\xi_{i}$, and where $\widetilde{\eta}$ is as defined in \req{defetatilde}. But we saw there that $\widetilde{\eta}$ is of infinite order, so $\widetilde{\eta}\ft$ is too, and thus $H$ is infinite.

\item [\underline{2\up{nd} case: $m\geqslant 3$.}] Then $n\geq 6+i$. Set $G_1=\ang{\alpha_{i}'^m,\garside}$ and $G_2=\xi_{i} G_1\xi_{i}^{-1}$, where $\xi_{i}=\alpha_{0}^{i/2}\delta_{m,i} \alpha_{0}^{-i/2}$, and since $\delta_{m,i}$ commutes with $\alpha_{i}^{m}$ by \relem{commalphaigen}(\ref{it:rgeq2}), $\xi_{i}$ commutes with $\alpha_{i}'^m$. Thus $G_2=\ang{\alpha_{i}'^m, \xi_{i} \garside \xi^{-1}}$, and $G_1\cap G_2 \supset \ang{\alpha_{i}'^m}\cong \Z_{2q}$. To complete the construction, it suffices to show that $H=\ang{G_{1} \cup G_{2}}$ contains an element of infinite order. Consider the element $[\xi_{i},\garside]=\xi_{i} \garside \xi_{i}^{-1}\ldotp \garside^{-1} \in H$. Then:
\begin{align*}
[\xi_{i},\garside] =&\xi_{i} \ldotp\garside \xi_{i}^{-1} \garside^{-1}= \sigma_{1+\frac{i}{2}}\sigma_{m+1+\frac{i}{2}} \cdots \sigma_{n-m+1-\frac{i}{2}}\ldotp\\
& \sigma_{m-1+\frac{i}{2}}^{-1}\sigma_{2m-1+\frac{i}{2}}^{-1}\cdots \sigma_{n-m-1-\frac{i}{2}}^{-1} \sigma_{n-1-\frac{i}{2}}^{-1}\\
=&\textstyle\left(\sigma_{1+\frac{i}{2}}\sigma_{m-1+\frac{i}{2}}^{-1}\right)
\left(\sigma_{m+1+\frac{i}{2}} \sigma_{2m-1+\frac{i}{2}}^{-1} \right)
\cdots \left(\sigma_{n-m+1-\frac{i}{2}}\sigma_{n-1-\frac{i}{2}}^{-1} \right), 
\end{align*}
where the bracketed terms commute pairwise. If $m=3$ then after having projected $[\xi_{i},\garside]^3$ into $P_4(\St)$ by forgetting all but the first four strings, we carry out a calculation similar to that of \req{sig4sig3}. If $m\geq 4$, we project $[\xi_{i},\garside]$ into $B_4(\St)$ by forgetting all but the strings numbered $1+\frac{i}{2}$, $2+\frac{i}{2}$, $m+1+\frac{i}{2}$ and $m+2+\frac{i}{2}$, which yields the braid $\sigma_{1}\sigma_{3}$ of infinite order. In both cases, we conclude that $[\xi_{i},\garside]$ is of infinite order.
Thus $H$ is of infinite order, and it follows from \repr{infincard} that $H\cong \dic{4q} \bigast_{\Z_{2q}} \dic{4q}$.
\end{enumerate}

\item\label{it:ggroupiiix} Let $q\geq 4$ be an even divisor of $n-i$, and set $m=(n-i)/q$, $G_1=\ang{\alpha_i'^m, \garside}$ and $G_2= \lambda_{i} G_1 \lambda_{i}^{-1}$, where 
\begin{equation}\label{eq:lambdai}
\lambda_{i}= \prod_{j=0}^{(q-2)/2} \sigma_{m(1+2j)+\frac{i}{2}}=\sigma_{m+\frac{i}{2}} \sigma_{3m+\frac{i}{2}} \cdots \sigma_{n-3m-\frac{i}{2}} \sigma_{n-m-\frac{i}{2}}.
\end{equation}
Then both $G_1$ and $G_2$ are isomorphic to $\dic{4q}$, and:
\begin{align}
\garside \lambda_i \garside^{-1}&=  \garside \sigma_{m+\frac{i}{2}} \sigma_{3m+\frac{i}{2}} \cdots \sigma_{n-3m-\frac{i}{2}} \sigma_{n-m-\frac{i}{2}} \garside^{-1}\notag\\
&= \sigma_{n-m-\frac{i}{2}} \sigma_{n-3m-\frac{i}{2}} \cdots \sigma_{3m+\frac{i}{2}} \sigma_{m+\frac{i}{2}}= \lambda_{i}\label{eq:zetagarside}
\end{align}
by \req{garsideconj}. Further, by equations~\reqref{fundaa} and~\reqref{fundab}, we have
\begin{align*}
\alpha_{i}'^{2m} \sigma_{n-m-\frac{i}{2}} \alpha_{2}'^{-2m}&=
\alpha_{0}^{i/2}\alpha_{i}^{2m}\alpha_{0}^{-i/2} \sigma_{n-m-\frac{i}{2}} \alpha_{0}^{i/2}\alpha_{i}^{-2m}\alpha_{0}^{-i/2}=\alpha_{0}^{i/2}\alpha_{i}^{2m}\sigma_{n-m-i} \alpha_{i}^{-2m}\alpha_{0}^{-i/2}\\
&=\alpha_{0}^{i/2}\alpha_{i}^{m+1}\sigma_{n-i-1} \alpha_{i}^{-(m+1)}\alpha_{0}^{-i/2}=\alpha_{0}^{i/2}\alpha_{i}^{m-1}\sigma_{1} \alpha_{i}^{-(m-1)}\alpha_{0}^{-i/2}= \sigma_{m+\frac{i}{2}},
\end{align*}
and from this and \req{lambdai} it follows that 
$\lambda_{i}$ also commutes with $\alpha_{i}'^{2m}$.
This fact and \req{zetagarside} imply that $G_1 \cap G_2 \supset \ang{\alpha_i'^{2m}, \garside}\cong \dic{2q}$. To complete the construction, it suffices to show that the subgroup $H=\ang{G_1\cup G_2}$ is infinite, or equivalently, that it contains an element of infinite order. We consider the two cases $m=1$ and $m\geq 2$ separately.


\begin{enumerate}
\item[\underline{1\up{st} case: $m=1$.}] Then $q=n-i\geq 4$, $n$ is even and $G_1\cong G_2\cong \dic{4(n-i)}$. If the element $\lambda_{i} \alpha_i' \lambda_{i}^{-1}$ of $G_{2}$ belonged also to $G_{1}$, since it is of order $2(n-i)\geqslant 8$, it would be an element of the subgroup $\ang{\alpha_i'}$ of $G_{1}$, and so $\lambda_{i}$ would belong to the normaliser of $\ang{\alpha_i'}$ in $B_n(\St)$. \repr{genhodgkin1} then implies that $\lambda_{i}$ is of finite order. However, $\lambda_{i}=\alpha_{0}^{i/2} \delta_{2,i} \alpha_{0}^{-i/2}$ is of infinite order by \relem{commalphaigen}(\ref{it:rgeq2}), which yields a contradiction, and so we conclude that $G_1\neq G_2$. If $\dic{4(n-i)}$ is maximal finite in $B_{n}(\St)$ then $H$ must then be infinite, which gives the result. So suppose that $\dic{4(n-i)}$ is not maximal. By \reth{finitebn}, we have $n=6$ and $i=2$, in which case $\lambda_{2}=\sigma_{2}\sigma_{4}$ and $\alpha_{2}'=\sigma_{2}\sigma_{3}\sigma_{4}\sigma_{5}^{2}$. Equation~\reqref{etainfinite} implies that the element $\lambda_{2} \alpha_{2}' \lambda_{2}^{-1}\ldotp \alpha_{2}'$ of $H$ is of infinite order as required.

\item[\underline{2\up{nd} case: $m\geqslant 2$.}] Consider the element $\rho_{i}=\alpha_i'^m \ldotp\lambda_{i} \alpha_i'^{-m} \lambda_{i}^{-1}$ of $H$. Then
\begin{align*}
\rho_{i} =& \alpha_i'^{(m-1)} \ldotp\alpha_i' 
\sigma_{m+\frac{i}{2}} \sigma_{3m+\frac{i}{2}} \cdots \sigma_{n-3m-\frac{i}{2}} \sigma_{n-m-\frac{i}{2}} \alpha_i'^{-1}\ldotp\\
&\alpha_i'^{-(m-1)}
\sigma_{n-m-\frac{i}{2}}^{-1} \sigma_{n-3m-\frac{i}{2}}^{-1} \cdots \sigma_{3m+\frac{i}{2}}^{-1} \sigma_{m+\frac{i}{2}}^{-1} \alpha_i'^{(m-1)} \ldotp \alpha_i'^{-(m-1)}\\
=& \alpha_i'^{(m-1)} 
\sigma_{m+1+\frac{i}{2}} \sigma_{3m+1+\frac{i}{2}} \cdots \sigma_{n-3m+1-\frac{i}{2}} \sigma_{n-m+1-\frac{i}{2}} \ldotp\\
& \sigma_{n-2m+1-\frac{i}{2}}^{-1} \sigma_{n-4m+1-\frac{i}{2}}^{-1} \cdots \sigma_{2m+1+\frac{i}{2}}^{-1} \sigma_{1+\frac{i}{2}}^{-1}    \alpha_i'^{-(m-1)}\\
=& \alpha_i'^{(m-1)} 
\sigma_{1+\frac{i}{2}}^{-1} \sigma_{m+1+\frac{i}{2}}  \sigma_{2m+1+\frac{i}{2}}^{-1} \sigma_{3m+1+\frac{i}{2}} \cdots \sigma_{n-2m+1-\frac{i}{2}}^{-1}\sigma_{n-m+1-\frac{i}{2}} \alpha_i'^{-(m-1)},
\end{align*}
using equations~\reqref{fundaa} and~\reqref{fundab}, the fact that $m\geq 2$, as well as the relations:
\begin{align*}
\alpha_i' \sigma_{n-m-\frac{i}{2}}\alpha_i'^{-1} &= \alpha_{0}^{i/2}\alpha_{i} \sigma_{n-m-i} \alpha_{i}^{-1} \alpha_{0}^{-i/2}\\
&= \alpha_{0}^{i/2}\sigma_{n-m-i+1} \alpha_{0}^{-i/2} \;\text{by \req{fundaa} as $n-m-i+1\leq n-i-1$}\\
&= \sigma_{n-m+1-\frac{i}{2}} \;\text{by \req{fundaa} since $n-m+1-\frac{i}{2}\leq n-1$,}
\end{align*}
and
\begin{align*}
\alpha_i'^{(m-1)} \sigma_{n-2m+1-\frac{i}{2}}\alpha_i'^{-(m-1)} & \!=\! \alpha_{0}^{i/2}\alpha_{i}^{m-1} \sigma_{n-2m+1-i} \alpha_{i}^{-(m-1)} \alpha_{0}^{-i/2}\\
&\!=\! \alpha_{0}^{i/2}\sigma_{n-m-i} \alpha_{0}^{-i/2}\;\text{by \req{fundaa} as $n-m-i\leq n-i-1$}\\
&\!=\! \sigma_{n-m-\frac{i}{2}}.
\end{align*}
So $\rho_{i}$ is conjugate to
\begin{equation*}
\sigma_{1+\frac{i}{2}}^{-1} \sigma_{m+1+\frac{i}{2}}  \sigma_{2m+1+\frac{i}{2}}^{-1} \sigma_{3m+1+\frac{i}{2}} \cdots \sigma_{n-2m+1-\frac{i}{2}}^{-1}\sigma_{n-m+1-\frac{i}{2}},
\end{equation*}
which under the projection onto $B_{4}(\St)$ that is obtained by forgetting all but the strings numbered $1+\frac{i}{2}$, $2+\frac{i}{2}$, $2m+1+\frac{i}{2}$ and $2m+2+\frac{i}{2}$ yields the element $\sigma_{1}^{-1} \sigma_{3}^{-1}$, which we know to be of infinite order in $B_{4}(\St)$,
so $H$ is also of infinite order. This completes the proof of the realisation of $\dic{4q}\bigast_{\dic{2q}}\dic{4q}$ as a subgroup of $B_{n}(\St)$, as well as that of \reth{realV2}.\qedhere
\end{enumerate}
\end{enumerate}
\end{proof}

\subsection{Realisation of $\oonestar\bigast_{\tonestar}\oonestar$ in $B_{n}(\St)$}\label{sec:oonestaramalg}

For the realisation of the Type~II subgroups of $B_{n}(\St)$ described in \redef{v1v2}(\ref{it:mainIIdef}), there is just one outstanding case not covered by \reth{realV2} to be dealt with, that of $\oonestar\bigast_{\tonestar}\oonestar$. We start by making some general comments. From \reth{finitebn}, there is no finite subgroup of $B_{n}(\St)$ that contains two copies of $\oonestar$. In particular, any subgroup of $B_{n}(\St)$ generated by two distinct copies $G_{1},G_{2}$ of $\oonestar$ is necessarily infinite. If further $G_{1}\cap G_{2} \cong \tonestar$ then it follows from \repr{infincard} that $\ang{G_{1}\cup G_{2}}\cong \oonestar\bigast_{\tonestar}\oonestar$, from which we also obtain a subgroup isomorphic to $\tonestar \rtimes \Z$ for one of the two actions of $\Z$ on $\tonestar$ of \redef{v1v2}(\ref{it:mainIdef})(\ref{it:maint}) and~(\ref{it:maing}). Notice also that in this case, \cite[Proposition~1.5]{GG7} implies that $G_{1}$ and $G_{2}$ are conjugate by an element that belongs to the normaliser of $G_{1}\cap G_{2}$ since $\oonestar$ contains a unique subgroup isomorphic to $\tonestar$. Conversely, if $\xi$ is an element of $ B_{n}(\St)$ that belongs to the normaliser of a subgroup $K$ of $B_{n}(\St)$ isomorphic to $\tonestar$, and if $n\nequiv 4\bmod{6}$ then $K$ is contained in a subgroup $G_{1}$ of $B_{n}(\St)$ isomorphic to $\oonestar$ by \reth{finitebn}. Either $\xi G_{1} \xi^{-1}=G_{1}$, in which case $\xi$ belongs to the normaliser of $G_{1}$, or else $G_{1}\neq \xi G_{1} \xi^{-1}$, in which case $\ang{G_{1}\cup \xi G_{1} \xi^{-1}}\cong \oonestar\bigast_{\tonestar}\oonestar$ in light of the above remarks.

We now prove the realisation of $\oonestar \bigast_{\tonestar} \oonestar$ in $B_{n}(\St)$ in the following cases.

\begin{prop}\label{prop:oto}
Let $n\equiv 0,2 \bmod 6$, and suppose that $n= 36$ or $n\geq 42$. Then $B_{n}(\St)$ possesses a subgroup that is isomorphic to $\oonestar \bigast_{\tonestar} \oonestar$.
\end{prop}

\begin{rem}\label{rem:notostar}
It follows from \reth{main}(\ref{it:mainII})(\ref{it:excepe}) and \repr{oto} that the condition given in \redef{v1v2}(\ref{it:mainII})(\ref{it:mainIIe}) for the existence of $\oonestar \bigast_{\tonestar} \oonestar$ as a subgroup of $B_{n}(\St)$ is necessary and sufficient, unless $n$ belongs to $\brak{6,8,12,14,18,20,24,26,30,32,38}$. For these values of $n$, which are those of \rerem{exceptions}(\ref{it:exceptionsostar}), it is an open question as to whether $\oonestar \bigast_{\tonestar} \oonestar$ is realised as a subgroup of $B_{n}(\St)$. 
\end{rem}

\begin{proof}[Proof of \repr{oto}.]
In order to obtain a subgroup of $B_{n}(\St)$ that is isomorphic to $\oonestar \bigast_{\tonestar} \oonestar$, we shall construct a copy of $\sn[4] \bigast_{\an[4]} \sn[4]$ in $\mcg$, and then take its inverse image by the homomorphism $\phi$ of \req{mcg}. Let $n\equiv 0,2 \bmod 6$, and set $n=12m+6\epsilon_{1}+8\epsilon_{2}$, where $m\in \N$ and $\epsilon_{1},\epsilon_{2}\in \brak{0,1}$. Since $n= 36$ or $n\geq 42$, we have that $m\geq 3$. We use the notation of the proof of \repr{ttimesz}, taking $\Delta$ to be a cube with $m$ (resp.\ $\epsilon_{1},\epsilon_{2}$) marked points lying on each edge (resp.\ at the centre of each face, at each vertex). As in that proof, we consider the group of rotations $\Gamma\cong \sn[4]$ of $\Delta$ to be a subgroup of $\homeo$, and we set $\widetilde{\Gamma}=\Psi(\Gamma)$, which is a subgroup of $\mcg$ isomorphic to $\sn[4]$. Choose an edge $e$ of $\Delta$, fix an orientation of $e$, and denote the marked points lying on $e$ by $p_{1},\ldots,p_{m}$; these points are numbered coherently with the orientation of $e$ (see Figure~\ref{fig:cube}). Let $h$ be the unique element of $\Gamma$ different from the identity and fixing $e$ setwise (so $h$ reverses the orientation of $e$). 
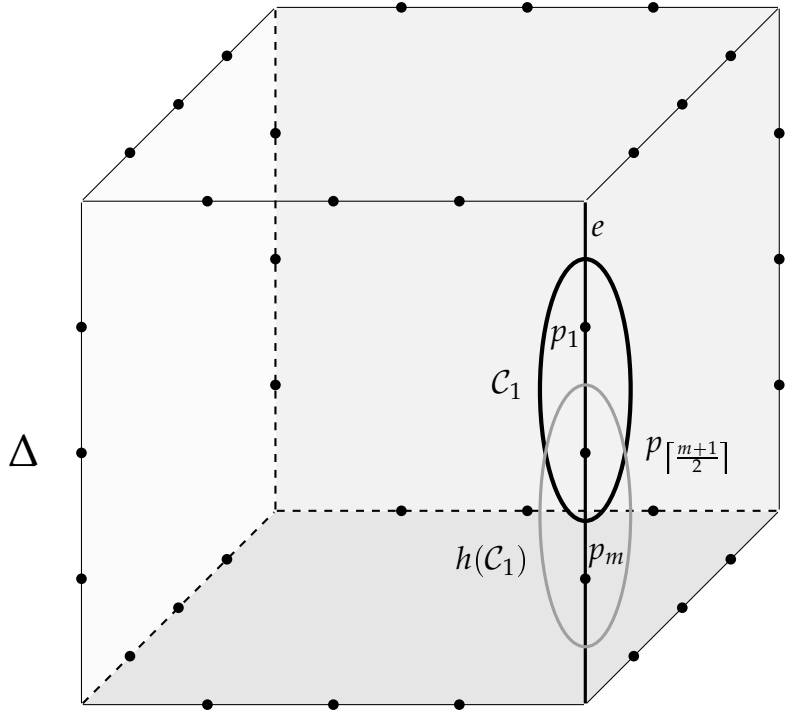
\begin{figure}[h]
\centering
\begin{tikzpicture}[scale=6.7]

\tikzset{inline/.style={outer sep=0mm,inner sep=0mm}}

\coordinate (A) at (0,0,0);
\coordinate (B) at (1,0,0);
\coordinate (C) at (1,1,0);
\coordinate (D) at (0,1,0);

\node[inline] (a) at (A){};
\node[inline] (b) at (B){};
\node[inline] (c) at (C){};
\node[inline] (d) at (D){};

\coordinate (E) at (0,0,1);
\coordinate (F) at (1,0,1);
\coordinate (G) at (1,1,1);
\coordinate (H) at (0,1,1);

\node[inline] (e) at (E){};
\node[inline] (f) at (F){};
\node[inline] (g) at (G){};
\node[inline] (h) at (H){};

\draw [white,fill=gray!10] (0,0,0)--(1,0,0)--(1,1,0)--(0,1,0)--cycle;
\draw [white,fill=gray!3] (0,0,0)--(0,0,1)--(0,1,1)--(0,1,0)--cycle;
\draw [white,fill=gray!20] (0,0,0)--(0,0,1)--(1,0,1)--(1,0,0)--cycle;


\draw[dashed,thick] (e)--(a)--(d) (a)--(b);
\draw (c)--(d)--(h)--(e)--(f)--(g)--(c)--(b)--(f) (h)--(g);
\draw[very thick] (g)--(f);

\foreach \k in {1/4,1/2,3/4}
{\fill [black] ($\k*(a) + (b)-\k*(b)$) circle (0.3pt);
\fill [black] ($\k*(a) + (e)-\k*(e)$) circle (0.3pt);
\fill [black] ($\k*(a) + (d)-\k*(d)$) circle (0.3pt);
\fill [black] ($\k*(c) + (g)-\k*(g)$) circle (0.3pt);
\fill [black] ($\k*(c) + (d)-\k*(d)$) circle (0.3pt);
\fill [black] ($\k*(c) + (b)-\k*(b)$) circle (0.3pt);
\fill [black] ($\k*(f) + (g)-\k*(g)$) circle (0.3pt);
\fill [black] ($\k*(f) + (e)-\k*(e)$) circle (0.3pt);
\fill [black] ($\k*(f) + (b)-\k*(b)$) circle (0.3pt);
\fill [black] ($\k*(h) + (g)-\k*(g)$) circle (0.3pt);
\fill [black] ($\k*(h) + (e)-\k*(e)$) circle (0.3pt);
\fill [black] ($\k*(h) + (d)-\k*(d)$) circle (0.3pt);
};


\draw[ultra thick] ($5/8*(g) + 3/8*(f)$) ellipse (0.09cm and 0.26cm);
\draw[very thick,color=gray!75] ($3/8*(g) + 5/8*(f)$) ellipse (0.09cm and 0.26cm);


\draw (0.46,0.25) node {$\mathcal{C}_{1}$};
\draw (0.43,-0.1) node {$h(\mathcal{C}_{1})$};
\draw (0.64,0.56) node {$e$};
\draw (-0.5,0.12) node {\Large$\Delta$};

\draw (0.575,0.343) node {$p_{1}$};
\draw (0.657,-0.09) node {$p_{m}$};
\draw (0.82,0.11) node {$p_{\left\lceil \frac{m+1}{2}\right\rceil}$};
\end{tikzpicture}
\caption{The construction of $\sn[4]\bigast_{\an[4]} \sn[4]$ in $B_{m}(\St)$, $m=36$.}\label{fig:cube}
\end{figure}

The group $\Gamma$ possesses a unique subgroup $\Omega=\brak{f_{1},\ldots, f_{12}}$ isomorphic to $\an[4]$, where we take $f_{1}=\id$. For $i=1,\ldots,12$, let $e_{i}=f_{i}(e)$, whose orientation is that induced by $e=e_{1}$. For any two edges $e'$ and $e''$ of $\Delta$, there are precisely two elements of $\Gamma$ that send $e'$ to $e''$ (as non-oriented edges). One of these elements respects the orientation, and belongs to $\Omega$, and the other reverses the orientation, and belongs to $\Gamma\setminus \Omega$. Thus $h\in \Gamma \setminus \Omega$ and $\Gamma=\Omega \coprod h\Omega$ since $[\Gamma:\Omega]=2$. Let $\mathcal{C}_{1}$ be a simple closed curve bounding a disc that is a small neighbourhood of the subsegment $[p_{1},p_{\left\lceil \frac{m+1}{2}\right\rceil}]$ of the edge $e$. Let $g_{1}$ be the positive Dehn twist along $\mathcal{C}_{1}$. For $i=1,\ldots, 12$, let $g_{i}=f_{i}\circ g_{1}\circ f_{i}^{-1}$ (resp.\ $g_{i}'=f_{i}\circ h \circ g_{1}\circ h^{-1}\circ f_{i}^{-1}$) be the positive Dehn twist along the simple closed curve $f_{i}(\mathcal{C}_{1})$ (resp.\ $f_{i}\circ h(\mathcal{C}_{1})$). Since the stabiliser of the edge $e$ in $\Gamma$ is $\ang{h}$, which is isomorphic to $\Z_{2}$, the condition on $\mathcal{C}_{1}$ implies that the $f_{i}(\mathcal{C}_{1})$ (resp.\ the $f_{i}\circ h(\mathcal{C}_{1})$) are pairwise disjoint, and that $f_{i}(\mathcal{C}_{1})$ and $f_{j}\circ h(\mathcal{C}_{1})$ are disjoint if $i\neq j$. We conclude that the $g_{i}$ (resp.\ the $g_{i}'$) commute pairwise, and that $g_{i}$ and $g_{j}'$ commute if $i\neq j$.

Let $g=g_{1}\circ \cdots \circ g_{12}$. If $j\in \brak{1,\ldots,12}$, conjugation of $g$ by $f_{j}$ permutes the $g_{i}$, which commute pairwise, so $g$ and $f_{j}$ commute, and thus $g$ belongs to the centraliser of $\Omega$. Let $\Gamma'=g \Gamma g^{-1}$. By construction, $\Gamma\cong\Gamma'\cong \sn[4]$ and $\Gamma \cap \Gamma' \supset \Omega$. Let $\widetilde{\Gamma}=\Psi(\Gamma)$, $\widetilde{\Gamma'}=\Psi(\Gamma')$ and $\widetilde{\Omega}=\Psi(\Omega)$. Then $\widetilde{\Gamma}\cong \widetilde{\Gamma'}\cong \sn[4]$, $\widetilde{\Omega}\subset \widetilde{\Gamma}\cap \widetilde{\Gamma'}$ and $\widetilde{\Omega}\cong \an[4]$. Let us show that $H=\ang{\widetilde{\Gamma} \cup \widetilde{\Gamma'}}\cong \sn[4] \bigast_{\an[4]} \sn[4]$. To do so, by \repr{infincard} it suffices to prove that $H$ is infinite, and in light of the maximality of $\sn[4]$ as a finite subgroup of $\mcg$~\cite{S}, this comes down to showing that $\widetilde{\Gamma}\neq \widetilde{\Gamma'}$. It is enough to prove that $\widetilde{h'}\notin \widetilde{\Gamma}$, where $\widetilde{h'}=\Psi(h')\in \widetilde{\Gamma'}$, and $h'=ghg^{-1}$. To achieve this, suppose on the contrary that $\widetilde{h'}\in \widetilde{\Gamma}$. Let $\widetilde{g}=\Psi(g)$ and $\widetilde{h}=\Psi(h)$. Since $\widetilde{h}\in \widetilde{\Gamma}$, we have $\widetilde{h'}\widetilde{h}^{-1}=\widetilde{\gamma} \in \widetilde{\Gamma}$, where $\widetilde{\gamma}=\Psi(\gamma)$ and $\gamma=[g,h]$. On the other hand, $g$ is the product of Dehn twists, so $\widetilde{g} \in \pmcg$, thus $\widetilde{\gamma} \in \pmcg$ by normality of $\pmcg$ in $\mcg$, which implies that 
$\widetilde{\gamma} \in \widetilde{\Gamma}\cap \pmcg$. As we mentioned in the Introduction, $\pmcg$ is torsion free, and the finiteness of $\widetilde{\Gamma}$ forces $\widetilde{\gamma}=1$. In particular, if $\map{\alpha}{\pmcg}[{\pmcg[4]}]$ is the projection induced by forgetting all of the marked points with the exception of $p_{1}, p_{\left\lceil \frac{m+1}{2}\right\rceil}, p_{m}$ and $f_{2}(p_{1})$ (for example) then $\alpha(\widetilde{\gamma})=1$.

In order to reach a contradiction, we now analyse $\gamma$ more closely. Since $\Omega\vartriangleleft \Gamma$, there exists a permutation $\lambda\in \sn[12]$ satisfying $\lambda(1)=1$ such that for all $i\in\brak{1,\ldots, 12}$, $f_{i}\circ h=h\circ f_{\lambda(i)}$. Then 
\begin{equation}\label{eq:conjlambdai}
h \circ g_{\lambda(i)} \circ h^{-1}=h\circ f_{\lambda(i)}\circ g_{1}\circ f_{\lambda(i)}^{-1} \circ h^{-1}=f_{i}\circ h\circ g_{1}\circ h^{-1}\circ f_{i}^{-1}=g_{i}'.
\end{equation}
Hence
\begin{align*}
\gamma&=g\ldotp hg^{-1}h^{-1}= g_{1}\circ \cdots \circ g_{12}\circ \left( h \left( g_{12}^{-1}\circ \cdots \circ g_{1}^{-1}\right)h^{-1}\right)\\
&=g_{1}\circ \cdots \circ g_{12}\circ\left(h \left( g_{\lambda(1)}^{-1}\circ \cdots \circ g_{\lambda(12)}^{-1}\right)h^{-1}\right) \quad \text{since the $g_{j}$ commute pairwise}\\
&=g_{1}\circ \cdots \circ g_{12}\circ g_{1}'^{-1}\circ \cdots \circ g_{12}'^{-1} \quad \text{by \req{conjlambdai}}\\ 
&= (g_{1}\circ g_{1}'^{-1})\circ \cdots \circ (g_{12}\circ g_{12}'^{-1}) \quad \text{by the commutativity relations on $g_{i}$ and $g_{j}'$.}
\end{align*}
Now for $i=1,\ldots, 12$, the Dehn twists $g_{i}$ and $g_{i}'$ are along curves contained in a small neighbourhood of the subsegment $[f_{i}(p_{1}), f_{i}(p_{m})]$ of the edge $e_{i}$ of $\Delta$, and since the homomorphism $\alpha$ forgets all of the marked points lying outside $e$ with the exception of $f_{2}(p_{1})$, we see that $\alpha\circ \Psi(g_{i})$ and $\alpha\circ \Psi(g_{i}')$ are trivial for all $i=2,\ldots, 12$. In particular, $\alpha(\widetilde{\gamma})=\alpha\circ \Psi(\gamma)= (\alpha\circ \Psi(g_{1}))\circ (\alpha\circ \Psi(g_{1}'^{-1}))$. Taking the four marked points of $\pmcg[4]$ in the given order, $\alpha\circ \Psi(g_{1})$ (resp.\ $\alpha\circ \Psi(g_{1}'^{-1})$) is a positive (resp.\ negative) Dehn twist along a curve that bounds a disc containing the first two (resp.\ the second and third) points, and so a preimage of $\alpha(\widetilde{\gamma})$ in $P_{4}(\St)$ under the homomorphism $\phi$ of \req{mcg} is given by $A_{1,2}A_{2,3}^{-1}$, where the $A_{i,j}$ are defined by \req{defaij}. But this element belongs to the free subgroup $\pi_{1}(\St\setminus \brak{x_{1},x_{3},x_{4}},x_{2})$ of $P_{4}(\St)$ of rank two, for which a basis is $\brak{A_{1,2},A_{2,3}}$. Thus $A_{1,2}A_{2,3}^{-1}$ is an element of $P_{4}(\St)$ of infinite order, and taking into account \req{mcg}, we deduce that $\alpha(\widetilde{\gamma})=\phi(A_{1,2}A_{2,3}^{-1})$ is also of infinite order, which contradicts the conclusion of the previous paragraph. Thus $\widetilde{h'}\notin \widetilde{\Gamma}$, and from the above arguments, we see that $H\cong \sn[4] \bigast_{\an[4]} \sn[4]$.
Taking $G=B_{n}(\St)$, $x=\ft$ and $p=\phi$, $\phi$ being as in the short exact sequence~\reqref{mcg}, in \repr{vcmcg}(\ref{it:vcmcgb})(\ref{it:vcmcgbii}), we deduce that $\phi^{-1}\left(\widetilde{\Gamma} \bigast_{\widetilde{\Omega}} \widetilde{\Gamma'}\right)$ is an infinite virtually cyclic subgroup of $B_{n}(\St)$ of Type~II isomorphic to $\phi^{-1}(\widetilde{\Gamma}) \bigast_{\phi^{-1}(\widetilde{\Omega})} \phi^{-1}(\widetilde{\Gamma'})$. But $\phi^{-1}(\widetilde{\Gamma})\cong \phi^{-1}(\widetilde{\Gamma'}) \cong \oonestar$ and $\phi^{-1}(\widetilde{\Gamma}) \cap \phi^{-1}({\widetilde\Gamma'})= \phi^{-1}(\widetilde{\Omega})\cong \tonestar$, so this subgroup is indeed isomorphic to $\oonestar \bigast_{\tonestar} \oonestar$, which proves the proposition.
\end{proof}

%
%

\section{Proof of the realisation of elements of $\mathbb{V}_{2}(n)$ in $B_{n}(\St)$}\label{sec:proofthm5}

In this section, we bring together the results of \resec{realtypeII} in order to prove \repr{realV2bis}. This will enable us to complete the proof of \reth{main}.

\begin{prop}\label{prop:realV2bis}
Let $n\geq 4$. The following Type~II virtually cyclic groups are realised as subgroups of $B_{n}(\St)$:
\begin{enumerate}[(a)]
\item\label{it:realIIa} $\Z_{4q}\bigast_{\Z_{2q}} \Z_{4q}$, where $q$ divides $(n-i)/2$ for some $i\in\brak{0,1,2}$.
\item $\Z_{4q}\bigast_{\Z_{2q}} \dic{4q}$, where $q\geq 2$ divides $(n-i)/2$ for some $i\in\brak{0,2}$.
\item $\dic{4q}\bigast_{\Z_{2q}} \dic{4q}$, where $q\geq 2$ divides $n-i$ strictly for some $i\in\brak{0,2}$. 
\item\label{it:realIId} $\dic{4q}\bigast_{\dic{2q}} \dic{4q}$, where $q\geq 4$ is even and divides $n-i$ for some $i\in\brak{0,2}$. 
\item\label{it:realIIe} $\oonestar \bigast_{\tonestar} \oonestar$, where $n\equiv 0,2 \bmod{6}$ and $n=36$ or $n\geq 42$.
\end{enumerate}
\end{prop}

\begin{proof}
Parts~(\ref{it:realIIa})--(\ref{it:realIId}) follow directly from \reth{realV2}, while part~(\ref{it:realIIe}) follows from \repr{oto}.
\end{proof}

\begin{proof}[Proof of \reth{main}]
\reth{main}(\ref{it:mainI}) was proved in Propositions~\ref{prop:necV1} and~\ref{prop:necV2} for virtually cyclic subgroups of Types~I and~II respectively. \reth{main}(\ref{it:mainII}) was proved in Propositions~\ref{prop:realV1} and~\ref{prop:realV2bis} for virtually cyclic subgroups of Types~I and~II respectively. Finally, as we mentioned in \rerems{proddirbin}(\ref{it:proddirbinb}), the proof of \reth{main}(\ref{it:mainIII}) is an immediate consequence of \repr{ttimesz}(\ref{it:ttimeszd}).
%
\end{proof}

\section{Isomorphism classes of virtually cyclic subgroups of $B_{n}(\St)$ of Type~II}\label{sec:isoclasses}

%

By \reth{main}, we know which elements of $\mathbb{V}_{2}(n)$ are realised as subgroups of $B_{n}(\St)$. Such subgroups are of one of the following forms:
\begin{enumerate}[(a)]
\item\label{it:mainIIab} $\Z_{4q}\bigast_{\Z_{2q}} \Z_{4q}$, where $q\in \N$.

\item\label{it:mainIIbb} $\Z_{4q}\bigast_{\Z_{2q}} \dic{4q}$, where $q\geq 2$.

\item\label{it:mainIIcb} $\dic{4q}\bigast_{\Z_{2q}} \dic{4q}$, where $q\geq 2$.

\item\label{it:mainIIdb} $\dic{4q}\bigast_{\dic{2q}} \dic{4q}$, where $q\geq 4$ is even.

\item\label{it:mainIIeb} $\oonestar \bigast_{\tonestar} \oonestar$.
\end{enumerate}
There are of course additional constraints on $q$ imposed by the value of $n$. The aim of this section is to study the isomorphism classes of these amalgamated products. As we shall see in \repr{isoamalg}, there is a single such class, with the exception of $\quat[16] \bigast_{\quat} \quat[16]$, for which there are two possible classes. In \reco{semiamalg}, we will also show that with one exception (that occurs for one of the two isomorphism classes $\quat[16] \bigast_{\quat} \quat[16]$), each of the above amalgamated products of the form $G\bigast_{H} G$ is isomorphic to a semi-direct product $\Z\rtimes G$. We stress that \repr{isoamalg} and \reco{semiamalg} are consequences of the groups considered abstractly, and do not depend on the fact that they are realised as subgroups of $B_{n}(\St)$. 

Let $G$ be a group and $H$ a normal subgroup. Let $\aut[H]{G}$ denote the subgroup of $\aut{G}$ whose elements induce an automorphism of $H$. In some cases (if $H$ is characteristic, for example), the two groups $\aut{G}$ and $\aut[H]{G}$ coincide. We will concentrate our attention on the cases where $G$ is either cyclic of order a multiple of $4$, dicyclic, or equal to $\oonestar$. These are precisely the groups that appear as factors in the above list.


\begin{lem}\label{lem:index2}\mbox{}
\begin{enumerate}[(a)]
\item\label{it:autdic4q} Let $G$ be isomorphic to $\Z_{4q}$, $q\geq 1$, or to $\dic{4q}$, $q\geq 3$. Then $G$ possesses a unique subgroup $H$ that is isomorphic to $\Z_{2q}$, which is characteristic. Further, the homomorphism $\aut{G}\to \aut{H}$ given by restriction is surjective.

\item\label{it:autostar} Let $G$ be isomorphic to $\oonestar$. Then $G$ possesses a unique subgroup $H$ isomorphic to $\tonestar$, which is characteristic. Further, the homomorphism $\aut{G}\to \aut{H}$ given by restriction is surjective.

\item Let $G$ be isomorphic to $\quat$. Then $G$ possesses three subgroups $H_1, H_2,H_{3}$ that are isomorphic to $\Z_{4}$. Further, there is an automorphism of $\quat$ that sends $H_{i}$ to $H_{j}$ for all $i,j=1,2,3$. For $i=1,2,3$, the homomorphism $\aut[H_{i}]{G}\to \aut{H_{i}}$ given by restriction is surjective.

\item\label{it:autdic4q2q} Let $G$ be isomorphic to $\dic{4q}$, where $q\geq 4$ is even. Then $G$ possesses two subgroups $H_1, H_2$ that are isomorphic to $\dic{2q}$, and there exists an automorphism of $G$ that sends $H_{1}$ to $H_{2}$. Further, if $q\geq 6$, for $i=1,2$, the homomorphism $\aut[H_{i}]{G}\to \aut{H_{i}}$ given by restriction is surjective.
\end{enumerate}
\end{lem}

\begin{proof}\mbox{}
\begin{enumerate}
\item If $G$ is cyclic then the uniqueness of $H$ is clear. Now let $G\cong \dic{4q}$, $q\geq 3$. If $q$ is even (resp.\ odd) then $G$ possesses three subgroups (resp.\ one subgroup) of index $2$ because the Abelianisation of $\dic{4q}$ is isomorphic to $\Z_2\oplus\Z_2$ (resp.\ $\Z_{2}$), and exactly one is isomorphic to $\Z_{2q}$. In both the cyclic and dicyclic cases, the uniqueness of $H$ implies that it is characteristic. The surjectivity of the given homomorphism $\aut{G}\to \aut{H}$ is a consequence of the isomorphisms $\aut{\Z_{4q}}\cong \Z_{4q}^{\times}$, the group of units of $\Z_{4q}$, and $\aut{\dic{4q}}\cong \Z_{2q} \rtimes \Z_{2q}^{\times}$ (if $\dic{4q}$ is described by the presentation~\reqref{presdic} then the elements of $\aut{\dic{4q}}$ are given by automorphisms of the form $x\mapsto x^{i}$, $y\mapsto x^{j}y$, where $1\leq i\leq 2q-1$ is coprime with $2q$, and $1\leq j\leq 2q$, see~\cite[Example~1.4]{gg5} for more details). 

\item Let $G\cong \oonestar$ be given by the presentation~\reqref{presostar}, and let $H=\ang{P,Q,X}\cong \tonestar$. Then $G/H\cong \Z_{2}$ is the Abelianisation of $G$, generated by the $H$-coset of $R$, and so $\Gamma_{2}(G)\cong H$. If $K$ is a subgroup of $G$ isomorphic to $\tonestar$ then the canonical projection $G\to G/K$ factors through the canonical projection $G\to G/H$, from which it follows that $G$ possesses a unique subgroup isomorphic to $\tonestar$. The surjectivity of $\aut{G}\to \aut{H}$ was proved in \cite[Proposition~4.1]{gg5}.

\item The first part is clear. Note that the automorphism $\alpha(1)$ of $\quat$ given in \redef{v1v2}(\ref{it:mainq8})(\ref{it:mainIcii}) may be used to permute the $H_{i}$. If $i\in \brak{1,2,3}$ then the non-trivial element of $\aut{H_{i}}\cong \Z_{2}$ is the restriction to $H_{i}$ of conjugation on $G$ by any element of $G\setminus H_{i}$.

\item Let $G$ be isomorphic to $\dic{4q}$, where $q\geq 4$ is even, and let $G$ have the presentation given by \req{presdic}. From part~(\ref{it:autdic4q}), $G$ possesses exactly two subgroups isomorphic to $\dic{2q}$, $H_{k}=\ang{x^{2},x^{k-1}y}$ for $k=1,2$. The automorphism of $G$ given by $x \mapsto x $ and $y \mapsto xy$ sends $H_{1}$ to $H_{2}$. Suppose further that $q\geq 6$, and let $f\in \aut{H_{k}}$. Using the description of $\aut{\dic{2q}}$ given in the proof of part~(\ref{it:autdic4q}), there exist $1\leq i\leq q-1$, $\gcd(i,q)=1$, and $1\leq j\leq q$ such that $f(x^{2})=x^{2i}$ and $f(x^{k-1}y)=x^{2j}\ldotp x^{k-1}y$. Since $q$ is even, $i$ is odd, so $\gcd(i,2q)=1$, and $f$ is the restriction to $H$ of the automorphism $x\mapsto x^{i}$, $y\mapsto x^{(1-i)(k-1)+2j}y$ of $G$. Hence the homomorphism $\aut[H_{i}]{G} \to \aut{H_{i}}$ is surjective.\qedhere
\end{enumerate}
\end{proof}

\begin{rems}\mbox{}
\begin{enumerate}[(a)]
\item In the case $q=4$ of \relem{index2}(\ref{it:autdic4q2q}), $\aut{\quat[16]}\cong \Z_{4}\rtimes \Z_{4}^{\times}$, while $\aut{\quat}\cong \sn[4]$, so the homomorphism $\aut[\quat]{\quat[16]}\to \aut{\quat}$ clearly cannot be surjective.
\item Note that \repr{isoamalg} depends only on the amalgamated products considered in an abstract sense, and does not use the fact that the groups are realised as subgroups of $B_{n}(\St)$.
\end{enumerate}
\end{rems}

We now come to the proof of \repr{isoamalg}.

\begin{proof}[Proof of \repr{isoamalg}.]
First suppose that $G_{1}\bigast_{F} G_{2}$ is one of the amalgamated products (\ref{it:mainIIab})--(\ref{it:mainIIeb}) appearing in the above list, with the exception of the group $\quat[16]\bigast_{\quat} \quat[16]$. Then for $k=1,2$, there exist embeddings $\map{i_{k}}{F}[G_{k}]$ that give rise to the amalgamated product $G_{1}\bigast_{F} G_{2}$. Suppose that there exists another amalgamated product $G_{1}\bigast'_{F} G_{2}$ involving the same groups, and for $k=1,2$, let $\map{j_{k}}{F}[G_{k}]$
be the associated embeddings. Let $\map{i_{k}^{-1}}{i_{k}(F)}[F]$ denote the inverse of the restriction $\map{i_{k}}{F}[i_{k}(F)]$. Then $\map{j_{k}\circ i_{k}^{-1}}{i_{k}(F)}[j_{k}(F)]$ is an isomorphism of subgroups of $G_{k}$ isomorphic to $F$, and so by \relem{index2}, there exists $\rho_{k}\in \aut{G_{k}}$ whose restriction to $j_{k}(F)$ is sent to $i_{k}(F)$, in other words, the upper left hand `square' of the diagram given in Figure~\ref{fig:commdiag}
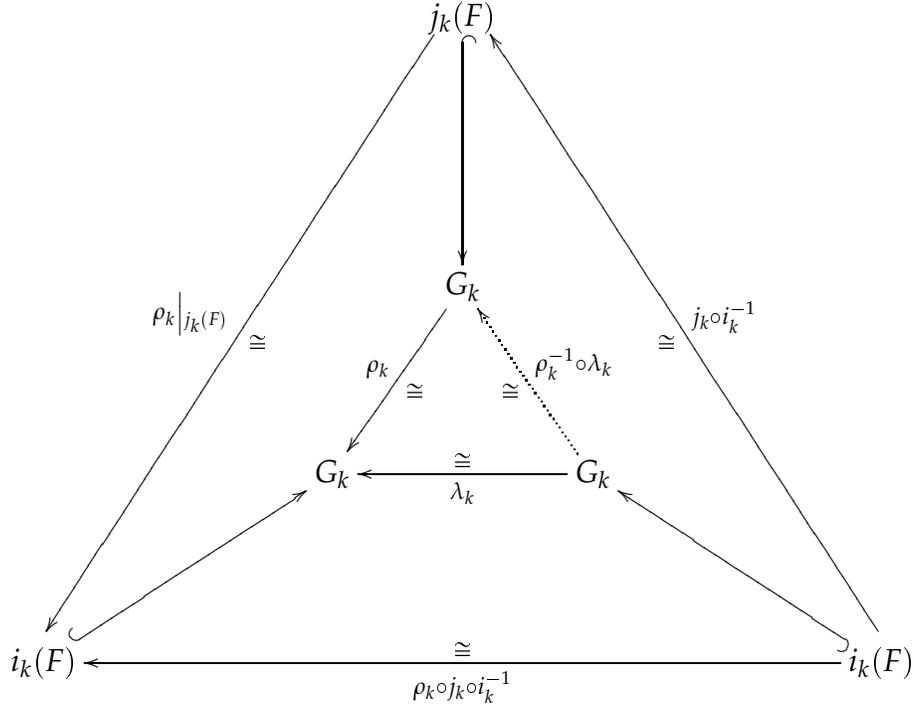
\begin{figure}[h]
\begin{equation*}
\xymatrix{%
&&& & {j_{k}(F)} \ar@{^{(}->}[ddd] \ar@<-1ex>[dddddddllll]_{\rho_{k}\left\lvert_{j_{k}(F)}\right.}^{\cong} & &&&\\
&& &&&&&&\\
&& &&&&&&\\
&& &&G_{k} \ar[ddl]_{\rho_{k}}^{\cong}  & &&\\
&&& &&&&&\\
&&&G_{k}  & & G_{k} \ar@{.>}[uul]_{{\rho_{k}^{-1}}\circ \lambda_{k}}^{\cong} \ar[ll]^{\lambda_{k}}_{\cong} &&\\
&& &&&&\\
i_{k}(F) \ar@{^{(}->}[rrruu] && & & && & & i_{k}(F) \ar@{^{(}->}[llluu] \ar[llllllll]^{\rho_{k}\circ j_{k}\circ i_{k}^{-1}}_{\cong} \ar@<-1ex>[uuuuuuullll]^{\cong}_{j_{k}\circ i_{k}^{-1}}
}
\end{equation*}
\caption{The commutative diagram involving the embeddings $i_{k}$ and $j_{k}$.}\label{fig:commdiag}
\end{figure}
commutes, where all of the arrows from $i_{k}(F)$ and $j_{k}(F)$ to $G_{k}$ are inclusions. Thus $\rho_{k}\circ j_{k}\circ i_{k}^{-1}$ is an automorphism of $i_{k}(F)$, and so once more by \relem{index2}, there exists $\lambda_{k}\in \aut{G_{k}}$ whose restriction to $i_{k}(F)$ is equal to $\rho_{k}\circ j_{k}\circ i_{k}^{-1}$, in other words, the lower `square' of the diagram commutes. Hence $\rho_{k}^{-1}\circ \lambda_{k}\in \aut{G_{k}}$, and the restriction of this automorphism to $i_{k}(F)$ yields the isomorphism $\map{j_{k}\circ i_{k}^{-1}}{i_{k}(F)}[j_{k}(F)]$. Taking $\theta=\rho_{k}^{-1}\circ \lambda_{k}$ and applying \repr{amalgiso}, we see that $G_{1}\bigast'_{F} G_{2}\cong G_{1}\bigast_{F} G_{2}$, which gives the result in this case.

We now turn to the exceptional case of $\quat[16]\bigast_{\quat} \quat[16]$. We have already seen that $\quat[16]$ possesses two subgroups isomorphic to $\quat$, and that there exists an automorphism of $\quat[16]$ that sends one subgroup into the other. Applying \repr{amalgiso} in a manner similar to that of the previous paragraph, it thus suffices to restrict our attention to one of these subgroups.  It remains to understand the amalgamated products obtained by considering all possible embeddings of $\quat$ whose image in each of the two copies of $\quat[16]$ is fixed. So let us consider the two copies of $\quat[16]$ of the form
\begin{equation*}
G_{1}=\setangr{x,y}{x^4=y^2,\; yxy^{-1}=x^{-1}}\quad \text{and}\quad G_{2}=\setangr{a,b}{a^4=b^2,\; bab^{-1}=a^{-1}}
\end{equation*}
respectively, and let $H_{1}=\ang{x^{2},y}$ and $H_{2}=\ang{a^{2},b}$ be their respective fixed subgroups isomorphic to $\quat$. Let $F=\setangr{P,Q}{P^{2}=Q^{2},\; QPQ^{-1}=P^{-1}}$ be an abstract copy of $\quat$. Up to isomorphism, every amalgamated product of $G_{1}$ and $G_{2}$ along $F$ is obtained via an isomorphism between $H_1$ and $H_2$. This leads to twenty-four possibilities that we identify with the elements of $\aut{F}\cong \sn[4]$ without further comment (see case~(\ref{it:caseq8}) of \resecglobal{generalities}{autout}). Let $\map{\delta}{F}[H_{1}]$ be a fixed isomorphism, which we shall take to be defined by $\delta(P)=x^{2}$ and $\delta(Q)=y$. Suppose that $\map{\phi,\phi'}{H_{1}}[H_{2}]$ are isomorphisms that differ by the inner automorphism $\iota_{h}$ of $H_{2}$, where $h\in H_{2}$, and let $G_{1}\bigast_{F} G_{2}$ and $G_{1}\bigast'_{F} G_{2}$ denote the respective amalgamated products. Then $\phi'=\iota_{h}\circ \phi$, and we have the following commutative diagram:
\begin{equation}\label{eq:diagamalg}
\begin{xy}*!C\xybox{%
\xymatrix{%
& G_{1}\ar[ld] \ar@{=}[rrrr] & & & &G_{1}\ar[rd]& \\
G_{1}\bigast'_{F} G_{2} &&& F \ar[llu]^{\delta} \ar[rru]_{\delta} \ar[lld]_{\phi'\circ\delta} \ar[rrd]^{\phi\circ\delta}
&& &G_{1}\bigast_{F} G_{2},\\
& G_{2} \ar[lu] & &&& \ar[llll]^{\iota_{h}}G_{2}\ar[ru]& \\
}}
\end{xy}
\end{equation}
where we also denote the extension of $\iota_{h}$ to $G_{2}$ by $\iota_{h}$. Taking $\theta_{1}=\id_{G_{1}}$ and $\theta_{2}=\iota_{h}$ in \repr{amalgiso} leads to the conclusion that $G_{1}\bigast_{F} G_{2}\cong G_{1}\bigast'_{F} G_{2}$ if $\phi$ and $\phi'$, considered as elements of $\aut{F}$, project to the same element of $\out{F}$. So it suffices to consider the six following coset representatives of $\out{F}$ in $\aut{F}$ (recall from case~(\ref{it:caseq8}) of \resecglobal{generalities}{autout}, that $\out{\quat}\cong\sn[3]$):
\begin{equation}\label{eq:sixphis}
\begin{aligned}
& \phi_1\colon\thinspace x^2 \mapsto a^2, \; y \mapsto b, \; x^2y \mapsto a^2b, & &\phi_2\colon\thinspace x^2 \mapsto b, \; y \mapsto a^2b, \; x^2y \mapsto a^2,\\
&\phi_3\colon\thinspace x^2 \mapsto a^2b, \; y \mapsto a^2, \; x^2y \mapsto b, && \phi_4\colon\thinspace x^2 \mapsto a^2, \; y \mapsto a^2b, \; x^2y \mapsto b^{-1},\\
& \phi_5\colon\thinspace x^2 \mapsto a^2b, \; y \mapsto b^{-1}, \; x^2y \mapsto a^{2},  &&\phi_6\colon\thinspace x^2 \mapsto b, \; y \mapsto a^{-2}, \; x^2y \mapsto a^2b.
\end{aligned}
\end{equation}
Let $\psi\in \aut{G_{2}}$ be defined by $a \mapsto a$, $b \mapsto a^2b$. Then $\phi_{4}=\psi\circ \phi_{1}$ (resp.\ $\phi_{5}=\psi\circ \phi_{2}$). Consider the above diagram~\reqref{diagamalg}, and replace $\phi$ by $\phi_{1}$ (resp.\ $\phi_{2}$), $\phi'$ by $\phi_{4}$ (resp.\ $\phi_{5}$), and $\iota_{h}$ by the automorphism of $G_{2}$ given by $a \mapsto a$, $b \mapsto a^2b$. Applying \repr{amalgiso} implies that the two automorphisms $\phi_1$ (resp.\ $\phi_{2}$) and $\phi_4$ (resp.\ $\phi_{5}$) give rise to isomorphic amalgamated products. If $\psi'\in \aut{G_{2}}$ is defined by $a \mapsto a^{-1}$, $b \mapsto a^2b$, then $\phi_{6}=\psi'\circ \phi_{3}$, and a similar argument shows that $\phi_{3}$ and $\phi_{6}$ also give rise to isomorphic amalgamated products. Now let $G_{1}\bigast_{F} G_{2}$ and $G_{1}\bigast'_{F} G_{2}$ be the amalgamated products associated with $\phi_{2}$ and $\phi_{3}$ respectively. Let $\map{\delta'}{F}[H_{1}]$ be the isomorphism defined by $\delta'(P)=x^{-2}$ and $\delta'(Q)=x^{2}y$, let $\theta_{1}\in \aut{G_{1}}$ be defined by $x \mapsto x^{-1}$, $y \mapsto x^2y$, and let $\theta_{2}\in \aut{G_{2}}$ be defined by $a \mapsto a$, $b \mapsto a^2 b^{-1}$. Then the following diagram commutes:
\begin{equation*}
\xymatrix{%
& G_{1}\ar[ld]  & & & & \ar[llll]_{\theta_{1}}^{\cong} G_{1}\ar[rd]& \\
G_{1}\bigast'_{F} G_{2} &&& F \ar[llu]^{\delta'} \ar[rru]_{\delta} \ar[lld]_{\phi_{3}\circ\delta'} \ar[rrd]^{\phi_{2}\circ\delta}
&& &G_{1}\bigast_{F} G_{2}.\\
& G_{2} \ar[lu] & &&& \ar[llll]^{\theta_{2}}_{\cong} G_{2}\ar[ru]& \\
}
\end{equation*}
So $\phi_{2}$ and $\phi_{3}$ give rise to isomorphic amalgamated products by \repr{amalgiso}. We conclude that there are at most two non-isomorphic amalgamated products of the form $K_{i}=G_{1}\bigast_{F} G_{2}$, defined by the automorphism $\phi_{i}$, where $i\in\brak{1,2}$.

To complete the proof, we now prove that $K_{1}\ncong K_{2}$. We start by showing that $K_{1}\cong \Z\rtimes \quat[16]$, where the action shall be defined presently. By definition,
\begin{equation}\label{eq:presK1}
K_{1}=\setangr{x,y,a,b}{x^{4}=y^{2},\; a^{4}=b^{2},\; yxy^{-1}= x^{-1},\; bab^{-1}=a^{-1},\; x^{2}=a^{2},\; y=b}.
\end{equation}
%
Let $N$ be the infinite cyclic subgroup of $K_{1}$ generated by $t=xa^{-1}$. Using the presentation~\reqref{presK1}, one may check that 
\begin{equation*}
vtv^{-1}=
\begin{cases}
t^{-1} & \text{if $v\in \brak{x,a}$}\\
t & \text{if $v\in \brak{y,b}$,}
\end{cases}
\end{equation*}
so $N$ is normal in $K_{1}$. A presentation of the quotient $K_{1}/N$ is obtained by adjoining the relation $x=a$ to that of $K_{1}$, from which it follows that
\begin{equation*}
K_{1}/N=\setangr{a,b}{a^{4}=b^{2},\; bab^{-1}=a^{-1}}\cong \quat[16].
\end{equation*}
Considered as a subgroup of $K_{1}$, $G_{2}=\ang{a,b}$ is isomorphic to $\quat[16]$, which implies that the short exact sequence
\begin{equation*}
1\to N\to K_{1}\to K_{1}/N\to 1
\end{equation*}
splits, and so $K_{1}\cong \Z \rtimes \quat[16]$. The action of $K_{1}/N$ on $N$ is defined as follows:
\begin{equation}\label{eq:zrtimesq16}
wtw^{-1}=
\begin{cases}
t^{-1} & \text{if $w\in G_{2}\setminus\ang{a^{2},b}$}\\
t & \text{if $w\in \ang{a^{2},b}$.}
\end{cases}
\end{equation}

To see that $K_{1}\ncong K_{2}$, let us suppose on the contrary that $K_{1}\cong K_{2}$ and argue for a contradiction. By definition,
\begin{equation}\label{eq:presK2}
K_{2}=\setangr{x,y,a,b}{x^{4}=y^{2},\; a^{4}=b^{2},\; yxy^{-1}= x^{-1},\; bab^{-1}=a^{-1},\; x^{2}=b,\; y=a^{2}b}.
\end{equation}
From this presentation, we obtain:
\begin{align*}
ax\ldotp x^{2}\ldotp x^{-1}a^{-1}&= ax^{2}a^{-1}=aba^{-1}=a^{2}b=y,\\
ax\ldotp y\ldotp x^{-1}a^{-1}&= ax^{2}ya^{-1}=aba^{2}ba^{-1}=a^{2}=x^{2}y\\
ax\ldotp x^{2}y\ldotp x^{-1}a^{-1}&=yx^{2}y=x^{2}.
\end{align*}
So $K_{2}$ possesses a copy $\ang{x^{2},y}$ of $\quat$ and an element $ax$ such that conjugation by $ax$ permutes the subgroups $\ang{x^{2}}$, $\ang{y}$ and $\ang{x^{2}y}$ cyclically. Since $K_{1}\cong K_{2}$ by hypothesis, $K_{1}$ thus possesses a subgroup $H$ isomorphic to $\quat$ and an element $z$ (of infinite order) such that $z L z^{-1}\neq L$ for every subgroup $L$ of $H$ of order $4$. We take $K_{1}$ to be described by the semi-direct product $\Z\rtimes G_{2}$, where the action is given by \req{zrtimesq16}. In particular, $K_{1}=\ang{a,b,t}$, and there exist $s,\lambda,\mu\in \Z$ such that $z=t^{s}a^{\lambda}b^{\mu}$. Consider the projection $\map{p}{\Z\rtimes G_{2}}[G_{2}]$ onto the second factor. 
Since $\ker{p}=\Z$ is torsion free, $p(H)$ is isomorphic to $\quat$, and thus must be equal to one of the two subgroups of $G_{2}$ isomorphic to $\quat$. These two subgroups both contain $a^{2}$, so there exists $u\in H$ of order $4$ such that $p(u)=a^{2}$. Now $p(a^{2})=a^{2}$, hence there exists $m\in \Z$ such that $u=t^{m}a^{2}$. But $t$ commutes with $a^{2}$ by \req{zrtimesq16}, and since $u$ and $a^{2}$ are of finite order, and $t$ is of infinite order, it follows that $m=0$, $u=a^{2}$ and:
\begin{equation*}
zuz^{-1}= t^{s}a^{\lambda}b^{\mu} a^{2} b^{-mu}a^{-\lambda}t^{-s}= t^{s}a^{\lambda} a^{2\epsilon} a^{-\lambda}t^{-s}=a^{2\epsilon},
\end{equation*}
where $\epsilon$ is equal to $1$ (resp.\ $-1$) if $\mu$ is even (resp.\ odd), and so $z\ang{u}z^{-1}=\ang{u}$. This contradicts the fact that $z L z^{-1}\neq L$ for every subgroup $L$ of $H$ of order $4$, and completes the proof of the fact that $K_{1}\ncong K_{2}$.
\end{proof}

%
%

Combining \repr{semiamalg} and \relem{index2} yields an alternative description of most of the amalgamated products of the form $G\bigast_{H} G$ appearing in \repr{realV2bis} as semi-direct products of $\Z$ by $G$.

\begin{cor}\label{cor:semiamalg}
Let $\Gamma=G\bigast_{H} G$ be an amalgamated product, where $G$ and $H$ satisfy one of the following conditions:
\begin{enumerate}
\item\label{it:amalgsemia} $G$ is isomorphic to $\Z_{4q}$ or $\dic{4q}$ and $H$ is isomorphic to $\Z_{2q}$.
\item\label{it:amalgsemib} $G$ is isomorphic to $\dic{4q}$, $q\geq 6$ is even and $H$ is isomorphic to $\dic{2q}$.
\item\label{it:amalgsemic} $q=4$, $G\cong \quat[16]$, $H\cong \quat$ and $\Gamma$ is isomorphic to $K_{1}$.
\item\label{it:amalgsemid} $G$ is isomorphic to $\oonestar$ and $H$ is isomorphic to $\tonestar$.
\end{enumerate}
Then $\Gamma\cong \Z\rtimes G$, where 
\begin{equation*}
gtg^{-1}=\begin{cases}
t & \text{if $g\in H$}\\
t^{-1} & \text{if $g\in G\setminus H$,}
\end{cases}
\end{equation*}
$t$ being a generator of the $\Z$-factor.
\end{cor}

\begin{proof}
First let $G$ and $H$ satisfy one of the conditions~(\ref{it:amalgsemia}),~(\ref{it:amalgsemib}) or (\ref{it:amalgsemid}). If $i_{1},i_{2}$ are the embeddings of $H$ into each of the $G$-factors of $\Gamma$ then $i_{2}\circ i_{1}^{-1}$ is an automorphism of $H$ that extends to an automorphism of $G$ by \relem{index2}. The result then follows from \repr{semiamalg}. Now suppose that $G$ and $H$ satisfy condition~(\ref{it:amalgsemic}). Since $\Gamma$ is isomorphic to $K_{1}$, using the presentation~\reqref{presK1}, we see that the isomorphism $\ang{x^{2},y}\to \ang{a^{2},b}$ of the amalgamating subgroup of $K_{2}$ isomorphic to $\quat$ that given by $x^{2}\mapsto a^{2}$ and $y\mapsto b$ extends to an isomorphism $\ang{x,y}\to \ang{a,b}$ of the factors that are isomorphic to $\quat[16]$, where the extension is given by $x\mapsto a$ and $y\mapsto b$. Once more, \repr{semiamalg} yields the result.
\end{proof}

The following two results will imply the existence of subgroups of $B_{n}(\St)$ isomorphic to $K_{1}$ and $K_{2}$ for all but a finite number of even values of $n$. The first proposition holds in general, while the second makes use of the structure of $B_{n}(\St)$.

\begin{prop}\label{prop:o2k2}
Let $G$ be a group that is isomorphic to $\oonestar\bigast_{\tonestar}\oonestar$. Then $G$ possesses a subgroup that is isomorphic to $K_{2}$.
\end{prop}

\begin{proof}
Suppose that $G$ is isomorphic to $\oonestar\bigast_{\tonestar} \oonestar$. Let $G_{1},G_{2}$ be subgroups of $G$ isomorphic to $\oonestar$ such that $F=G_{1}\cap G_{2}\cong \tonestar$ and $G=\ang{G_{1}\cup G_{2}}\cong \oonestar\bigast_{\tonestar} \oonestar$. Let $Q$ be the unique subgroup of $F\cong \quat\rtimes\Z_{3}$ that is isomorphic to $\quat$. By \relem{index2}(\ref{it:autostar}), $F$ is the unique subgroup of $G_{i}$ isomorphic to $\tonestar$ for $i=1,2$. From the proof of \repr{maxsubgp}(\ref{it:subgpsostar}), if $i\in\brak{1,2}$, the Sylow $2$-subgroups of $G_{i}$ consist of three conjugate subgroups isomorphic to $\quat[16]$ that contain $Q$. Let $H_{1}$ be one of the Sylow $2$-subgroups of $G_{1}$ with presentation
\begin{equation*}
H_{1}=\setangr{a,b}{a^{4}=b^{2},\; bab^{-1}=a^{-1}}.
\end{equation*}
Since the subgroups of $H_{1}$ isomorphic to $\quat$ are of the form $\ang{a^{2}, a^{\epsilon}b}$, $\epsilon\in\brak{0,1}$, by replacing $b$ by $ab$ if necessary in the presentation of $H_{1}$, we may suppose that $Q=\ang{a^{2},b}$. Now let $H_{2}$ be a subgroup of $G_{2}$ that is isomorphic to $\quat[16]$. Since for $i\in\brak{1,2}$, $H_{i}\nsubset F$, it follows that $H_{1}\cap H_{2}=Q$ and that $H_{i}$ contains elements of $G_{i}\setminus F$, whence $H=\ang{H_{1}\cup H_{2}}\cong \quat[16]\bigast_{\quat} \quat[16]$. The proof of \repr{isoamalg} implies that $H$ is isomorphic to one of $K_{1}$ and $K_{2}$. If $H\cong K_{2}$ then we are done. So suppose that $H\cong K_{1}$.
Then by \req{presK1}, there exist generators $x,y$ of $H_{2}$ such that $x^{4}=y^{2}$, $yxy^{-1}=x^{-1}$, $x^{2}=a^{2}$, $y=b$ and $Q=\ang{x^{2},y}$. Since $Q$ is the unique subgroup of $F$ that is isomorphic to $\quat$, there exists $t\in F$ such that $tx^{2}t^{-1}=y$ and $tyt^{-1}=x^{2}y$ corresponding to the action of $\Z_{3}$ on $\quat$. Now $F\subset G_{2}$, so $tG_{2}t^{-1}=G_{2}$. Let $H_{2}'=tH_{2}t^{-1}\subset G_{2}$. Then $x'=txt^{-1}$ and $y'=tyt^{-1}$ are generators of $H_{2}'$ satisfying $x'^{4}=y'^{2}$ and $y'x'y'^{-1}=x'^{-1}$. Now $x'^{2}=tx^{2}t^{-1}=y=b$ and $y'=tyt^{-1}=x^{2}y=a^{2}b$, and since $H_{1}\cap H_{2}'=Q$, it follows from \req{presK2} that $\ang{H_{1}\cup H_{2}'}\cong K_{2}$ as required.
\end{proof}

\begin{prop}\label{prop:existk1k2}
Let $n\geq 4$ be even.
\begin{enumerate}[(a)]
\item\label{it:existk1k2a} There exists a subgroup of $B_{n}(\St)$ isomorphic to $K_{1}$.
\item\label{it:existk1k2b} Suppose that either $n\equiv 0\bmod{4}$ or $n\equiv 10 \bmod{12}$. There exists a subgroup of $B_{n}(\St)$ isomorphic to $K_{2}$.
\end{enumerate}
\end{prop}

\begin{rem}\label{rem:existk1k2}
Let $n\geq 4$ be even. Propositions~\ref{prop:oto},~\ref{prop:o2k2} and~\ref{prop:existk1k2} imply that $B_{n}(\St)$ possesses subgroups isomorphic to $K_{1}$ and $K_{2}$ with the possible exception of $K_{2}$ when $n\in \brak{6,14,18,26,30,38}$. 
\end{rem}

\begin{proof}[Proof of \repr{existk1k2}.]\mbox{}
\begin{enumerate}[(a)]
\item Suppose that $n\geq 4$ is even, let $i\in \brak{0,2}$ be such that $4 \divides n-i$, and let $m=(n-i)/4$. In the construction of $\quat[16]\bigast_{\quat} \quat[16]$ in $B_{n}(\St)$ given in part~(\ref{it:realV2d}) of the proof of \reth{realV2}, we have that $G_{1}=\ang{x,y}$ and $G_{2}=\ang{a,b}$, where $x=\alpha_{i}'^{m}$, $a=\lambda_{i} \alpha_{i}'^{m}\lambda_{i}^{-1}$ and $y=b=\garside$, where $\lambda_{i}=\sigma_{m+\frac{i}{2}}\sigma_{3m+\frac{i}{2}}$. Since $\lambda_{i}$ commutes with $\alpha_{i}'^{2m}$, we have also that $x^{2}=a^{2}$. So $\ang{G_{1}\cup G_{2}}$ is isomorphic to $K_{1}$ by \req{presK1}.
\item We consider the two cases separately.
\begin{enumerate}[(1)]
\item \underline{$n\equiv 0 \bmod{4}$.} Set $G_{1}=\ang{a,b}$, where $a=\alpha_{0}^{n/4}$ and $b=\garside$, let $G_{2}=\nu G_{1} \nu^{-1}$, where $\nu=\alpha_{0}^{n/4} \Omega_{2}$ is as in the proof of \repr{constq8}(\ref{it:q8parta}), and let $x=\nu a \nu^{-1}$ and $y=\nu b^{-1} \nu^{-1}$ be generators of $G_{2}$. Then $G_{1}\cong G_{2}\cong \quat[16]$, and $F=\ang{a^{2},b}$ is isomorphic to $\quat$. Since $\nu F \nu^{-1}=F$ by \repr{constq8}(\ref{it:q8parta}), it follows that $G_{1}\cap G_{2} \supset F$. Suppose that $x\in G_{1}$. Since $x$ is of order $8$, there exists $j\in \brak{1,3,5,7}$ such that $x=a^{j}$, and so $x^{2}\in \brak{\alpha_{0}^{n/2},\alpha_{0}^{-n/2}}$. On the other hand, using \relem{propsomega}(\ref{it:omegav}) and \req{basicconj}, we have:
\begin{equation}\label{eq:x2a2b}
x^{2}= \alpha_{0}^{n/4} \Omega_{2} \alpha_{0}^{n/2}\Omega_{2}^{-1} \alpha_{0}^{-n/4} = \alpha_{0}^{n/4} \garside \alpha_{0}^{-n/4}= \alpha_{0}^{n/2}\garside=a^{2}b,
\end{equation}
and so $x^{2}$ does not belong to $\brak{\alpha_{0}^{n/2},\alpha_{0}^{-n/2}}$, which gives a contradiction. We thus conclude that $x\notin G_{1}$, and so $G_{1}\cap G_{2}=F$. Let $K=\ang{G_{1}\cup G_{2}}$. By \req{nuconjgarside}, we have 
\begin{equation}\label{eq:ya2}
y=\nu b^{-1} \nu^{-1}=\nu \garside^{-1} \nu^{-1}= \alpha_{0}^{n/2}=a^{2}.
\end{equation}
Equations~\reqref{x2a2b} and~\reqref{ya2} correspond to the automorphism $\phi_{3}$ of \req{sixphis}, which using the proof of \repr{isoamalg}, will imply that $K\cong K_{2}$ provided that $K$ is indeed isomorphic to $\quat[16] \bigast_{\quat} \quat[16]$. By \repr{infincard}, it thus suffices to show that $K$ is infinite. To see this, we consider the following three cases.
\begin{enumerate}[(i)]
\item \underline{$n=4$.} Since the maximal finite subgroups of $B_{4}(\St)$ are isomorphic to $\quat[16]$ and $\tonestar$ by \reth{finitebn}, the fact that $G_{1}\neq G_{2}$ implies that $K$ is infinite.
\item \underline{$n=8$.} Let $\gamma=a^{-2} \nu a \nu^{-1} a \in K$.
Then
\begin{align*}
\gamma&=\alpha_{0}^{-2} \Omega_{2} \alpha_{0}^{2} \Omega_{2}^{-1}= \sigma_{3} \sigma_{4} \sigma_{5} \sigma_{3} \sigma_{4} \sigma_{3} \sigma_{5}^{-1} \sigma_{6}^{-1} \sigma_{5}^{-1} \sigma_{7}^{-1} \sigma_{6}^{-1} \sigma_{5}^{-1}\\
&=\sigma_{5}(\sigma_{3}\sigma_{4}\sigma_{5}\sigma_{3}\sigma_{4}\sigma_{5}^{-1} \sigma_{6}^{-1} \sigma_{5}^{-1} \sigma_{7}^{-1} \sigma_{6}^{-1})\sigma_{5}^{-1}\\
&= \sigma_{5}(\sigma_{4}\sigma_{3}\sigma_{5}\sigma_{4}\sigma_{6}^{-1} \sigma_{5}^{-1} \sigma_{7}^{-1} \sigma_{6}^{-1})\sigma_{5}^{-1}.
\end{align*}
by equations~\reqref{fundaa} and~\reqref{omegadef}. The braid $\widehat{\gamma}=\sigma_{4}\sigma_{3}\sigma_{5}\sigma_{4}\sigma_{6}^{-1} \sigma_{5}^{-1} \sigma_{7}^{-1} \sigma_{6}^{-1}$ is conjugate to $\gamma$, and its geometric representation is given in Figure~\ref{fig:gamma}.
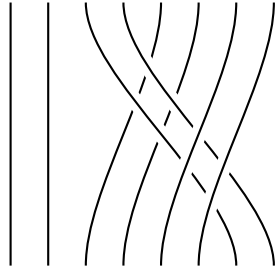
\begin{figure}[h]
\hfill
\begin{tikzpicture}[scale=0.5]
\foreach \j in {5,6}
{\draw[thick] (\j,8) .. controls (\j,6) and (\j-2,3) .. (\j-2,1);};
\foreach \j in {3,4}
{\draw[white,line width=6pt] (\j,8) .. controls (\j,6) and (\j+4,3) .. (\j+4,1);
\draw[thick] (\j,8) .. controls (\j,6) and (\j+4,3) .. (\j+4,1);};
\foreach \j in {7,8}
{\draw[white,line width=6pt] (\j,8) .. controls (\j,6) and (\j-2,3) .. (\j-2,1);
\draw[thick] (\j,8) .. controls (\j,6) and (\j-2,3) .. (\j-2,1);};
\foreach \k in {1,2}
{\draw[thick] (\k,1)--(\k,8);};
\end{tikzpicture}
\hspace*{\fill}
\caption{The braid $\widehat{\gamma}$ in $B_8(\St)$.}\label{fig:gamma}
\end{figure}
Forgetting the $2\up{nd}$, $4\up{th}$, $6\up{th}$ and $8\up{th}$ strings of $\widehat{\gamma}$ yields the braid $\sigma_{2}\sigma_{3}^{-1}$ of $B_{4}(\St)$, which may be seen to be of infinite order using an argument similar to that of \req{sig4sig3}. Hence $\widehat{\gamma}$ and $\gamma$ are of infinite order, and so $K$ is infinite.

\item \underline{$n\geq 12$.} Consider the element $\gamma'=\nu a \nu\ldotp a^{-1}=\alpha_{0}^{n/4} \Omega_{2} \alpha_{0}^{n/4} \Omega_{2}^{-1}\alpha_{0}^{-n/2}\in K$, and let $n/2+1\leq t\leq 3n/4$. Then $\pi(\alpha_{0}^{n/4})(t)=t-n/4$, $\pi(\Omega_{2})(t-n/4)=t-n/4$ since $t-n/4\leq n/2$, $\pi(\alpha_{0}^{n/4})(t-n/4)=t-n/2$, and $\pi(\Omega_{2}^{-1}\alpha_{0}^{-n/2})(t-n/2)=t$, so $\pi(\gamma')(t)=t$. Further, $\pi(\gamma')(1)=n$. Thus the cycle decomposition of $\pi(\gamma')$ has at least $n/4\geq 3$ fixed points, and at least one non-trivial cycle. \reth{murasugi} then implies that $\gamma'$ is of infinite order, and so $K$ is infinite.
\end{enumerate}
\item \underline{$n\equiv 10\bmod{12}$.} 
Then $4$ divides $n-2$, and we may write $n-2=2^{r}s$, where $r\geq 2$ and $s\in \N$ is odd. Since $n$ is even, $B_{n}(\St)$ possesses a subgroup $L$ isomorphic to $\tonestar\cong \quat\rtimes \Z_{3}$ by \reth{finitebn}. The fact that the action by conjugation of the generator of $\Z_{3}$ permutes cyclically the elements $i,j$ and $k$ of the subgroup $Q$ of $L$ isomorphic to $\quat$ implies that these elements are pairwise conjugate in $L$. On the other hand, by~\cite[Proposition~1.5]{GG7}, $\ang{\alpha_{2}'^{s},\garside}$ represents the unique conjugacy class of the group $\quat[2^{r+2}]$ in $B_{n}(\St)$, and it possesses two subgroups $\ang{\alpha_{2}'^{2s},\garside}$ and $\ang{\alpha_{2}'^{2s},\alpha_{2}'^{s}\garside}$ isomorphic to $\quat[2^{r+1}]$ that contain respectively $\Gamma_{0}=\ang{\alpha_{2}'^{2^{r-1}s},\garside}$ and $\Gamma_{1}= \ang{\alpha_{2}'^{2^{r-1}s},\alpha_{2}'^{s}\garside}$ which are subgroups isomorphic to $\quat$. The fact that $n/2$ and $s$ are odd implies that $\pi(\garside)= \pi(\alpha_{2}'^{s})=n-1\bmod{2(n-1)}$. In particular, $\alpha_{2}'^{2^{r-1}s}$ and $\garside$ are not conjugate, so $\Gamma_{0}$ is neither conjugate to $\Gamma_{1}$ nor to $Q$, and since $B_{n}(\St)$ possesses two conjugacy classes of subgroups isomorphic to $\quat$~\cite[Proposition~1.5]{GG7}, we deduce that $\Gamma_{1}$ and $Q$ are conjugate. Set $G_{1}=\ang{a,b}$, where $a=\alpha_{2}'^{2^{r-2}s}$ and $b=\alpha_{2}'^{s}\garside$. Then $G_{1}$ is isomorphic to $\quat[16]$, and it contains $\Gamma_{1}=\ang{a^{2}, b}$. Since $B_{n}(\St)$ possesses two isomorphism classes of subgroups isomorphic to $\quat$, we deduce that $\Gamma_{1}$ and $Q$ are conjugate, and using the fact that $Q$ is a subgroup of $L$, there exists an element $z\in B_{n}(\St)$ conjugate to an element of $L\setminus Q$ for which $za^{2}z^{-1}=b$, $zbz^{-1}=a^{2}b$ and $za^{2}b z^{-1}=a^{2}$.
Let $G_{2}=zG_{1} z ^{-1}$. 
From the action by conjugation of $z$ on $\Gamma_{1}$, we have that $G_{1}\cap G_{2}\supset \Gamma_{1}$. Suppose that $G_{1}=G_{2}$. Then $zaz^{-1}\in G_{2}$, which is of order $8$, would be equal to an element of order $8$ of $G_{1}$, and so would be of the form $a^{j}$, $j\in\brak{1,3,5,7}$. Thus $za^{2}z^{-1}\in \brak{a^{2},a^{-2}}$, which is not possible. This implies that $G_{1}\neq G_{2}$, and thus $G_{1}\cap G_{2}=\Gamma_{1}$. Let $K=\ang{G_{1}\cup G_{2}}$. To see that $K\cong \quat[16] \bigast_{\quat} \quat[16]$,  by \repr{infincard} it suffices to prove that $K$ is infinite. Suppose on the contrary that $K$ is finite, and let $M$ be a finite maximal subgroup of $B_{n}(\St)$ that contains $K$. Since $K$ contains copies of $\quat[16]$, $M$ cannot be cyclic, nor can it be isomorphic to $\tonestar$ or $\istar$ by \repr{maxsubgp}. By the hypothesis on $n$, $\oonestar$ is not realised as a subgroup of $B_{n}(\St)$ by \reth{finitebn}, so $M\ncong \oonestar$, and thus $M\cong \dic{4(n-2)}$. Let $u\in M$ be an element of order $2(n-2)$. Since $G_{1}$ and $G_{2}$ are subgroups of $M$ isomorphic to $\quat[16]$, they both contain the unique cyclic subgroup $\ang{u^{(n-2)/4}}$ of $M$ of order $8$, but this contradicts the fact that $G_{1}\cap G_{2}=\Gamma_{1}$. So $M\ncong \dic{4(n-2)}$, and thus $K$ is infinite by \reth{finitebn}. \repr{infincard} then implies that $K\cong \quat[16]\bigast_{\quat} \quat[16]$. It remains to show that $K\cong K_{2}$. This may be seen as follows. 
To see this, let $x=zaz^{-1}$ and $y=zbz^{-1}$ be generators of $G_{2}$. Then $x^{4}=y^{2}$, $yxy^{-1}=x^{-1}$, $x^2=za^2 z^{-1}=za^2 z^{-1}=b$ and $y=zbz^{-1}=a^2 b$. Equation~\reqref{presK2} implies that $K\cong K_{2}$ as required.\qedhere
\end{enumerate}
\end{enumerate}
\end{proof}

\section{Classification of the virtually cyclic subgroups of the mapping class group $\mcg$}\label{sec:genmcg}


We apply \reth{main} and \repr{corrbnmcg} to deduce \reth{classvcmcg}, which up to a finite number of exceptions, yields the classification of the virtually cyclic subgroups of $\mcg$.

\begin{proof}[Proof of \reth{classvcmcg}.]
Let $n\geq 4$. The homomorphism $\phi$ of the short exact sequence~\reqref{mcg} satisfies the hypothesis of \repr{corrbnmcg} with $x=\ft$. \reth{main} and \repr{corrbnmcg} then imply the result, using the fact that if a finite subgroup $F$ of $B_{n}(\St)$ is isomorphic to $\Z_{q}$ (resp.\ $\dic{4m}, \quat, \tonestar, \oonestar,\istar$) then $\phi(F)$ is isomorphic to $\Z_{q/2}$ if $q$ is even and to $\Z_{q}$ if $q$ is odd (resp.\ is isomorphic to $\dih{2m}$, $\Z_{2}\oplus \Z_{2}, \an[4], \sn[4], \an[5]$). Note that the only cases where the conditions given in \redef{v1v2} on the order $q$ of $F$ differ from those on the order $q'$ of $\phi(F)$ given by \redef{v1v2mcg} is when $F$ is cyclic, and correspond to cases~(\ref{it:mainzq}) and~(\ref{it:mainzqt}) of these definitions. To see that one does indeed obtain the given conditions in parts~(\ref{it:mainzqmcg}) and~(\ref{it:mainzqtmcg}) of \redef{v1v2mcg}, suppose that $q$ satisfies the corresponding conditions given in parts~(\ref{it:mainzq}) and~(\ref{it:mainzqt}) of \redef{v1v2}. In particular, $q$ is a strict divisor of $2(n-i)$, and $q\neq n-i$ if $n-i$ is odd. If $q$ is even then $q'=q/2$, and $q'$ is a strict divisor of $n-i$. So suppose that $q$ is odd, in which case $q'=q$ and $q'$ divides $n-i$. Clearly, if $n-i$ is even then $q'\neq n-i$. On the other hand, if $n-i$ is odd then $q\neq n-i$. In both cases it follows once more that $q'$ is a strict divisor of $n-i$, which yields the condition on the order of the finite cyclic factor in parts~(\ref{it:mainzqmcg}) and~(\ref{it:mainzqtmcg}) of \redef{v1v2mcg}.
\end{proof}


One may ask a similar question to that of \resecglobal{realisation}{isoclasses} concerning the isomorphism classes of the amalgamated products that are realised as subgroups of $\mcg$. From \redef{v1v2mcg} and \reth{classvcmcg}, these subgroups are of the form:
\begin{enumerate}[(a)]
\item\label{it:type2mcga} $\Z_{2q}\bigast_{\Z_{q}} \Z_{2q}$, where $q$ divides $(n-i)/2$ for some $i\in\brak{0,1,2}$.
\item $\Z_{2q}\bigast_{\Z_{q}} \dih{2q}$, where $q\geq 2$ divides $(n-i)/2$ for some $i\in\brak{0,2}$.
\item $\dih{2q}\bigast_{\Z_{q}} \dih{2q}$, where $q\geq 2$ divides $n-i$ strictly for some $i\in\brak{0,2}$.
\item $\dih{2q}\bigast_{\dih{q}} \dih{2q}$, where $q\geq 4$ is even and divides $n-i$ for some $i\in\brak{0,2}$. 
\item\label{it:type2mcge} $\sn[4] \bigast_{\an[4]} \sn[4]$, where $n\equiv 0,2 \bmod{6}$. 
\end{enumerate}
A key element in the analysis of the isomorphism classes of the amalgamated products that are realised as subgroups of $B_{n}(\St)$ was the use of \relem{index2}. This may be generalised as follows to the groups that appear as factors in the above list.

\begin{lem}
Let $G'$ be a group isomorphic to $\Z_{2q}$, $q\geq 1$ (resp.\ to $\dih{2q}$, $q\geq 2$, to $\sn[4]$), let $G$ be a group isomorphic to $\Z_{4q}$, $q\geq 1$ (resp.\ to $\dic{4q}$, $q\geq 2$, to $\oonestar$), and let $\map{\phi}{G}[G']$ be the canonical homomorphism, where we identify $G'$ with the quotient of $G$ by its unique subgroup $K$ of order $2$. Let $H'$ be a subgroup of $G'$ of index $2$, non isomorphic to $\Z_{2}\oplus \Z_{2}$ if $G'\cong \dih{8}$. Then the homomorphism $\aut[H']{G'}\to \aut{H'}$ given by restriction is surjective. 
\end{lem}

\begin{proof}
Let $H=\phi^{-1}(H')$. Then $H$ is of index $2$ in $G$, and if $G\cong \quat[16]$ then $H\ncong \quat$. Let $\alpha'\in \aut{H'}$. We must show that there exists an automorphism of $G'$ that leaves $H'$ invariant, and whose restriction to $H'$ is equal to $\alpha'$. Note that:
\begin{enumerate}[(a)]
\item\label{it:authgi} if $G'\cong \Z_{2q}$, $q\geq 1$, then $G\cong \Z_{4q}$, $H'\cong \Z_{q}$ and $H\cong \Z_{2q}$.
\item\label{it:authgii} if $G'\cong \dih{2q}$, $q\geq 2$ then  $G\cong \dic{4q}$. If $q$ is odd then $H'\cong\Z_{q}$ and $H\cong \Z_{2q}$. If $q$ is even then $H'$ is isomorphic to $\Z_{q}$ or to $\dih{q}$, and $H$ is isomorphic to $\Z_{2q}$ or to $\dic{2q}$ respectively.
\item\label{it:authgiii} if $G\cong \oonestar$ then $G'\cong \sn[4]$, $H'\cong \an[4]$ and $H\cong \tonestar$.
\end{enumerate}
The kernel of $\map{\phi\left\lvert_{H}\right.}{H}[H']$ is that of $\phi$, equal to $K$. Since $K$ is characteristic in $G$ (resp.\ $H$), for each automorphism $f\in \aut{G}$ (resp.\ $f\in \aut{H}$), there exists a unique automorphism $f'\in\aut{G'}$ (resp.\ $f'\in \aut{H'}$) such that $\phi\circ f=f'\circ \phi$, and the correspondence $f\mapsto f'$ gives rise to a homomorphism $\map{\widetilde{\Phi}}{\aut{G}}[\aut{G'}]$ (resp.\ $\map{\Phi}{\aut{H}}[\aut{H'}]$) satisfying $\widetilde{\Phi}(f)\circ \phi=\phi\circ f$ (resp.\ $\Phi(f)\circ \phi=\phi\circ f$).

Let $r$ and $r'$ denote the restriction homomorphisms $\aut[H]{G}\to \aut{H}$ and $\aut[H']{G'}\to \aut{H'}$ respectively, and let $\iota$ and $\iota'$ denote the inclusions $\aut[H]{G}\to \aut{G}$ and $\aut[H']{G'}\to \aut{G'}$ respectively. Then we have the following commutative diagram:
\begin{equation}\label{eq:commautgh}
\begin{xy}*!C\xybox{%
\xymatrix{\aut{G} \ar[d]^{\widetilde{\Phi}} & \ar[l]_{\iota} \aut[H]{G} \ar[r]^{r} \ar[d]^{\widetilde{\Phi}\left\lvert_{\aut[H]{G}}\right.} & \aut{H}\ar[d]^{\Phi} \\ 
\aut{G'} & \ar[l]_{\iota'} \aut[H']{G'} \ar[r]^{r'} & \aut{H'}}}
\end{xy}
\end{equation}
Note that the restriction of $\widetilde{\Phi}$ to $\aut[H]{G}$ is well defined. Indeed, let $f\in \aut[H]{G}$, and let $h'\in H'$. Then there exists $h\in H$ such that $\phi(h)=h'$, and
\begin{equation*}
\Phi(f)(h')=\Phi(f)\circ\phi(h)=\phi\circ f(h)\in H, \quad\text{since $f(h)\in H$.}
\end{equation*}
We claim that for the groups $H,H'$ described in~(\ref{it:authgi})--(\ref{it:authgiii}) above, $\Phi$ is surjective. If $H\cong\Z_{2q}$, which covers  case~(\ref{it:authgi}) above and part of case~(\ref{it:authgii}), $\aut{H}\cong \Z_{2q}^{\times}$, $\aut{H'}\cong \Z_{q}^{\times}$, and if $\alpha'\in \aut{H'}$ is given by multiplication by $j$, where $1\leq j\leq q-1$, $\gcd{(j,q)}=1$, then $\Phi(\alpha)=\alpha'$, where $\alpha\in \aut{H}$ is given by multiplication by $j+\epsilon q$, where $\epsilon=0$ if $j$ is odd, and $\epsilon=1$ if $j$ is even. Let us now consider the remaining part of case~(\ref{it:authgii}) where $q$ is even, $H\cong\dic{2q}$ and $H'\cong \dih{q}$. Let $H$ admit the presentation
\begin{equation*}
H=\setangr{x,y}{x^{q/2}=y^{2},\; yxy^{-1}=x^{-1}},
\end{equation*}
and let $\overline{x}=\phi(x)$ and $\overline{y}=\phi(y)$, so that
\begin{equation*}
H'=\setangr{\overline{x}, \overline{y}}{\overline{x}^{q/2}=\overline{y}^{2}=1,\; \overline{y}\, \overline{x}\,\overline{y}^{-1}=\overline{x}^{-1}}.
\end{equation*}
Any automorphism $\alpha'$ of $H'$ is given by $\overline{x}\mapsto \overline{x}^{j}$, $\overline{y}\mapsto \overline{x}^{k}\,\overline{y}$, where $1\leq j\leq q/2-1$, $\gcd{(j,q/2)}=1$ and $0\leq k\leq q/2-1$. The presentation of $H$ implies that the map $\map{\alpha}{H}$ given by $x\mapsto x^{j+\epsilon q/2}$, $y\mapsto x^{k}y$, where $\epsilon=0$ if $j$ is odd, and $\epsilon=1$ if $j$ is even, is an automorphism. Further, $\Phi(\alpha)(\overline{x})= \phi(\alpha(x))=\overline{x}^{j}=\alpha'(\overline{x})$, and $\Phi(\alpha)(\overline{y})= \phi(\alpha(y))=\overline{x}^{k}\overline{y}=\alpha'(\overline{y})$, which proves the surjectivity of $\Phi$ in this case. Finally, in case~(\ref{it:authgiii}), the result is a consequence of~\cite[Theorem~3.3]{gg5}

It just remains to show that $r'$ is surjective. By the commutative diagram~\reqref{commautgh}, this follows from the surjectivity of $\Phi$, and that of $r$, which is a consequence of \relem{index2}.
\end{proof}


In principle, if we are given finite groups $H, G_1, G_2$, where $H$ is an index $2$ subgroup of both $G_{1}$ and $G_{2}$, there may be various non-isomorphic amalgamated products of the form $G_{1}\bigast_{H} G_{2}$. As for $B_{n}(\St)$, such a situation occurs exceptionally in $\mcg$, and we obtain a similar result to that of \repr{isoamalg} for the virtually cyclic subgroups of $\mcg$ of Type~II.
\begin{prop}\label{prop:isoamalgmcg}
Let $n\geq 4$ be even.
\begin{enumerate}[(a)]
\item\label{it:isoamalgmcga} Let $H_{1}',H_{2}'$ be subgroups of $\mcg$ that are both isomorphic to one of the amalgamated products given in (\ref{it:type2mcga})--(\ref{it:type2mcge}) 
above, with the exception of $\dih{8}\bigast_{\dih{4}} \dih{8}$. Then $H_{1}'\cong H_{2}'$.
\item\label{it:isoamalgmcgb} Let $H'$ be a subgroup of $\mcg$ that is isomorphic to an amalgamated product of the form $\dih{8}\bigast_{\dih{4}} \dih{8}$. Then $H'$ is isomorphic to exactly one of the following two groups:
\begin{equation}\label{eq:presK1prime}
K_{1}'=\setangr{x,y,a,b}{x^{4}=y^{2}=a^{4}=b^{2}=1,\, yxy^{-1}= x^{-1},\, bab^{-1}=a^{-1},\, x^{2}=a^{2},\, y=b},
\end{equation}
and
\begin{equation}\label{eq:presK2prime}
K_{2}'=\setangr{x,y,a,b}{x^{4}=y^{2}=a^{4}=b^{2}=1,\, yxy^{-1}= x^{-1},\, bab^{-1}=a^{-1},\, x^{2}=b,\, y=a^{2}b}.
\end{equation}
\end{enumerate}
\end{prop}

\begin{rem}
One may mimic the proof of \repr{isoamalg} to obtain an analogous result for the amalgamated products given in (\ref{it:type2mcga})--(\ref{it:type2mcge}) above, that is, abstractly there is a single isomorphism class, with the exception of $\dih{8}\bigast_{\dih{4}} \dih{8}$, for which there are two isomorphism classes, for which $K_{1}'$ and $K_{2}'$ are representatives. However, using \repr{isoamalg}, we shall give an alternative proof in the case that interests us, where the groups in question are realised as subgroups of $\mcg$. 
\end{rem}

\begin{proof}[Proof of \repr{isoamalgmcg}.]
Consider one of the amalgamated products given in the list (\ref{it:type2mcga})--(\ref{it:type2mcge}) above, and suppose that $n\geq 4$ is such that this amalgamated product is realised as a subgroup $G_{1}'\bigast_{F'} G_{2}'$ of $\mcg$, where $G_{1}',G_{2}'$ and $F'$ are finite subgroups of $\mcg$, and $[G_{i}'\colon F']=2$ for $i=1,2$. 
Taking $G=B_{n}(\St)$, $G'=\mcg$, $x=\ft$ and $p=\phi$ in the statement of \repr{vcmcg}, where $\phi$ is the homomorphism of \req{mcg}, we have that $\phi^{-1}(G_{1}')\bigast_{\phi^{-1}(F')} \phi^{-1}(G_{2}')$ is a subgroup of $B_{n}(\St)$ by part~(\ref{it:vcmcgb})(\ref{it:vcmcgbii}) of that proposition.  We claim that the number of isomorphism classes of subgroups of $B_{n}(\St)$ that are isomorphic to an amalgamated product of the form $\phi^{-1}(G_{1}')\bigast_{\phi^{-1}(F')} \phi^{-1}(G_{2}')$ (which are the amalgamated products~(\ref{it:mainIIab})--(\ref{it:mainIIeb}) that appear at the beginning of \resecglobal{realisation}{isoclasses}) is greater than or equal to the number of isomorphism classes of subgroups of $\mcg$ that are isomorphic to an amalgamated product of the form $G_{1}'\bigast_{F'} G_{2}'$. To prove the claim, let $H_{1}'$, $H_{2}'$ be subgroups of $\mcg$ that may be written in the form $G_{1}'\bigast_{F'} G_{2}'$, and for $i=1,2$, let $H_{i}=\phi^{-1}(H_{i}')$. From above, $H_{1}$ and $H_{2}$ are subgroups of $B_{n}(\St)$ that may be written in the form $\phi^{-1}(G_{1}')\bigast_{\phi^{-1}(F')} \phi^{-1}(G_{2}')$. If they are isomorphic then $H_{1}'=p(H_{1})$ and $H_{2}'=p(H_{2})$ are isomorphic by \repr{vcmcg}(\ref{it:vcmcgc}), which proves the claim. If $G_{1}'\bigast_{F'} G_{2}' \ncong \dih{8}\bigast_{\dih{4}} \dih{8}$ then $\phi^{-1}(G_{1}')\bigast_{\phi^{-1}(F')} \phi^{-1}(G_{2}')\ncong \quat[16] \bigast_{\quat} \quat[16]$, and combining the claim with \repr{isoamalg} implies that $\mcg$ possesses a single isomorphism class of subgroups that are isomorphic to amalgamated products of the form $G_{1}'\bigast_{F'} G_{2}'$, which proves part~(\ref{it:isoamalgmcga}) of the proposition. Similarly, if $G_{1}'\bigast_{F'} G_{2}' \cong \dih{8}\bigast_{\dih{4}} \dih{8}$ then $\phi^{-1}(G_{1}')\bigast_{\phi^{-1}(F')} \phi^{-1}(G_{2}')\cong \quat[16] \bigast_{\quat} \quat[16]$, and $\mcg$ possesses at most two isomorphism class of subgroups that are isomorphic to amalgamated products of the form $\dih{8}\bigast_{\dih{4}} \dih{8}$, and these isomorphism classes are represented by subgroups of $\mcg$ that are isomorphic to $K_{1}'$ and $K_{2}'$. To complete the proof of part~(\ref{it:isoamalgmcgb}) of the proposition, it thus suffices to show that $K_{1}'\ncong K_{2}'$. Taking $K_{1}'$ and $K_{2}'$ to be presented by equations~\reqref{presK1prime} and~\reqref{presK2prime} respectively, following the proof of \repr{isoamalg} from \req{presK1} onwards, and letting $N'$ be the infinite cyclic subgroup of $K_{1}'$ generated by $t=xa^{-1}$, we see that $N'$ is normal in $K_{1}'$,
\begin{equation*}
K_{1}'/N'=\setangr{a,b}{a^{4}=b^{2}=1,\; bab^{-1}=a^{-1}}\cong \dih{8},
\end{equation*}
and $K_{1}'\cong \ang{t}\rtimes \dih{8}$, where the action of $K_{1}'/N'$ on $N'$ is given by \req{zrtimesq16}, $G_{2}$ being in this case the subgroup $\ang{a,b}$ of $K_{1}'$. The rest of the proof of \repr{isoamalg} then goes through, where $\quat$ is replaced by $\Z_{2}\oplus \Z_{2}$, and the subgroups $L$ of $H$ are now of order $2$. We conclude that $K_{1}'\ncong K_{2}'$ as required.
\end{proof}

\pagebreak

We thus obtain the following result on the existence of subgroups of $\mcg$ isomorphic to $K_{1}'$ and $K_{2}'$.
\begin{prop}\label{prop:s4a4}
Let $n\geq 4$ be even.
\begin{enumerate}[(a)]
\item\label{it:K1prime} There exists a subgroup of $\mcg$ isomorphic to $K_{1}'$.
\item\label{it:K2prime} Suppose that either $n\equiv 0\bmod{4}$ or $n\equiv 10 \bmod{12}$ and $n\notin \brak{6,14,18,26,30,38}$. There exists a subgroup of $\mcg$ isomorphic to $K_{2}'$.
\end{enumerate}
\end{prop}

\begin{proof}
Let $n$ be even, and let $\phi$ be the homomorphism of \req{mcg}. If $n\geq 4$ (resp.\ $n\nequiv 4 \bmod{12}$ and $n\notin \brak{6,14,18,26,30,38}$) then \repr{existk1k2} and \rerem{existk1k2} imply that $B_{n}(\St)$ possesses a subgroup $H$ that is isomorphic to $K_{1}$ (resp.\ $K_{2}$). The presentations of $K_{1}$ and $K_{1}'$ (resp.\ $K_{2}$ and $K_{2}'$) given by equations~\reqref{presK1} and~\reqref{presK1prime} (resp.\ equations~\reqref{presK2} and~\reqref{presK2prime}) imply that $\phi(H)$ is isomorphic to $K_{1}'$ (resp.\ to $K_{2}'$).
%
\end{proof}

\begin{rem}
As in the case of $B_{n}(\St)$, we do not know whether $\mcg$ possesses a subgroup isomorphic to $K_{2}'$ if $n\in \brak{6,14,18,26,30,38}$.
%
\end{rem}

\appendix

\addcontentsline{toc}{chapter}{Appendix: The subgroups of the binary polyhedral groups}

\addtocontents{toc}{\protect\setcounter{tocdepth}{-1}}

\chapter*{Appendix: The subgroups of the\\ binary polyhedral
groups}\label{part:append}

In this Appendix, we derive the structure of the subgroups of the binary polyhedral groups $\tonestar,\oonestar,\istar$ that we refer to in the main body of the manuscript. More information on these groups may be found in~\cite{AM,Co,CM,Th}.

\begin{prop}\mbox{}\label{prop:maxsubgp}
\begin{enumerate}[(a)]
\item\label{it:subgpststar} The proper subgroups of the binary tetrahedral group $\tonestar$ are $\brak{e}$, $\Z_2$, $\Z_3$, $\Z_4$, $\Z_6$ and $\quat$. Its maximal subgroups are isomorphic to $\Z_6$ or $\quat$, its maximal cyclic subgroups are isomorphic to $\Z_4$ or $\Z_6$, and its non-trivial normal subgroups are isomorphic to $\Z_{2}$ or $\quat$.
\item\label{it:subgpsostar} The proper subgroups of the binary octahedral group $\oonestar$ are isomorphic to $\brak{e}$, $\Z_2$, $\Z_3$, $\Z_4$, $\Z_6$, $\Z_8$, $\quat$, $\dic{12}$, $\quat[16]$ or $\tonestar$. Its
maximal subgroups are isomorphic to $\dic{12}$, $\quat[16]$ or $\tonestar$, its maximal cyclic subgroups are isomorphic to $\Z_4$, $\Z_6$  or $\Z_8$, and its non-trivial normal subgroups are isomorphic to $\Z_{2}$, $\quat$ or $\tonestar$.
\item The proper subgroups of the binary icosahedral group $\istar$ are isomorphic to $\brak{e}$, $\Z_2$, $\Z_3$, $\Z_4$, $\Z_5$, $\Z_6$, $\quat$, $\Z_{10}$, $\dic{12}$, $\dic{20}$ or $\tonestar
$, its maximal subgroups are isomorphic to $\dic{12}$, $\dic{20}$ or $\tonestar$, its maximal cyclic subgroups are isomorphic to $\Z_4$, $\Z_6$ or $\Z_{10}$, and it has a unique non-trivial normal subgroup, isomorphic to $\Z_{2}$.
\end{enumerate}
\end{prop}

\begin{proof}
Recall first that if $G$ is a binary polyhedral group, it is periodic~\cite{AM} and has a unique element of order $2$ that generates $Z(G)$. By periodicity, the group $G$ satisfies the $p^2$–condition (if $p$ is prime and divides the order of $G$ then $G$ has no subgroup isomorphic to $\Z_{p}\times \Z_{p}$), which implies that every Sylow $p$–subgroup of $G$ is cyclic or generalised quaternion, as well as the $2p$–condition (each subgroup of order $2p$ is cyclic).
\begin{enumerate}[(a)]
\item Consider first the binary tetrahedral group $\tonestar$. It is isomorphic to $\quat \rtimes \Z_3$. Using the presentation given by \req{preststar}, one may check that $\tonestar\setminus \quat$ consists of the eight elements of
\begin{equation*}
\setl{S^{-j}X^{j}}{\text{$j\in \brak{-1,1}$ and $S\in \brak{1,P,Q,PQ}$}},
\end{equation*}
and of the eight elements of order $6$ which are obtained from those of order $3$ by multiplying by the unique (central) element $P^{2}$ of order $2$. The proper non-trivial subgroups of $\tonestar$ are isomorphic to $\Z_2$, $\Z_3$, $\Z_4$, $\Z_6$ and $\quat$. The fact that $\tonestar$ has a unique element of order $2$ rules out the existence of subgroups isomorphic to $\sn[3]$. Since $\quat$ is a Sylow $2$-subgroup of $\tonestar$, $\Z_8$ cannot be a subgroup of $\tonestar$. Further, since $\tonestar/Z(\tonestar)\cong \an[4]$, the quotient by $Z(\tonestar)$ of any order~$12$ subgroup of $\tonestar$ would be a subgroup of $\an[4]$ of order~$6$, which is impossible. Also, any copy of $\Z_{3}$ (resp.\ $\Z_{4}$) is contained in a copy of $\Z_{6}$ (resp.\ $\quat$). The maximal subgroups of $\tonestar$ are thus isomorphic to $\Z_6$ or $\quat$, and its maximal cyclic subgroups are isomorphic to $\Z_{4}$ or $\Z_6$. Among these possible subgroups, it is straightforward to check that the normal non-trivial subgroups are those isomorphic to $\Z_{2}$ or $\quat$.

\item Consider the binary octahedral group $\oonestar$, with presentation given by \req{presostar}. Recall from \relem{index2}(\ref{it:autostar}) that $\ang{P,Q,X}$ is the unique subgroup of $\oonestar$ isomorphic to $\tonestar$. The twenty-four elements of $\oonestar\setminus \tonestar$ are comprised of twelve elements of order $4$ and twelve of order $8$. Under the canonical projection onto $\oonestar/Z(\oonestar)\cong \sn[4]$, these elements are sent to the six transpositions and the six $4$-cycles of $\sn[4]$ respectively. The squares of the elements of order $8$ are the elements of $\tonestar$ of order $4$. Consequently, the elements of $\oonestar\setminus \tonestar$ of order $4$ generate maximal cyclic subgroups. 
Thus $\oonestar$ has three subgroups isomorphic to $\Z_{8}$. The Sylow $2$-subgroups are copies of $\quat[16]$, and since each copy of $\Z_{8}$ is contained in a copy of $\quat[16]$ and each copy of $\quat[16]$ contains a unique copy of $\Z_{8}$, it follows from Sylow's Theorems that $\oonestar$ possesses exactly three (maximal and non-normal) copies of $\quat[16]$, and that the subgroups of $\oonestar$ of order $8$ are isomorphic to $\Z_{8}$ or $\quat$.


It remains to determine the subgroups of order $12$. Under the projection onto the quotient $\oonestar/Z(\oonestar)$, such a subgroup would be sent to a subgroup of $\sn[4]$ of order $6$, so is the inverse image under this projection of a copy of $\sn[3]$, isomorphic to $\dic{12}$. It is not normal because the subgroups of $\sn[4]$ isomorphic to $\sn[3]$ are not normal. Further it cannot be a subgroup of $\ang{P,Q,X}$ since projection onto $\oonestar/Z(\oonestar)$ would imply that the image of $\ang{P,Q,X}$, which is isomorphic to $\an[4]$, would have a subgroup of order $6$, which is impossible. We thus obtain the isomorphism classes of the subgroups of $\oonestar$ given in the statement, as well as the isomorphism classes of the maximal and maximal cyclic subgroups. 
 
We now determine the normal subgroups of $\oonestar$. As we already mentioned,  the subgroups of $\oonestar$ isomorphic to $\dic{12}$ or $\quat[16]$ are not normal, and the fact that each of the three cyclic subgroups of order $8$ belongs to a single copy of $\quat[16]$ implies that these subgroups are not normal in $\oonestar$. Clearly $Z(\oonestar)\cong \Z_{2}$ and $\ang{P,Q,X}\cong \tonestar$ are normal in $\oonestar$. Since $\tonestar$ is normal in $\oonestar$ and possesses a unique copy $\ang{P,Q}$ of $\quat$, this copy of $\quat$ is normal in $\oonestar$. The subgroups isomorphic to $\Z_{3}$ or $\Z_{6}$ are not normal because they are contained in $\ang{P,Q,X}$ and are not normal there. The same is true for the subgroups isomorphic to $\Z_{4}$ and lying in $\ang{P,Q,X}$. Finally, under the canonical projection onto $\oonestar/Z(\oonestar)$, any subgroup of order $4$ generated by an element $\oonestar\setminus \tonestar$ is sent to subgroup of $\sn[4]$ generated by a transposition, so cannot be normal in $\oonestar$. This yields the list of isomorphism classes of normal subgroups of $\oonestar$.

\item Finally, consider the binary icosahedral group $\istar$ of order~$120$. It is well known that $\istar$ admits the presentation $\setangr{S,T}{(ST)^2=S^3=T^5}$, is isomorphic to the group $\operatorname{SL}_2(\F[5])$, and $\istar/Z(\istar)\cong \an[5]$. The group $\istar$ has thirty elements of order $4$ (which project to the fifteen elements of $\an[5]$ of order $2$), twenty elements each of order $3$ and $6$ (which project to the twenty $3$-cycles of $\an[5]$), and twenty-four elements each of order $5$ and $10$ (which project to the twenty-four $5$-cycles of $\an[5]$). Its proper subgroups of order less than or equal to $10$ are $\Z_2$, $\Z_3$, $\Z_4$, $\Z_5$, $\Z_6$, $\quat$ and $\Z_{10}$. The only difficulty here is the case of order $8$ subgroups: $\istar$ has no element of order $8$ since under the projection onto $\istar/Z(\istar)$, such an element would project onto an element of $\an[5]$ of order $4$, which is not possible. Since $\istar$ possesses a unique element of order $2$, the Sylow $2$-subgroups of $\istar$, which are of order $8$, are isomorphic to $\quat$. Any subgroup of order $15$ or $30$ (resp.\ $60$) would project to a subgroup of $\an[5]$ of order $15$ (resp.\ $30$), which is not possible either. Note that $\istar$ has no element of order $12$ (resp.\ $20$) since such an element would project to one of order $6$ (resp.\ $10$) in $\an[5]$. Since $\istar$ has a unique element of order $2$, any subgroup of order $12$ (resp.\ $20$) must thus be isomorphic to $\dic{12}$ (resp.\ $\dic{20}$) using the classification of the groups of these orders up to isomorphism.  Such a subgroup exists by taking the inverse image of the projection of any subgroup of $\an[5]$ isomorphic to $\dih{6}$ (resp.\ $\dih{10}$).


Any subgroup of $\istar$ of order~$24$ projects to a subgroup of $\an[5]$ of order~$12$, which must be a copy of $\an[4]$. Hence any subgroup of $\an[5]$ of order $12$, which is isomorphic to $\an[4]$, lifts to a subgroup of $\istar$ isomorphic to $\tonestar$. So any subgroup of $\istar$ of order $24$ is isomorphic to $\tonestar$, and such a subgroup exists. 

Let $G$ be a subgroup of $\istar$ of order $40$, and let $G'$ be its projection in $\istar/Z(\istar)$. Then $G'$ is of order $20$, and the Sylow $5$-subgroup $K$ of $G'$ is normal. Now $G'$ has no element of order $4$ since $\istar$ has no element of order $8$, so $G'/K\cong \Z_{2}\oplus \Z_{2}$. We thus have a short exact sequence:
\begin{equation*}
1 \to \Z_5 \to G' \to \Z_2 \oplus \Z_2 \to 1
\end{equation*}
which splits since the kernel and the quotient have coprime orders~\cite[Theorem~10.5]{McL}. Since $\operatorname{Aut}(\Z_5)\cong \Z_4$, the action of any non-trivial element of $\Z_2 \oplus \Z_2$ on $\Z_5$ must be multiplication by $-1$ (it could not be the identity, for otherwise $\an[5]$ would have an element of order $10$, which is impossible), but this is not compatible with the structure of $\Z_2 \oplus \Z_2$. Hence $\istar$ has no subgroup of order $40$. We thus obtain the list of subgroups of $\istar$ given in the statement. The cyclic subgroups of order $3$ and $5$ of $\istar$ are contained in the cyclic subgroups of order $6$ and $10$ respectively obtained by multiplying a generator by the central element of order $2$. Thus the maximal cyclic subgroups of $\istar$ are isomorphic to $\Z_{4}$, $\Z_{6}$ or $\Z_{10}$.

We now consider the maximal subgroups. Clearly, any subgroup of $\istar$ isomorphic to $\dic{12}$ or $\tonestar$ is maximal. Further, since $\tonestar$ has no subgroup of order $12$, any subgroup of $\istar$ isomorphic to $\dic{12}$ is also maximal. The subgroups of $\istar$ isomorphic to $\quat$ are its Sylow $2$-subgroups, so are conjugate, and since one of these subgroups is contained in a copy of $\tonestar$, the same is true for any such subgroup. Thus the subgroups of $\istar$ isomorphic to $\quat$ are not maximal. Replacing $\quat$ by $\Z_{3}$ (resp.\ $\quat$ by $\Z_{5}$ and $\tonestar$ by $\dic{20}$) and applying a similar argument shows that the subgroups of $\istar$ isomorphic to $\Z_{6}$ (resp.\ $\Z_{10}$) are not maximal either. This yields the list of the isomorphism classes of the maximal subgroups of $\istar$ given in the statement. Finally, since $\an[5]$ is simple, the only non-trivial normal subgroup of $\istar$ is its unique subgroup of order $2$.\qedhere
\end{enumerate}
\end{proof}


\backmatter

\addtocontents{toc}{\protect\setcounter{tocdepth}{0}}
\addcontentsline{toc}{chapter}{Bibliography}
\addtocontents{toc}{\protect\setcounter{tocdepth}{-1}}

\end{document}